\documentclass{amsart}
\usepackage{amsmath,amssymb,amsthm,bm,bbm}
\usepackage{graphicx,enumerate}
\usepackage{layout}
\usepackage[OT2,T1]{fontenc}
\usepackage{longtable}
\usepackage{hyperref}
\usepackage{url}

\DeclareSymbolFont{cyrletters}{OT2}{wncyr}{m}{n}
\DeclareMathSymbol{\Sha}{\mathalpha}{cyrletters}{"58}
\newcommand\scalemath[2]{\scalebox{#1}{\mbox{\ensuremath{\displaystyle #2}}}}

\theoremstyle{plain}
\newtheorem{theorem}{Theorem}[section]
\newtheorem*{theorem-nn}{Theorem}
\newtheorem{lemma}[theorem]{Lemma}
\newtheorem{proposition}[theorem]{Proposition}
\newtheorem*{proposition-nn}{Proposition}
\newtheorem{corollary}[theorem]{Corollary}
\newtheorem{problem}[theorem]{Problem}

\theoremstyle{definition}
\newtheorem{definition}[theorem]{Definition}
\newtheorem{example}[theorem]{Example}
\newtheorem{remark}[theorem]{Remark}
\newtheorem*{acknowledgments}{Acknowledgments}

\theoremstyle{remark}

\newcommand{\bZ}{\mathbbm{Z}}\newcommand{\bQ}{\mathbbm{Q}}
\newcommand{\bG}{\mathbbm{G}}
\newcommand{\bF}{\mathbbm{F}}\newcommand{\bP}{\mathbbm{P}}
\newcommand{\cC}{\mathcal{C}}\newcommand{\cD}{\mathcal{D}}
\newcommand{\cH}{\mathcal{H}}\newcommand{\cS}{\mathcal{S}}
\newcommand{\cT}{\mathcal{T}}
\newcommand{\GL}{{\rm GL}}\newcommand{\SL}{{\rm SL}}
\newcommand{\Syl}{{\rm Syl}}
\newcounter{sub}
{\begin{list}{(\arabic{sub})}{\usecounter{sub}%
\setlength{\leftmargin}{2em}}}{\end{list}}

\def\rank{\mbox{\rm rank }}

\def\End{\mbox{\rm End}}

\newcommand{\Id}{\mathrm{Id}}

\title{Birational classification for algebraic tori}

\author[A. Hoshi]{Akinari Hoshi}
\address{Department of Mathematics,
Niigata University, Niigata 950-2181, Japan}
\email{hoshi@math.sc.niigata-u.ac.jp}

\author[A. Yamasaki]{Aiichi Yamasaki}
\address{Department of Mathematics,
Kyoto University, Kyoto 606-8502, Japan}
\email{aiichi.yamasaki@gmail.com}

\thanks{
This work was partially supported by KAKENHI (24540019, 25400027, 16K05059, 19K03418, 20H00115, 20K03511, 24K00519, 24K06647).}


\begin{document}
\begin{abstract}
We give a stably birational classification 
for algebraic $k$-tori of dimensions $3$ and $4$ over a field $k$. 
Kunyavskii \cite{Kun90} 
proved that there exist $15$ not stably $k$-rational cases 
among $73$ cases of algebraic $k$-tori of dimension $3$. 
Hoshi and Yamasaki \cite{HY17} 
showed that there exist exactly $487$ (resp. $7$, resp. $216$) 
stably $k$-rational (resp. not stably but retract $k$-rational, 
resp. not retract $k$-rational) cases of 
algebraic $k$-tori of dimension $4$. 
First, we define the weak stably $k$-equivalence of algebraic $k$-tori 
and show that there exist $13$ (resp. $128$) 
weak stably $k$-equivalent classes of 
algebraic $k$-tori $T$ of dimension $3$ (resp. $4$) 
which are not stably $k$-rational 
by computing some cohomological stably birational invariants, 
e.g. the Brauer-Grothendieck group of $X$ 
where $X$ is a smooth $k$-compactification of $T$, 
provided by Kunyavskii, Skorobogatov and Tsfasman \cite{KST89}. 
We make a procedure to compute such stably birational invariants 
effectively 
and the computations are done by using the computer algebra system GAP. 
Second, we define the $p$-part of the flabby class $[\widehat{T}]^{fl}$ 
as a $\bZ_p[\Syl_p(G)]$-lattice and prove that 
they are faithful and indecomposable $\bZ_p[\Syl_p(G)]$-lattices 
unless it vanishes 
for $p=2$ (resp. $p=2,3$) in dimension $3$ (resp. $4$) 
via $p$-adic analysis. 
The $\bZ_p$-ranks of them are also given. 
Third, we give a necessary and sufficient condition for 
which two not stably $k$-rational algebraic $k$-tori $T$ and $T^\prime$ of 
dimensions $3$ (resp. $4$) are stably birationally $k$-equivalent 
in terms of the splitting fields 
and the weak stably $k$-equivalent classes of $T$ and $T^\prime$. 
In particular, the splitting fields of them should coincide 
if $\widehat{T}$ and $\widehat{T}^\prime$ are indecomposable. 
Fourth, for each $7$ cases of not stably but retract $k$-rational 
algebraic $k$-tori of dimension $4$, 
we find an algebraic $k$-torus $T^\prime$ of dimension $4$ 
which satisfies that $T\times_k T^\prime$ is stably $k$-rational. 
Finally, we give a criterion to 
determine whether two algebraic $k$-tori $T$ and $T^\prime$ 
of general dimensions are stably birationally $k$-equivalent 
when $T$ (resp. $T^\prime$) is stably birationally $k$-equivalent 
to some algebraic $k$-torus $T^{\prime\prime}$ of dimension up to $4$. 
\end{abstract}

\maketitle

\tableofcontents


%
\section{Introduction}\label{S1}

We first recall some relevant definitions of rationality of function fields. 
\begin{definition}
Let $K/k$ and $K^\prime/k$ be finitely generated field extensions of a base field $k$.\\ 
(1) $K$ is called {\it rational over $k$} 
(or {\it $k$-rational} for short) 
if $K$ is purely transcendental over $k$, 
i.e. $K$ is isomorphic to $k(x_1,\ldots,x_n)$, 
the rational function field over $k$ with $n$ variables $x_1,\ldots,x_n$ 
for some integer $n$;\\
(2) $K$ is called {\it stably $k$-rational} 
if $K(y_1,\ldots,y_m)$ is $k$-rational for some algebraically 
independent elements $y_1,\ldots,y_m$ over $K$;\\
(3) $K$ and $K^\prime$ are called 
{\it stably $k$-isomorphic} 
if $K(y_1,\ldots,y_m)\simeq K^\prime(z_1,\ldots,z_n)$ for some algebraically 
independent elements $y_1,\ldots,y_m$ over $K$ 
and $z_1,\ldots,z_n$ over $K^\prime$;\\
(4) When $k$ is an infinite field, $K$ is called {\it retract $k$-rational} 
if there exists a $k$-algebra domain $R\subset K$ such that 
(i) $K$ is the quotient field of $R$ and (ii) 
the identity map $1_R : R\rightarrow R$ 
factors through a localized polynomial ring over $k$, i.e. 
there exist a non-zero polynomial $f\in k[x_1,\ldots,x_n]$ 
and $k$-algebra homomorphisms 
$\varphi : R\rightarrow k[x_1,\ldots,x_n][1/f]$ and 
$\psi : k[x_1,\ldots,x_n][1/f]\rightarrow R$ satisfying 
$\psi\circ\varphi=1_R$ (cf. Saltman \cite[Definition 3.1]{Sal84});\\
(5) $K$ is called {\it $k$-unirational} 
if $k\subset K\subset k(x_1,\ldots,x_n)$ for some integer $n$. 
\end{definition}
It is not difficult to see that
\begin{center}
``$k$-rational'' $\Rightarrow$ ``stably $k$-rational'' $\Rightarrow$ 
``retract $k$-rational'' $\Rightarrow$ ``$k$-unirational''. 
\end{center}

Let $\overline{k}$ be a fixed separable closure of the base field $k$. 
Let $T$ be an algebraic $k$-torus, 
i.e. a group $k$-scheme with fiber product (base change) 
$T\times_k \overline{k}=
T\times_{{\rm Spec}\, k}\,{\rm Spec}\, \overline{k}
\simeq (\bG_{m,\overline{k}})^n$; 
$k$-form of the split torus $(\bG_{m,k})^n$. 
Let $\mathcal{G}={\rm Gal}(\overline{k}/k)$ be 
the absolute Galois group of $k$ 
and $\widehat{T}_{\overline{k}}={\rm Hom}(T,\bG_{m,\overline{k}})$ 
be the character module of $T$ which becomes 
$\mathcal{G}$-lattice, 
i.e. finitely generated $\bZ[\mathcal{G}]$-module 
which is $\bZ$-free as an abelian group. 

\begin{definition}
An algebraic $k$-torus $T$ is said to be {\it $k$-rational} 
(resp. {\it stably $k$-rational}, {\it retract $k$-rational}, 
{\it $k$-unirational}) 
if the function field $k(T)$ of $T$ over $k$ is $k$-rational 
(resp. stably $k$-rational, retract $k$-rational, $k$-unirational). 
Two algebraic $k$-tori $T$ and $T^\prime$ 
are said to be 
{\it birationally $k$-equivalent} 
(resp. {\it stably birationally $k$-equivalent}), 
denoted by $T\approx T^\prime$ 
(resp. $T\stackrel{\rm s.b.}{\approx} T^\prime$), 
if their function fields $k(T)$ and $k(T^\prime)$ are 
$k$-isomorphic (resp. stably $k$-isomorphic). 
\end{definition}

We see that 
$T$ is $k$-rational if and only if $T$ is birational 
to the projective space $\bP^n$ of dimension $n$ for some integer $n$. 
$T$ is stably $k$-rational if and only if $T\times \bP^m$ is $k$-rational 
for some integer $m$.  
$T$ is retract $k$-rational if there exists a dominant rational map 
$f: \bP^n\dashrightarrow T$ for some integer $n$ and 
a rational map $g: T\dashrightarrow \bP^n$ such that 
$f\circ g=\Id_T$ where $\Id_T$ is the identity map of $T$, 
i.e. there exists a dominant rational map $f: \bP^n\dashrightarrow T$ which 
admits a rational section. 
$T$ is $k$-unirational if and only if there exists a dominant rational map 
$f: \bP^n\dashrightarrow T$. 
For an equivalent definition in the language of algebraic geometry, 
see e.g. 
Manin \cite{Man74}, \cite{Man86}, 
Manin and Tsfasman \cite{MT86}, 
Colliot-Th\'{e}l\`{e}ne and Sansuc \cite[Section 1]{CTS07}, 
Beauville \cite{Bea16}, 
Merkurjev \cite[Section 3]{Mer17}. 

\begin{remark}\label{rem1.13}
Colliot-Th\'el\`ene and Sansuc {\cite[Proposition 7.4, page 181]{CTS87}} 
showed that the following conditions are also equivalent 
for algebraic $k$-tori $T$: 
(i) $T$ is retract $k$-rational; 
(ii) there exists an algebraic $k$-torus $T^\prime$ 
such that $T\times_k T^\prime$ is $k$-rational; 
(iii) there exists an integral $k$-variety $Y$ such that 
$T\times_k Y$ is $k$-rational. 
\end{remark}

Let $\mathcal{T}$ be the category of algebraic $k$-tori and 
$\mathcal{T}_n$ be the category of algebraic $k$-tori of dimension $n$. 
%
We consider the stably birational classification for algebraic $k$-tori, 
i.e. the structure of $\mathcal{T}/\stackrel{\rm s.b.}{\approx}$ 
(resp. $\mathcal{T}_n/\stackrel{\rm s.b.}{\approx}$). 
\begin{problem}[Stably birational classification for algebraic $k$-tori (resp. for algebraic $k$-tori of dimension $n$)]\label{prob1.3}
Determine the structure of $\mathcal{T}/\stackrel{\rm s.b.}{\approx}$ 
$($resp. $\mathcal{T}_n/\stackrel{\rm s.b.}{\approx}$$)$. 
In particular, 
for given two algebraic $k$-tori $T$ and $T^\prime$ 
$($resp. $T$ and $T^\prime$ of dimension $n$$)$ 
determine whether $T$ and $T^\prime$ are stably birationally $k$-equivalent. 
\end{problem}


Let $G$ be a finite group and 
$M$ be a $\bZ[G]$-lattice (or just $G$-lattice for short), 
i.e. finitely generated $\bZ[G]$-module which is $\bZ$-free 
as an abelian group. 
There exists the minimal finite Galois extension $L/k$ 
with Galois group $G={\rm Gal}(L/k)$ such that 
$T$ splits over $L$: $T\times_k L\simeq (\bG_{m,L})^n$. 
The category of $G$-lattices is anti-equivalent to the category of 
algebraic $k$-tori which split over $L$ 
(see Ono \cite[Section 1.2]{Ono61}, 
Voskresenskii \cite[page 27, Example 6]{Vos98}, 
Knus, Merkurjev, Rost and Tignol \cite[Proposition 20.17]{KMRT98}). 
Indeed, if $T$ is an algebraic $k$-torus, then the character module 
$\widehat{T}={\rm Hom}(T,\bG_{m,L})$ of $T$ becomes a $G$-lattice. 
Conversely, for a $G$-lattice $M$, there exists an algebraic $k$-torus 
$T={\rm Spec}(L[M]^G)$ which splits over $L$ such that 
$\widehat{T}\simeq M$ as $G$-lattices. 
The invariant field $L(M)^G$ may be identified with the function field 
$k(T)$ of $T$ 
and hence it is $k$-unirational (see Endo and Miyata \cite{EM75}, 
\cite[page 40, Example 21]{Vos98}).
Isomorphism classes of $k$-tori of dimension $n$ 
correspond bijectively to the elements of 
the set $H^1(\mathcal{G},\GL(n,\bZ))$ via Galois descent 
where $\mathcal{G}={\rm Gal}(\overline{k}/k)$ because 
${\rm Aut}((\bG_{m,\overline{k}})^n)=\GL(n,\bZ)$ 
where $\GL(n,\bZ)$ is the general linear group of degree $n$ over $\bZ$. 
The algebraic $k$-torus $T$ of dimension $n$ 
is uniquely determined by the integral representation 
$h : \mathcal{G}\rightarrow \GL(n,\bZ)$ up to conjugacy, 
and $G=h(\mathcal{G})$ is a finite subgroup of $\GL(n,\bZ)$ 
(see \cite[page 57, Section 4.9]{Vos98}). 
The minimal splitting field $L$ of $T$ corresponds to 
$\mathcal{H}$ where 
$\rho: \mathcal{G}\rightarrow G$ and $\mathcal{H}={\rm Ker}(\rho)$ 
with $G\simeq \mathcal{G}/\mathcal{H}$. \\

We recall some basic tools on lattice arguments 
in order to study the 
(stably, retract) rationality problem of algebraic $k$-tori 
(see Manin \cite{Man74}, \cite[Appendix]{Man86}, 
Colliot-Th\'el\`ene and Sansuc \cite[Section 1]{CTS77}, 
Swan \cite[Section 8]{Swa83}, \cite{Swa10}, 
Manin and Tsfasman \cite[Section 4]{MT86}, 
Kunyavskii, Skorobogatov, Tsfasman \cite[Section 2]{KST89}, 
Voskresenskii \cite[Chapter 2]{Vos98}, 
Cortella and Kunyavskii \cite{CK00}, 
Lorenz \cite[Chapter 2]{Lor05}, 
Colliot-Th\'el\`ene \cite[Section 5]{CT07}). 


\begin{definition}
Let $\mathcal{G}$ be a profinite group and $M$ be a (continuous) 
$\mathcal{G}$-lattice.\\ 
(1) $M$ is called {\it permutation} $\mathcal{G}$-lattice 
if $M$ has a $\bZ$-basis
permuted by $\mathcal{G}$, i.e. $M\simeq \oplus_{1\leq i\leq m}\bZ[\mathcal{G}/\mathcal{H}_i]$ 
for some open subgroups $\mathcal{H}_1,\ldots,\mathcal{H}_m$ of $\mathcal{G}$.\\
(2) $M$ is called {\it stably permutation} 
$\mathcal{G}$-lattice if $M\oplus P\simeq P^\prime$ 
for some permutation $\mathcal{G}$-lattices $P$ and $P^\prime$.\\
(3) $M$ is called {\it invertible} 
if it is a direct summand of a permutation $\mathcal{G}$-lattice, 
i.e. $P\simeq M\oplus M^\prime$ for some permutation $\mathcal{G}$-lattice 
$P$ and a $\mathcal{G}$-lattice $M^\prime$.\\ 
(4) $M$ is called {\it coflabby} if $H^1(\mathcal{H},M)=0$
for any closed subgroup $\mathcal{H}$ of $\mathcal{G}$.\\ 
(5) $M$ is called {\it flabby} if $\widehat{H}^{-1}(\mathcal{H},M)=0$ 
for any closed subgroup $\mathcal{H}$ of $\mathcal{G}$ 
where $\widehat{H}$ is the Tate cohomology.
\end{definition}

It is not difficult to see 
\begin{align*}
\textrm{``}M\textrm{\ is\ permutation''}
\Rightarrow
\textrm{``}M\textrm{\ is\ stably\ permutation''}
\Rightarrow 
\textrm{``}M\textrm{\ is\ invertible''}
\Rightarrow
\textrm{``}M\textrm{\ is\ flabby\ and\ coflabby''}.
\end{align*}

Note that the above implications in each step cannot be reversed 
(see e.g. Hoshi and Yamasaki \cite[Chapter 1]{HY17}).\\

Let $X$ be a smooth $k$-compactification of $T$, 
i.e. smooth projective $k$-variety $X$ 
containing $T$ as a dense open subvariety, 
and $\overline{X}=X\times_k\overline{k}$ 
where $\overline{k}$ is a fixed separable closure of the base field $k$. 
There exists such a smooth $k$-compactification of an algebraic $k$-torus $T$ 
over any field $k$ (due to Hironaka \cite{Hir64} for ${\rm char}\, k=0$, 
see Colliot-Th\'{e}l\`{e}ne, Harari and Skorobogatov 
\cite[Corollaire 1]{CTHS05} for any field $k$). 

\begin{theorem}[{Voskresenskii \cite[Section 4, page 1213]{Vos69}, \cite[Section 3, page 7]{Vos70}, see also \cite{Vos74}, \cite[Section 4.6]{Vos98}, Kunyavskii \cite[Theorem 1.9]{Kun07} and Colliot-Th\'el\`ene \cite[Theorem 5.1, page 19]{CT07} for any field $k$}]\label{thVos69}
Let $k$ be a field 
and $\mathcal{G}={\rm Gal}(\overline{k}/k)$. 
Let $T$ be an algebraic $k$-torus, 
$X$ be a smooth $k$-compactification of $T$ 
and $\overline{X}=X\times_k\overline{k}$. 
Then there exists an exact sequence of $\mathcal{G}$-lattices 
\begin{align*}
0\to \widehat{T}_{\overline{k}}\to \widehat{Q}\to {\rm Pic}\,\overline{X}\to 0
\end{align*}
where 
$\widehat{Q}$ is permutation 
and ${\rm Pic}\ \overline{X}$ is flabby. 
\end{theorem}
%

\begin{theorem}[{Voskresenskii \cite[Theorem 1]{Vos70}, \cite[Theorem 3]{Vos73}, \cite[Chapter 2]{Vos98}}]\label{thVos70} 
Let $k$ be a field and $\mathcal{G}={\rm Gal}(\overline{k}/k)$.
Let $T$ 
be an algebraic $k$-torus, 
$X$ be a smooth $k$-compactification of $T$ 
and $\overline{X}=X\times_k\overline{k}$.\\
{\rm (i)} $T$ is stably $k$-rational if and only if 
$[{\rm Pic}\,\overline{X}]=0$ as a $\mathcal{G}$-lattice.\\
{\rm (ii)} $T$ and $T^\prime$ are stably birationally $k$-equivalent 
if and only if $[{\rm Pic}\,\overline{X}]=[{\rm Pic}\,\overline{X^\prime}]$ 
as $\mathcal{G}$-lattices.
\end{theorem}

\begin{definition}[{see Endo and Miyata \cite[Section 1]{EM75}, Voskresenskii \cite[Section 4.7]{Vos98}}]
Let $\mathcal{G}$ be a profinite group. 
Let $\cC(\mathcal{G})$ be the category of all $\mathcal{G}$-lattices. 
Let $\cS(\mathcal{G})$ be the full subcategory of $\cC(\mathcal{G})$ of all permutation 
$\mathcal{G}$-lattices and $\cD(\mathcal{G})$ be the full subcategory of $\cC(\mathcal{G})$ 
of all invertible $\mathcal{G}$-lattices. 
Let 
\begin{align*}
\cH^i(\mathcal{G})=\{M\in \cC(\mathcal{G})\mid \widehat{H}^i(\mathcal{H},M)=0\ {\rm for\ any\ closed\ subgroup}\ \mathcal{H}\leq \mathcal{G}\}\ (i=\pm 1)
\end{align*}
be the class of ``$\widehat{H}^i$-vanish'' $\mathcal{G}$-lattices 
where $\widehat{H}^i$ is the Tate cohomology. 
Then one has the inclusions 
$\cS(\mathcal{G})\subset \cD(\mathcal{G})\subset \cH^i(\mathcal{G})\subset \cC(\mathcal{G})$ $(i=\pm 1)$. 
\end{definition}

\begin{definition}[{see Manin \cite[Appendix, page 285]{Man86}, Voskresenskii \cite[Section 4.7]{Vos98}}]
Let $\mathcal{G}$ be a profinite group. 
We say that two $\mathcal{G}$-lattices $M_1$ and $M_2$ are {\it similar} 
if there exist permutation $\mathcal{G}$-lattices $P_1$ and $P_2$ such that 
$M_1\oplus P_1\simeq M_2\oplus P_2$. 
We denote the set of similarity classes $\cC(\mathcal{G})/\cS(\mathcal{G})$ by $\cT(\mathcal{G})$ 
and the similarity class of $M$ by $[M]$. 
$\cT(\mathcal{G})$ becomes a commutative monoid 
with respect to the sum $[M_1]+[M_2]:=[M_1\oplus M_2]$ 
and the zero $0=[P]$ where $P\in \cS(\mathcal{G})$. 
\end{definition}

\begin{definition}[{see Endo and Miyata \cite[Lemma 1.1]{EM75}, Colliot-Th\'el\`ene and Sansuc \cite[Lemme 3]{CTS77}, Manin \cite[Appendix, page 286]{Man86}}]\label{defF}
Let $\mathcal{G}$ be a profinite group. 
For a $\mathcal{G}$-lattice $M$, there exists a short exact sequence of 
$\mathcal{G}$-lattices
$0 \rightarrow M \rightarrow P \rightarrow F \rightarrow 0$
where $P$ is permutation and $F$ is flabby which is called a 
{\it flabby resolution} of $M$. 
The similarity class $[F]\in \cT(\mathcal{G})$ 
of $F$ is uniquely determined and is called {\it the flabby class} of $M$. 
We denote the flabby class $[F]$ of $M$ by $[M]^{fl}$. 
We say that $[M]^{fl}$ is invertible if $[M]^{fl}=[E]$ for some 
invertible $\mathcal{G}$-lattice $E$. 
\end{definition}
Note that (i) if $M$ is (stably) permutation (resp. invertible), 
then $[M]^{fl}=0$ (resp. $[M]^{fl}$ is invertible) 
and (ii) if $[M]^{fl}=0$ (resp. $[M]^{fl}$ is invertible) 
as a $\mathcal{G}$-lattice, 
then $[M]^{fl}=0$ (resp. $[M]^{fl}$ is invertible) as an $\mathcal{H}$-lattice 
for any closed subgroup $\mathcal{H}\leq \mathcal{G}$ (see Hoshi and Yamasaki 
\cite[Lemma 2.17]{HY17}).\\ 

{\it From $\mathcal{G}$-lattice to $G$-lattice.}\\

Let $L$ be the minimal splitting field of 
$T$ with Galois group $G={\rm Gal}(L/k)\simeq \mathcal{G}/\mathcal{H}$ 
where $\mathcal{G}={\rm Gal}(\overline{k}/k)$ 
and $\mathcal{H}={\rm Gal}(\overline{k}/L)$. 
By Theorem \ref{thVos69}, 
we obtain a flabby resolution of $\widehat{T}={\rm Hom}(T,\bG_{m,L})$: 
\begin{align*}
0\to \widehat{T}\to \widehat{Q}\to {\rm Pic}\, X_L\to 0
\end{align*}
as $G$-lattices with $[\widehat{T}]^{fl}=[{\rm Pic}\ X_L]$ 
where $\widehat{Q}$ is permutation 
and ${\rm Pic}\ X_L$ is flabby (see also \cite[Section 1]{Vos74}). 
By the inflation-restriction exact sequence 
$0\to H^1(G,{\rm Pic}\, X_L)\xrightarrow{\rm inf} 
H^1(k,{\rm Pic}\,\overline{X})\xrightarrow{\rm res} 
H^1(L,{\rm Pic}\,\overline{X})$, 
we get 
${\rm inf}: H^1(G,{\rm Pic}\, X_L)\xrightarrow{\sim} 
H^1(k,{\rm Pic}\,\overline{X})$ 
because $H^1(L,{\rm Pic}\,\overline{X})=0$. 

It follows from Theorem \ref{thVos70} (ii) that\\~\\
${\rm (ii)}^\prime$ 
$T$ and $T^\prime$ are stably birationally $k$-equivalent 
if and only if 
$[{\rm Pic}\, X_{\widetilde{L}}]=[{\rm Pic}\, X^\prime_{\widetilde{L}}]$ 
as $\widetilde{H}$-lattices where 
$\widetilde{L}=LL^\prime$, 
$L$ $($resp. $L^\prime$$)$ 
is the minimal splitting field of $T$ $($resp. $T^\prime$$)$ 
and $\widetilde{H}={\rm Gal}(\widetilde{L}/k)$ 
(cf. \cite[Proposition 1]{Vos74}).\\

The group $\widetilde{H}$ becomes a 
{\it subdirect product} of $G={\rm Gal}(L/k)$ and 
$G^\prime={\rm Gal}(L^\prime/k)$, i.e. 
a subgroup $\widetilde{H}$ of $G\times G^\prime$ 
with surjections $\varphi_1:\widetilde{H}\rightarrow G$ and 
$\varphi_2:\widetilde{H}\rightarrow G^\prime$. 
This observation yields a concept of ``{\it weak stably $k$-equivalence}''.  
\begin{definition}\label{d1.10}
Let $L/k$ (resp. $L^\prime/k$) be a finite Galois extension of fields 
with Galois group $G={\rm Gal}(L/k)$ 
(resp. $G^\prime={\rm Gal}(L^\prime/k))$. 
Let $M$ (resp. $M^\prime$) be a $G$-lattice (resp. $G^\prime$-lattice).\\
{\rm (i)} Two flabby classes $[M]^{fl}$ and $[M^\prime]^{fl}$ are 
said to be {\it weak stably $k$-equivalent}, 
denoted by $[M]^{fl}\sim [M^\prime]^{fl}$, if 
there exists a subdirect product $\widetilde{H}\leq G\times G^\prime$ 
of $G$ and $G^\prime$ 
with surjections 
$\varphi_1:\widetilde{H}\rightarrow G$ and 
$\varphi_2:\widetilde{H}\rightarrow G^\prime$ 
such that $[M]^{fl}=[M^\prime]^{fl}$ 
as $\widetilde{H}$-lattices 
where $\widetilde{H}$ acts on $M$ (resp. $M^\prime)$ 
through the surjection $\varphi_1$ (resp. $\varphi_2$).\\
{\rm (ii)} Two algebraic $k$-tori $T$ and $T^\prime$ are 
said to be {\it weak stably birationally $k$-equivalent}, 
denoted by $T\stackrel{\rm s.b.}{\sim} T^\prime$, 
if $[\widehat{T}]^{fl}$ and $[\widehat{T}^\prime]^{fl}$ are 
weak stably $k$-equivalent.\\
{\rm (iii)} A weak stably $k$-equivalence for flabby classes 
$[\widehat{T}]^{fl}$ 
(or a weak stably birationally $k$-equivalence 
for algebraic $k$-tori $T$) indeed yields 
an equivalence relation (see Remark \ref{r1.11} (2) below) 
and we call this equivalent class {\it the weak stably $k$-equivalent 
class} of $[\widehat{T}]^{fl}$ (or {\it the weak stably birationally 
$k$-equivalent class} of $T$) 
denoted by ${\rm WSEC}_r$ $(r\geq 0)$ 
with the stably $k$-rational class ${\rm WSEC}_0$.
\end{definition}
\begin{remark}\label{r1.11}
(1) $T\stackrel{\rm s.b.}{\approx} T^\prime$ $\Rightarrow$ 
$T\stackrel{\rm s.b.}{\sim} T^\prime$ 
(see the paragraph before Definition \ref{d1.10}).\\
(2) 
Indeed, the weak stably $k$-equivalence becomes 
equivalence relation: 
(i) Reflexivity: $[M]^{fl}\sim [M]^{fl}$ follows 
if we take a subdirect product 
$\widetilde{H}=\{(g,g)\mid g\in G\}\leq G\times G$ of $G$ and $G$ itself. 
(ii) Symmetry is obvious. 
(iii) Transitivity: 
If $[M]^{fl}\sim [M^\prime]^{fl}$ and 
$[M^\prime]^{fl}\sim [M^{\prime\prime}]^{fl}$, 
then $[M]^{fl}=[M^\prime]^{fl}$ as $\widetilde{H}$-lattices 
and $[M^\prime]^{fl}=[M^{\prime\prime}]^{fl}$ 
as $\widetilde{H}^\prime$-lattices 
where $\widetilde{H}\leq G\times G^\prime$ is 
a subdirect product of $G$ and $G^\prime$ 
with surjections 
$\varphi_1:\widetilde{H}\rightarrow G$ and 
$\varphi_2:\widetilde{H}\rightarrow G^\prime$ 
and  
$\widetilde{H}^\prime\leq G^\prime\times G^{\prime\prime}$ is 
a subdirect product of $G^\prime$ and $G^{\prime\prime}$ 
with surjections 
$\varphi^\prime_1:\widetilde{H}\rightarrow G^\prime$ and 
$\varphi^\prime_2:\widetilde{H}\rightarrow G^{\prime\prime}$. 
We may take a subdirect product 
$\widetilde{S}:=\{(h,h^\prime)\in\widetilde{H}\times\widetilde{H}^\prime\mid 
\varphi_2(h)=\varphi^\prime_1(h^\prime)\}\leq \widetilde{H}\times\widetilde{H}^\prime$ of $\widetilde{H}$ and $\widetilde{H}^\prime$ 
with surjections 
$\theta_1: \widetilde{S}\rightarrow \widetilde{H}, (h,h^\prime)\mapsto h$ and 
$\theta_2: \widetilde{S}\rightarrow \widetilde{H}^\prime, (h,h^\prime)\mapsto h^\prime$. 
Then 
$[M]^{fl}=[M^{\prime\prime}]^{fl}$ as $\widetilde{S}$-lattices. 
We consider a homomorphism 
$\Phi: \widetilde{S}\rightarrow G\times G^{\prime\prime}, 
s\mapsto (\varphi_1(\theta_1(s)),\varphi_2^\prime(\theta_2(s)))$ 
and the image $\widetilde{H}^{\prime\prime}:={\rm Im}(\Phi)$. 
Then $\widetilde{H}^{\prime\prime}\leq G\times G^{\prime\prime}$ 
is a subdirect product of $G$ and $G^{\prime\prime}$ 
and $[M]^{fl}=[M^{\prime\prime}]^{fl}$ as 
$\widetilde{H}^{\prime\prime}$-lattices because 
$\widetilde{H}^{\prime\prime}\simeq \widetilde{S}/{\rm Ker}(\Phi)$
(see Colliot-Th\'el\`ene and Sansuc \cite[Lemme 2 (i), (ix)]{CTS77}). 
This implies that $[M]^{fl}\sim [M^{\prime\prime}]^{fl}$. 
\end{remark} 

We can obtain the set of all subdirect products $\widetilde{H}$ 
of $G$, $G^\prime$ as follows (see also the proof of Theorem \ref{th8.1} (1)): 
\begin{lemma}\label{lem1.13} 
Let $G$ and $G^\prime$ be finite groups. 
For a subdirect product $\widetilde{H}\leq G\times G^\prime$ of $G$, $G^\prime$ 
with surjections $\varphi_1: \widetilde{H}\to G$, $\varphi_2: \widetilde{H}\to G^\prime$, define 
\begin{align*}
N_1:=\varphi_1({\rm Ker}(\varphi_2))\lhd G,\ 
N_2:=\varphi_2({\rm Ker}(\varphi_1))\lhd G^\prime,\ 
\pi_1:G\rightarrow G/N_1,\ 
\pi_2:G^\prime\rightarrow G^\prime/N_2.
\end{align*}
Then we have 
$\overline{\varphi}_1=\pi_1\circ\varphi_1: \widetilde{H}\rightarrow G/N_1$, 
$\overline{\varphi}_2=\pi_2\circ\varphi_2: \widetilde{H}\rightarrow G^\prime/N_2$, 
and $\overline{\varphi}=(\overline{\varphi}_2)(\overline{\varphi}_1)^{-1}:
G/N_1\xrightarrow{\sim}G^\prime/N_2$. 
Conversely, 
we find that a subdirect product 
$\widetilde{H}\leq G\times G^\prime$ of $G$, $G^\prime$ 
with surjections $\varphi_1: \widetilde{H}\to G$, $\varphi_2: \widetilde{H}\to G^\prime$ 
is given by 
\begin{align*}
\widetilde{H}=\{(g_1,g_2)\in G\times G^\prime\mid \overline{\varphi}(\pi_1(g_1))=\pi_2(g_2), 
\overline{\varphi}=(\overline{\varphi}_2)(\overline{\varphi}_1)^{-1}:
G/N_1\xrightarrow{\sim}G^\prime/N_2\}
\end{align*}
and hence there exists a one-to-one correspondence 
between the set of all subdirect products $\widetilde{H}$ of $G$, $G^\prime$ 
with surjections $\varphi_1: \widetilde{H}\to G$, $\varphi_2: \widetilde{H}\to G^\prime$ and 
\begin{align*}
\{(N_1,N_2,\overline{\varphi})\mid N_1\lhd G, N_2\lhd G^\prime, 
\overline{\varphi}=(\overline{\varphi}_2)(\overline{\varphi}_1)^{-1}:
G/N_1\xrightarrow{\sim}G^\prime/N_2\}. 
\end{align*}
\end{lemma}

%
Let $L/k$ be a finite Galois extension of fields with Galois group 
$G={\rm Gal}(L/k)$.  
Let $M=\bigoplus_{1\leq i\leq n}\bZ\cdot u_i$ be a $G$-lattice with 
a $\bZ$-basis $\{u_1,\ldots,u_n\}$, i.e. finitely generated $\bZ[G]$-module 
which is $\bZ$-free as an abelian group. 
Let $G$ act on the rational function field $L(x_1,\ldots,x_n)$ 
over $L$ with $n$ variables $x_1,\ldots,x_n$ by 
\begin{align}
x_i^{\sigma}=\prod_{j=1}^n x_j^{a_{i,j}},\quad 1\leq i\leq n\label{acts}
\end{align}
for any $\sigma\in G$, when $u_i^{\sigma}=\sum_{j=1}^n a_{i,j} u_j$, 
$a_{i,j}\in\bZ$. 
The field $L(x_1,\ldots,x_n)$ with this action of $G$ will be denoted by 
$L(M)$. 
The invariant field $L(M)^G$ of $L(M)$ under the action of $G$ 
as in (\ref{acts}) 
may be identified with the function field $k(T)$ 
of the algebraic $k$-torus $T$ over $k$ 
where $M=\widehat{T}$ (see Endo and Miyata \cite{EM75}).\\

The flabby class $[M]^{fl}=[\widehat{T}]^{fl}$ 
plays a crucial role in the rationality problem 
for $L(M)^G\simeq k(T)$ 
as follows (see also 
Theorem \ref{thVos70}, 
Colliot-Th\'el\`ene and Sansuc \cite[Section 2]{CTS77}, 
Voskresenskii \cite[Section 4.6]{Vos98}, 
Kunyavskii \cite[Theorem 1.7]{Kun07}, 
Colliot-Th\'el\`ene \cite[Theorem 5.4]{CT07}): 
%
%
\begin{theorem}
\label{thEM73}
Let $L/k$ be a finite Galois extension with Galois group $G={\rm Gal}(L/k)$ 
and $M$ and $M^\prime$ be $G$-lattices. 
Let $T$ and $T^\prime$ be algebraic $k$-tori with $\widehat{T}\simeq M$ 
and $\widehat{T}^\prime\simeq M^\prime$, 
i.e. $L(M)^G\simeq k(T)$ and $L(M^\prime)^G\simeq k(T^\prime)$.\\
{\rm (i)} $(${\rm Endo and Miyata} \cite[Theorem 1.6]{EM73}$)$ 
$[M]^{fl}=0$ if and only if $k(T)$ is stably $k$-rational.\\
{\rm (ii)} $(${\rm Voskresenskii} \cite[Theorem 2]{Vos74}$)$ 
$[M]^{fl}=[M^\prime]^{fl}$ if and only if $k(T)$ and $k(T^\prime)$ 
are stably $k$-isomorphic.\\
{\rm (iii)} $(${\rm Saltman} \cite[Theorem 3.14]{Sal84}$)$ 
$[M]^{fl}$ is invertible if and only if $k(T)$ is 
retract $k$-rational.
\end{theorem}

%
Before stating main theorems, we recall some known results 
about the (stably/retract) rationality of $k$-tori of small dimensions.

For dimension $1$, 
all the $1$-dimensional algebraic $k$-tori $T$, 
i.e. the trivial torus $\bG_{m,k}$ 
and the norm one torus $R_{L/k}^{(1)}(\bG_{m,L})$ of $L/k$ 
with $[L:k]=2$, are $k$-rational. 
\begin{theorem}[{Voskresenskii \cite{Vos67}}]\label{thVo}
All the $2$-dimensional algebraic $k$-tori $T$ are $k$-rational. 
In particular, for any finite subgroup $G\leq \GL(2,\bZ)$, 
$L(x_1,x_2)^G\simeq k(T)$ is $k$-rational. 
\end{theorem}
Note that $T$ is $k$-rational $\Rightarrow$ 
$T$ is stably $k$-rational $\Rightarrow$ $T$ is retract $k$-rational 
$\Rightarrow$ $H^1(k,{\rm Pic}\,\overline{X})=0$ 
and if $k$ is a global field, then 
$H^1(k,{\rm Pic}\,\overline{X})=0\Rightarrow A(T)=\Sha(T)=0$ 
(see Voskresenskii \cite[Theorem 5, page 1213]{Vos69}, 
Manin \cite[Section 30]{Man74}, Manin and Tsfasman \cite{MT86} 
and also Hoshi, Kanai and Yamasaki \cite[Section 1]{HKY22}). 

Let $G\simeq {\rm Gal}(L/k)$ be a finite group and 
$M\simeq M_1\oplus M_2$ be a decomposable $G$-lattice. 
Let $N_i=\{\sigma\in G\mid v^\sigma=v\ {\rm for\ any}\ v\in M_i\}$ 
be the kernel of the action of $G$ on $M_i$ $(i=1,2)$. 
Then $L(M)^G$ is the function field of an algebraic torus $T$ 
with $M=\widehat{T}$ 
and is $k$-isomorphic to the free composite of 
$L(M_1)^G$ and $L(M_2)^G$ over $k$ where 
$L(M_i)^G=(L^{N_i})(M_i^{N_i})^{G/N_i}$ 
is the function field of some torus $T_i$ ($i=1,2$) with 
$M_i=\widehat{T}_i$, 
$T=T_1\times_k T_2$ 
and $M_i$ may be regarded as a $G/N_i$-lattice 
(see Theorem \ref{th1.24} for more detailed explanations). 
There exist $13$ non-conjugate finite subgroups of $\GL(2,\bZ)$ 
and $4$ (resp. $9$) of them are decomposable (resp. indecomposable). 
\begin{lemma}\label{lemp1}
Let $G$ be a finite group and
$M \simeq M_1 \oplus M_2$ be a decomposable $G$-lattice with 
the flabby class $\rho_G(M)=[M]^{fl}$. 
Let $N_1$ be a normal subgroup of $G$ which acts on $M_1$ trivially. 
The $G$-lattice $M_1$ may be regarded as a $G/N_1$-lattice with the 
flabby class $\rho_{G/N_1}(M_1)$ as a $G/N_1$-lattice. 
Then\\
{\rm (i)} $\rho_G(M)=\rho_G(M_1)+\rho_G(M_2)$.\\
{\rm (ii)} $\rho_{G}(M_1)=0$ if and only if $\rho_{G/N_1}(M_1)=0$.\\
{\rm (iii)} $\rho_G(M_1)$ is invertible if and only if 
$\rho_{G/N_1}(M_1)$ is invertible.
\end{lemma}
\begin{proof}
(i) Let $0 \rightarrow M_i \rightarrow P_i \rightarrow F_i
\rightarrow 0$ be flabby resolutions of $M_i$ as $G$-lattices $(i=1,2)$. 
Then 
$0 \rightarrow M \rightarrow P_1 \oplus P_2
\rightarrow F_1 \oplus F_2 \rightarrow 0$
is a flabby resolution of $M$. 
Hence $\rho_G(M)=[F_1 \oplus F_2]=\rho_G(M_1)+\rho_G(M_2)$. 
(ii), (iii) See \cite[Lemme 2]{CTS77} and \cite[Lemma 4.1]{Kan09}.
\end{proof}
By Theorem \ref{thVo} and Lemma \ref{lemp1}, 
for decomposable $3$-dimensional $k$-tori $T=T_1\times_k T_2$ 
with dim $T_1=1$ and dim $T_2=2$, 
the function fields $k(T)\simeq L(M)^G$ are $k$-rational where 
$M=M_1\oplus M_2$ with rank $M_1=1$ and rank $M_2=2$. 
There exist $73$ non-conjugate finite subgroups of $\GL(3,\bZ)$ 
and $39$ (resp. $34$) of them are decomposable (resp. indecomposable). 
%

Let $S_n$ (resp. $A_n$, $D_n$, $C_n$) be the symmetric 
(resp. the alternating, the dihedral, the cyclic) group 
of degree $n$ of order $n!$ (resp. $n!/2$, $2n$, $n$). 
Let ${}^tG$ be the group of transposed matrices of $G\leq \GL(n,\bZ)$.

Rationality (resp. stably rationality, resp. retract rationality) 
problem for $3$-dimensional algebraic 
$k$-tori was solved by Kunyavskii \cite{Kun90} 
(see also Definition \ref{def7.1}, Remark \ref{rem7.2}, Example \ref{ex7.3}, 
Lemma \ref{lem7.2b}, Lemma \ref{lem7.3b} and Theorem \ref{th7.12}). 
\begin{theorem}[{Kunyavskii \cite[Theorem 1, Theorem 2]{Kun90}, see Kang \cite[page 25, the fifth paragraph]{Kan12} for the last statement}]\label{thKu}
Let $L/k$ 
be a Galois extension and 
$G\simeq {\rm Gal}(L/k)$ 
be a finite subgroup of $\GL(3,\bZ)$ 
which acts on $L(M)=L(x_1,x_2,x_3)$ 
via $(\ref{acts})$. 
Let $T$ be an algebraic $k$-torus with $\widehat{T}\simeq M$ 
and $\dim_k\, T=3$. 
Then 
$L(M)^G\simeq k(T)$ is not $k$-rational if and only if 
$G$ is conjugate to one of the $15$ groups $N_{3,i}$ $(1\leq i\leq 15)$ 
which are given 
as in {\rm Table} $1$ $($see also {\rm Table} $6$$)$.  
Moreover, if $L(M)^G\simeq k(T)$ is not $k$-rational, 
then it is not retract $k$-rational. 
\end{theorem}

\vspace*{2mm}
\begin{center}
Table $1$: $L(M)^G$ not retract $k$-rational, 
rank $M=3$, $M$: indecomposable ($15$ cases)\vspace*{2mm}\\
\fontsize{9pt}{11pt}\selectfont
\begin{tabular}{ll} 
${}^tG$ in \cite{Kun90} & 
$G$ \\\hline
$U_1$ 
& $C_2^2$\\
$U_2$ 
& $C_2^3$\\
$U_3$ 
& $D_4$\\
$U_4$ 
& $D_4$\\
$U_5$ 
& $A_4$
\end{tabular}\quad
\begin{tabular}{ll}
${}^tG$ in \cite{Kun90} & 
$G$ \\\hline
$U_6$ 
& $D_4\times C_2$\\
$U_7$ 
& $A_4\times C_2$\\
$U_8$ 
& $S_4$\\
$U_9$ 
& $S_4$\\
$U_{10}$ 
& $S_4$
\end{tabular}\quad
\begin{tabular}{ll} 
${}^tG$ in \cite{Kun90} & 
$G$ \\\hline
$U_{11}$ 
& $S_4\times C_2$\\
$U_{12}$ 
& $S_4\times C_2$\\
$W_1$ 
& $C_4\times C_2$\\
$W_2$ 
& $C_2^3$\\
$W_3$ 
& $A_4\times C_2$
\end{tabular}
\end{center}
%
\begin{remark}\label{rem1.18}
(1) For the last statement of Theorem \ref{thKu}, 
see also Lemma \ref{lem7.3b}.\\ 
(2) 
Let $L/k$ $($resp. $L^\prime/k)$ 
be a Galois extension and 
$G\simeq {\rm Gal}(L/k)$ $($resp. $G^\prime\simeq {\rm Gal}(L^\prime/k))$ 
be a finite subgroup of $\GL(3,\bZ)$ 
which acts on $L(M)$ 
$($resp. $L^\prime(M^\prime))$ 
via $(\ref{acts})$. 
Let $T$ (resp. $T^\prime$) be an algebraic $k$-torus 
with $\widehat{T}\simeq M$ (resp. $\widehat{T}^\prime\simeq M^\prime$). 
Kunyavskii \cite{Kun90} also claimed that\\
(i) (\cite[Theorem 1]{Kun90}) 
If $G$ and $G^\prime$ are not isomorphic as abstract groups, 
then $L(M)^G\simeq k(T)$ and 
$L^\prime(M^\prime)^{G^\prime}\simeq k(T^\prime)$ are not stably 
$k$-isomorphic.\\
(ii) (\cite[Remark 1, page 19]{Kun90}) 
If $G\simeq U_3\simeq D_4$ and $G^\prime\simeq U_4\simeq D_4$ 
and $L=L^\prime$, 
then $L(M)^G\simeq k(T)$ and 
$L^\prime(M^\prime)^{G^\prime}\simeq k(T^\prime)$ are stably 
$k$-isomorphic.\\
(iii) (\cite[Remark 2, page 20]{Kun90}) 
If $G\simeq U_i\simeq S_4$ and $G^\prime\simeq U_j\simeq S_4$ 
$(i,j\in\{8,9,10\}, i\neq j)$, then 
$L(M)^G\simeq k(T)$ and 
$L^\prime(M^\prime)^{G^\prime}\simeq k(T^\prime)$ are not stably 
$k$-isomorphic. 

For (i), Kunyavskii \cite{Kun90}
gave a part of the proof for the case where $G$ is a $2$-group, 
i.e. $G\simeq U_1\simeq C_2^2$, $U_2\simeq C_2^3$, $U_3\simeq U_4\simeq D_4$, 
$U_6\simeq D_4\times C_2$, $W_1\simeq C_4\times C_2$, $W_2\simeq C_2^3$  
(see Section \ref{S7}, Kunyavskii \cite[Section 4 and Section 5]{Kun90}). 
We will give a full proof of (i) 
as in Theorem \ref{thmain2} (Main theorem 2). 
For (ii), (iii), we will disprove these assertions as in Theorem \ref{thmain2} (Main theorem 2) (see also Example \ref{ex1.26}). 
Note that we also have isomorphisms $U_{11}\simeq U_{12}\simeq S_4\times C_2$, 
$U_2\simeq W_2\simeq C_2^3$, $U_7\simeq W_3\simeq A_4\times C_2$. 
\end{remark}

Stably $k$-rationality and retract $k$-rationality 
for algebraic $k$-tori of dimensions $4$ and $5$ 
are given in Hoshi and Yamasaki \cite[Chapter 1]{HY17}.   
There exist $710$ non-conjugate finite subgroups of $\GL(4,\bZ)$. 
The case of dimension $4$ can be stated as follows 
(see \cite[Theorem 1.12]{HY17} for the case of dimension $5$): 
\begin{theorem}[Hoshi and Yamasaki {\cite[Theorem 1.9]{HY17}}]\label{thHY17}
Let $L/k$ be a Galois extension and $G\simeq 
{\rm Gal}(L/k)$ be a finite subgroup of $\GL(4,\bZ)$ 
which acts on $L(M)=L(x_1,x_2,x_3,x_4)$ via $(\ref{acts})$. 
Let $T$ be an algebraic $k$-torus with $\widehat{T}\simeq M$ 
and $\dim_k\, T=4$.\\
{\rm (i)} 
$L(M)^G\simeq k(T)$ is stably $k$-rational 
if and only if 
$G$ is conjugate to one of the $487$ groups which are not in 
{\rm Tables} $7, 8, 9$ 
$($\cite[{\rm Tables} $2$, $3$, $4$]{HY17}$)$.\\
{\rm (ii)} 
$L(M)^G\simeq k(T)$ is not stably but retract $k$-rational 
if and only if $G$ is conjugate to one of the $7$ groups $I_{4,i}$ 
$(1\leq i\leq 7)$ 
which are given as in {\rm Table} $9$ $($\cite[{\rm Table} $2$]{HY17}$)$.\\
{\rm (iii)} 
$L(M)^G\simeq k(T)$ is not retract $k$-rational if and only if 
$G$ is conjugate to one of the $216$ groups which consist of 
the $64$ groups $N_{31,i}$ $(1\leq i\leq 64)$ 
in {\rm Table} $7$ $($\cite[{\rm Table} $3$]{HY17}$)$
and 
the $152$ groups $N_{4,i}$ $(1\leq i\leq 152)$ 
in {\rm Table} $8$ $($\cite[{\rm Table} $4$]{HY17}$)$. 
\end{theorem}

\begin{definition}[Hoshi and Yamasaki {\cite[Table $2$, Table $3$, Table $4$, Example 4.12]{HY17}}]\label{defN3N4}
~\\
(1) The $15$ groups $G=N_{3,i}\leq \GL(3,\bZ)$ $(1\leq i\leq 15)$ 
for which $L(M)^G$ is not retract $k$-rational 
are defined as in Table $6$ (see also Table $1$). \\
(2) The $64$ groups 
$G=N_{31,i}\leq \GL(4,\bZ)$ $(1\leq i\leq 64)$ 
for which $L(M)^G$ is not retract $k$-rational 
where $M\simeq M_1\oplus M_2$ and $M_1$ is indecomposable 
with ${\rm rank}$ $M_1=3$ 
and ${\rm rank}$ $M_2=1$ are defined as in Table $7$.\\
(3) The $152$ groups $G=N_{4,i}\leq \GL(4,\bZ)$ $(1\leq i\leq 152)$ 
for which $L(M)^G$ is not retract $k$-rational 
where $M$ is indecomposable with ${\rm rank}$ $M=4$ 
are defined as in Table $8$.\\
(4) The $7$ groups $G=I_{4,i}\leq \GL(4,\bZ)$ $(1\leq i\leq 7)$
for which $L(M)^G$ is not stably but retract $k$-rational 
with $M$ indecomposable are defined as in Table $9$.
\end{definition}
Tables $6,7,8,9$ will be given in Section \ref{S2}.

\begin{remark}\label{rem1.19}
(1) We do not know the $k$-rationality of 
$L(M)^G\simeq k(T)$ in Theorem \ref{thHY17} (i) 
for the stably $k$-rational algebraic $k$-tori $T$ 
of dimension $4$ with $\widehat{T}=M$. 
Voskresenskii conjectured that any stably $k$-rational 
algebraic $k$-torus 
is $k$-rational (see \cite[Section 6.2]{Vos98}).\\
(2) 
For $G=I_{4,i}$ $(1\leq i\leq 7)$ with $\widehat{T}=M_G$, 
there exists an algebraic $k$-torus $T^\prime$ such that 
$T\times_k T^\prime$ is stably $k$-rational (see Remark \ref{rem1.13}). 
Indeed, by $-[M_G]^{fl}=[[M_G]^{fl}]^{fl}$ (see Swan \cite[Lemma 3.1]{Swa10} 
and Hoshi and Yamasaki \cite[Lemma 2.15]{HY17}), 
we may find such $T^\prime$ with dim $T^\prime={\rm rank}$ 
$[M_G]^{fl}=16,16,16,16,16,16,20$ 
for $1\leq i\leq 7$ respectively 
by using the function 
{\tt FlabbyResolutionLowRank(}$G${\tt ).actionF} 
(see Section \ref{S6}). 
Hoshi and Yamasaki \cite[Theorem 1.27]{HY17} also showed that 
$T\times_k T^\prime$ is stably $k$-rational for $\widehat{T}=M_G$ 
and $\widehat{T}^\prime=M_{G^\prime}$ with 
$G=I_{4,1}\simeq F_{20}$ and $G^\prime=I_{4,2}\simeq F_{20}$ 
(resp. 
$G=I_{4,4}\simeq S_5$ and $G^\prime=I_{4,5}\simeq S_5$). 
We will discuss the case where $G=I_{4,i}$ later (Theorem \ref{thmain6}).
\end{remark}

Let $K/k$ be a separable field extension of degree $n$ 
and $L/k$ be the Galois closure of $K/k$. 
Let $G={\rm Gal}(L/k)$ and $H={\rm Gal}(L/K)$. 
The Galois group $G$ may be regarded as a transitive subgroup of 
the symmetric group $S_n$ of degree $n$. 
Let $R^{(1)}_{K/k}(\bG_{m,K})$ be the norm one torus of $K/k$,
i.e. the kernel of the norm map $R_{K/k}(\bG_{m,K})\rightarrow \bG_{m,k}$ 
where 
$R_{K/k}$ is the Weil restriction (see Ono \cite[Section 1.4]{Ono61}, 
Voskresenskii \cite[page 37, Section 3.12]{Vos98}). 
The norm one torus $R^{(1)}_{K/k}(\bG_{m,K})$ has the 
Chevalley module $J_{G/H}$ as its character module 
and the field $L(J_{G/H})^G$ as its function field 
where $J_{G/H}=(I_{G/H})^\circ={\rm Hom}_\bZ(I_{G/H},\bZ)$ 
is the dual lattice of $I_{G/H}={\rm Ker}(\varepsilon)$ and 
$\varepsilon : \bZ[G/H]\rightarrow \bZ$ is the augmentation map 
(see \cite[Section 4.8]{Vos98}). 
We have the exact sequence $0\rightarrow \bZ\rightarrow \bZ[G/H]
\rightarrow J_{G/H}\rightarrow 0$ and rank $J_{G/H}=n-1$. 
Then the action of $G$ on $L(J_{G/H})=L(x_1,\ldots,x_{n-1})$ is 
of the form (\ref{acts}). 

The rationality problem for norm one tori is investigated 
by Endo and Miyata \cite{EM75}, 
Colliot-Th\'{e}l\`{e}ne and Sansuc \cite{CTS77}, \cite{CTS87}, 
H\"{u}rlimann \cite{Hur84}, 
Le Bruyn \cite{LeB95}, 
Cortella and Kunyavskii \cite{CK00}, 
Lemire and Lorenz \cite{LL00}, 
Florence \cite{Flo}, 
Endo \cite{End11}, 
Hoshi and Yamasaki \cite{HY17}, 
\cite{HY21}, 
Hasegawa, Hoshi and Yamasaki \cite{HHY20}.\\
%

The aim of this paper is to give a stably 
birational classification for algebraic $k$-tori 
of dimensions $3$ and $4$. 
For the case of dimension $\leq 2$, by Theorem \ref{thVo}, 
there exists only one algebraic $k$-torus, 
i.e. rational $k$-torus, up to stably birationally $k$-equivalence. 

\begin{theorem}[Colliot-Th\'el\`ene and Sansuc {\cite[Corollaire, page 191]{CTS77}}]\label{thV4}
Let $L/k$ and $L^\prime/k$ be Galois extensions with 
${\rm Gal}(L/k)\simeq {\rm Gal}(L^\prime/k)\simeq C_2^2$. 
Let $T=R_{L/k}^{(1)}(\bG_{m,L})$ and 
$T^\prime=R_{L^\prime/k}^{(1)}(\bG_{m,L^\prime})$ 
be the corresponding norm one tori. 
If $T$ and $T^\prime$ are stably birationally $k$-equivalent, 
then 
$L=L^\prime$. 
\end{theorem}

The norm one tori $T=R_{L/k}^{(1)}(\bG_{m,L})$ 
and $T^\prime=R_{L^\prime/k}^{(1)}(\bG_{m,L^\prime})$ as in Theorem \ref{thV4} 
correspond to $U_1$ in Table $1$. 
In particular, if $k$ is a number field, then 
there exist infinitely many stably birationally $k$-equivalent classes of 
(non-rational) $k$-tori 
which correspond to $U_1$ in Table $1$ 
(see also Voskresenskii \cite[Example, page 14]{Vos70}). 

\begin{definition}[The $G$-lattice $M_G$ of a finite subgroup $G$ of 
$\GL(n,\bZ)$]\label{d2.2} 
Let $G$ be a finite subgroup of $\GL(n,\bZ)$. 
{\it The $G$-lattice $M_G$ of rank $n$} 
is defined to be the $G$-lattice with a $\bZ$-basis $\{u_1,\ldots,u_n\}$ 
on which $G$ acts by $u_i^{\sigma}=\sum_{j=1}^n a_{i,j}u_j$ for any $
\sigma=[a_{i,j}]\in G$ (corresponding to the equation (\ref{acts})). 
\end{definition}

Note that if $G$ and $G^\prime$ are $\GL(n,\bZ)$-conjugate, 
then $[M_G]^{fl}$ and $[M_{G^\prime}]^{fl}$ 
are weak stably $k$-equivalent: $[M_G]^{fl}\sim [M_{G^\prime}]^{fl}$ 
(see Definition \ref{d1.10} and Remark \ref{r1.11}).\\

When $M_G$ is decomposable, i.e. 
$M_G\simeq M_{G^\prime}\oplus M_{G^{\prime\prime}}$, 
then 
$L(M_G)^G\simeq L(M_{G^\prime}\oplus M_{G^{\prime\prime}})^G$ 
and hence $T$ and $T^\prime\times_k T^{\prime\prime}$ 
are birationally $k$-equivalent 
where $\widehat{T}\simeq M_G$, 
$\widehat{T}^\prime\simeq M_{G^\prime}$ and 
$\widehat{T}^{\prime\prime}\simeq M_{G^{\prime\prime}}$. 
In particular, we have 
$[M_G]^{fl}=[M_{G^\prime}]^{fl}+[M_{G^{\prime\prime}}]^{fl}$ 
as $G$-lattices and hence $T$ and $T^\prime\times_k T^{\prime\prime}$ 
are stably birationally $k$-equivalent. 
If $[M_{G^{\prime\prime}}]^{fl}=0$ as a $G^{\prime\prime}$-lattice 
(see Lemma \ref{lemp1}), i.e. 
$T^{\prime\prime}$ is stably $k$-rational, 
then $[M_G]^{fl}=[M_{G^\prime}]^{fl}$ as $G$-lattices, 
i.e. $T$ and $T^\prime$ are stably birationally $k$-equivalent: 
%
\begin{theorem}[{$M_G$: decomposable case}]\label{th1.24}
Let $G\leq \GL(n,\bZ)$ $($resp. $G^\prime\leq \GL(n_1,\bZ)$, 
$G^{\prime\prime}\leq \GL(n_2,\bZ)$$)$ 
and $M_G$ $($resp. $M_{G^\prime}$, $M_{G^{\prime\prime}}$$)$
be the corresponding $G$-lattice 
$($resp. $G^\prime$-lattice, $G^{\prime\prime}$-lattice$)$ 
of $\bZ$-rank $n$ $($resp. $n_1$, $n_2$$)$ 
as in Definition \ref{d2.2}. 
Let $T$ $($resp. $T^\prime$, $T^{\prime\prime}$$)$ 
be an algebraic $k$-torus with $\widehat{T}\simeq M_G$ 
$($resp. $\widehat{T}^\prime\simeq M_{G^\prime}$, 
$\widehat{T}^{\prime\prime}\simeq M_{G^{\prime\prime}}$$)$. 
We assume that $M_G\simeq M_{G^\prime}\oplus M_{G^{\prime\prime}}$ 
with $n=n_1+n_2$. 
%
Then $G\leq G^\prime\times G^{\prime\prime}$ is a subdirect product 
of $G^\prime$ and $G^{\prime\prime}$ which 
acts on $M_{G^\prime}$ $($resp. $M_{G^{\prime\prime}}$$)$ 
through the surjection 
$\varphi_1: G \rightarrow G^\prime$ 
$($resp. $\varphi_2:G\rightarrow G^{\prime\prime}$$)$. 
Then $L(M_G)^G\simeq L(M_{G^\prime}\oplus M_{G^{\prime\prime}})^G$ 
and hence $T$ and $T^\prime\times_k T^{\prime\prime}$ 
are birationally $k$-equivalent. 
In particular, we have 
$[M_G]^{fl}=[M_{G^\prime}]^{fl}+[M_{G^{\prime\prime}}]^{fl}$ 
as $G$-lattices 
and hence $T$ and $T^\prime\times_k T^{\prime\prime}$ 
are stably birationally $k$-equivalent. 

If $[M_{G^{\prime\prime}}]^{fl}=0$ as a $G^{\prime\prime}$-lattice 
$($this can be regarded as a $G$-lattice 
because $G^{\prime\prime}\simeq G/{\rm Ker}(\varphi_2)$, 
see Lemma \ref{lemp1}$)$, 
i.e. $T^{\prime\prime}$ is stably $k$-rational,  
e.g. $n_2\leq 2$ by Voskresenskii's theorem $($Theorem \ref{thVo}$)$, 
then $[M_G]^{fl}=[M_{G^\prime}]^{fl}$ as $G$-lattices 
and hence $T$ and $T^\prime$ are stably birationally $k$-equivalent. 
In particular, we have $[M_G]^{fl} \sim [M_{G^\prime}]^{fl}$. 
\end{theorem}

Let $G$, $G^\prime\leq {\rm GL}(n,\bZ)$ be $\GL(n,\bZ)$-conjugate. 
Let $T$ and $T^\prime$ 
be algebraic $k$-tori of dimension $n$ 
with the minimal splitting fields $L$ and $L^\prime$ 
and the character modules 
$\widehat{T}=M_G$ and $\widehat{T}^\prime=M_{G^\prime}$ 
respectively. 
We assume that $L=L^\prime$. 
Then $G\simeq G^\prime\simeq {\rm Gal}(L/k)$ and 
we have
$\varphi_1: {\rm Gal}(L/k)\xrightarrow{\sim} G\leq {\rm GL}(n,\bZ)$, 
$f\mapsto \varphi_1(f)$, 
$\varphi_2: {\rm Gal}(L^\prime/k)\xrightarrow{\sim} G^\prime\leq {\rm GL}(n,\bZ)$, 
$f\mapsto \varphi_2(f)$ and a subdirect product 
$\widetilde{H}=\{(\varphi_1(f),\varphi_2(f))\mid f\in {\rm Gal}(L/k)={\rm Gal}(L^\prime/k)\}\leq G\times G^\prime$ of $G$, $G^\prime$ with $\widetilde{H}\simeq G$. 
On the other hand, 
we have 
$\psi: G\xrightarrow{\sim} G^\prime$, $g\mapsto u^{-1}gu$ 
$(u\in {\rm GL}(n,\bZ))$. 
%
Hence we can obtain $\sigma\in {\rm Aut}(G)$ such that 
$(\psi^{-1})(\varphi_2\varphi_1^{-1})(g)=g^\sigma$ 
for any $g\in G$ where 
${\rm Aut}(G)$ is the group of automorphisms on $G$ and hence 
we can identify $\psi^{-1}:G^\prime\simeq G^\sigma$ $(\sigma\in {\rm Aut}(G))$. 
Let $T^\sigma$ be an algebraic $k$-torus of dimension $n$ 
with the minimal splitting field $L$ and the character module 
$\widehat{T}^\sigma=M_{G^\sigma}$ $(\sigma\in {\rm Aut}(G))$
with 
$G\simeq G^\sigma\simeq {\rm Gal}(L/k)$. 
Note that ${\rm Aut}(G)$ acts on the set $\{K\mid k\subset K\subset L\}$ 
of intermediate fields. 
The subdirect product 
$\widetilde{H}=\{(g,g^\sigma)\mid g\in G\}\simeq G$ 
of $G$, $G^\sigma$ acts on $M_G$ and $M_{G^\sigma}$ diagonally. 
Then the set 
\begin{align*}
\{T^\sigma\mid \sigma\in {\rm Aut}(G)\}
\end{align*}
gives all algebraic $k$-tori of dimension $n$ 
with the minimal splitting field $L$ and the character module 
$\widehat{T}^\sigma\simeq M_G$. 

\begin{definition}\label{defXYZ}
We define the following subgroups of ${\rm Aut}(G)$ for 
$G\leq {\rm GL}(n\,\bZ)$: 
\begin{center}
${\rm Inn}(G)\leq X\leq Y\leq Z\leq {\rm Aut}(G)$,
\end{center}
\begin{align*}
X&={\rm Aut}_{\GL(n,\bZ)}(G)=
\{\sigma\in{\rm Aut}(G)\mid 
{\rm there}\ {\rm exists}\ u\in {\rm GL}(n,\bZ)\ {\rm such}\ {\rm that}\ 
u^{-1}gu=g^\sigma\ {\rm for}\ {\rm any}\ g\in G\}\\
&\simeq N_{\GL(n,\bZ)}(G)/Z_{\GL(n,\bZ)}(G),\\
Y&=\{\sigma\in{\rm Aut}(G)\mid [M_G]^{fl}=[M_{G^\sigma}]^{fl}\ {\rm as}\ 
\widetilde{H}\textrm{-{\rm lattices}}\ {\rm where}\ \widetilde{H}=\{(g,g^\sigma)\mid g\in G\}\simeq G\},\\
Z&=\{\sigma\in{\rm Aut}(G)\mid [M_H]^{fl}\sim [M_{H^\sigma}]^{fl}\ {\rm for}\ {\rm any}\ H\leq G\} 
\end{align*}
where 
${\rm Inn}(G)$ is the group of inner automorphisms on $G$, 
${\rm Aut}(G)$ is the group of automorphisms on $G$, 
$N_{\GL(n,\bZ)}(G)$ is the normalizer of $G$ in $\GL(n,\bZ)$ and 
$Z_{\GL(n,\bZ)}(G)$ is the centralizer of $G$ in $\GL(n,\bZ)$.
%
%
\end{definition}
We see the meaning and the importance of 
${\rm Inn}(G)\leq X\leq Y\leq Z\leq {\rm Aut}(G)$ as follows: 
\begin{theorem}\label{thmXYZ}
Let $k$ be a field. 
Let $G, G^\sigma\leq {\rm GL}(n,\bZ)$ $(\sigma\in {\rm Aut}(G))$ 
and $M_G$, $M_{G^\sigma}$ be the corresponding $G$-lattices 
as in Definition \ref{d2.2}. 
Let $T$ and $T^\sigma$ be algebraic $k$-tori of dimension $n$ 
with the minimal splitting field $L$ and the character modules 
$\widehat{T}=M_G$ and $\widehat{T}^\sigma=M_{G^\sigma}$ 
$(\sigma\in {\rm Aut}(G))$
with 
$G\simeq G^\sigma\simeq {\rm Gal}(L/k)$. 
Let ${\rm Inn}(G)\leq X\leq Y\leq Z\leq {\rm Aut}(G)$ 
be as in Definition \ref{defXYZ}.\\
{\rm (1)} For $\sigma\in {\rm Aut}(G)$, 
$T$ and $T^\sigma$ are weak stably birationally $k$-equivalent;\\
{\rm (2)} For $\sigma\in {\rm Aut}(G)$, 
$\sigma\in X$ if and only if 
$M_G\simeq M_{G^\sigma}$ as ${\rm Gal}(L/k)$-lattices. 
In particular,  
$T$ and $T^\sigma$ are birationally $k$-equivalent, i.e. 
$k(T)\simeq L(M)^G\simeq L(M^\sigma)^{G^\sigma}\simeq k(T^\sigma)$;\\
{\rm (3)} For $\sigma\in {\rm Aut}(G)$, 
$\sigma\in Y$ if and only if 
$T$ and $T^\sigma$ are stably birationally $k$-equivalent. 
In particular, 
$\{T^\sigma\mid \sigma\in {\rm Aut}(G)\}$ splits into 
$\lambda$ stably birationally $k$-equivalent classes 
where 
$\lambda=|Y\backslash {\rm Aut}(G)|$;\\ 
{\rm (4)} For $\sigma\in {\rm Aut}(G)$, 
$\sigma\in Z$ if and only if 
$T\times_k K$ and $T^\sigma\times_k K$ 
are weak stably birationally $K$-equivalent for any $k\subset K\subset L$. 

In particular, {\rm (i)} if $Y=Z$, then 
$\sigma \in Z$ 
if and only if 
$T$ and $T^\sigma$ are stably birationally $k$-equivalent; 
and {\rm (ii)} if $X=Y=Z$ $($resp. $X=Y$$)$, then 
$\sigma\in Z$ $($resp. $\sigma\in Y$$)$
if and only if 
$T$ and $T^\sigma$ are birationally $k$-equivalent. 
\end{theorem}
\begin{proof}
(1) follows from 
$[M_G]^{fl}\sim [M_{G^\sigma}]^{fl}$ 
because 
$[M_G]^{fl}=[M_{G^\sigma}]^{fl}$ 
as $\widetilde{H}$-lattices where 
$\widetilde{H}=\{(g,(g^{\sigma^{-1}})^{\sigma})\mid 
g\in G\}\leq G\times G^\sigma$ and $\widetilde{H}\simeq G\simeq G^\sigma$. 

Because $T$ and $T^\sigma$ have the same splitting field $L$ 
with ${\rm Gal}(L/k)\simeq G\simeq G^\sigma$, 
we have\\
(2) $\sigma\in X$
$\Leftrightarrow$ 
there exists $u\in {\rm GL}(n,\bZ)$ such that 
$u^{-1}gu=g^\sigma$ for any $g\in G$ 
$\Leftrightarrow$ 
$M_G\simeq M_{G^\sigma}$ as ${\rm Gal}(L/k)$-lattices;\\
(3) 
$\sigma\in Y$
$\Leftrightarrow$ 
$[M_G]^{fl}=[M_{G^\sigma}]^{fl}$ as 
$\widetilde{H}$-lattices where 
$\widetilde{H}=\{(g,g^\sigma)\mid g\in G\}\simeq G$ 
$\Leftrightarrow$ 
$T$ and $T^\sigma$ are stably birationally $k$-equivalent;\\
(4) 
$\sigma\in Z$
$\Leftrightarrow$ 
$[M_H]^{fl}\sim [M_{H^\sigma}]^{fl}$ for any 
$H\leq G$ 
$\Leftrightarrow$ 
$T\times_k K$ and $T^\sigma\times_k K$ 
are weak stably birationally $K$-equivalent 
for any $k\subset K\subset L$. 
\end{proof}
\begin{corollary}
Let $G\leq {\rm GL}(n,\bZ)$, $M_G$ be the corresponding $G$-lattice as in Definition \ref{d2.2}, 
$T$ be the algebraic $k$-torus of dimension $n$ with $\widehat{T}\simeq M_G$ 
and $L$ be the minimal splitting field of $T$ as in Theorem \ref{thmXYZ}. 
Let ${\rm Inn}(G)\leq X\leq Y\leq {\rm Aut}(G)$ 
be as in Definition \ref{defXYZ}. 
Then the set $\{T^\sigma\mid \sigma\in {\rm Aut}(G)\}$ 
gives all algebraic $k$-tori of dimension $n$ 
with the minimal splitting field $L$ 
with $\widehat{T}^\sigma\simeq M_G$ 
and 
splits into $\lambda$ stably birationally $k$-equivalent classes consist of 
$\mu$ birationally $k$-equivalent classes where 
$\lambda=|Y\backslash {\rm Aut}(G)|\leq \mu\leq |X\backslash {\rm Aut}(G)|$.
\end{corollary}

We state main theorems of this paper.

\begin{theorem}[Main theorem 1]\label{thmain1}
Let $k$ be a field. 
There exist exactly $14$ weak stably birationally 
$k$-equivalent classes 
of algebraic $k$-tori $T$ of dimension $3$ which consist of 
the stably rational class ${\rm WSEC}_0$ 
and $13$ classes ${\rm WSEC}_r$ $(1\leq r\leq 13)$ 
for $[\widehat{T}]^{fl}$ with $\widehat{T}=M_G$ and 
$G=N_{3,i}$ $(1\leq i\leq 15)$ as in {\rm Table} $2$. 
\end{theorem}

Table $2$: Weak stably $k$-equivalent classes for $[\widehat{T}]^{fl}$ with $\widehat{T}=M_G$ and $G=N_{3,i}$ $(1\leq i\leq 15)$
{\small 
\begin{longtable}{cllcc}
$r$ & $G=N_{3,i}: [\widehat{T}]^{fl}=[M_G]^{fl}\in {\rm WSEC}_r$ & ${\rm G}_r\simeq G$ & $Y\backslash {\rm Aut}(G)$ & $\lambda_r=|{\rm WSEC}_{r,L}|=|Y\backslash {\rm Aut}(G)|$\\\hline\vspace*{-4mm}
\endfirsthead
$1$ & \fbox{$N_{3,1}$} $\simeq U_1$ & $C_2^2$ & $\{1\}$ & $1$\\
$2$ & $N_{3,2}\simeq U_2$ & $C_2^3$ & ${\rm PGL}(3,\bF_2)\leq S_7$ & $7$\\
$3$ & $N_{3,3}\simeq W_2$ & $C_2^3$ & $\{1\}$ & $1$\\
$4$ & $N_{3,4}\simeq W_1$ & $C_4 \times C_2$ & \{1\} & $1$\\
$5$ & \fbox{$N_{3,5}$} $\simeq U_3$, \fbox{$N_{3,6}$} $\simeq U_4$ & $D_4$ & $C_2\leq S_2$ & $2$\\
$6$ & $N_{3,7}\simeq U_6$ & $D_4\times C_2$ & $D_4\leq S_4$ & $4$\\
$7$ & \fbox{$N_{3,8}$} $\simeq U_5$ & $A_4$ & $\{1\}$ & $1$\\
$8$ & \fbox{$N_{3,9}$} $\simeq U_7$ & $A_4\times C_2$ & $\{1\}$ & $1$\\
$9$ & \fbox{$N_{3,10}$} $\simeq W_3$ & $A_4\times C_2$ & $\{1\}$ & $1$\\
$10$ & \fbox{$N_{3,11}$} $\simeq U_9$, \fbox{$N_{3,13}$} $\simeq U_{10}$ & $S_4$ & $\{1\}$ & $1$\\
$11$ & \fbox{$N_{3,12}$} $\simeq U_8$ & $S_4$ & $\{1\}$ & $1$\\
$12$ & $N_{3,14}\simeq U_{12}$ & $S_4\times C_2$ & $\{1\}$ & $1$\\
$13$ & \fbox{$N_{3,15}$} $\simeq U_{11}$ & $S_4\times C_2$ & $C_2\leq S_2$ & $2$
\end{longtable}
}\vspace*{-2mm}
The box cases \fbox{$N_{3,i}$} satisfy $X=Y=Z$ 
where ${\rm Inn}(G)\leq X\leq Y\leq Z\leq {\rm Aut}(G)$ 
are as in Definition \ref{defXYZ} (see Example \ref{ex9.1} for GAP computations).\\

In Table $2$, 
the fourth column $Y\backslash {\rm Aut}(G)$ means that 
the action of ${\rm Aut}(G)$ on ${\rm WSEC}_{r,L}$ 
with the stabilizer $Y$ is given as the 
$\lambda_r$-th transitive subgroup 
$\lambda_rTm\leq S_{\lambda_r}$ 
where 
$\lambda_r=|{\rm WSEC}_{r,L}|=|Y\backslash {\rm Aut}(G)|$
(see Theorem \ref{thmain2}). 

The cases $N_{3,1}\simeq U_1\simeq C_2^2$, 
$N_{3,6}\simeq U_4\simeq D_4$, 
$N_{3,8}\simeq U_5\simeq A_4$, 
$N_{3,13}\simeq U_{10}\simeq S_4$ 
correspond to the norm one tori 
$T=R^{(1)}_{K/k}(\bG_{m,K})$ with $G\simeq {\rm Gal}(L/k)$ where 
$L/k$ is the Galois closure of $K/k$. 
\vspace*{2mm}

\begin{theorem}[Main theorem 2]\label{thmain2}
Let $k$ be a field.  
Let $T_i$ and $T_j^\prime$ $(1\leq i,j\leq 15)$ 
be algebraic $k$-tori of dimension $3$ 
with the minimal splitting fields $L_i$ and $L^\prime_j$ 
and the character modules 
$\widehat{T}_i=M_G$ and $\widehat{T}^\prime_j=M_{G^\prime}$ 
which satisfy that 
$G\leq \GL(3,\bZ)$ and 
$G^\prime\leq \GL(3,\bZ)$ are 
$\GL(3,\bZ)$-conjugate to $N_{3,i}$ and $N_{3,j}$ respectively 
with $G\simeq {\rm Gal}(L_i/k)$ and 
$G^\prime\simeq {\rm Gal}(L_j^\prime/k)$. 
For $1\leq i,j\leq 15$, 
the following conditions are equivalent:\\
{\rm (1)} $T_i$ and $T_j^\prime$ are stably birationally $k$-equivalent;\\
{\rm (2)} 
$G\simeq G^\prime$, 
$L_i=L_j^\prime$, 
$T_i\times_k K$ and $T_j^\prime\times_k K$ 
are weak stably birationally $K$-equivalent 
for any $k\subset K\subset L_i$;\\
{\rm (3)} 
$G\simeq G^\prime$, 
$L_i=L_j^\prime$, 
$T_i\times_k K$ and $T_j^\prime\times_k K$ 
are weak stably birationally $K$-equivalent 
for any $k\subset K\subset L_i$ corresponding to 
${\rm WSEC}_r$ $(1\leq r\leq 13)$ as in Table $2$ 
with $T_i\times_k K$ not retract $K$-rational;\\
{\rm (4)} 
$G\simeq G^\prime$, 
$L_i=L_j^\prime$, 
$T_i\times_k K$ and $T_j^\prime\times_k K$ 
are weak stably birationally $K$-equivalent 
for any $k\subset K\subset L_i$ corresponding to 
${\rm WSEC}_r$ $(1\leq r\leq 13)$ as in Table $2$ 
with $T_i\times_k K$ not retract $K$-rational 
and $[K:k]=d$ 
$($the ``bold'' cases as in {\rm Table} $10$$)$ where 
\begin{align*}
d=
\begin{cases}
1&(i=1,3,4,8,9,10,11,12,13,14),\\
1,2&(i=2,5,6,7,15).
\end{cases}
\end{align*}
In particular, for $i=1,3,4,8,9,10,11,12,13,14$ with $d=1$, 
i.e. ${\rm Aut}(G)=Y$ with $\lambda_{r}=1$, 
we get that 
${(\rm 1)}$ $T_i$ and $T_j^\prime$ are stably birationally $k$-equivalent 
if and only if ${(\rm 4)}$ $G\simeq G^\prime$, $L_i=L_j^\prime$, 
i.e. $\widetilde{H}\simeq G\simeq G^\prime$. 

Moreover, if $i=j$ with $G\simeq G^\prime$, $L_i=L_j^\prime$, 
i.e. $\widetilde{H}\simeq G\simeq G^\prime$, 
then $Y=Z$ $($which is equivalent to $(1)\Leftrightarrow (2)$$)$ 
where ${\rm Inn}(G)\leq X\leq Y\leq Z\leq {\rm Aut}(G)$ 
are as in Definition \ref{defXYZ} 
%
%
and for $1\leq r\leq 13$, we get the following disjoint union decompositions 
\begin{align*}
{\rm WSEC}_r=\coprod_{L/k\atop {\rm Gal}(L/k)\simeq {\rm G}_r}
{\rm WSEC}_{r,L},\ 
{\rm WSEC}_{r,L}=
\coprod_{t=1}^{\lambda_{r}} {\rm SEC}_{r,L,t}
\end{align*}
modulo stably birationally $k$-equivalence $\stackrel{\rm s.b.}{\approx}$ 
where 
${\rm SEC}_{r,L,t}$ $(1\leq t\leq \lambda_{r})$ is the $t$-th 
stably $k$-equivalent class of $T$ of dimension $3$ in 
${\rm WSEC}_{r,L}$ which corresponds to the fixed minimal splitting field $L$ 
in 
${\rm WSEC}_r=\coprod_{L/k\atop {\rm Gal}(L/k)\simeq {\rm G}_r} {\rm WSEC}_{r,L}$ 
with 
$[\widehat{T}]^{fl}\in {\rm WSEC}_{r,L}
=\coprod_{t=1}^{\lambda_r} {\rm SEC}_{r,L,t}$, 
${\rm Gal}(L/k)\simeq {\rm G}_r\simeq N_{3,i}$ 
$(1\leq r\leq 13)$ and 
$\lambda_r=|{\rm WSEC}_{r,L}|=|Y\backslash {\rm Aut}(G)|$
is given as in Table $2$. 
Furthermore, 
for the box cases \fbox{$N_{3,i}$} with $X=Y$ as in Table $2$,
the following conditions are also equivalent:\\
{\rm (0)} $T_i$ and $T_i^\prime$ are birationally $k$-equivalent;\\
{\rm (1)} $T_i$ and $T_i^\prime$ are stably birationally $k$-equivalent. 

In particular, for the box cases \fbox{$N_{3,i}$}, 
all the algebraic tori $T$ of dimension $3$ 
with $\widehat{T}=M_G$ and $G=N_{3,i}$ 
which correspond to the fixed minimal splitting field $L$ 
split into 
$\lambda_r=|Y\backslash {\rm Aut}(G)|$ birationally $k$-equivalent classes. 
\end{theorem}

As a corollary of Theorem \ref{thmain1} (Main theorem $1$) and 
Theorem \ref{thmain2} (Main theorem $2$), 
we get an answer of Problem \ref{prob1.3} 
for algebraic $k$-tori $T$ of dimensions $3$, 
i.e. the structure of $\mathcal{T}_3/\stackrel{\rm s.b.}{\approx}$: 

\begin{corollary}[The structure of $\mathcal{T}_3/\stackrel{\rm s.b.}{\approx}$]\label{cor1.27}
Let $\mathcal{T}_3$ be the category of algebraic $k$-tori of dimension $3$. 
We get a classification $($disjoint union decomposition$)$ 
of $\mathcal{T}_3$ with respect to 
the stably birationally $k$-equivalence $\stackrel{\rm s.b.}{\approx}$: 
\begin{align*}
\mathcal{T}_3=\coprod_{r=0}^{13} {\rm WSEC}_r
={\rm SEC}_0\coprod
\left(\coprod_{r=1}^{13} 
\coprod_{L/k\atop {\rm Gal}(L/k)\simeq {\rm G}_r}
\coprod_{t=1}^{\lambda_{r}} {\rm SEC}_{r,L,t}\right)
\end{align*}
modulo stably birationally $k$-equivalence $\stackrel{\rm s.b.}{\approx}$ 
where ${\rm SEC}_0$ is the stably $k$-equivalent class consists of 
stably $k$-rational tori $T\in \mathcal{T}_3$ 
and 
${\rm SEC}_{r,L,t}$ $(1\leq t\leq \lambda_{r})$ is the $t$-th 
stably $k$-equivalent class of $T\in \mathcal{T}_3$ 
in 
${\rm WSEC}_{r,L}$ which corresponds to the fixed minimal splitting field $L$ 
in 
${\rm WSEC}_r=\coprod_{L/k\atop {\rm Gal}(L/k)\simeq {\rm G}_r} {\rm WSEC}_{r,L}$ 
with 
$[\widehat{T}]^{fl}\in {\rm WSEC}_{r,L}
=\coprod_{t=1}^{\lambda_r} {\rm SEC}_{r,L,t}$, 
${\rm Gal}(L/k)\simeq {\rm G}_r\simeq N_{3,i}$ 
$(1\leq r\leq 13)$ and 
$\lambda_r=|{\rm WSEC}_{r,L}|=|Y\backslash {\rm Aut}(G)|$
is given as in Table $2$. 
\end{corollary}
\begin{remark}
The case $i=j=1$, i.e. $G=N_{3,1}\simeq U_1\simeq C_2^2$ and 
$G^\prime=N_{3,1}\simeq U_1\simeq C_2^2$ with $\lambda_1=1$, 
of Theorem \ref{thmain2} 
is obtained by Colliot-Th\'el\`ene and Sansuc (see Theorem \ref{thV4}). 
\end{remark}
\begin{example}[${\rm WSEC}_5$: $G=N_{3,6}\simeq U_4\simeq D_4$, $G^\prime=N_{3,6}\simeq U_4\simeq D_4$ with $\lambda_5=2$]\label{ex1.26}
Let $k=\bQ$, $K_4=\bQ(\sqrt[4]{2})$ and 
$K_4^\prime=\bQ(\sqrt[4]{2}\zeta_8)$ be algebraic number fields 
with $[K_4:\bQ]=[K_4^\prime:\bQ]=4$ 
and the same Galois closure $L=\bQ(\sqrt[4]{2},\sqrt{-1})$, 
${\rm Gal}(L/\bQ)\simeq D_4=\langle \sigma,\tau\rangle$ and $[L:\bQ]=8$ 
where $\zeta_8=e^{2\pi\sqrt{-1}/8}$ and 
\begin{align*}
&\sigma: \sqrt[4]{2}\mapsto \sqrt[4]{2}\sqrt{-1},\ \sqrt{-1}\mapsto \sqrt{-1},\\ 
&\tau: \sqrt[4]{2}\mapsto \sqrt[4]{2},\ \sqrt{-1}\mapsto -\sqrt{-1}. 
\end{align*} 
We find that $K_4=L^{\langle\tau\rangle}$, 
$(K_4)^\sigma=L^{\langle\sigma^2\tau\rangle}$, 
$K_4^\prime=L^{\langle\sigma\tau\rangle}$, 
$(K_4^\prime)^\sigma=L^{\langle\sigma^{-1}\tau\rangle}$, 
$\bQ(\sqrt{2})=L^{\langle\sigma^2,\tau\rangle}$
$\bQ(\sqrt{-1})=L^{\langle\sigma\rangle}$, 
$\bQ(\sqrt{-2})=L^{\langle\sigma^2,\sigma\tau\rangle}$. 
Let 
$T=R_{K_4/\bQ}^{(1)}(\bG_{m,K_4})$ and 
$T^\prime=R_{K_4^\prime/\bQ}^{(1)}(\bG_{m,K_4^\prime})$ 
be norm one tori of $K_4/\bQ$ and of $K_4^\prime/\bQ$ respectively 
with $\widehat{T}\simeq M_G\simeq J_{D_4/\langle\tau\rangle}$ 
and $\widehat{T}^\prime\simeq M_{G^\prime}\simeq J_{D_4/\langle\sigma\tau\rangle}$ 
where $G, G^\prime \leq \GL(3,\bZ)$ 
which correspond to $N_{3,6}\simeq U_4\simeq D_4$. 
We see that $X^4-2$ is irreducible over $\bQ(\sqrt{-2})$ and 
hence it follows from Kunyavskii's theorem (Theorem \ref{thKu}) that 
$T\times_\bQ \bQ(\sqrt{-2})$ is not retract $\bQ(\sqrt{-2})$-rational 
because it corresponds to $N_{3,1}\simeq U_1\simeq C_2^2$ and 
$\bQ(T\times_\bQ \bQ(\sqrt{-2}))\simeq L(M_H)^H\simeq L(J_H)^H=L(x_1,x_2,x_3)^H$ and 
$H\simeq {\rm Gal}(L/\bQ(\sqrt{-2}))
=\langle \sigma^2,\sigma\tau\rangle\simeq C_2^2$ 
with 
\begin{align*}
&\sigma^2: x_1\mapsto x_3,\ x_2\mapsto \frac{1}{x_1x_2x_3},\ x_3\mapsto x_1,\\ 
&\sigma\tau: x_1\mapsto \frac{1}{x_1x_2x_3},\ x_2\mapsto x_3,\ x_3\mapsto x_2.
\end{align*}
On the other hand, we see that $X^4+2=(X^2-\sqrt{-2})(X^2+\sqrt{-2})$ 
over $\bQ(\sqrt{-2})$ and by 
Kunyavskii's theorem (Theorem \ref{thKu}) again  
$T^\prime\times_\bQ \bQ(\sqrt{-2})$ is stably $\bQ(\sqrt{-2})$-rational
because it corresponds to 
$\bQ(T^\prime)\simeq L(M_{H^\prime})^{H^\prime}
=L(y_1,y_2,y_3)^{H^\prime}$ and 
$H^\prime\simeq {\rm Gal}(L/\bQ(\sqrt{-2}))=
\langle \sigma^2,\sigma\tau\rangle\simeq C_2^2$ with 
\begin{align*}
&\sigma^2: y_1\mapsto y_3,\ y_2\mapsto \frac{1}{y_1 y_2 y_3},\ 
y_3\mapsto y_1,\\ 
&\sigma\tau: y_1\mapsto y_1,\ y_2\mapsto \frac{1}{y_1 y_2 y_3},\ 
y_3\mapsto y_3.
\end{align*}
It follows from Theorem \ref{thmain2} that 
$T$ and $T^\prime$ are not stably birationally $\bQ$-equivalent 
although both $T$ and $T^\prime$ correspond to $N_{3,6}\simeq U_4\simeq D_4$ 
with the same splitting field $L$. 

Because $\lambda_5=2$ as in Table $2$, by Theorem \ref{thmain2}, 
if we take an algebraic $\bQ$-torus 
$T^{\prime\prime}$ of dimension $3$ with 
the minimal splitting field $L=\bQ(\sqrt[4]{2},\sqrt{-1})$ and 
$[T^{\prime\prime}]^{fl}\in N_{3,5}\simeq U_3$, 
then $T^{\prime\prime}$ is stably birationally 
$\bQ$-equivalent to either $T=R_{K_4/\bQ}^{(1)}(\bG_{m,K_4})$ or 
$T^\prime=R_{K_4^\prime/\bQ}^{(1)}(\bG_{m,K_4^\prime})$.
\end{example}
\begin{example}[${\rm WSEC}_2$: $G=N_{3,2}\simeq U_2\simeq C_2^3$, $G^\prime=N_{3,2}\simeq U_2\simeq C_2^3$ with $\lambda_2=7$]\label{exU2}
Let $G=N_{3,2}=\langle s_1, s_2, s_3\rangle\leq {\rm GL}(3,\bZ)$ 
with $G\simeq U_2\simeq C_2^3$ where 
\begin{align*}
s_1=\left(
\begin{array}{ccc}
0&0&1\\
-1&-1&-1\\
1&0&0
\end{array}\right),\ 
s_2=\left(
\begin{array}{ccc}
0&1&0\\
1&0&0\\
-1&-1&-1
\end{array}\right),\ 
s_3=\left(
\begin{array}{ccc}
-1&0&0\\
0&-1&0\\
0&0&-1
\end{array}\right). 
\end{align*}
Let $k=\bQ$ and 
$T$ be an algebraic $\bQ$-torus with the minimal splitting field 
$L=\bQ(\sqrt{2},\sqrt{3},\sqrt{5})$ with $[L:\bQ]=8$. 
By Kunyavskii's theorem (Theorem \ref{thKu}), 
we see that $T$ with $\widehat{T}=M_G\simeq N_{3,2}$, 
$[\widehat{T}]^{fl}\in {\rm WSEC}_2$ 
is not retract $\bQ$-rational. 
We see that there exist exactly $7$ subgroups $H_i\leq G$ $(1\leq i\leq 7)$ 
with $[G:H_i]=2$: 
\begin{align*}
H_1&=\langle s_1, s_2\rangle\simeq U_1,\ 
H_2=\langle s_1, s_3\rangle,\ 
H_3=\langle s_2, s_3\rangle,\\ 
H_4&=\langle s_1, s_2s_3\rangle,\ 
H_5=\langle s_2, s_1s_3\rangle,\ 
H_6=\langle s_3, s_1s_2\rangle,\ 
H_7=\langle s_1s_3, s_2s_3\rangle 
\end{align*}
and 
there exist exactly $7$ subgroups $H_i^\prime\leq G$ $(1\leq i\leq 7)$ 
with $[G:H_i^\prime]=4$: 
\begin{align*}
H_1^\prime=\langle s_1\rangle,\ 
H_2^\prime=\langle s_2\rangle,\ 
H_3^\prime=\langle s_3\rangle,\ 
H_4^\prime=\langle s_1s_2\rangle,\ 
H_5^\prime=\langle s_1s_3\rangle,\ 
H_6^\prime=\langle s_2s_3\rangle,\ 
H_7^\prime=\langle s_1s_2s_3\rangle. 
\end{align*}
Then we see that $[M_G\mid_{H_i}]^{fl}=[M_{H_i}]^{fl}=0$ 
$(2\leq i\leq 7)$ 
and $[M_G\mid_{H_i^\prime}]^{fl}=[M_{H_i^\prime}]^{fl}=0$ $(1\leq i\leq 7)$ 
and $[M_G\mid_{H_1}]^{fl}=[M_{H_1}]^{fl}$ is not invertible 
with $H_1\simeq U_1$ 
where $M_G\mid H_i$ is an $H_i$-lattice obtained 
by restricting the action of $G$ on $M_G$ to $H_i$. 
Let 
\begin{align*}
Y=\{\sigma\in{\rm Aut}(G)\mid [M_G]^{fl}=[M_{G^\sigma}]^{fl}\ {\rm as}\ 
\widetilde{H}\textrm{-{\rm lattices}}\ {\rm where}\ \widetilde{H}=\{(g,g^\sigma)\mid g\in G\}\simeq G\}
\end{align*}
and $G^\prime=G^\sigma\leq {\rm GL}(3,\bZ)$ with 
$G^\prime\simeq G\simeq  C_2^3\simeq (\bF_2)^3$ 
$(\sigma\in {\rm Aut}(G))$ where 
${\rm Aut}(G)\simeq {\rm GL}(3,\bF_2)\simeq 
{\rm PGL}(3,\bF_2)\simeq {\rm SL}(3,\bF_2)\simeq {\rm PSL}(3,\bF_2)
\simeq {\rm PSL}(2,\bF_7)$ 
with $|{\rm Aut}(G)|=168$. 
By Theorem \ref{thmain2}, we have 
$\lambda_2=|{\rm WSEC}_{r,L}|=|Y\backslash {\rm Aut}(G)|=7$ 
where $Y={\rm Stab}_{H_1}({\rm Aut}(G))\simeq S_4$ with $|Y|=24$. 
This implies that ${\rm WSEC}_{2,L}=\{[\widehat{T}_i]^{fl}\mid 1\leq i\leq 7\}$
and there exist algebraic $k$-tori $T_i$ $(1\leq i\leq 7)$ with 
$\widehat{T}_i=M_{G^{\sigma_i}}$ $(\overline{\sigma_i} 
\in Y\backslash {\rm Aut}(G))$ 
such that $T_i$  are not stably birationally $\bQ$-equivalent each other 
although all $T_i$ correspond to $N_{3,2}\simeq U_2\simeq C_2^3$ 
with the same splitting field $L$. 
Note that $Y\backslash {\rm Aut}(G)\simeq {\rm Gr}_{\bF_2}(2,3)\simeq {\rm Gr}_{\bF_2}(1,3)\simeq \bP^2_{\bF_2}$ with $|\bP^2_{\bF_2}|=7$ 
where ${\rm Gr}_{\bF_2}(d,3)$ is 
the Grassmannian of $d$-dimensional subspaces of $(\bF_2)^3$ 
(see Example \ref{ex9.10} for GAP computations). 
\end{example}
\begin{theorem}[Main theorem 3]\label{thmain3}
Let $k$ be a field. 
There exist exactly $129$ weak stably birationally $k$-equivalent classes 
of algebraic $k$-tori $T$ of dimension $4$ which consist of 
the stably rational class ${\rm WSEC}_0$, 
$121$ classes ${\rm WSEC}_r$ $(1\leq r\leq 121)$ for 
$[\widehat{T}]^{fl}$ with $\widehat{T}=M_G$ and $G=N_{31,i}$ $(1\leq i\leq 64)$ as in {\rm Table} $3$ 
and for 
$[\widehat{T}]^{fl}$ with $\widehat{T}=M_G$ and 
$G=N_{4,i}$ $(1\leq i\leq 152)$ as in {\rm Table} $4$, 
and $7$ classes ${\rm WSEC}_r$ $(122\leq r\leq 128)$ for 
$[\widehat{T}]^{fl}$ with $\widehat{T}=M_G$ and 
$G=I_{4,i}$ $(1\leq i\leq 7)$ as in {\rm Table} $5$. 
Moreover, 
the first $13$ weak stably birationally $k$-equivalent classes 
${\rm WSEC}_r$ $(1\leq r\leq 13)$ are common in 
dimension $3$ $(${\rm Theorem} $\ref{thmain1}$ and {\rm Table} $2$$)$ 
and dimension $4$ $(${\rm Table} $3$ and {\rm Table} $4$$)$. 
\end{theorem}
%

Table $3$: Weak stably $k$-equivalent classes for $[\widehat{T}]^{fl}$ with $\widehat{T}=M_G$ and $G=N_{31,i}$ $(1\leq i\leq 64)$
{\small 
\begin{longtable}{cll}
$r$ & $G=N_{31,i}: [\widehat{T}]^{fl}=[M_G]^{fl}\in {\rm WSEC}_r$ & ${\rm G}_r$: $G\simeq {\rm G}_r$ or $G\simeq {\rm G}_r\times C_2$\\\hline\vspace*{-4mm}
\endfirsthead
$1$ & $N_{31,1}$, \fbox{$N_{31,4}$}, $N_{31,11}$ & $C_2^2$\\
$2$ & $N_{31,3}$, $N_{31,6}$, $N_{31,8}$, $N_{31,10}$, $N_{31,13}$ & $C_2^3$\\
$3$ & $N_{31,2}$, $N_{31,5}$, $N_{31,7}$, $N_{31,9}$, $N_{31,12}$ & $C_2^3$\\
$4$ & $N_{31,14}$, $N_{31,18}$, $N_{31,25}$, $N_{31,26}$, $N_{31,29}$ & $C_4 \times C_2$\\
$5$ & \fbox{$N_{31,15}$}, \fbox{$N_{31,16}$}, \fbox{$N_{31,19}$}, \fbox{$N_{31,20}$}, \fbox{$N_{31,21}$}, \fbox{$N_{31,22}$}, \fbox{$N_{31,27}$}, \fbox{$N_{31,28}$}, $N_{31,31}$, $N_{31,35}$ & $D_4$\\
$6$ & $N_{31,17}$, $N_{31,23}$, $N_{31,24}$, $N_{31,30}$, $N_{31,32}$, $N_{31,33}$, $N_{31,34}$, $N_{31,36}$, $N_{31,37}$ & $D_4\times C_2$\\
$7$ & \fbox{$N_{31,38}$}, \fbox{$N_{31,48}$} & $A_4$\\
$8$ & \fbox{$N_{31,40}$}, \fbox{$N_{31,47}$}, $N_{31,53}$ & $A_4\times C_2$\\
$9$ & \fbox{$N_{31,39}$}, \fbox{$N_{31,46}$}, $N_{31,52}$ & $A_4\times C_2$\\
$10$ & \fbox{$N_{31,41}$}, \fbox{$N_{31,43}$}, \fbox{$N_{31,50}$}, \fbox{$N_{31,51}$}, $N_{31,60}$, $N_{31,62}$ & $S_4$\\
$11$ & \fbox{$N_{31,42}$}, \fbox{$N_{31,49}$}, $N_{31,61}$ & $S_4$\\
$12$ & $N_{31,45}$, $N_{31,55}$, $N_{31,57}$, $N_{31,59}$, $N_{31,64}$ & $S_4\times C_2$\\
$13$ & \fbox{$N_{31,44}$}, \fbox{$N_{31,54}$}, \fbox{$N_{31,56}$}, \fbox{$N_{31,58}$}, $N_{31,63}$ & $S_4\times C_2$
\end{longtable}
}
\vspace*{-2mm}
The box cases \fbox{$N_{31,i}$} satisfy $X=Y=Z$ 
where ${\rm Inn}(G)\leq X\leq Y\leq Z\leq {\rm Aut}(G)$ 
are as in Definition \ref{defXYZ} (see Example \ref{ex9.7} for GAP computations).\\

Table $4$: Weak stably $k$-equivalent classes for $[\widehat{T}]^{fl}$ with $\widehat{T}=M_G$ and $G=N_{4,i}$ $(1\leq i\leq 152)$
{\small 
\setlength{\tabcolsep}{2pt}
\begin{longtable}{cllcc}
 & & & & $\lambda_r=|{\rm WSEC}_{r,L}|$\\
$r$ & $G=N_{4,i}: [\widehat{T}]^{fl}=[M_G]^{fl}\in {\rm WSEC}_r$ & ${\rm G}_r\simeq G$ & $Y\backslash {\rm Aut}(G)$ & \hspace*{5.5mm}$=|Y\backslash {\rm Aut}(G)|$\\\hline\vspace*{-4mm}
\endfirsthead
$1$ & $N_{4,1}$ & $C_2^2$ & $\{1\}$ & $1$\\
$2$ & $N_{4,6}$ & $C_2^3$ & ${\rm PGL}(3,\bF_2)\leq S_7$ & $7$\\
$3$ & $N_{4,2}$, $N_{4,5}$, $N_{4,7}$ & $C_2^3$ & $\{1\}$ & $1$\\
$4$ & $N_{4,13}$, $N_{4,23}$, $N_{4,24}$ & $C_4 \times C_2$ & $\{1\}$ & $1$\\
$5$ & \fbox{$N_{4,15}$}, \fbox{$N_{4,16}$}, \fbox{$N_{4,17}$}, \fbox{$N_{4,25}$}, \fbox{$N_{4,26}$} & $D_4$ & $C_2\leq S_2$ & $2$\\
$6$ & $N_{4,19}$, $N_{4,29}$, $N_{4,30}$ & $D_4\times C_2$ & $D_4\leq S_4$ & $4$\\
$7$ & $$ & $$\\
$8$ & $$ & $$\\
$9$ & \fbox{$N_{4,69}$}, \fbox{$N_{4,74}$}, \fbox{$N_{4,75}$} & $A_4\times C_2$ & $\{1\}$ & $1$\\
$10$ & $$ & $$\\
$11$ & \fbox{$N_{4,71}$}, \fbox{$N_{4,77}$}, \fbox{$N_{4,78}$} & $S_4$ & $\{1\}$ & $1$\\
$12$ & $$ & $$\\
$13$ & \fbox{$N_{4,72}$}, \fbox{$N_{4,83}$}, \fbox{$N_{4,85}$} & $S_4\times C_2$ & $C_2\leq S_2$ & $2$\\\vspace*{-4mm}\\\hline
$14$ & $N_{4,3}$, $N_{4,8}$ & $C_2^3$ & ${\rm PGL}(3,\bF_2)\leq S_{21}$ & $21$\\
$15$ & $N_{4,4}$ & $C_2^3$ & $\{1\}$ & $1$\\
$16$ & \fbox{$N_{4,9}$} & $C_2^3$ & ${\rm PGL}(3,\bF_2)\leq S_{28}$ & $28$\\
$17$ & $N_{4,10}$ & $C_2^4$ & $\{1\}$ & $1$\\
$18$ & $N_{4,11}$ & $C_2^4$ & $A_8\leq S_{420}$ & $420$\\
$19$ & \fbox{$N_{4,12}$} & $C_2^4$ & $A_8\leq S_{840}$ & $840$\\
$20$ & $N_{4,14}$, \fbox{$N_{4,39}$} & $C_4\times C_2$ & $\{1\}$ & $1$\\
$21$ & \fbox{$N_{4,18}$} & $D_4$ & $C_2\leq S_2$ & $2$\\
$22$ & $N_{4,20}$, $N_{4,31}$, $N_{4,33}$, $N_{4,35}$ & $D_4\times C_2$ & $C_2\leq S_2$ & $2$\\
$23$ & \fbox{$N_{4,21}$} & $D_4\times C_2$ & $D_4\times C_2\leq S_8$ & $8$\\
$24$ & $N_{4,22}$, \fbox{$N_{4,46}$} & $D_4\times C_2$ & $C_2\leq S_2$ & $2$\\
$25$ & $N_{4,27}$ & $C_4\times C_2^2$ & $\{1\}$ & $1$\\
$26$ & $N_{4,28}$ & $C_4\times C_2^2$ & $(C_2^4\rtimes C_3)\rtimes C_2\leq S_{24}$ & $24$\\
$27$ & \fbox{$N_{4,32}$}, \fbox{$N_{4,34}$} & $D_4\times C_2$ & $D_4\times C_2\leq S_{16}$ & $16$\\
$28$ & $N_{4,36}$ & $D_4\times C_2$ & $C_2^2\leq S_4$ &$4$\\
$29$ & $N_{4,37}$ & $D_4\times C_2^2$ & $((C_2^4\rtimes C_2)\rtimes C_3)\rtimes C_2\leq S_8$ & $8$\\
$30$ & $N_{4,38}$ & $D_4\times C_2^2$ & $((C_2^2\times (C_2^4\rtimes C_2))\rtimes C_3)\rtimes C_2\leq S_{192}$ & $192$\\
$31$ & \fbox{$N_{4,40}$} & $(C_4 \times C_2) \rtimes C_2$ & $C_2^2\leq S_4$ &$4$\\
$32$ & $N_{4,41}$, \fbox{$N_{4,44}$} & $(C_4\times C_2)\rtimes C_2$ & $\{1\}$ & $1$\\
$33$ & \fbox{$N_{4,42}$} & $(C_4 \times C_2) \rtimes C_2$ & $\{1\}$ & $1$\\
$34$ & \fbox{$N_{4,43}$} & $(C_4 \times C_2) \rtimes C_2$ & $C_2\leq S_2$ & $2$\\
$35$ & \fbox{$N_{4,45}$} & $D_4\times C_2$ & $C_2\leq S_2$ & $2$\\
$36$ & \fbox{$N_{4,47}$} & $C_2^4 \rtimes C_2$ & $S_4\times C_2\simeq 12T22\leq S_{12}$ & $12$\\
$37$ & $N_{4,48}$ & $C_2^4 \rtimes C_2$ & $S_4\times C_2\simeq 12T21\leq S_{12}$ & $12$\\
$38$ & $N_{4,49}$ & $C_2^4 \rtimes C_2$ & $\{1\}$ & $1$\\
$39$ & $N_{4,50}$ & $C_4 \rtimes C_4$ & $C_2\leq S_2$ & $2$\\
$40$ & $N_{4,51}$ & $C_4^2$ & $S_3\leq S_3$ & $3$\\
$41$ & $N_{4,52}$ & $D_4 \times C_4$ & $D_4\leq S_4$ & $4$\\
$42$ & $N_{4,53}$ & $(C_4 \times C_2^2) \rtimes C_2$ & $C_2^2\leq S_4$ &$4$\\
$43$ & $N_{4,54}$ & $(C_4 \times C_2^2) \rtimes C_2$ & $C_2\leq S_2$ & $2$\\
$44$ & $N_{4,55}$ & $C_4^2 \rtimes C_2$ & $S_4\leq S_{12}$ & $12$\\
$45$ & $N_{4,56}$ & $D_4^2$ & $((C_2^4\rtimes C_2)\rtimes C_2)\rtimes C_2\leq S_{16}$ & $16$\\
$46$ & $N_{4,57}$ & $C_3^2$ & $\{1\}$ & $1$\\
$47$ & $N_{4,58}$ & $C_6 \times C_3$ & $\{1\}$ & $1$\\
$48$ & \fbox{$N_{4,59}$}, \fbox{$N_{4,91}$} & $S_3\times C_3$ & $\{1\}$ & $1$\\
$49$ & \fbox{$N_{4,60}$}, \fbox{$N_{4,92}$} & $S_3\times C_3$ & $\{1\}$ & $1$\\
$50$ & $N_{4,61}$ & $C_3^2 \rtimes C_2$ & $\{1\}$ & $1$\\
$51$ & $N_{4,62}$ & $C_3^2 \rtimes C_2$ & $\{1\}$ & $1$\\
$52$ & \fbox{$N_{4,63}$}, \fbox{$N_{4,93}$} & $S_3\times C_6$ & $C_2\leq S_2$ & $2$\\
$53$ & $N_{4,64}$ & $(C_3^2 \rtimes C_2) \times C_2$ & $C_2\leq S_2$ & $2$\\
$54$ & \fbox{$N_{4,65}$}, \fbox{$N_{4,94}$} & $S_3^2$ & $\{1\}$ & $1$\\
$55$ & \fbox{$N_{4,66}$}, \fbox{$N_{4,96}$} & $S_3^2$ & $C_2\leq S_2$ & $2$\\
$56$ & \fbox{$N_{4,67}$}, \fbox{$N_{4,95}$} & $S_3^2$ & $\{1\}$ & $1$\\
$57$ & \fbox{$N_{4,68}$}, \fbox{$N_{4,99}$} & $S_3^2 \times C_2$ & $D_4\leq S_4$ & $4$\\
$58$ & \fbox{$N_{4,70}$} & $A_4\times C_2$ & $\{1\}$ & $1$\\
$59$ & $N_{4,73}$ & $S_4\times C_2$ & $\{1\}$ & $1$\\
$60$ & \fbox{$N_{4,76}$} & $A_4\times C_2$ & $\{1\}$ & $1$\\
$61$ & $N_{4,79}$ & $A_4 \times C_2^2$ & $\{1\}$ & $1$\\
$62$ & \fbox{$N_{4,80}$} & $A_4 \times C_2^2$ & $S_3\leq S_6$ & $6$\\
$63$ & $N_{4,81}$, $N_{4,84}$, $N_{4,87}$ & $S_4 \times C_2$ & $\{1\}$ & $1$\\
$64$ & $N_{4,82}$ & $S_4\times C_2$ & $\{1\}$ & $1$\\
$65$ & \fbox{$N_{4,86}$}, \fbox{$N_{4,88}$} & $S_4 \times C_2$ & $C_2\leq S_2$ & $2$\\
$66$ & $N_{4,89}$ & $S_4 \times C_2^2$ & $S_4\leq S_4$ & $4$\\
$67$ & $N_{4,90}$ & $S_4 \times C_2^2$ & $S_4\leq S_{12}$ & $12$\\
$68$ & $N_{4,97}$ & $C_3^2 \rtimes C_4$ & $\{1\}$ & $1$\\
$69$ & $N_{4,98}$ & $C_3^2 \rtimes C_4$ & $\{1\}$ & $1$\\
$70$ & $N_{4,100}$ & $(C_3^2 \rtimes C_4) \times C_2$ & $C_2\leq S_2$ & $2$\\
$71$ & $N_{4,101}$ & $S_3^2 \rtimes C_2$ & $\{1\}$ & $1$\\
$72$ & \fbox{$N_{4,102}$}, \fbox{$N_{4,103}$} & $S_3^2\rtimes C_2$ & $C_2\leq S_2$ & $2$\\
$73$ & $N_{4,104}$ & $S_3^2 \rtimes C_2$ & $\{1\}$ & $1$\\
$74$ & $N_{4,105}$ & $(S_3^2 \rtimes C_2) \times C_2$ & $D_4\leq S_4$ & $4$\\
$75$ & \fbox{$N_{4,106}$} & $Q_8$ & $\{1\}$ & $1$\\
$76$ & \fbox{$N_{4,107}$} & $C_8 \rtimes C_2$ & $\{1\}$ & $1$\\
$77$ & \fbox{$N_{4,108}$} & $QD_8$ & $\{1\}$ & $1$\\
$78$ & \fbox{$N_{4,109}$} & $(C_4 \times C_2) \rtimes C_2$ & $\{1\}$ & $1$\\
$79$ & \fbox{$N_{4,110}$} & $\SL(2,\bF_3)$ & $\{1\}$ & $1$\\
$80$ & \fbox{$N_{4,111}$} & $\SL(2,\bF_3)$ & $\{1\}$ & $1$\\
$81$ & \fbox{$N_{4,112}$} & $C_8 \rtimes C_2^2$ & $C_2\leq S_2$ & $2$\\
$82$ & \fbox{$N_{4,113}$} & $(C_8 \rtimes C_2) \rtimes C_2$ & $C_2\leq S_2$ & $2$\\
$83$ & $N_{4,114}$ & $C_4^2 \rtimes C_2$ & $\{1\}$ & $1$\\
$84$ & \fbox{$N_{4,115}$} & $C_2^3 \rtimes C_4$ & $\{1\}$ & $1$\\
$85$ & \fbox{$N_{4,116}$} & $C_2^3 \rtimes C_4$ & $\{1\}$ & $1$\\
$86$ & \fbox{$N_{4,117}$} & $C_2^3 \rtimes C_2^2$ & $C_2\leq S_2$ & $2$\\
$87$ & $N_{4,118}$ & $\GL(2,\bF_3)$ & $\{1\}$ & $1$\\
$88$ & $N_{4,119}$ & $\GL(2,\bF_3)$ & $\{1\}$ & $1$\\
$89$ & $N_{4,120}$ & $((C_8 \rtimes C_2) \rtimes C_2) \rtimes C_2$ & $C_2\leq S_2$ & $2$\\
$90$ & $N_{4,121}$ & $((C_8 \rtimes C_2) \rtimes C_2) \rtimes C_2$ & $\{1\}$ & $1$\\
$91$ & \fbox{$N_{4,122}$} & $((C_8 \rtimes C_2) \rtimes C_2) \rtimes C_2$ & $C_2\leq S_2$ & $2$\\
$92$ & \fbox{$N_{4,123}$} & $(C_2^4 \rtimes C_2) \rtimes C_2$ & $C_2\leq S_2$ & $2$\\
$93$ & $N_{4,124}$ & $(C_2^4 \rtimes C_2) \rtimes C_2$ & $\{1\}$ & $1$\\
$94$ & $N_{4,125}$ & $(C_2^3 \rtimes C_4) \rtimes C_2$ & $C_2\leq S_2$ & $2$\\
$95$ & \fbox{$N_{4,126}$} & $(C_2^3 \rtimes C_2^2) \rtimes C_3$ & $\{1\}$ & $1$\\
$96$ & \fbox{$N_{4,127}$} & $(C_2^3 \rtimes C_2^2) \rtimes C_3$ & $\{1\}$ & $1$\\
$97$ & $N_{4,128}$ & $D_4^2 \rtimes C_2$ & $D_4\leq S_4$ & $4$\\
$98$ & \fbox{$N_{4,129}$} & $((C_2^4 \rtimes C_2) \rtimes C_2) \rtimes C_3$ & $C_2\leq S_2$ & $2$\\
$99$ & $N_{4,130}$ & $((C_2^4 \rtimes C_2) \rtimes C_2) \rtimes C_3$ & $\{1\}$ & $1$\\
$100$ & \fbox{$N_{4,131}$} & $((C_2^3 \rtimes C_2^2) \rtimes C_3) \rtimes C_2$ & $C_2\leq S_2$ & $2$\\
$101$ & $N_{4,132}$ & $((C_2^3 \rtimes C_2^2) \rtimes C_3) \rtimes C_2$ & $\{1\}$ & $1$\\
$102$ & $N_{4,133}$ & $((C_2^3 \rtimes C_2^2) \rtimes C_3) \rtimes C_2$ & $\{1\}$ & $1$\\
$103$ & $N_{4,134}$ & $((C_2^3 \rtimes C_2^2) \rtimes C_3) \rtimes C_2$ & $\{1\}$ & $1$\\
$104$ & $N_{4,135}$ & $(((C_2^3 \rtimes C_2^2) \rtimes C_3) \rtimes C_2) \rtimes C_2$ & $C_2\leq S_2$ & $2$\\
$105$ & $N_{4,136}$ & $(((C_2^3 \rtimes C_2^2) \rtimes C_3) \rtimes C_2) \rtimes C_2$ & $C_2\leq S_2$ & $2$\\
$106$ & \fbox{$N_{4,137}$} & $Q_8 \times C_3$ & $C_2\leq S_2$ & $2$\\
$107$ & \fbox{$N_{4,138}$} & $\SL(2,\bF_3)$ & $\{1\}$ & $1$\\
$108$ & \fbox{$N_{4,139}$} & $(Q_8 \times C_3) \rtimes C_2$ & $C_2^2\leq S_4$ &$4$\\
$109$ & \fbox{$N_{4,140}$} & $((C_4 \times C_2) \rtimes C_2) \rtimes C_3$ & $\{1\}$ & $1$\\
$110$ & \fbox{$N_{4,141}$} & $\GL(2,\bF_3)$ & $\{1\}$ & $1$\\
$111$ & \fbox{$N_{4,142}$} & $\SL(2,\bF_3) \times C_3$ & $S_3\leq S_6$ & $6$\\
$112$ & \fbox{$N_{4,143}$} & $(C_2^3 \rtimes C_2^2) \rtimes C_3$ & $C_2\leq S_2$ & $2$\\
$113$ & \fbox{$N_{4,144}$} & $\GL(2,\bF_3) \rtimes C_2$ & $C_2\leq S_2$ & $2$\\
$114$ & \fbox{$N_{4,145}$} & $\SL(2,\bF_3) \rtimes C_4$ & $C_2\leq S_2$ & $2$\\
$115$ & \fbox{$N_{4,146}$} & $(\SL(2,\bF_3) \times C_3) \rtimes C_2$ & $D_6\leq S_{12}$ & $12$\\
$116$ & \fbox{$N_{4,147}$} & $(\GL(2,\bF_3) \rtimes C_2) \rtimes C_2$ & $C_2^2\leq S_4$ &$4$\\
$117$ & \fbox{$N_{4,148}$} & $(C_2^3 \rtimes C_2^2) \rtimes C_3^2$ & $C_2\leq S_2$ & $2$\\
$118$ & \fbox{$N_{4,149}$} & $(((C_2^3 \rtimes C_2^2) \rtimes C_3) \rtimes C_2) \rtimes C_3$ & $C_2\leq S_2$ & $2$\\
$119$ & \fbox{$N_{4,150}$} & $(((C_2^3 \rtimes C_2^2) \rtimes C_3) \rtimes C_2) \rtimes C_3$ & $C_2\leq S_2$ & $2$\\
$120$ & \fbox{$N_{4,151}$} & $((C_2^3 \rtimes C_2^2) \rtimes C_3^2) \rtimes C_2$ & $C_2^2\leq S_4$ &$4$\\
$121$ & \fbox{$N_{4,152}$} & $(((C_2^3 \rtimes C_2^2) \rtimes C_3^2) \rtimes C_2) \rtimes C_2$ & $D_4\leq S_8$ & $8$
\end{longtable}
}
\vspace*{-2mm}
The box cases \fbox{$N_{4,i}$} satisfy $X=Y$ 
where ${\rm Inn}(G)\leq X\leq Y\leq Z\leq {\rm Aut}(G)$ 
are as in Definition \ref{defXYZ} (see Example \ref{ex9.3} for GAP computations).\\

Table $5$: Weak stably $k$-equivalent classes for $[\widehat{T}]^{fl}$ with $\widehat{T}=M_G$ and $G=I_{4,i}$ $(1\leq i\leq 7)$
{\small 
\begin{longtable}{cllcc}
$r$ & $G=I_{4,i}: [\widehat{T}]^{fl}=[M_G]^{fl}\in {\rm WSEC}_r$ & ${\rm G}_r\simeq G$ & $Y\backslash {\rm Aut}(G)$ & $\lambda_r=|{\rm WSEC}_{r,L}|=|Y\backslash {\rm Aut}(G)|$\\\hline\vspace*{-4mm}
\endfirsthead
$122$ & \fbox{$I_{4,1}$} & $F_{20}$ & $\{1\}$ & $1$\\
$123$ & \fbox{$I_{4,2}$} & $F_{20}$ & $\{1\}$ & $1$\\
$124$ & \fbox{$I_{4,3}$} & $F_{20} \times C_2$ & $C_2\leq S_2$ & $2$\\
$125$ & \fbox{$I_{4,4}$} & $S_5$ & $\{1\}$ & $1$\\
$126$ & \fbox{$I_{4,5}$} & $S_5$ & $\{1\}$ & $1$\\
$127$ & \fbox{$I_{4,6}$} & $S_5 \times C_2$ & $C_2\leq S_2$ & $2$\\
$128$ & \fbox{$I_{4,7}$} & $C_3 \rtimes C_8$ & $C_2\leq S_2$ & $2$
\end{longtable}
}
\vspace*{-2mm}
The box cases \fbox{$I_{4,i}$} satisfy $X=Y$ 
where ${\rm Inn}(G)\leq X\leq Y\leq Z\leq {\rm Aut}(G)$ 
are as in Definition \ref{defXYZ} (see Example \ref{ex9.4} for GAP computations).\\

In Table $4$ and Table $5$, 
the fourth column $Y\backslash {\rm Aut}(G)$ means that 
the action of ${\rm Aut}(G)$ on ${\rm WSEC}_{r,L}$ 
with the stabilizer $Y$ is given as the 
$\lambda_r$-th transitive subgroup 
$\lambda_rTm\leq S_{\lambda_r}$ 
where 
$\lambda_r=|{\rm WSEC}_{r,L}|=|Y\backslash {\rm Aut}(G)|$ 
(see Theorem \ref{thmain4} and Theorem \ref{thmain5}). 

The cases $I_{4,1}\simeq F_{20}$, 
$I_{4,4}\simeq S_5$ 
correspond to the norm one tori 
$T=R^{(1)}_{K/k}(\bG_{m,K})$ with $G\simeq {\rm Gal}(L/k)$ where 
$L/k$ is the Galois closure of $K/k$. 
\vspace*{2mm}

\begin{remark}\label{rem1.25}
For $G=N_{3,i}$, $N_{4,i}$ or $I_{4,i}$, 
and $G^\prime=N_{3,j}$, $N_{4,j}$ or $I_{4,j}$, 
it follows from Theorem \ref{thmain1} and Theorem \ref{thmain3} that 
if $[M_G]^{fl}\sim [M_{G^\prime}]^{fl}$, then $G\simeq G^\prime$. 
However, in general, it occurs that 
$[M_G]^{fl}\sim [M_{G^\prime}]^{fl}$ and $G\not\simeq G^\prime$ 
(see Theorem \ref{th1.24}, Table $3$ and Remark \ref{rem1.41} (2)). 
We can also find that there exists $G^{\prime\prime}=N_{5,324}\leq \GL(5,\bZ)$ 
such that $[M_{G^{\prime\prime}}]^{fl}\sim [M_G]^{fl}$ 
with indecomposable $G^{\prime\prime}$-lattice $M_{G^{\prime\prime}}$, 
$G^{\prime\prime}\simeq D_4$ with CARAT ID (5,100,12) 
(see Hoshi and Yamasaki \cite[Chapter 3]{HY17}) 
and $G=N_{3,1}\simeq C_2^2$ 
(see Remark \ref{rem1.41} (3), Example \ref{ex1.42}). 
\end{remark}


%
\begin{theorem}[Main theorem 4]\label{thmain4}
Let $k$ be a field.  
Let $T_i$ and $T_j^\prime$ $(1\leq i,j\leq 152)$ 
be algebraic $k$-tori of dimension $4$ 
with the minimal splitting fields $L_i$ and $L^\prime_j$ 
and the character modules 
$\widehat{T}_i=M_G$ and $\widehat{T}^\prime_j=M_{G^\prime}$ 
which satisfy that 
$G\leq \GL(4,\bZ)$ and $G^\prime\leq \GL(4,\bZ)$ are 
$\GL(4,\bZ)$-conjugate to $N_{4,i}$ and $N_{4,j}$ respectively. 
For $1\leq i,j\leq 152$ 
except for the cases $i=j=137,139,145,147$, 
the following conditions are equivalent:\\
{\rm (1)} $T_i$ and $T_j^\prime$ are stably birationally $k$-equivalent;\\
{\rm (2)} 
$G\simeq G^\prime$, 
$L_i=L_j^\prime$, 
$T_i\times_k K$ and $T_j^\prime\times_k K$ 
are weak stably birationally $K$-equivalent 
for any $k\subset K\subset L_i$;\\
{\rm (3)} 
$G\simeq G^\prime$, 
$L_i=L_j^\prime$, 
$T_i\times_k K$ and $T_j^\prime\times_k K$ 
are weak stably birationally $K$-equivalent 
for any $k\subset K\subset L_i$ corresponding to 
${\rm WSEC}_r$ $(1\leq r\leq 128)$ as in Table $4$ and Table $5$ 
with $T_i\times_k K$ not stably $K$-rational;\\
{\rm (4)} 
$G\simeq G^\prime$, 
$L_i=L_j^\prime$, 
$T_i\times_k K$ and $T_j^\prime\times_k K$ 
are weak stably birationally $K$-equivalent 
for any $k\subset K\subset L_i$ corresponding to 
${\rm WSEC}_r$ $(1\leq r\leq 128)$ as in Table $4$ and Table $5$ 
with $T_i\times_k K$ not stably $K$-rational 
and $[K:k]=d$ 
$($the ``bold'' cases as in {\rm Table} $12$$)$ 
where 
\begin{align*}
d=
\begin{cases}
1&(i=1,2,4,5,7,10,13,14,23,24,27,39,41,42,44,49,57,58,59,60,61,62,65,67,69,70,71,73,\\
&\hspace*{7mm}74,75,76,77, 78,79,81,82,84,87,91,92,94,95,97,98,101,104,106,107,108,109,110,111,\\
&\hspace*{7mm}114,115,116,118,119,121,124,126,127,130,132,133,134,138
,137,139,
140,141,
145),\\
1,2&(i=3,6,8,9,11,12,15,16,17,18,19,20,22,25,26,29,30,31,33,35,37,43,46,50,51,54,55,\\
&\hspace*{7mm}63,64,66,68,72,80,83,85,89,90,93,96,99,100,102,103,105,125),\\
1,3&(i=142,148),\\
1,4&(i=21,36,38,45,47,48,52,53,56,113,117,120),\\
1,2,4&(i=28,32,34,40),\\
1,6&(i=86,88),\\
1,8&(i=112,122,123,128,129,131,135,136),\\
1,3,8&(i=146),\\
1,12&(i=143,147),\\
1,24&(i=144,149),\\
1,32&(i=150),\\
1,2,32&(i=152),\\
1,3,32&(i=151).
\end{cases}
\end{align*}
In particular, for the cases with $d=1$, we get that 
${(\rm 4)}$ 
if and only if $G\simeq G^\prime$, $L_i=L_j^\prime$, 
i.e. $\widetilde{H}\simeq G\simeq G^\prime$. 

For the exceptional cases $i=j=137,139,145,147$ 
$($$G\simeq Q_8\times C_3$, $(Q_8\times C_3)\rtimes C_2$, 
$\SL(2,\bF_3)\rtimes C_4$, $(\GL(2,\bF_3)\rtimes C_2)\rtimes C_2\simeq 
(\SL(2,\bF_3)\rtimes C_4)\rtimes C_2$$)$, 
we have the implications 
${\rm (1)}\Rightarrow {\rm (2)}\Leftrightarrow {\rm (3)}
\Leftrightarrow {\rm (4)}$ and if $G\simeq G^\prime$, $L_i=L_j^\prime$, then 
there exists $\sigma\in {\rm Aut}(G)$ such that $G^\prime=G^\sigma$ 
and $X=Y\lhd Z$ with $Z/Y\simeq C_2,C_2^2,C_2,C_2$ respectively 
$($this means that the implication $(1)\Rightarrow (2)$ 
can not be reversed in general$)$
where ${\rm Inn}(G)\leq X\leq Y\leq Z\leq {\rm Aut}(G)$ 
are as in Definition \ref{defXYZ} 
%
and we have 
${\rm (1)}$ 
$\Leftrightarrow M_G\simeq M_{G^\sigma}$ as $\widetilde{H}$-lattices 
$\Leftrightarrow M_G\otimes_\bZ \bF_p\simeq M_{G^\sigma}\otimes_\bZ \bF_p$ 
as $\bF_p[\widetilde{H}]$-lattices for
\begin{align*}
p= 
\begin{cases}
2&(i=j=137),\\
2\ \,{\rm and}\ \,3&(i=j=139),\\
3&(i=j=145,147).
\end{cases}
\end{align*}

Moreover, for $1\leq r\leq 121$, $r\neq 7,8,10,12$, 
we get the following disjoint union decompositions 
\begin{align*}
{\rm WSEC}_r=\coprod_{L/k\atop {\rm Gal}(L/k)\simeq {\rm G}_r}
{\rm WSEC}_{r,L},\ 
{\rm WSEC}_{r,L}=
\coprod_{t=1}^{\lambda_{r}} {\rm SEC}_{r,L,t}
\end{align*}
modulo stably birationally $k$-equivalence $\stackrel{\rm s.b.}{\approx}$ 
where 
${\rm SEC}_{r,L,t}$ $(1\leq t\leq \lambda_{r})$ is the $t$-th 
stably $k$-equivalent class of $T$ of dimension $4$ in 
${\rm WSEC}_{r,L}$ which corresponds to the fixed minimal splitting field $L$ 
in 
${\rm WSEC}_r=\coprod_{L/k\atop {\rm Gal}(L/k)\simeq {\rm G}_r} {\rm WSEC}_{r,L}$ 
with 
$[\widehat{T}]^{fl}\in {\rm WSEC}_{r,L}
=\coprod_{t=1}^{\lambda_r} {\rm SEC}_{r,L,t}$, 
${\rm Gal}(L/k)\simeq {\rm G}_r\simeq N_{4,i}$ 
$(1\leq r\leq 121, r\neq 7,8,10,12)$ and 
$\lambda_r=|{\rm WSEC}_{r,L}|=|Y\backslash {\rm Aut}(G)|$ 
is given as in Table $4$. 
Furthermore, 
for the box cases \fbox{$N_{4,i}$} with $X=Y$ as in Table $4$,
the following conditions are also equivalent:\\
{\rm (0)} $T_i$ and $T_i^\prime$ are birationally $k$-equivalent;\\
{\rm (1)} $T_i$ and $T_i^\prime$ are stably birationally $k$-equivalent. 

In particular, for the box cases \fbox{$N_{4,i}$}, 
all the algebraic tori $T$ of dimension $4$ 
with $\widehat{T}=M_G$ and $G=N_{4,i}$ 
which correspond to the fixed minimal splitting field $L$ 
split into 
$\lambda_r=|Y\backslash {\rm Aut}(G)|$ birationally $k$-equivalent classes. 
\end{theorem}
\begin{remark}
(1) 
Under the assumption $G\simeq G^\prime$, $L_i=L^\prime_j$, i.e. 
$\widetilde{H}\simeq G\simeq G^\prime$, 
${\rm (1)}\Leftrightarrow {\rm (2)}$ of Theorem \ref{thmain4} 
is equivalent to $Y=Z$ (see Theorem \ref{thmXYZ}). 
In particular, if ${\rm Inn}(G)={\rm Aut}(G)$, then 
${\rm (1)}\Leftrightarrow {\rm (2)}$ of Theorem \ref{thmain4} holds.\\ 
(2) 
We give the condition of Theorem \ref{thmain4} {\rm (4)} 
for $d=[K:k]\in l$ as 
the list $l=[l_1,\ldots,l_s]$ $(l_1\leq \cdots\leq l_s)$ 
of integers with the minimal $l_s,\ldots,l_1$ (see Section \ref{S10}).\\
(3) 
In Theorem \ref{thmain4} {\rm (4)}, 
${\rm Aut}(G)$ may permute subgroups $H, H^\prime\leq G$ 
with $K=L_i^H$, $K^\prime=L_i^{H^\prime}$ corresponding to 
${\rm WSEC}_0$, ${\rm WSEC}_r$ $(r\geq 1)$ respectively.\\
(4) 
By Theorem \ref{thmain2} and Theorem \ref{thmain4} that 
if $T_i$ and $T_j^\prime$ are stably birationally $k$-equivalent, 
then $G\simeq G^\prime$, $L_i=L_j^\prime$. 
However, in general, it occurs that 
(i) $T_i$ and $T_j^\prime$ are stably birationally $k$-equivalent 
and $G\not\simeq G^\prime$; 
(ii) $T_i$ and $T_j^\prime$ are stably birationally $k$-equivalent, 
$G\simeq G^\prime$, $L_i\neq L_j^\prime$.  
We see that there exists $G^{\prime\prime}=N_{5,324}\leq \GL(5,\bZ)$ 
such that $[M_{G^{\prime\prime}}]^{fl}=[M_G]^{fl}$ 
as $G^{\prime\prime}$-lattices 
with indecomposable $G^{\prime\prime}$-lattice $M_{G^{\prime\prime}}$, 
$G^{\prime\prime}\simeq D_4$ with CARAT ID (5,100,12) 
(see Hoshi and Yamasaki \cite[Chapter 3]{HY17}) 
and $G=N_{3,1}\simeq C_2^2\simeq G^{\prime\prime}/Z(G^{\prime\prime})$ 
(see Remark \ref{rem1.41} (2), (3), Example \ref{ex1.42}). 
\end{remark}

\begin{example}
(1) 
For the exceptional case $i=j=137$ 
$($$G\simeq Q_8\times C_3$$)$ 
in Theorem \ref{thmain4}, we have 
$G=N_{4,137}=\langle\gamma_1,\gamma_2\mid \gamma_1^{12}=I, \gamma_2^2=\gamma_1^6, [\gamma_1,\gamma_2]=\gamma_1^6\rangle$ with GAP ID $(4,33,1,1)$ 
(see Table $8$ and \cite[Table 1C, page 257]{BBNWZ78}), 
$G^\tau=\langle\gamma_1^\tau,\gamma_2^\tau\rangle$ and 
$Z/Y=\{{\rm id}_G,\tau\}$ 
where 
\begin{align*}
\gamma_1=\gamma_1^\tau=\left(\scalemath{0.9}{
\begin{array}{cccc}
0&1&0&-1\\
0&0&-1&1\\
-1&0&0&1\\
0&1&0&0
\end{array}}\right),\ 
\gamma_2=\left(\scalemath{0.9}{
\begin{array}{cccc}
0&0&-1&0\\
-1&-1&-1&2\\
1&0&0&0\\
0&-1&-1&1
\end{array}}\right),\ 
\gamma_2^\tau=\left(\scalemath{0.9}{
\begin{array}{cccc}
0&1&0&0\\
-1&0&0&0\\
-1&-1&-1&2\\
-1&0&-1&1
\end{array}}\right). 
\end{align*}
(2) 
For the exceptional case $i=j=139$ 
$($$G\simeq (Q_8\times C_3)\rtimes C_2$$)$ 
in Theorem \ref{thmain4}, we have 
$G=N_{4,137}=\langle\gamma_1,\gamma_2,\gamma_3\mid \gamma_1^{12}=I, \gamma_2^2=\gamma_1^6, \gamma_3^2=I, [\gamma_1,\gamma_2]=\gamma_1^6, 
[\gamma_1,\gamma_3]=\gamma_1^{10}, 
[\gamma_2,\gamma_3]=\gamma_1^9\rangle$ 
with GAP ID $(4,33,4,1)$ 
(see Table $8$ and \cite[Table 1C, page 257]{BBNWZ78}) 
and $Z/Y=\{{\rm id}_G,\tau,\tau^\prime,\tau\tau^\prime\}$ 
where $\gamma_1=\gamma_1^\tau$, $\gamma_2$ and $\gamma_2^\tau$ 
are given as in (1), 
\begin{align*}
\gamma_3=\left(\scalemath{0.9}{
\begin{array}{cccc}
1&0&0&0\\
0&0&-1&0\\
0&-1&0&0\\
1&0&0&-1
\end{array}}\right),\ 
\gamma_3^\tau=\left(\scalemath{0.9}{
\begin{array}{cccc}
-1&-1&-1&2\\
0&1&0&0\\
0&0&-1&0\\
0&0&-1&1
\end{array}}\right),\ 
\gamma_1^{\tau^\prime}=\gamma_1^{\tau\tau^\prime}=\left(\scalemath{0.9}{
\begin{array}{cccc}
1&0&1&-1\\
0&0&0&-1\\
1&1&0&-1\\
1&0&0&-1
\end{array}}\right), 
\end{align*}
$\gamma_2^{\tau^\prime}=\gamma_2$, 
$\gamma_2^{\tau\tau^\prime}=\gamma_2^\tau$, 
$\gamma_3^{\tau^\prime}=-\gamma_3$, 
$\gamma_3^{\tau\tau^\prime}=-\gamma_3^\tau$.\\
(3) 
For the exceptional case $i=j=145$ 
$($$G\simeq \SL(2,\bF_3)\rtimes C_4$$)$ 
in Theorem \ref{thmain4}, we have 
$G=N_{4,137}=\langle\alpha_1,\alpha_2,\alpha_3,\alpha_4\mid 
\alpha_1^4=I, \alpha_2^2=\alpha_1^2, \alpha_3^3=I, \alpha_4^4=I, 
[\alpha_1,\alpha_2]=\alpha_1^2, 
[\alpha_1,\alpha_3]=\alpha_2\alpha_1^2, 
[\alpha_1,\alpha_4]=I, 
[\alpha_2,\alpha_3]=\alpha_2\alpha_1, 
[\alpha_2,\alpha_4]=\alpha_1, 
[\alpha_3,\alpha_4]=\alpha_3\alpha_2\alpha_1^2
\rangle$ 
with GAP ID $(4,33,10,1)$ 
(see Table $8$ and \cite[Table 1C, page 258]{BBNWZ78}) 
and $Z/Y=\{{\rm id}_G,\tau\}$ 
where 
\begin{align*}
\alpha_1=\left(\scalemath{0.9}{
\begin{array}{cccc}
-1&-1&-1&2\\
0&0&1&0\\
0&-1&0&0\\
-1&-1&0&1
\end{array}}\right),\ 
\alpha_2=\left(\scalemath{0.9}{
\begin{array}{cccc}
0&0&-1&0\\
-1&-1&-1&2\\
1&0&0&0\\
0&-1&-1&1
\end{array}}\right),\ 
\alpha_3=\left(\scalemath{0.9}{
\begin{array}{cccc}
0&0&0&-1\\
0&-1&-1&1\\
1&1&0&-1\\
1&0&0&-1
\end{array}}\right),\ 
\alpha_4=\left(\scalemath{0.9}{
\begin{array}{cccc}
-1&-1&-1&2\\
0&1&0&0\\
0&0&1&0\\
-1&0&0&1
\end{array}}\right),
\end{align*}
$\alpha_1^\tau=\alpha_1$, 
$\alpha_2^\tau=\alpha_2$, 
$\alpha_3^\tau=\alpha_3$, 
$\alpha_4^\tau=-\alpha_4$.\\
(4) 
For the exceptional case $i=j=147$ 
$($$G\simeq (\GL(2,\bF_3)\rtimes C_2)\rtimes C_2\simeq 
(\SL(2,\bF_3)\rtimes C_4)\rtimes C_2$$)$ 
in Theorem \ref{thmain4}, we have 
$G=N_{4,137}=\langle\beta_1,\beta_2,\beta_3,\beta_4,\beta_5\mid 
\beta_1^2=I, \beta_2^4=I, \beta_3^2=I, \beta_4^6=\beta_2^2, \beta_5^2=I, 
[\beta_1,\beta_2]=I, 
[\beta_1,\beta_3]=\beta_2^2, 
[\beta_1,\beta_4]=\beta_3\beta_2^3, 
[\beta_1,\beta_5]=I, 
[\beta_2,\beta_3]=I, 
[\beta_2,\beta_4]=\beta_2^2, 
[\beta_2,\beta_5]=\beta_2^2, 
[\beta_3,\beta_4]=\beta_3\beta_2^2\beta_1, 
[\beta_3,\beta_5]=\beta_2\beta_1, 
[\beta_4,\beta_5]=\beta_4^4\beta_3\beta_2\beta_1
\rangle$ 
with GAP ID $(4,33,12,1)$ 
(see Table $8$ and \cite[Table 1C, page 259]{BBNWZ78}) 
and $Z/Y=\{{\rm id}_G,\tau\}$ 
where  
\begin{align*}
\beta_1&=\left(\scalemath{0.9}{
\begin{array}{cccc}
1&0&0&0\\
0&-1&0&0\\
0&0&-1&0\\
0&-1&-1&1
\end{array}}\right),\ 
\beta_2=\left(\scalemath{0.9}{
\begin{array}{cccc}
-1&-1&-1&2\\
0&0&1&0\\
0&-1&0&0\\
-1&-1&0&1
\end{array}}\right),\ 
\beta_3=\left(\scalemath{0.9}{
\begin{array}{cccc}
0&0&-1&0\\
-1&-1&-1&2\\
-1&0&0&0\\
-1&0&-1&1
\end{array}}\right),\ 
\beta_4=\left(\scalemath{0.9}{
\begin{array}{cccc}
-1&0&-1&1\\
0&0&0&-1\\
-1&-1&0&1\\
-1&-1&-1&1
\end{array}}\right),
\end{align*}
\vspace*{-2mm}
\begin{align*}
\beta_5&=\left(\scalemath{0.9}{
\begin{array}{cccc}
1&0&0&0\\
0&0&1&0\\
0&1&0&0\\
1&1&1&-1
\end{array}}\right),\ 
\beta_1^\tau=\left(\scalemath{0.9}{
\begin{array}{cccc}
0&0&1&0\\
1&1&1&-2\\
1&0&0&0\\
1&0&1&-1
\end{array}}\right),\ 
\beta_3^\tau=\left(\scalemath{0.9}{
\begin{array}{cccc}
0&-1&0&0\\
-1&0&0&0\\
1&1&1&-2\\
0&0&0&-1
\end{array}}\right),\ 
\beta_5^\tau\left(\scalemath{0.9}{
\begin{array}{cccc}
0&-1&-1&1\\
0&0&0&-1\\
-1&-1&0&1\\
0&-1&0&0
\end{array}}\right),
\end{align*}
$\beta_2^\tau=-\beta_2$, $\beta_4^\tau=\beta_4$ (see Example \ref{ex9.5} for GAP computations). 
\end{example}
\begin{theorem}[Main theorem 5]\label{thmain5}
Let $k$ be a field. 
Let $T_i$ and $T_j^\prime$ $(1\leq i,j\leq 7)$ 
be algebraic $k$-tori of dimension $4$ 
with the minimal splitting fields $L_i$ and $L^\prime_j$ 
and the character modules 
$\widehat{T}_i=M_G$ and $\widehat{T}^\prime_j=M_{G^\prime}$ 
which satisfy that 
$G\leq \GL(4,\bZ)$ and $G^\prime\leq \GL(4,\bZ)$ are 
$\GL(4,\bZ)$-conjugate to $I_{4,i}$ and $I_{4,j}$ respectively. 
For $1\leq i,j\leq 7$ except for the case $i=j=7$, 
the following conditions are equivalent:\\
{\rm (1)} $T_i$ and $T_j^\prime$ are stably birationally $k$-equivalent;\\
{\rm (2)} 
$G\simeq G^\prime$, 
$L_i=L_j^\prime$, 
$T_i\times_k K$ and $T_j^\prime\times_k K$ 
are weak stably birationally $K$-equivalent 
for any $k\subset K\subset L_i$;\\
{\rm (3)} 
$G\simeq G^\prime$, 
$L_i=L_j^\prime$, 
$T_i\times_k K$ and $T_j^\prime\times_k K$ 
are weak stably birationally $K$-equivalent 
for any $k\subset K\subset L_i$ corresponding to 
${\rm WSEC}_r$ $(122\leq r\leq 128)$ as in Table $5$ 
with $T_i\times_k K$ not stably but retract $K$-rational;\\
{\rm (4)} 
$G\simeq G^\prime$, 
$L_i=L_j^\prime$, 
$T_i\times_k K$ and $T_j^\prime\times_k K$ 
are weak stably birationally $K$-equivalent 
for any $k\subset K\subset L_i$ corresponding to 
${\rm WSEC}_r$ $(122\leq r\leq 128)$ as in Table $5$ 
with $T_i\times_k K$ not stably but retract $K$-rational 
and $[K:k]=d$ 
$($the ``bold'' cases as in {\rm Table} $13$$)$
where 
\begin{align*}
d=
\begin{cases}
1&(i=1,2,4,5,7),\\
1,2&(i=3,6).
\end{cases}
\end{align*}
In particular, for $i=1,2,4,5,7$ with $d=1$, we get that 
${(\rm 4)}$ 
if and only if $G\simeq G^\prime$, $L_i=L_j^\prime$, 
i.e. $\widetilde{H}\simeq G\simeq G^\prime$. 

For the exceptional case  $i=j=7$ $($$G\simeq C_3\rtimes C_8$$)$, 
we have the implications 
${\rm (1)}\Rightarrow {\rm (2)}\Leftrightarrow {\rm (3)}
\Leftrightarrow {\rm (4)}$ and if $G\simeq G^\prime$, $L_i=L_j^\prime$, then 
there exists $\sigma\in {\rm Aut}(G)$ such that $G^\prime=G^\sigma$ 
and $X=Y\lhd Z$ with $Z/Y\simeq C_2$ 
$($this means that the implication $(1)\Rightarrow (2)$ 
can not be reversed in general$)$ 
where ${\rm Inn}(G)\simeq S_3\leq X=Y\simeq D_6\leq Z={\rm Aut}(G)\simeq S_3\times C_2^2$ 
are as in Definition \ref{defXYZ} 
%
and we have 
${\rm (1)}$ 
$\Leftrightarrow M_G\simeq M_{G^\sigma}$ as $\widetilde{H}$-lattices 
$\Leftrightarrow M_G\otimes_\bZ \bF_3\simeq M_{G^\sigma}\otimes_\bZ \bF_3$ 
as $\bF_3[\widetilde{H}]$-lattices.  

Moreover, for $122\leq r\leq 128$, 
we get the following disjoint union decompositions 
\begin{align*}
{\rm WSEC}_r=\coprod_{L/k\atop {\rm Gal}(L/k)\simeq {\rm G}_r}
{\rm WSEC}_{r,L},\ 
{\rm WSEC}_{r,L}=
\coprod_{t=1}^{\lambda_{r}} {\rm SEC}_{r,L,t}
\end{align*}
modulo stably birationally $k$-equivalence $\stackrel{\rm s.b.}{\approx}$ 
where 
${\rm SEC}_{r,L,t}$ $(1\leq t\leq \lambda_{r})$ is the $t$-th 
stably $k$-equivalent class of $T$ of dimension $4$ in 
${\rm WSEC}_{r,L}$ which corresponds to the fixed minimal splitting field $L$ 
in 
${\rm WSEC}_r=\coprod_{L/k\atop {\rm Gal}(L/k)\simeq {\rm G}_r} {\rm WSEC}_{r,L}$ 
with 
$[\widehat{T}]^{fl}\in {\rm WSEC}_{r,L}
=\coprod_{t=1}^{\lambda_r} {\rm SEC}_{r,L,t}$, 
${\rm Gal}(L/k)\simeq {\rm G}_r\simeq I_{4,i}$ 
$(122\leq r\leq 128)$ and 
$\lambda_r=|{\rm WSEC}_{r,L}|=|Y\backslash {\rm Aut}(G)|$ 
is given as in Table $5$.
Furthermore, 
for the box cases \fbox{$I_{4,i}$} with $X=Y$ as in Table $5$,
the following conditions are also equivalent:\\
{\rm (0)} $T_i$ and $T_i^\prime$ are birationally $k$-equivalent;\\
{\rm (1)} $T_i$ and $T_i^\prime$ are stably birationally $k$-equivalent. 

In particular, for the box cases \fbox{$I_{4,i}$}, 
all the algebraic tori $T$ of dimension $4$ 
with $\widehat{T}=M_G$ and $G=I_{4,i}$ 
which correspond to the fixed minimal splitting field $L$ 
split into 
$\lambda_r=|Y\backslash {\rm Aut}(G)|$ birationally $k$-equivalent classes. 
\end{theorem}
\begin{example}
For the exceptional case  $i=j=7$ $($$G\simeq C_3\rtimes C_8$$)$ 
in Theorem \ref{thmain5}, we have 
$G=I_{4,7}=\langle\gamma_1,\gamma_4\mid\gamma_1^{12}=I, \gamma_4^2=\gamma_1^9, 
[\gamma_1,\gamma_4]=\gamma_1^4\rangle$ with GAP ID $(4,33,2,1)$ 
(see Table $9$ and \cite[Table 1C, page 257]{BBNWZ78}) 
and 
$Z/Y=\{{\rm id}_G,\tau\}$ 
where 
\begin{align*}
\gamma_1=\left(\scalemath{0.9}{\begin{array}{cccc}
0&1&0&-1\\
0&0&-1&1\\
-1&0&0&1\\
0&1&0&0
\end{array}}\right),\ 
\gamma_4=
\left(\scalemath{0.9}{\begin{array}{cccc}
0&0&-1&0\\
-1&0&0&0\\
1&1&1&-2\\
0&1&0&-1
\end{array}}\right),
\end{align*}
$\gamma_1^\tau=\gamma_1$, 
$\gamma_4^\tau=-\gamma_4$ (see Example \ref{ex9.6} for GAP computations). 
\end{example}

As a corollary of Theorem \ref{thmain3} (Main theorem $3$), 
Theorem \ref{thmain4} (Main theorem $4$) 
and Theorem \ref{thmain5} (Main theorem $5$), 
we get an answer of Problem \ref{prob1.3} 
for algebraic $k$-tori $T$ of dimensions $4$, 
i.e. the structure of $\mathcal{T}_4/\stackrel{\rm s.b.}{\approx}$: 

\begin{corollary}[The structure of $\mathcal{T}_4/\stackrel{\rm s.b.}{\approx}$]\label{cor1.38}
Let $\mathcal{T}_4$ be the category of algebraic $k$-tori of dimension $4$. 
We get a classification $($disjoint union decomposition$)$ 
of $\mathcal{T}_4$ with respect to 
the stably birationally $k$-equivalence $\stackrel{\rm s.b.}{\approx}$: 
\begin{align*}
\mathcal{T}_4=\coprod_{r=0}^{128} {\rm WSEC}_r
={\rm SEC}_0\coprod
\left(\coprod_{r=1}^{128} 
\coprod_{L/k\atop {\rm Gal}(L/k)\simeq {\rm G}_r}
\coprod_{t=1}^{\lambda_{r}} {\rm SEC}_{r,L,t}\right)
\end{align*}
modulo stably birationally $k$-equivalence $\stackrel{\rm s.b.}{\approx}$ 
where ${\rm SEC}_0$ is the stably $k$-equivalent class consists of 
stably $k$-rational tori $T\in \mathcal{T}_4$ 
and 
${\rm SEC}_{r,L,t}$ $(1\leq t\leq \lambda_{r})$ is the $t$-th 
stably $k$-equivalent class of $T\in \mathcal{T}_4$ 
in 
${\rm WSEC}_{r,L}$ which corresponds to the fixed minimal splitting field $L$ 
in 
${\rm WSEC}_r=\coprod_{L/k\atop {\rm Gal}(L/k)\simeq {\rm G}_r} {\rm WSEC}_{r,L}$ 
with 
$[\widehat{T}]^{fl}\in {\rm WSEC}_{r,L}
=\coprod_{t=1}^{\lambda_r} {\rm SEC}_{r,L,t}$, 
${\rm Gal}(L/k)\simeq {\rm G}_r\simeq N_{3,i}\in {\rm WSEC}_r$ $(1\leq r\leq 13)$, 
$N_{4,i}\in {\rm WSEC}_r$ $(14\leq r\leq 121)$, $I_{4,i}\in {\rm WSEC}_r$ $(122\leq r\leq 128)$ and 
$\lambda_r=|{\rm WSEC}_{r,L}|=|Y\backslash {\rm Aut}(G)|$ 
is given as in Table $2$, Table $4$ and Table $5$. 
\end{corollary}

For algebraic $k$-tori $T_i$ $(1\leq i\leq 7)$ as in Theorem \ref{thmain5}, 
we find the following relationships among them 
(the statements (1) and (3) of Theorem \ref{thmain6} were proved in 
Hoshi and Yamasaki \cite[Theorem 1.27]{HY17} and we will give 
an alternative proof of them, see also Remark \ref{rem1.19} (2)):

\begin{theorem}[Main theorem 6]\label{thmain6}
Let $k$ be a field. 
Let $T_i$ $(1\leq i\leq 7)$ be an algebraic $k$-torus of dimension $4$ 
with the minimal splitting field $L_i$ and 
the character module $\widehat{T}_i=M_{G_i}$ which satisfies that 
$G_i\leq \GL(4,\bZ)$ is $\GL(4,\bZ)$-conjugate to $I_{4,i}$. 
Let $T_i^\sigma$ be the algebraic $k$-torus 
with the minimal splitting field $L_i$ and 
$\widehat{T}_i^\sigma=M_{G_i^\sigma}$ $(\sigma\in {\rm Aut}(G_i))$. 
Then $T_i$ and $T_i^\sigma$ are not stably but retract $k$-rational, 
i.e. $[M_{G_i}]^{fl}\neq 0$, $[M_{G_i^{\sigma}}]^{fl}\neq 0$ but invertible of infinite order $($see Hoshi and Yamasaki \cite[Theorem 1.27]{HY17} for the last statement$)$. 
We have:\\
{\rm (1)} $(G_1\simeq G_2\simeq F_{20})$ If $L_1=L_2$, then
$[M_{G_1}]^{fl}+[M_{G_2}]^{fl}=0$ as $\widetilde{H}$-lattices 
where $\widetilde{H}=\{(\varphi_1(g),\varphi_2(g))\mid g\in {\rm Gal}(L_1L_2/k)\}\simeq G_1\simeq G_2$ with surjections $\varphi_1: {\rm Gal}(L_1L_2/k)\xrightarrow{\sim} G_1$, $\varphi_2: {\rm Gal}(L_1L_2/k)\xrightarrow{\sim} G_2$. 
In particular, $T_1\times_k T_2$ is stably $k$-rational 
although $T_1$ and $T_2$ are not stably $k$-rational and stably birationally $k$-equivalent;\\
{\rm (2)} $(G_3\simeq F_{20}\times C_2)$ Let $\sigma\in {\rm Aut}(G_3)$ with 
$1\neq \overline{\sigma}\in {\rm Aut}(G_3)/{\rm Inn}(G_3)\simeq C_2$, 
$\sigma(G_1)=G_2$, $\sigma(G_2)=G_1$. 
We have
$[M_{G_3}]^{fl}+[M_{G_3^\sigma}]^{fl}=0$ as $\widetilde{H}$-lattices 
where $\widetilde{H}=\{(g,g^\sigma)\mid g\in G_3\}\simeq G_3$. 
In particular, $T_3\times_k T_3^\prime$ is stably $k$-rational 
although $T_3$ and $T_3^\prime$ are not stably $k$-rational and not stably birationally $k$-equivalent;\\
{\rm (3)} $(G_4\simeq G_5\simeq S_5)$ If $L_4=L_5$, then
$[M_{G_4}]^{fl}+[M_{G_5}]^{fl}=0$ as $\widetilde{H}$-lattices 
where $\widetilde{H}=\{(\varphi_1(g),\varphi_2(g))\mid g\in {\rm Gal}(L_4L_5/k)\}\simeq G_4\simeq G_5$ with surjections $\varphi_1: {\rm Gal}(L_4L_5/k)\xrightarrow{\sim} G_4$, $\varphi_2: {\rm Gal}(L_4L_5/k)\xrightarrow{\sim} G_5$. 
In particular, $T_4\times_k T_5$ is stably $k$-rational 
although $T_4$ and $T_5$ are not stably $k$-rational and stably birationally $k$-equivalent;\\
{\rm (4)} $(G_6\simeq S_5\times C_2)$ Let $\sigma\in {\rm Aut}(G_6)$ with 
$1\neq \overline{\sigma}\in {\rm Aut}(G_6)/{\rm Inn}(G_6)\simeq C_2$, 
$\sigma(G_4)=G_5$, $\sigma(G_5)=G_4$. 
We have 
$[M_{G_6}]^{fl}+[M_{G_6^\sigma}]^{fl}=0$ as $\widetilde{H}$-lattices 
where $\widetilde{H}=\{(g,g^\sigma)\mid g\in G_6\}\simeq G_6$. 
In particular, $T_6\times_k T_6^\prime$ is stably $k$-rational 
although $T_6$ and $T_6^\prime$ are not stably $k$-rational and not stably birationally $k$-equivalent;\\
{\rm (5)} $(G_7\simeq C_3\rtimes C_8)$ Let $\sigma\in {\rm Aut}(G_7)$ with 
$1\neq\overline{\sigma}\in {\rm Aut}(G_7)/X\simeq C_2$ 
where 
$X\simeq D_6$ is as in Definition \ref{defXYZ}. 
We have 
$[M_{G_7}]^{fl}+[M_{G_7^\sigma}]^{fl}=0$ as $\widetilde{H}$-lattices 
where $\widetilde{H}=\{(g,g^\sigma)\mid g\in G_7\}\simeq G_7$. 
In particular, $T_7\times_k T_7^\prime$ is stably $k$-rational 
although $T_7$ and $T_7^\prime$ are not stably $k$-rational and not stably birationally $k$-equivalent. 
\end{theorem}
\begin{remark}
For $I_{4,1}$, $I_{4,2}$, $I_{4,3}$, $I_{4,4}$, $I_{4,5}$, $I_{4,6}$, 
$I_{4,7}$, it follows from Theorem \ref{thmain5} (Main theorem 5) that 
${\rm WSEC}_{r,L}=
\coprod_{t=1}^{\lambda_{r}} {\rm SEC}_{r,L,t}$ with 
$\lambda_r=1$, $1$, $2$, $1$, $1$, $2$, $2$ for 
$r=122$, $123$, $124$, $125$, $126$, $127$, $128$ 
(see Table $5$). 
Theorem \ref{thmain6} (2) (resp. (4), (5)) 
with the aid of Theorem \ref{thmain6} (Main theorem 6) 
implies that $T_3$ and $T_3^\prime$ 
(resp. $T_6$ and $T_6^\prime$, $T_7$ and $T_7^\prime$) 
become a complete system of representatives 
in the $124$th (resp. $127$th, $128$th) weak stably $k$-equivalent class 
${\rm WSEC}_{124,L_3}/\stackrel{\rm s.b.}{\approx}$ 
(resp. ${\rm WSEC}_{127,L_6}/\stackrel{\rm s.b.}{\approx}$, 
${\rm WSEC}_{128,L_7}/\stackrel{\rm s.b.}{\approx}$) 
which corresponds to the fixed minimal splitting field $L_3$ 
(resp. $L_6$, $L_7$) modulo stably $k$-equivalence. 
\end{remark}

The following theorem can answer Problem \ref{prob1.3} 
for algebraic $k$-tori $T$ and $T^\prime$ of 
dimensions $m\geq 3$ and $n\geq 3$ respectively 
with $[\widehat{T}]^{fl}, [\widehat{T}^\prime]^{fl}\in {\rm WSEC}_r$ 
$(1\leq r\leq 128)$ via some algebraic $k$-torus 
$T^{\prime\prime}$ with dim $T^{\prime\prime}=3$ or $4$ 
as in 
{Theorem \ref{thmain2}}, {Theorem \ref{thmain4}} 
or {Theorem \ref{thmain5}}.

\begin{theorem}[Main theorem 7]\label{thmain7}
Let $k$ be a field. 
Let $T$ be an algebraic $k$-torus of dimension $m\geq 3$ 
with the minimal splitting field $L$, 
$\widehat{T}=M_G$, $G\leq \GL(m,\bZ)$ and 
$[\widehat{T}]^{fl}\in {\rm WSEC}_r$ $(1\leq r\leq 128)$. 
Then there exists an algebraic $k$-torus $T^{\prime\prime}$ 
of dimension $3$ or $4$ 
with the minimal splitting field $L^{\prime\prime}$, 
$\widehat{T}^{\prime\prime}=M_{G^{\prime\prime}}$, and 
$G^{\prime\prime}\simeq {\rm Gal}(L^{\prime\prime}/k)\simeq {\rm G}_r\simeq N_{3,i}\in {\rm WSEC}_r$ $(1\leq r\leq 13)$, 
$N_{4,i}\in {\rm WSEC}_r$ $(14\leq r\leq 121)$, $I_{4,i}\in {\rm WSEC}_r$ $(122\leq r\leq 128)$ 
such that $T^{\prime\prime}$ and $T$ are stably birationally $k$-equivalent 
and $L^{\prime\prime}\subset L$, i.e. 
$[M_{G^{\prime\prime}}]^{fl}=[M_G]^{fl}$ as $G$-lattices and 
$G$ acts on $[M_{G^{\prime\prime}}]^{fl}$ though 
$G^{\prime\prime}\simeq G/N$ for the corresponding 
normal subgroup $N\lhd G$. 
\end{theorem}
As a corollary of Theorem \ref{thmain7} (Main theorem $7$), 
we get an answer of Problem \ref{prob1.3} 
for algebraic $k$-tori $T$ 
with $[\widehat{T}]^{fl}\in {\rm WSEC}_r$ $(0\leq r\leq 128)$, 
i.e. the structure of $\mathcal{T}^\prime/\stackrel{\rm s.b.}{\approx}$: 
\begin{corollary}\label{cor1.41}
Let $\mathcal{T}^\prime$ be the category of algebraic $k$-tori $T$ 
with $[\widehat{T}]^{fl}\in {\rm WSEC}_r$ $(0\leq r\leq 128)$. 
We get a classification $($disjoint union decomposition$)$ 
of $\mathcal{T}^\prime$ with respect to 
the stably birationally $k$-equivalence $\stackrel{\rm s.b.}{\approx}$: 
\begin{align*}
\mathcal{T}^\prime=\coprod_{r=0}^{128} {\rm WSEC}_r
={\rm SEC}_0\coprod
\left(\coprod_{r=1}^{128} 
\coprod_{L/k\atop {\rm Gal}(L/k)\simeq {\rm G}_r}
\coprod_{t=1}^{\lambda_{r}} {\rm SEC}_{r,L,t}\right)
\end{align*}
modulo stably birationally $k$-equivalence $\stackrel{\rm s.b.}{\approx}$ 
where ${\rm SEC}_0$ is the stably $k$-equivalent class consists of 
stably $k$-rational tori $T\in \mathcal{T}^\prime$ 
and 
${\rm SEC}_{r,L,t}$ $(1\leq t\leq \lambda_{r})$ is the $t$-th 
stably $k$-equivalent class of $T\in \mathcal{T}^\prime$ in 
${\rm WSEC}_{r,L}$ which corresponds to the fixed minimal splitting field $L$ 
in 
${\rm WSEC}_r=\coprod_{L/k\atop {\rm Gal}(L/k)\simeq {\rm G}_r} {\rm WSEC}_{r,L}$ 
with 
$[\widehat{T}]^{fl}\in {\rm WSEC}_{r,L}
=\coprod_{t=1}^{\lambda_r} {\rm SEC}_{r,L,t}$, 
${\rm Gal}(L/k)\simeq {\rm G}_r\simeq N_{3,i}\in {\rm WSEC}_r$ $(1\leq r\leq 13)$, 
$N_{4,i}\in {\rm WSEC}_r$ $(14\leq r\leq 121)$, $I_{4,i}\in {\rm WSEC}_r$ $(122\leq r\leq 128)$ and 
$\lambda_r=|{\rm WSEC}_{r,L}|=|Y\backslash {\rm Aut}(G)|$ 
is given as in Table $2$, Table $4$ and Table $5$. 
\end{corollary}
\begin{remark}\label{rem1.41}
(1) It follows from Theorem {Theorem \ref{thmain2}}, 
{Theorem \ref{thmain4}} and {Theorem \ref{thmain5}} that 
the field $L^{\prime\prime}$ (the normal subgroup $N\lhd G$) 
can be uniquely determined by $T$. 
For example, consider $G=N_{31,i}$ (see Table $3$).\\
(2) Theorem \ref{thmain7} does not hold for $G^{\prime\prime}=N_{31,i}$. 
For example, take $G=N_{31,11}\simeq G^{\prime\prime}=N_{31,11}\simeq C_2^3$. 
Then it occurs that $[L\cap L^{\prime\prime}:k]=4$ and 
$[M_G]^{fl}=[M_{G^{\prime\prime}}]^{fl}$ as $\widetilde{H}$-lattices 
where $\widetilde{H}={\rm Gal}(LL^{\prime\prime}/k)$. 
This implies that $T$ and $T^{\prime\prime}$ are 
stably birationally $k$-equivalent 
where $\widehat{T}\simeq M_G$ and 
$\widehat{T}^\prime\simeq M_{G^\prime}$ 
but $L\neq L^{\prime\prime}$. 
By Theorem \ref{thmain7}, 
we see that $ G/N\simeq G^{\prime\prime}/N^{\prime\prime}\simeq {\rm G}_1\simeq 
N_{3,1}\simeq C_2^2$ 
for some normal subgroups $N\lhd G$, $N^{\prime\prime}\lhd G^{\prime\prime}$ 
but $G^{\prime\prime}\not\simeq {\rm G}_1\simeq C_2^2$. \\
(3) Theorem \ref{thmain7} does not hold for 
$G^{\prime\prime}=N_{5,324}\leq {\rm GL}(5,\bZ)$. 
Indeed, we can take $G=N_{5,324}$, $G^{\prime\prime}=N_{5,324}\leq \GL(5,\bZ)$ 
with CARAT ID (5,100,12), 
indecomposable $G$-lattice $M_G$, 
indecomposable $G^{\prime\prime}$-lattice $M_{G^{\prime\prime}}$, 
$G\simeq G^{\prime\prime}\simeq D_4$. 
Then it occurs that 
$[L\cap L^{\prime\prime}:k]=4$ and 
$[M_G]^{fl}=[M_{G^{\prime\prime}}]^{fl}$ as $\widetilde{H}$-lattices 
where $\widetilde{H}={\rm Gal}(LL^{\prime\prime}/k)$. 
This implies that $T$ and $T^{\prime\prime}$ are 
stably birationally $k$-equivalent 
where $\widehat{T}\simeq M_G$ and 
$\widehat{T}^\prime\simeq M_{G^\prime}$ 
but $L\neq L^{\prime\prime}$. 
By Theorem \ref{thmain7}, 
we see that $G/N\simeq G^{\prime\prime}/N^{\prime\prime}\simeq {\rm G}_1\simeq 
N_{3,1}\simeq C_2^2$ but $G^{\prime\prime}\not\simeq {\rm G}_1\simeq C_2^2$ 
(see Example \ref{ex1.42}).\\ 
(4) If $G\leq \GL(m,\bZ)$ has no normal subgroup $N$ 
which satisfies that 
$G/N$ is in Table $6$ ($G=N_{3,i}$), Table $8$ ($G=N_{4,i}$) or 
Table $9$ ($G=I_{4,i}$), then 
$[\widehat{T}]^{fl}\not\in {\rm WSEC}_r$ $(1\leq r\leq 128)$. 
For example, consider the norm one torus 
$T=R^{(1)}_{K/k}$ of dimension $n-1$ with $[K:k]=n$ and 
$G={\rm Gal}(L/k)\simeq A_n\leq \GL(n-1,\bZ)$ $(n\geq 6)$ 
which is not stably $k$-rational 
where $L/k$ is the Galois closure of $K/k$ 
(see Endo \cite[Theorem 4.4]{End11}, 
Hoshi and Yamasaki \cite[Corollary 1.11]{HY17}). 
\end{remark}
\begin{example}[{$G^{\prime\prime}\leq {\rm GL}(5,\bZ)$ with CARAT ID $(5,100,12)$, $G^{\prime\prime}\simeq D_4$, $\widehat{T}^{\prime\prime}=M_{G^{\prime\prime}}$, $[\widehat{T}^{\prime\prime}]^{fl}\in {\rm WSEC}_1$}]\label{ex1.42}
We take $G^{\prime\prime}=N_{5,324}=\langle t_1,t_2\rangle\leq {\rm GL}(5,\bZ)$, 
$G^{\prime\prime}\simeq D_4$ with CARAT ID (5,100,12) where 
\begin{align*}
t_1=\left(\scalemath{0.9}{\begin{array}{ccccc}
0&-1&-1&0&1\\
-1&0&0&0&-1\\
0&0&1&0&0\\
0&0&0&-1&0\\
0&0&1&0&-1
\end{array}}\right),\ 
t_2=\left(\scalemath{0.9}{\begin{array}{ccccc}
1&1&1&1&0\\
0&-1&0&0&0\\
0&0&-1&0&0\\
0&0&0&-1&0\\
0&-1&-1&-1&1
\end{array}}\right).
\end{align*} 
Then $G^{\prime\prime}$-lattice 
$\widehat{T}^{\prime\prime}=M_{G^{\prime\prime}}$ is indecomposable 
(see Hoshi and Yamasaki \cite[Table $15$, page 187]{HY17}). 
We also take 
$G=N_{3,1}=\langle s_1,s_2\rangle\leq {\rm GL}(3,\bZ)$, 
$G\simeq {\rm G}_1\simeq U_1\simeq C_2^2$ where 
\begin{align*}
s_1=\left(\scalemath{0.9}{\begin{array}{ccc}
0&0&1\\
-1&-1&-1\\
1&0&0
\end{array}}\right),\ 
s_2=\left(\scalemath{0.9}{\begin{array}{ccc}
0&1&0\\
1&0&0\\
-1&-1&-1
\end{array}}\right).
\end{align*}
Then we can take a subdirect product 
$\widetilde{H}\leq G\times G^{\prime\prime}$ 
with $\widetilde{H}=\langle f_1,f_2,f_3\rangle\simeq D_4$, 
$\varphi_1:\widetilde{H}\to G$, 
$f_1\mapsto s_1$, 
$f_2\mapsto s_2$, 
$f_3\mapsto s_2$, 
$\varphi_2:\widetilde{H}\to G^{\prime\prime}$
$f_1\mapsto t_2$ , 
$f_2\mapsto t_2t_1t_2$, 
$f_3\mapsto t_1$. 
This implies that $[M_{G^{\prime\prime}}]^{fl}\sim [M_G]^{fl}$ 
and $[M_{G^{\prime\prime}}]^{fl}\in {\rm WSEC}_1$ 
but $G^{\prime\prime}\not\simeq {\rm G}_1\simeq G\simeq C_2^2$. 

Conversely, by Theorem \ref{thmain7}, 
we see that 
$T$ and $T^{\prime\prime}$ are stably birationally $k$-equivalent 
and $L\subset L^{\prime\prime}$, i.e. 
$[M_G]^{fl}=[M_{G^{\prime\prime}}]^{fl}$ as $G^{\prime\prime}$-lattices 
and $G^{\prime\prime}$ acts on $[M_G]^{fl}$ though 
$G\simeq G^{\prime\prime}/N$ for the corresponding normal subgroup 
$N=Z(G^{\prime\prime})\lhd G^{\prime\prime}$ 
(for the computations of $G^{\prime\prime}=N_{5,324}\leq \GL(5,\bZ)$ 
with CARAT ID (5,100,12), we need {\tt MultInvField.gap} 
which is available as in \cite{MultInvField}, 
see Example \ref{ex12.1} for GAP computations). 
\end{example}

We organize this paper as follows. 
In Section \ref{S2}, 
the groups $N_{3,i}$, $N_{31,i}$, $N_{4,i}$ and $I_{4,i}$ 
in Definition \ref{defN3N4} are given as in Tables $6, 7, 8, 9$. 
In Section \ref{S3}, 
all of the subgroups 
$N_{3,j}\leq N_{3,i}$ $(1\leq i,j\leq 15)$, 
$N_{31,j}\leq N_{31,i}$ $(1\leq i,j\leq 64)$, 
$N_{31,l}, N_{4,j}, I_{4,m}\leq N_{4,i}$ 
$(1\leq i,j\leq 152, 1\leq l\leq 64, 1\leq m\leq 7)$ and 
$I_{4,j}\leq I_{4,i}$ $(1\leq i,j\leq 7)$ 
are given as in Tables $10, 11, 12$. 
In Section \ref{S4}, 
we compute the torus invariants $TI_G=[l_1,l_2,l_3,l_4]$ of $[M_G]^{fl}$ 
(see Definition \ref{defTI}) 
for $G=N_{3,i}$, $N_{31,i}$, $N_{4,i}$ and $I_{4,i}$ 
where $l_1=0$ (resp. $1$, $2$) if 
$[M_G]^{fl}=0$ (resp. $[M_G]^{fl}\neq 0$ but invertible, 
$[M_G]^{fl}$ is not invertible), 
$l_2$, $l_3$ and $l_4$ are 
the abelian invariants of 
$\Sha^1_\omega(G,[M_G]^{fl})$, 
$\Sha^2_\omega(G,([M_G]^{fl})^\circ)$ and 
$\Sha^2_\omega(G,[M_G]^{fl}$) respectively. 
In particular, we see that 
$\Sha^1_\omega(G,[M_G]^{fl})\simeq {\rm Br}(X)/{\rm Br}(k)$ 
where $X$ is a smooth $k$-compactification of an algebraic $k$-torus $T$ 
with $\widehat{T}=M_G$ 
and ${\rm Br}(X)$ is the Brauer-Grothendieck group of $X$, 
see Section \ref{S4} for details. 
In Section \ref{S5}, by using the torus invariants $TI_G$, 
we give a partial result for a classification of 
the weak stably birationally $k$-equivalent classes 
of algebraic $k$-tori $T$ up to dimension $4$.
In Section \ref{S6}, 
the proofs of {Theorem \ref{thmain1}} and {Theorem \ref{thmain3}} are given. 
In particular, we get the complete classification of 
weak stably $k$-equivalent classes of $[M_G]^{fl}$ (resp. $T$) 
up to dimension $4$. 
In Section \ref{S7}, 
for $G=N_{3,i}$, $N_{4,i}$ 
we prove that 
the $2$-part $\widetilde{N}_2$ $($resp. $3$-part $\widetilde{N}_3$$)$ 
of $[M_G]^{fl}$ as a $\bZ_p[\Syl_p(G)]$-lattice (see Definition \ref{def7.1}) 
is a faithful and indecomposable $\bZ_2[\Syl_2(G)]$-lattice 
$($resp. $\bZ_3[\Syl_3(G)]$-lattice$)$ unless it vanishes 
(Tables $14, 15$). 
In Section \ref{S8}, 
the proofs of ${\rm (1)}\Rightarrow {\rm (2)}$ 
of {Theorem \ref{thmain2}}, {Theorem \ref{thmain4}} 
and {Theorem \ref{thmain5}} are given. 
In Section \ref{S9}, 
we prove the implications ${\rm (2)}\Rightarrow {\rm (1)}$ of 
{Theorem \ref{thmain2}}, {Theorem \ref{thmain4}} and 
{Theorem \ref{thmain5}} 
except for the cases $i=j=137,139,145,147$ (resp. $i=j=7$) 
of $N_{4,i}$ and $N_{4,j}$ (resp. $I_{4,i}$ and $I_{4,j}$). 
In particular, Corollary \ref{cor1.27} and Corollary \ref{cor1.38} 
can answer Problem \ref{prob1.3} for 
algebraic $k$-tori of dimensions $3$ and $4$. 
In Section \ref{S10}, 
the proofs of 
${\rm (2)}\Leftrightarrow {\rm (3)}\Leftrightarrow {\rm (4)}$ 
of {Theorem \ref{thmain2}}, {Theorem \ref{thmain4}} and 
{Theorem \ref{thmain5}} are given. 
In Section \ref{S11}, 
we show Theorem \ref{thmain6} 
which gives relationships among the $I_{4,i}$ cases. 
In Section \ref{S12}, 
we prove Theorem \ref{thmain7} which can answer Problem \ref{prob1.3} 
for algebraic $k$-tori $T$ and $T^\prime$ 
of general dimensions 
when $T$ (resp. $T^\prime$) is stably birationally $k$-equivalent 
to some algebraic $k$-torus of dimension up to $4$. 

\begin{acknowledgments}
The authors would like to thank Shizuo Endo 
for giving them useful and valuable comments 
related to Krull-Schmidt-Azumaya theorem, in particular, 
Lemma \ref{lem7.11} in Section \ref{S7}. 
They also would like to thank Boris Kunyavskii 
for giving them valuable comments related to 
Theorem \ref{thKu} and his strong encouragement. 
They also would like to thank the referee 
for very careful reading and valuable comments and suggestions 
which greatly improved this paper. 
In particular, by referee's comments, they made Remark \ref{rem1.18}, 
modified Definition \ref{def7.1} and made Remark \ref{rem7.2} and 
Example \ref{ex7.3} to clarify the contents of Kunyavskii \cite{Kun90} 
and the differences between Kunyavskii \cite{Kun90} and this paper. 
Many helpful comments of the referee improve main theorems with some examples 
which illustrate the situation and enable them to reach a full 
stably birational classification of algebraic $k$-tori of dimensions 
$3$ and $4$ as in Corollary \ref{cor1.27} and Corollary \ref{cor1.38}. 
\end{acknowledgments}


\section{The groups $N_{3,i}$, $N_{31,i}$, $N_{4,i}$ and $I_{4,i}$: not stably $k$-rational cases of dimensions $3$ and $4$}\label{S2}

We give Tables $6, 7, 8, 9$ for the groups 
$N_{3,i}$, $N_{31,i}$, $N_{4,i}$ and $I_{4,i}$ respectively 
as in Definition \ref{defN3N4}. 
They are also available as 
{\tt NonInv.dat} in \cite{BCAlgTori}. 

For $2\leq n\leq 4$, the GAP ID $(n,i,j,k)$ of a finite subgroup 
$G$ of $\GL(n,\bZ)$ stands for 
the $k$-th $\bZ$-class of the $j$-th $\bQ$-class of 
the $i$-th crystal system of dimension $n$ in GAP 
(see also Hoshi and Yamasaki \cite[Chapter 3]{HY17}, 
\cite[Table 1]{BBNWZ78}). 
The SmallGroup ID $[n,i]$ stands for 
the $i$-th group of order $n$ in GAP \cite{GAP}. 
Let ${}^tG$ be the group of transposed matrices of $G\leq \GL(n,\bZ)$. 
For three groups 
$\Sha^1_\omega(G,[M_G]^{fl})$, 
$\Sha^2_\omega(G,([M_G]^{fl})^\circ)$, 
$\Sha^2_\omega(G,[M_G]^{fl})$ 
where $M_G$ is the corresponding $G$-lattice of $\bZ$-rank $n$ 
as in Definition \ref{d2.2}, 
we will explain and evaluate them in Section \ref{S4}. 


%
%
%

By Theorem \ref{th1.24},  
for $G^\prime=N_{31,j}\leq \GL(4,\bZ)$, 
we can find $G=N_{3,i}\leq \GL(3,\bZ)$ ($1\leq i\leq 15$) 
such that $[M_G]^{fl}=[M_{G^\prime}]^{fl}$ as $G$-lattices 
and hence $T$ and $T^\prime$ are stably birationally $k$-equivalent
where $\widehat{T}\simeq M_G$ and 
$\widehat{T}^\prime\simeq M_{G^\prime}$. 
In particular, we have  
$[M_G]^{fl} \sim [M_{G^\prime}]^{fl}$ (see Table $7$).\\

Table $6$: 
Groups $G=N_{3,i}$ $(1\leq i\leq 15)$ 
with $L(M)^G$ not retract $k$-rational, 
rank $M=3$, $M=M_G$: indecomposable 
{\small 
\begin{longtable}{lllll|ccc}
$i$ & GAP ID & $N_{3,i}$ & SmallGroup ID & ${}^tG$ in \cite{Kun90} 
& $\Sha^1_\omega(G,[M]^{fl})$ & $\Sha^2_\omega(G,([M]^{fl})^\circ)$ & $\Sha^2_\omega(G,[M]^{fl})$\\\hline
\endfirsthead
$1$ & $( 3, 3, 1, 3 )$ & $C_2^2$ & $[ 4, 2 ]$ & $U_1$ & $\bZ/2\bZ$ & $\bZ/2\bZ$ & $0$\\
$2$ & $( 3, 3, 3, 3 )$ & $C_2^3$ & $[ 8, 5 ]$ & $U_2$& $0$ & $0$ & $0$\\
$3$ & $( 3, 3, 3, 4 )$ & $C_2^3$ & $[ 8, 5 ]$ & $W_2$& $0$ & $0$ & $0$\\
$4$ & $( 3, 4, 3, 2 )$ & $C_4 \times C_2$ & $[ 8, 2 ]$ & $W_1$& $0$ & $0$ & $\bZ/2\bZ$\\
$5$ & $( 3, 4, 4, 2 )$ & $D_4$ & $[ 8, 3 ]$ & $U_3$& $0$ & $0$ & $0$\\
$6$ & $( 3, 4, 6, 3 )$ & $D_4$ & $[ 8, 3 ]$ & $U_4$& $0$ & $0$ & $0$\\
$7$ & $( 3, 4, 7, 2 )$ & $D_4\times C_2$ & $[ 16, 11 ]$ & $U_6$& $0$ & $0$ & $0$\\
$8$ & $( 3, 7, 1, 2 )$ & $A_4$ & $[ 12, 3 ]$ & $U_5$& $\bZ/2\bZ$ & $\bZ/2\bZ$ & $0$\\
$9$ & $( 3, 7, 2, 2 )$ & $A_4\times C_2$ & $[ 24, 13 ]$ & $U_7$& $0$ & $0$ & $0$\\
$10$ & $( 3, 7, 2, 3 )$ & $A_4\times C_2$ & $[ 24, 13 ]$ & $W_3$& $0$ & $0$ & $0$\\
$11$ & $( 3, 7, 3, 2 )$ & $S_4$ & $[ 24, 12 ]$ & $U_9$& $0$ & $0$ & $0$\\
$12$ & $( 3, 7, 3, 3 )$ & $S_4$ & $[ 24, 12 ]$ & $U_8$& $0$ & $0$ & $0$\\
$13$ & $( 3, 7, 4, 2 )$ & $S_4$ & $[ 24, 12 ]$ & $U_{10}$& $0$ & $0$ & $0$\\
$14$ & $( 3, 7, 5, 2 )$ & $S_4\times C_2$ & $[ 48, 48 ]$ & $U_{12}$& $0$ & $0$ & $0$\\
$15$ & $( 3, 7, 5, 3 )$ & $S_4\times C_2$ & $[ 48, 48 ]$ & $U_{11}$ & $0$ & $0$ & $0$
\end{longtable}
}
\vspace*{2mm}

Table $7$: 
Groups $G=N_{31,j}$ $(1\leq j\leq 64)$ 
with $[M_G]^{fl}=[M_{G^\prime}]^{fl}$ as $G$-lattices 
$($in particular, $[M_G]^{fl}\sim [M_{G^\prime}]^{fl})$, 
$G^\prime=N_{3,i}$ $(1\leq i\leq 15)$ and 
$L(M)^G$ not retract $k$-rational, 
rank $M=4$, $M=M_G\simeq M_{G^\prime}\oplus M_{G^{\prime\prime}}$: 
decomposable with rank $M_{G^\prime}=3$ ($M_{G^\prime}$: indecomposable) and rank $M_{G^{\prime\prime}}=1$
{\small 
\begin{longtable}{llllll|ccc}
$j$ & GAP ID & $N_{31,j}$ & SmallGroup ID & $i:N_{3,i}$ & $N_{31,j}$ 
& $\Sha^1_\omega(G,[M]^{fl})$ & $\Sha^2_\omega(G,([M]^{fl})^\circ)$ & $\Sha^2_\omega(G,[M]^{fl})$\\\hline
\endfirsthead
$1$ & $( 4, 4, 3, 6 )$ & $C_2^2$ & $[ 4, 2 ]$ & $1$ & $$& $\bZ/2\bZ$ & $\bZ/2\bZ$ & $0$\\ 
$2$ & $( 4, 4, 4, 4 )$ & $C_2^3$ & $[ 8, 5 ]$ & $3$ & $$& $0$ & $0$ & $0$\\ 
$3$ & $( 4, 4, 4, 6 )$ & $C_2^3$ & $[ 8, 5 ]$ & $2$ & $$& $0$ & $0$ & $0$\\ 
$4$ & $( 4, 5, 1, 9 )$ & $C_2^2$ & $[ 4, 2 ]$ & $1$ & $N_{3,1}\times\{1\}$& $\bZ/2\bZ$ & $\bZ/2\bZ$ & $0$\\
$5$ & $( 4, 5, 2, 4 )$ & $C_2^3$ & $[ 8, 5 ]$ & $3$ & $$& $0$ & $0$ & $0$\\ 
$6$ & $( 4, 5, 2, 7 )$ & $C_2^3$ & $[ 8, 5 ]$ & $2$ & $$& $0$ & $0$ & $0$\\ 
$7$ & $( 4, 6, 1, 4 )$ & $C_2^3$ & $[ 8, 5 ]$ & $3$ & $N_{3,3}\times\{1\}$& $0$ & $0$ & $0$\\
$8$ & $( 4, 6, 1, 8 )$ & $C_2^3$ & $[ 8, 5 ]$ & $2$ & $N_{3,2}\times\{1\}$& $0$ & $0$ & $0$\\
$9$ & $( 4, 6, 2, 4 )$ & $C_2^3$ & $[ 8, 5 ]$ & $3$ & $$& $0$ & $0$ & $0$\\ 
$10$ & $( 4, 6, 2, 8 )$ & $C_2^3$ & $[ 8, 5 ]$ & $2$ & $$& $0$ & $0$ & $0$\\ 
$11$ & $( 4, 6, 2, 9 )$ & $C_2^3$ & $[ 8, 5 ]$ & $1$ & $N_{3,1}\times\{\pm 1\}$& $\bZ/2\bZ$ & $\bZ/2\bZ$ & $0$\\
$12$ & $( 4, 6, 3, 3 )$ & $C_2^4$ & $[ 16, 14 ]$ & $3$ & $N_{3,3}\times\{\pm 1\}$& $0$ & $0$ & $0$\\
$13$ & $( 4, 6, 3, 6 )$ & $C_2^4$ & $[ 16, 14 ]$ & $2$ & $N_{3,2}\times\{\pm 1\}$& $0$ & $0$ & $0$\\
$14$ & $( 4, 7, 3, 2 )$ & $C_4 \times C_2$ & $[ 8, 2 ]$ & $4$ & $$& $0$ & $0$ & $\bZ/2\bZ$\\ 
$15$ & $( 4, 7, 4, 3 )$ & $D_4$ & $[ 8, 3 ]$ & $6$ & $$& $0$ & $0$ & $0$\\ 
$16$ & $( 4, 7, 5, 2 )$ & $D_4$ & $[ 8, 3 ]$ & $5$ & $$& $0$ & $0$ & $0$\\ 
$17$ & $( 4, 7, 7, 2 )$ & $D_4\times C_2$ & $[ 16, 11 ]$ & $7$ & $$& $0$ & $0$ & $0$\\ 
$18$ & $( 4, 12, 2, 4 )$ & $C_4 \times C_2$ & $[ 8, 2 ]$ & $4$ & $$& $0$ & $0$ & $\bZ/2\bZ$\\ 
$19$ & $( 4, 12, 3, 7 )$ & $D_4$ & $[ 8, 3 ]$ & $6$ & $N_{3,6}\times\{1\}$& $0$ & $0$ & $0$\\
$20$ & $( 4, 12, 4, 6 )$ & $D_4$ & $[ 8, 3 ]$ & $6$ & $$& $0$ & $0$ & $0$\\ 
$21$ & $( 4, 12, 4, 8 )$ & $D_4$ & $[ 8, 3 ]$ & $5$ & $$& $0$ & $0$ & $0$\\ 
$22$ & $( 4, 12, 4, 9 )$ & $D_4$ & $[ 8, 3 ]$ & $5$ & $$& $0$ & $0$ & $0$\\ 
$23$ & $( 4, 12, 5, 6 )$ & $D_4\times C_2$ & $[ 16, 11 ]$ & $7$ & $$& $0$ & $0$ & $0$\\ 
$24$ & $( 4, 12, 5, 7 )$ & $D_4\times C_2$ & $[ 16, 11 ]$ & $7$ & $$& $0$ & $0$ & $0$\\ 
$25$ & $( 4, 13, 1, 3 )$ & $C_4 \times C_2$ & $[ 8, 2 ]$ & $4$ & $N_{3,4}\times\{1\}$& $0$ & $0$ & $\bZ/2\bZ$\\
$26$ & $( 4, 13, 2, 4 )$ & $C_4 \times C_2$ & $[ 8, 2 ]$ & $4$ & $$& $0$ & $0$ & $\bZ/2\bZ$\\ 
$27$ & $( 4, 13, 3, 4 )$ & $D_4$ & $[ 8, 3 ]$ & $6$ & $$& $0$ & $0$ & $0$\\ 
$28$ & $( 4, 13, 4, 3 )$ & $D_4$ & $[ 8, 3 ]$ & $5$ & $N_{3,5}\times\{1\}$& $0$ & $0$ & $0$\\
$29$ & $( 4, 13, 5, 3 )$ & $C_4 \times C_2^2$ & $[ 16, 10 ]$ & $4$ & $N_{3,4}\times\{\pm 1\}$& $0$ & $0$ & $\bZ/2\bZ$\\
$30$ & $( 4, 13, 6, 3 )$ & $D_4\times C_2$ & $[ 16, 11 ]$ & $7$ & $N_{3,7}\times\{1\}$& $0$ & $0$ & $0$\\
$31$ & $( 4, 13, 7, 6 )$ & $D_4\times C_2$ & $[ 16, 11 ]$ & $6$ & $N_{3,6}\times\{\pm 1\}$& $0$ & $0$ & $0$\\
$32$ & $( 4, 13, 7, 7 )$ & $D_4\times C_2$ & $[ 16, 11 ]$ & $7$ & $$& $0$ & $0$ & $0$\\ 
$33$ & $( 4, 13, 7, 8 )$ & $D_4\times C_2$ & $[ 16, 11 ]$ & $7$ & $$& $0$ & $0$ & $0$\\ 
$34$ & $( 4, 13, 8, 3 )$ & $D_4\times C_2$ & $[ 16, 11 ]$ & $7$ & $$& $0$ & $0$ & $0$\\ 
$35$ & $( 4, 13, 8, 4 )$ & $D_4\times C_2$ & $[ 16, 11 ]$ & $5$ & $N_{3,5}\times\{\pm 1\}$& $0$ & $0$ & $0$\\
$36$ & $( 4, 13, 9, 3 )$ & $D_4\times C_2$ & $[ 16, 11 ]$ & $7$ & $$& $0$ & $0$ & $0$\\ 
$37$ & $( 4, 13, 10, 3 )$ & $D_4\times C_2^2$ & $[ 32, 46 ]$ & $7$ & $N_{3,7}\times\{\pm 1\}$& $0$ & $0$ & $0$\\
$38$ & $( 4, 24, 1, 5 )$ & $A_4$ & $[ 12, 3 ]$ & $8$ & $N_{3,8}\times\{1\}$& $\bZ/2\bZ$ & $\bZ/2\bZ$ & $0$\\
$39$ & $( 4, 24, 2, 3 )$ & $A_4\times C_2$ & $[ 24, 13 ]$ & $10$ & $$& $0$ & $0$ & $0$\\ 
$40$ & $( 4, 24, 2, 5 )$ & $A_4\times C_2$ & $[ 24, 13 ]$ & $9$ & $$& $0$ & $0$ & $0$\\ 
$41$ & $( 4, 24, 3, 5 )$ & $S_4$ & $[ 24, 12 ]$ & $13$ & $N_{3,13}\times\{1\}$& $0$ & $0$ & $0$\\
$42$ & $( 4, 24, 4, 3 )$ & $S_4$ & $[ 24, 12 ]$ & $12$ & $$& $0$ & $0$ & $0$\\ 
$43$ & $( 4, 24, 4, 5 )$ & $S_4$ & $[ 24, 12 ]$ & $11$ & $$& $0$ & $0$ & $0$\\ 
$44$ & $( 4, 24, 5, 3 )$ & $S_4\times C_2$ & $[ 48, 48 ]$ & $15$ & $$& $0$ & $0$ & $0$\\ 
$45$ & $( 4, 24, 5, 5 )$ & $S_4\times C_2$ & $[ 48, 48 ]$ & $14$ & $$& $0$ & $0$ & $0$\\ 
$46$ & $( 4, 25, 1, 2 )$ & $A_4\times C_2$ & $[ 24, 13 ]$ & $10$ & $N_{3,10}\times\{1\}$& $0$ & $0$ & $0$\\
$47$ & $( 4, 25, 1, 4 )$ & $A_4\times C_2$ & $[ 24, 13 ]$ & $9$ & $N_{3,9}\times\{1\}$& $0$ & $0$ & $0$\\
$48$ & $( 4, 25, 2, 4 )$ & $A_4\times C_2$ & $[ 24, 13 ]$ & $8$ & $N_{3,8}\times\{\pm 1\}$& $\bZ/2\bZ$ & $\bZ/2\bZ$ & $0$\\
$49$ & $( 4, 25, 3, 2 )$ & $S_4$ & $[ 24, 12 ]$ & $12$ & $N_{3,12}\times\{1\}$& $0$ & $0$ & $0$\\
$50$ & $( 4, 25, 3, 4 )$ & $S_4$ & $[ 24, 12 ]$ & $11$ & $N_{3,11}\times\{1\}$& $0$ & $0$ & $0$\\
$51$ & $( 4, 25, 4, 4 )$ & $S_4$ & $[ 24, 12 ]$ & $13$ & $$& $0$ & $0$ & $0$\\ 
$52$ & $( 4, 25, 5, 2 )$ & $A_4 \times C_2^2$ & $[ 48, 49 ]$ & $10$ & $N_{3,10}\times\{\pm 1\}$& $0$ & $0$ & $0$\\
$53$ & $( 4, 25, 5, 4 )$ & $A_4 \times C_2^2$ & $[ 48, 49 ]$ & $9$ & $N_{3,9}\times\{\pm 1\}$& $0$ & $0$ & $0$\\
$54$ & $( 4, 25, 6, 2 )$ & $S_4\times C_2$ & $[ 48, 48 ]$ & $15$ & $$& $0$ & $0$ & $0$\\ 
$55$ & $( 4, 25, 6, 4 )$ & $S_4\times C_2$ & $[ 48, 48 ]$ & $14$ & $$& $0$ & $0$ & $0$\\ 
$56$ & $( 4, 25, 7, 2 )$ & $S_4\times C_2$ & $[ 48, 48 ]$ & $15$ & $N_{3,15}\times\{1\}$& $0$ & $0$ & $0$\\
$57$ & $( 4, 25, 7, 4 )$ & $S_4\times C_2$ & $[ 48, 48 ]$ & $14$ & $N_{3,14}\times\{1\}$& $0$ & $0$ & $0$\\
$58$ & $( 4, 25, 8, 2 )$ & $S_4\times C_2$ & $[ 48, 48 ]$ & $15$ & $$& $0$ & $0$ & $0$\\ 
$59$ & $( 4, 25, 8, 4 )$ & $S_4\times C_2$ & $[ 48, 48 ]$ & $14$ & $$& $0$ & $0$ & $0$\\ 
$60$ & $( 4, 25, 9, 4 )$ & $S_4\times C_2$ & $[ 48, 48 ]$ & $13$ & $N_{3,13}\times\{\pm 1\}$& $0$ & $0$ & $0$\\
$61$ & $( 4, 25, 10, 2 )$ & $S_4\times C_2$ & $[ 48, 48 ]$ & $12$ & $N_{3,12}\times\{\pm 1\}$& $0$ & $0$ & $0$\\
$62$ & $( 4, 25, 10, 4 )$ & $S_4\times C_2$ & $[ 48, 48 ]$ & $11$ & $N_{3,11}\times\{\pm 1\}$& $0$ & $0$ & $0$\\
$63$ & $( 4, 25, 11, 2 )$ & $S_4 \times C_2^2$ & $[ 96, 226 ]$ & $15$ & $N_{3,15}\times\{\pm 1\}$& $0$ & $0$ & $0$\\
$64$ & $( 4, 25, 11, 4 )$ & $S_4 \times C_2^2$ & $[ 96, 226 ]$ & $14$ & $N_{3,14}\times\{\pm 1\}$ & $0$ & $0$ & $0$
\end{longtable}
}
\vspace*{2mm}

Table $8$: 
Groups $G=N_{4,i}$ $(1\leq i\leq 152)$ 
with $L(M)^G$ not retract $k$-rational, 
rank $M=4$, $M=M_G$: indecomposable
{\small 
\begin{longtable}{llll|ccc}
$i$ & GAP ID & $N_{4,i}$ & SmallGroup ID 
& $\Sha^1_\omega(G,[M]^{fl})$ & $\Sha^2_\omega(G,([M]^{fl})^\circ)$ & $\Sha^2_\omega(G,[M]^{fl})$\\\hline
\endfirsthead
$1$ & $( 4, 5, 1, 12 )$ & $C_2^2$ & $[ 4, 2 ]$& $\bZ/2\bZ$ & $\bZ/2\bZ$ & $0$\\
$2$ & $( 4, 5, 2, 5 )$ & $C_2^3$ & $[ 8, 5 ]$& $0$ & $0$ & $0$\\
$3$ & $( 4, 5, 2, 8 )$ & $C_2^3$ & $[ 8, 5 ]$& $\bZ/2\bZ$ & $0$ & $0$\\
$4$ & $( 4, 5, 2, 9 )$ & $C_2^3$ & $[ 8, 5 ]$& $0$ & $0$ & $\bZ/2\bZ$\\
$5$ & $( 4, 6, 1, 6 )$ & $C_2^3$ & $[ 8, 5 ]$& $0$ & $0$ & $0$\\
$6$ & $( 4, 6, 1, 11 )$ & $C_2^3$ & $[ 8, 5 ]$& $0$ & $0$ & $0$\\
$7$ & $( 4, 6, 2, 6 )$ & $C_2^3$ & $[ 8, 5 ]$& $0$ & $0$ & $0$\\
$8$ & $( 4, 6, 2, 10 )$ & $C_2^3$ & $[ 8, 5 ]$& $\bZ/2\bZ$ & $0$ & $0$\\
$9$ & $( 4, 6, 2, 12 )$ & $C_2^3$ & $[ 8, 5 ]$& $0$ & $\bZ/2\bZ$ & $0$\\
$10$ & $( 4, 6, 3, 4 )$ & $C_2^4$ & $[ 16, 14 ]$& $0$ & $0$ & $0$\\
$11$ & $( 4, 6, 3, 7 )$ & $C_2^4$ & $[ 16, 14 ]$& $0$ & $0$ & $0$\\
$12$ & $( 4, 6, 3, 8 )$ & $C_2^4$ & $[ 16, 14 ]$& $0$ & $0$ & $0$\\
$13$ & $( 4, 12, 2, 5 )$ & $C_4 \times C_2$ & $[ 8, 2 ]$& $0$ & $0$ & $\bZ/2\bZ$\\
$14$ & $( 4, 12, 2, 6 )$ & $C_4 \times C_2$ & $[ 8, 2 ]$& $\bZ/2\bZ$ & $0$ & $\bZ/2\bZ$\\
$15$ & $( 4, 12, 3, 11 )$ & $D_4$ & $[ 8, 3 ]$& $0$ & $0$ & $0$\\
$16$ & $( 4, 12, 4, 10 )$ & $D_4$ & $[ 8, 3 ]$& $0$ & $0$ & $0$\\
$17$ & $( 4, 12, 4, 11 )$ & $D_4$ & $[ 8, 3 ]$& $0$ & $0$ & $0$\\
$18$ & $( 4, 12, 4, 12 )$ & $D_4$ & $[ 8, 3 ]$& $\bZ/2\bZ$ & $\bZ/2\bZ$ & $0$\\
$19$ & $( 4, 12, 5, 8 )$ & $D_4\times C_2$ & $[ 16, 11 ]$& $0$ & $0$ & $0$\\
$20$ & $( 4, 12, 5, 9 )$ & $D_4\times C_2$ & $[ 16, 11 ]$& $0$ & $0$ & $0$\\
$21$ & $( 4, 12, 5, 10 )$ & $D_4\times C_2$ & $[ 16, 11 ]$& $\bZ/2\bZ$ & $0$ & $0$\\
$22$ & $( 4, 12, 5, 11 )$ & $D_4\times C_2$ & $[ 16, 11 ]$& $0$ & $0$ & $0$\\
$23$ & $( 4, 13, 1, 5 )$ & $C_4 \times C_2$ & $[ 8, 2 ]$& $0$ & $0$ & $\bZ/2\bZ$\\
$24$ & $( 4, 13, 2, 5 )$ & $C_4 \times C_2$ & $[ 8, 2 ]$& $0$ & $0$ & $\bZ/2\bZ$\\
$25$ & $( 4, 13, 3, 5 )$ & $D_4$ & $[ 8, 3 ]$& $0$ & $0$ & $0$\\
$26$ & $( 4, 13, 4, 5 )$ & $D_4$ & $[ 8, 3 ]$& $0$ & $0$ & $0$\\
$27$ & $( 4, 13, 5, 4 )$ & $C_4 \times C_2^2$ & $[ 16, 10 ]$& $0$ & $0$ & $(\bZ/2\bZ)^{\oplus 2}$\\
$28$ & $( 4, 13, 5, 5 )$ & $C_4 \times C_2^2$ & $[ 16, 10 ]$& $0$ & $0$ & $0$\\
$29$ & $( 4, 13, 6, 5 )$ & $D_4\times C_2$ & $[ 16, 11 ]$& $0$ & $0$ & $0$\\
$30$ & $( 4, 13, 7, 9 )$ & $D_4\times C_2$ & $[ 16, 11 ]$& $0$ & $0$ & $0$\\
$31$ & $( 4, 13, 7, 10 )$ & $D_4\times C_2$ & $[ 16, 11 ]$& $0$ & $0$ & $0$\\
$32$ & $( 4, 13, 7, 11 )$ & $D_4\times C_2$ & $[ 16, 11 ]$& $0$ & $0$ & $0$\\
$33$ & $( 4, 13, 8, 5 )$ & $D_4\times C_2$ & $[ 16, 11 ]$& $0$ & $0$ & $0$\\
$34$ & $( 4, 13, 8, 6 )$ & $D_4\times C_2$ & $[ 16, 11 ]$& $0$ & $0$ & $0$\\
$35$ & $( 4, 13, 9, 4 )$ & $D_4\times C_2$ & $[ 16, 11 ]$& $0$ & $0$ & $0$\\
$36$ & $( 4, 13, 9, 5 )$ & $D_4\times C_2$ & $[ 16, 11 ]$& $0$ & $0$ & $0$\\
$37$ & $( 4, 13, 10, 4 )$ & $D_4\times C_2^2$ & $[ 32, 46 ]$& $0$ & $0$ & $0$\\
$38$ & $( 4, 13, 10, 5 )$ & $D_4\times C_2^2$ & $[ 32, 46 ]$& $0$ & $0$ & $0$\\
$39$ & $( 4, 18, 1, 3 )$ & $C_4 \times C_2$ & $[ 8, 2 ]$& $\bZ/2\bZ$ & $0$ & $\bZ/2\bZ$\\
$40$ & $( 4, 18, 2, 4 )$ & $(C_4 \times C_2) \rtimes C_2$ & $[ 16, 3 ]$& $0$ & $0$ & $0$\\
$41$ & $( 4, 18, 2, 5 )$ & $(C_4 \times C_2) \rtimes C_2$ & $[ 16, 3 ]$& $0$ & $0$ & $0$\\
$42$ & $( 4, 18, 3, 5 )$ & $(C_4 \times C_2) \rtimes C_2$ & $[ 16, 3 ]$& $0$ & $0$ & $\bZ/2\bZ$\\
$43$ & $( 4, 18, 3, 6 )$ & $(C_4 \times C_2) \rtimes C_2$ & $[ 16, 3 ]$& $0$ & $0$ & $0$\\
$44$ & $( 4, 18, 3, 7 )$ & $(C_4 \times C_2) \rtimes C_2$ & $[ 16, 3 ]$& $0$ & $0$ & $0$\\
$45$ & $( 4, 18, 4, 4 )$ & $D_4\times C_2$ & $[ 16, 11 ]$& $\bZ/2\bZ$ & $0$ & $0$\\
$46$ & $( 4, 18, 4, 5 )$ & $D_4\times C_2$ & $[ 16, 11 ]$& $0$ & $0$ & $0$\\
$47$ & $( 4, 18, 5, 5 )$ & $C_2^4 \rtimes C_2$ & $[ 32, 27 ]$& $0$ & $0$ & $0$\\
$48$ & $( 4, 18, 5, 6 )$ & $C_2^4 \rtimes C_2$ & $[ 32, 27 ]$& $0$ & $0$ & $0$\\
$49$ & $( 4, 18, 5, 7 )$ & $C_2^4 \rtimes C_2$ & $[ 32, 27 ]$& $0$ & $0$ & $0$\\
$50$ & $( 4, 19, 1, 2 )$ & $C_4 \rtimes C_4$ & $[ 16, 4 ]$& $0$ & $0$ & $\bZ/2\bZ$\\
$51$ & $( 4, 19, 2, 2 )$ & $C_4^2$ & $[ 16, 2 ]$& $0$ & $0$ & $\bZ/2\bZ$\\
$52$ & $( 4, 19, 3, 2 )$ & $D_4 \times C_4$ & $[ 32, 25 ]$& $0$ & $0$ & $\bZ/2\bZ$\\
$53$ & $( 4, 19, 4, 3 )$ & $(C_4 \times C_2^2) \rtimes C_2$ & $[ 32, 28 ]$& $0$ & $0$ & $0$\\
$54$ & $( 4, 19, 4, 4 )$ & $(C_4 \times C_2^2) \rtimes C_2$ & $[ 32, 28 ]$& $0$ & $0$ & $\bZ/2\bZ$\\
$55$ & $( 4, 19, 5, 2 )$ & $C_4^2 \rtimes C_2$ & $[ 32, 34 ]$& $0$ & $0$ & $0$\\
$56$ & $( 4, 19, 6, 2 )$ & $D_4^2$ & $[ 64, 226 ]$& $0$ & $0$ & $0$\\
$57$ & $( 4, 22, 1, 1 )$ & $C_3^2$ & $[ 9, 2 ]$& $0$ & $0$ & $\bZ/3\bZ$\\
$58$ & $( 4, 22, 2, 1 )$ & $C_6 \times C_3$ & $[ 18, 5 ]$& $0$ & $0$ & $0$\\
$59$ & $( 4, 22, 3, 1 )$ & $S_3 \times C_3$ & $[ 18, 3 ]$& $0$ & $0$ & $0$\\
$60$ & $( 4, 22, 4, 1 )$ & $S_3 \times C_3$ & $[ 18, 3 ]$& $0$ & $0$ & $\bZ/3\bZ$\\
$61$ & $( 4, 22, 5, 1 )$ & $C_3^2 \rtimes C_2$ & $[ 18, 4 ]$& $0$ & $0$ & $\bZ/3\bZ$\\
$62$ & $( 4, 22, 5, 2 )$ & $C_3^2 \rtimes C_2$ & $[ 18, 4 ]$& $0$ & $0$ & $0$\\
$63$ & $( 4, 22, 6, 1 )$ & $S_3 \times C_6$ & $[ 36, 12 ]$& $0$ & $0$ & $0$\\
$64$ & $( 4, 22, 7, 1 )$ & $(C_3^2 \rtimes C_2) \times C_2$ & $[ 36, 13 ]$& $0$ & $0$ & $0$\\
$65$ & $( 4, 22, 8, 1 )$ & $S_3^2$ & $[ 36, 10 ]$& $0$ & $0$ & $0$\\
$66$ & $( 4, 22, 9, 1 )$ & $S_3^2$ & $[ 36, 10 ]$& $0$ & $0$ & $0$\\
$67$ & $( 4, 22, 10, 1 )$ & $S_3^2$ & $[ 36, 10 ]$& $0$ & $0$ & $\bZ/3\bZ$\\
$68$ & $( 4, 22, 11, 1 )$ & $S_3^2 \times C_2$ & $[ 72, 46 ]$& $0$ & $0$ & $0$\\
$69$ & $( 4, 24, 2, 4 )$ & $A_4\times C_2$ & $[ 24, 13 ]$& $0$ & $0$ & $0$\\
$70$ & $( 4, 24, 2, 6 )$ & $A_4\times C_2$ & $[ 24, 13 ]$& $0$ & $0$ & $\bZ/2\bZ$\\
$71$ & $( 4, 24, 4, 4 )$ & $S_4$ & $[ 24, 12 ]$& $0$ & $0$ & $0$\\
$72$ & $( 4, 24, 5, 4 )$ & $S_4\times C_2$ & $[ 48, 48 ]$& $0$ & $0$ & $0$\\
$73$ & $( 4, 24, 5, 6 )$ & $S_4\times C_2$ & $[ 48, 48 ]$& $0$ & $0$ & $0$\\
$74$ & $( 4, 25, 1, 3 )$ & $A_4\times C_2$ & $[ 24, 13 ]$& $0$ & $0$ & $0$\\
$75$ & $( 4, 25, 2, 3 )$ & $A_4\times C_2$ & $[ 24, 13 ]$& $0$ & $0$ & $0$\\
$76$ & $( 4, 25, 2, 5 )$ & $A_4\times C_2$ & $[ 24, 13 ]$& $0$ & $\bZ/2\bZ$ & $0$\\
$77$ & $( 4, 25, 3, 3 )$ & $S_4$ & $[ 24, 12 ]$& $0$ & $0$ & $0$\\
$78$ & $( 4, 25, 4, 3 )$ & $S_4$ & $[ 24, 12 ]$& $0$ & $0$ & $0$\\
$79$ & $( 4, 25, 5, 3 )$ & $A_4 \times C_2^2$ & $[ 48, 49 ]$& $0$ & $0$ & $0$\\
$80$ & $( 4, 25, 5, 5 )$ & $A_4 \times C_2^2$ & $[ 48, 49 ]$& $0$ & $0$ & $0$\\
$81$ & $( 4, 25, 6, 3 )$ & $S_4\times C_2$ & $[ 48, 48 ]$& $0$ & $0$ & $0$\\
$82$ & $( 4, 25, 6, 5 )$ & $S_4\times C_2$ & $[ 48, 48 ]$& $0$ & $0$ & $0$\\
$83$ & $( 4, 25, 7, 3 )$ & $S_4\times C_2$ & $[ 48, 48 ]$& $0$ & $0$ & $0$\\
$84$ & $( 4, 25, 8, 3 )$ & $S_4\times C_2$ & $[ 48, 48 ]$& $0$ & $0$ & $0$\\
$85$ & $( 4, 25, 9, 3 )$ & $S_4\times C_2$ & $[ 48, 48 ]$& $0$ & $0$ & $0$\\
$86$ & $( 4, 25, 9, 5 )$ & $S_4\times C_2$ & $[ 48, 48 ]$& $0$ & $0$ & $0$\\
$87$ & $( 4, 25, 10, 3 )$ & $S_4\times C_2$ & $[ 48, 48 ]$& $0$ & $0$ & $0$\\
$88$ & $( 4, 25, 10, 5 )$ & $S_4\times C_2$ & $[ 48, 48 ]$& $0$ & $0$ & $0$\\
$89$ & $( 4, 25, 11, 3 )$ & $S_4 \times C_2^2$ & $[ 96, 226 ]$& $0$ & $0$ & $0$\\
$90$ & $( 4, 25, 11, 5 )$ & $S_4 \times C_2^2$ & $[ 96, 226 ]$& $0$ & $0$ & $0$\\
$91$ & $( 4, 29, 1, 1 )$ & $S_3 \times C_3$ & $[ 18, 3 ]$& $0$ & $0$ & $0$\\
$92$ & $( 4, 29, 1, 2 )$ & $S_3 \times C_3$ & $[ 18, 3 ]$& $0$ & $0$ & $\bZ/3\bZ$\\
$93$ & $( 4, 29, 2, 1 )$ & $S_3 \times C_6$ & $[ 36, 12 ]$& $0$ & $0$ & $0$\\
$94$ & $( 4, 29, 3, 1 )$ & $S_3^2$ & $[ 36, 10 ]$& $0$ & $0$ & $0$\\
$95$ & $( 4, 29, 3, 2 )$ & $S_3^2$ & $[ 36, 10 ]$& $0$ & $0$ & $\bZ/3\bZ$\\
$96$ & $( 4, 29, 3, 3 )$ & $S_3^2$ & $[ 36, 10 ]$& $0$ & $0$ & $0$\\
$97$ & $( 4, 29, 4, 1 )$ & $C_3^2 \rtimes C_4$ & $[ 36, 9 ]$& $0$ & $0$ & $\bZ/3\bZ$\\
$98$ & $( 4, 29, 4, 2 )$ & $C_3^2 \rtimes C_4$ & $[ 36, 9 ]$& $0$ & $0$ & $0$\\
$99$ & $( 4, 29, 5, 1 )$ & $S_3^2 \times C_2$ & $[ 72, 46 ]$& $0$ & $0$ & $0$\\
$100$ & $( 4, 29, 6, 1 )$ & $(C_3^2 \rtimes C_4) \times C_2$ & $[72, 45 ]$& $0$ & $0$ & $0$\\
$101$ & $( 4, 29, 7, 1 )$ & $S_3^2 \rtimes C_2$ & $[ 72, 40 ]$& $0$ & $0$ & $0$\\
$102$ & $( 4, 29, 7, 2 )$ & $S_3^2 \rtimes C_2$ & $[ 72, 40 ]$& $0$ & $0$ & $0$\\
$103$ & $( 4, 29, 8, 1 )$ & $S_3^2 \rtimes C_2$ & $[ 72, 40 ]$& $0$ & $0$ & $0$\\
$104$ & $( 4, 29, 8, 2 )$ & $S_3^2 \rtimes C_2$ & $[ 72, 40 ]$& $0$ & $0$ & $\bZ/3\bZ$\\
$105$ & $( 4, 29, 9, 1 )$ & $(S_3^2 \rtimes C_2) \times C_2$ & $[ 144, 186 ]$& $0$ & $0$ & $0$\\
$106$ & $( 4, 32, 1, 2 )$ & $Q_8$ & $[ 8, 4 ]$& $(\bZ/2\bZ)^{\oplus 2}$ & $0$ & $(\bZ/2\bZ)^{\oplus 2}$\\
$107$ & $( 4, 32, 2, 2 )$ & $C_8 \rtimes C_2$ & $[ 16, 6 ]$& $\bZ/2\bZ$ & $0$ & $\bZ/2\bZ$\\
$108$ & $( 4, 32, 3, 2 )$ & $QD_8$ & $[ 16, 8 ]$& $\bZ/2\bZ$ & $0$ & $\bZ/2\bZ$\\
$109$ & $( 4, 32, 4, 2 )$ & $(C_4 \times C_2) \rtimes C_2$ & $[ 16, 13 ]$& $\bZ/2\bZ$ & $0$ & $0$\\
$110$ & $( 4, 32, 5, 2 )$ & $\SL(2,\bF_3)$ & $[ 24, 3 ]$& $0$ & $0$ & $0$\\
$111$ & $( 4, 32, 5, 3 )$ & $\SL(2,\bF_3)$ & $[ 24, 3 ]$& $0$ & $0$ & $(\bZ/2\bZ)^{\oplus 2}$\\
$112$ & $( 4, 32, 6, 2 )$ & $C_8 \rtimes C_2^2$ & $[ 32, 43 ]$& $\bZ/2\bZ$ & $0$ & $0$\\
$113$ & $( 4, 32, 7, 2 )$ & $(C_8 \rtimes C_2) \rtimes C_2$ & $[ 32, 7 ]$& $0$ & $0$ & $0$\\
$114$ & $( 4, 32, 8, 2 )$ & $C_4^2 \rtimes C_2$ & $[ 32, 11 ]$& $0$ & $0$ & $0$\\
$115$ & $( 4, 32, 9, 4 )$ & $C_2^3 \rtimes C_4$ & $[ 32, 6 ]$& $0$ & $0$ & $0$\\
$116$ & $( 4, 32, 9, 5 )$ & $C_2^3 \rtimes C_4$ & $[ 32, 6 ]$& $0$ & $0$ & $0$\\
$117$ & $( 4, 32, 10, 2 )$ & $C_2^3 \rtimes C_2^2$ & $[ 32, 49 ]$& $0$ & $0$ & $0$\\
$118$ & $( 4, 32, 11, 2 )$ & $\GL(2,\bF_3)$ & $[ 48, 29 ]$& $0$ & $0$ & $0$\\
$119$ & $( 4, 32, 11, 3 )$ & $\GL(2,\bF_3)$ & $[ 48, 29 ]$& $0$ & $0$ & $\bZ/2\bZ$\\
$120$ & $( 4, 32, 12, 2 )$ & $((C_8 \rtimes C_2) \rtimes C_2) \rtimes C_2$ & $[ 64, 134 ]$& $0$ & $0$ & $0$\\
$121$ & $( 4, 32, 13, 3 )$ & $((C_8 \rtimes C_2) \rtimes C_2) \rtimes C_2$ & $[ 64, 32 ]$& $0$ & $0$ & $0$\\
$122$ & $( 4, 32, 13, 4 )$ & $((C_8 \rtimes C_2) \rtimes C_2) \rtimes C_2$ & $[ 64, 32 ]$& $0$ & $0$ & $0$\\
$123$ & $( 4, 32, 14, 3 )$ & $(C_2^4 \rtimes C_2) \rtimes C_2$ & $[ 64, 138 ]$& $0$ & $0$ & $0$\\
$124$ & $( 4, 32, 14, 4 )$ & $(C_2^4 \rtimes C_2) \rtimes C_2$ & $[ 64, 138 ]$& $0$ & $0$ & $0$\\
$125$ & $( 4, 32, 15, 2 )$ & $(C_2^3 \rtimes C_4) \rtimes C_2$ & $[ 64, 34 ]$& $0$ & $0$ & $0$\\
$126$ & $( 4, 32, 16, 2 )$ & $(C_2^3 \rtimes C_2^2) \rtimes C_3$ & $[ 96, 204 ]$& $0$ & $0$ & $0$\\
$127$ & $( 4, 32, 16, 3 )$ & $(C_2^3 \rtimes C_2^2) \rtimes C_3$ & $[ 96, 204 ]$& $0$ & $0$ & $0$\\
$128$ & $( 4, 32, 17, 2 )$ & $D_4^2 \rtimes C_2$ & $[ 128, 928 ]$& $0$ & $0$ & $0$\\
$129$ & $( 4, 32, 18, 2 )$ & $((C_2^4 \rtimes C_2) \rtimes C_2) \rtimes C_3$ & $[ 192, 201 ]$& $0$ & $0$ & $0$\\
$130$ & $( 4, 32, 18, 3 )$ & $((C_2^4 \rtimes C_2) \rtimes C_2) \rtimes C_3$ & $[ 192, 201 ]$& $0$ & $0$ & $0$\\
$131$ & $( 4, 32, 19, 2 )$ & $((C_2^3 \rtimes C_2^2) \rtimes C_3) \rtimes C_2$ & $[ 192, 1493 ]$& $0$ & $0$ & $0$\\
$132$ & $( 4, 32, 19, 3 )$ & $((C_2^3 \rtimes C_2^2) \rtimes C_3) \rtimes C_2$ & $[ 192, 1493 ]$& $0$ & $0$ & $0$\\
$133$ & $( 4, 32, 20, 2 )$ & $((C_2^3 \rtimes C_2^2) \rtimes C_3) \rtimes C_2$ & $[ 192, 1494 ]$& $0$ & $0$ & $0$\\
$134$ & $( 4, 32, 20, 3 )$ & $((C_2^3 \rtimes C_2^2) \rtimes C_3) \rtimes C_2$ & $[ 192, 1494 ]$& $0$ & $0$ & $0$\\
$135$ & $( 4, 32, 21, 2 )$ & $(((C_2^3 \rtimes C_2^2) \rtimes C_3) \rtimes C_2) \rtimes C_2$ & $[ 384, 5602 ]$& $0$ & $0$ & $0$\\
$136$ & $( 4, 32, 21, 3 )$ & $(((C_2^3 \rtimes C_2^2) \rtimes C_3) \rtimes C_2) \rtimes C_2$ & $[ 384, 5602 ]$& $0$ & $0$ & $0$\\
$137$ & $( 4, 33, 1, 1 )$ & $Q_8 \times C_3$ & $[ 24, 11 ]$& $0$ & $0$ & $0$\\
$138$ & $( 4, 33, 3, 1 )$ & $\SL(2,\bF_3)$ & $[ 24, 3 ]$& $(\bZ/2\bZ)^{\oplus 2}$ & $0$ & $0$\\
$139$ & $( 4, 33, 4, 1 )$ & $(Q_8 \times C_3) \rtimes C_2$ & $[ 48, 17 ]$& $0$ & $0$ & $0$\\
$140$ & $( 4, 33, 5, 1 )$ & $((C_4 \times C_2) \rtimes C_2) \rtimes C_3$ & $[ 48, 33 ]$& $\bZ/2\bZ$ & $0$ & $0$\\
$141$ & $( 4, 33, 6, 1 )$ & $\GL(2,\bF_3)$ & $[ 48, 29 ]$& $\bZ/2\bZ$ & $0$ & $0$\\
$142$ & $( 4, 33, 7, 1 )$ & $\SL(2,\bF_3) \times C_3$ & $[ 72, 25 ]$& $0$ & $0$ & $0$\\
$143$ & $( 4, 33, 8, 1 )$ & $(C_2^3 \rtimes C_2^2) \rtimes C_3$ & $[ 96, 201 ]$& $0$ & $0$ & $0$\\
$144$ & $( 4, 33, 9, 1 )$ & $\GL(2,\bF_3) \rtimes C_2$ & $[ 96, 193 ]$& $\bZ/2\bZ$ & $0$ & $0$\\
$145$ & $( 4, 33, 10, 1 )$ & $\SL(2,\bF_3) \rtimes C_4$ & $[ 96, 67 ]$& $0$ & $0$ & $0$\\
$146$ & $( 4, 33, 11, 1 )$ & $(\SL(2,\bF_3) \times C_3) \rtimes C_2$ & $[ 144, 125 ]$& $0$ & $0$ & $0$\\
$147$ & $( 4, 33, 12, 1 )$ & $(\GL(2,\bF_3) \rtimes C_2) \rtimes C_2$ & $[192, 988 ]$& $0$ & $0$ & $0$\\
$148$ & $( 4, 33, 13, 1 )$ & $(C_2^3 \rtimes C_2^2) \rtimes C_3^2$ & $[ 288, 860 ]$& $0$ & $0$ & $0$\\
$149$ & $( 4, 33, 14, 1 )$ & $(((C_2^3 \rtimes C_2^2) \rtimes C_3) \rtimes C_2) \rtimes C_3$ & $[ 576, 8277 ]$& $0$ & $0$ & $0$\\
$150$ & $( 4, 33, 14, 2 )$ & $(((C_2^3 \rtimes C_2^2) \rtimes C_3) \rtimes C_2) \rtimes C_3$ & $[ 576, 8277 ]$& $0$ & $0$ & $0$\\
$151$ & $( 4, 33, 15, 1 )$ & $((C_2^3 \rtimes C_2^2) \rtimes C_3^2) \rtimes C_2$ & $[ 576, 8282 ]$& $0$ & $0$ & $0$\\
$152$ & $( 4, 33, 16, 1 )$ & $(((C_2^3 \rtimes C_2^2) \rtimes C_3^2) \rtimes C_2) \rtimes C_2$ & $[ 1152, 157478 ]$ & $0$ & $0$ & $0$
\end{longtable}
}
\vspace*{2mm}

Table $9$: 
Groups $G=I_{4,i}$ $(1\leq i\leq 7)$ 
with $L(M)^G$ not stably $k$-rational but retract $k$-rational, 
rank $M=4$, $M=M_G$: indecomposable
{\small 
\begin{longtable}{llll|ccc}
$i$ & GAP ID & $I_{4,i}$ & SmallGroup ID 
& $\Sha^1_\omega(G,[M]^{fl})$ & $\Sha^2_\omega(G,([M]^{fl})^\circ)$ & $\Sha^2_\omega(G,[M]^{fl})$\\\hline
\endfirsthead
$1$ & $( 4, 31, 1, 3 )$ & $F_{20}$ & $[ 20, 3 ]$& $0$ & $0$ & $0$\\
$2$ & $( 4, 31, 1, 4 )$ & $F_{20}$ & $[ 20, 3 ]$& $0$ & $0$ & $0$\\
$3$ & $( 4, 31, 2, 2 )$ & $F_{20} \times C_2$ & $[ 40, 12 ]$& $0$ & $0$ & $0$\\
$4$ & $( 4, 31, 4, 2 )$ & $S_5$ & $[ 120, 34 ]$& $0$ & $0$ & $0$\\
$5$ & $( 4, 31, 5, 2 )$ & $S_5$ & $[ 120, 34 ]$& $0$ & $0$ & $0$\\
$6$ & $( 4, 31, 7, 2 )$ & $S_5 \times C_2$ & $[ 240, 189 ]$& $0$ & $0$ & $0$\\ 
$7$ & $( 4, 33, 2, 1 )$ & $C_3 \rtimes C_8$ & $[ 24, 1 ]$ & $0$ & $0$ & $0$
\end{longtable}
}
\vspace*{2mm}

Tables $6, 7, 8, 9$ can be obtained by GAP computations as follows: 

\bigskip
\begin{example}[The groups $N_{3,i}$ as in Table $6$]~\vspace*{-5mm}\\
\begin{verbatim}
gap> Read("BCAlgTori.gap");
gap> N3table:=List(N3,x->[Position(N3,x),x,
> StructureDescription(MatGroupZClass(x[1],x[2],x[3],x[4])),
> IdSmallGroup(MatGroupZClass(x[1],x[2],x[3],x[4]))]);
[ [ 1, [ 3, 3, 1, 3 ], "C2 x C2", [ 4, 2 ] ], 
  [ 2, [ 3, 3, 3, 3 ], "C2 x C2 x C2", [ 8, 5 ] ], 
  [ 3, [ 3, 3, 3, 4 ], "C2 x C2 x C2", [ 8, 5 ] ], 
  [ 4, [ 3, 4, 3, 2 ], "C4 x C2", [ 8, 2 ] ], 
  [ 5, [ 3, 4, 4, 2 ], "D8", [ 8, 3 ] ], 
  [ 6, [ 3, 4, 6, 3 ], "D8", [ 8, 3 ] ], 
  [ 7, [ 3, 4, 7, 2 ], "C2 x D8", [ 16, 11 ] ], 
  [ 8, [ 3, 7, 1, 2 ], "A4", [ 12, 3 ] ], 
  [ 9, [ 3, 7, 2, 2 ], "C2 x A4", [ 24, 13 ] ], 
  [ 10, [ 3, 7, 2, 3 ], "C2 x A4", [ 24, 13 ] ], 
  [ 11, [ 3, 7, 3, 2 ], "S4", [ 24, 12 ] ], 
  [ 12, [ 3, 7, 3, 3 ], "S4", [ 24, 12 ] ], 
  [ 13, [ 3, 7, 4, 2 ], "S4", [ 24, 12 ] ], 
  [ 14, [ 3, 7, 5, 2 ], "C2 x S4", [ 48, 48 ] ], 
  [ 15, [ 3, 7, 5, 3 ], "C2 x S4", [ 48, 48 ] ] ]
\end{verbatim}
\end{example}

\bigskip
\begin{example}[The groups $N_{31,i}$ as in Table $7$]~\vspace*{-5mm}\\
\begin{verbatim}
gap> Read("BCAlgTori.gap");
gap> N31table:=List(N31,x->[Position(N31,x),x,
> StructureDescription(MatGroupZClass(x[1],x[2],x[3],x[4])),
> IdSmallGroup(MatGroupZClass(x[1],x[2],x[3],x[4])),0,0]);;
gap> for x in N3 do
> i:=Position(N3,x);
> N3M1:=AllSubdirectProducts(MatGroupZClass(x[1],x[2],x[3],x[4]),Group([[-1]]));
> for G in N3M1 do
> j:=Position(N31,CrystCatZClass(SubdirectProduct2Matrix(G)));
> if N31table[j][5]<>0 then Print("error"); fi;
> N31table[j][5]:=i;
> if Order(G)=Order(MatGroupZClass(x[1],x[2],x[3],x[4])) then N31table[j][6]:=-1;
> else N31table[j][6]:=2; fi;
> od;od;
gap> for x in N3 do
> i:=Position(N3,x);
> N3Z:=AllSubdirectProducts(MatGroupZClass(x[1],x[2],x[3],x[4]),Group([[1]]));
> for G in N3Z do
> j:=Position(N31,CrystCatZClass(SubdirectProduct2Matrix(G)));
> if N31table[j][5]<>0 then Print("error"); fi;
> N31table[j][5]:=i;
> N31table[j][6]:=1;
> od;od;
gap> N31table;
[ [ 1, [ 4, 4, 3, 6 ], "C2 x C2", [ 4, 2 ], 1, -1 ], 
  [ 2, [ 4, 4, 4, 4 ], "C2 x C2 x C2", [ 8, 5 ], 3, -1 ], 
  [ 3, [ 4, 4, 4, 6 ], "C2 x C2 x C2", [ 8, 5 ], 2, -1 ], 
  [ 4, [ 4, 5, 1, 9 ], "C2 x C2", [ 4, 2 ], 1, 1 ], 
  [ 5, [ 4, 5, 2, 4 ], "C2 x C2 x C2", [ 8, 5 ], 3, -1 ], 
  [ 6, [ 4, 5, 2, 7 ], "C2 x C2 x C2", [ 8, 5 ], 2, -1 ], 
  [ 7, [ 4, 6, 1, 4 ], "C2 x C2 x C2", [ 8, 5 ], 3, 1 ], 
  [ 8, [ 4, 6, 1, 8 ], "C2 x C2 x C2", [ 8, 5 ], 2, 1 ], 
  [ 9, [ 4, 6, 2, 4 ], "C2 x C2 x C2", [ 8, 5 ], 3, -1 ], 
  [ 10, [ 4, 6, 2, 8 ], "C2 x C2 x C2", [ 8, 5 ], 2, -1 ], 
  [ 11, [ 4, 6, 2, 9 ], "C2 x C2 x C2", [ 8, 5 ], 1, 2 ], 
  [ 12, [ 4, 6, 3, 3 ], "C2 x C2 x C2 x C2", [ 16, 14 ], 3, 2 ], 
  [ 13, [ 4, 6, 3, 6 ], "C2 x C2 x C2 x C2", [ 16, 14 ], 2, 2 ], 
  [ 14, [ 4, 7, 3, 2 ], "C4 x C2", [ 8, 2 ], 4, -1 ], 
  [ 15, [ 4, 7, 4, 3 ], "D8", [ 8, 3 ], 6, -1 ], 
  [ 16, [ 4, 7, 5, 2 ], "D8", [ 8, 3 ], 5, -1 ], 
  [ 17, [ 4, 7, 7, 2 ], "C2 x D8", [ 16, 11 ], 7, -1 ], 
  [ 18, [ 4, 12, 2, 4 ], "C4 x C2", [ 8, 2 ], 4, -1 ], 
  [ 19, [ 4, 12, 3, 7 ], "D8", [ 8, 3 ], 6, 1 ], 
  [ 20, [ 4, 12, 4, 6 ], "D8", [ 8, 3 ], 6, -1 ], 
  [ 21, [ 4, 12, 4, 8 ], "D8", [ 8, 3 ], 5, -1 ], 
  [ 22, [ 4, 12, 4, 9 ], "D8", [ 8, 3 ], 5, -1 ], 
  [ 23, [ 4, 12, 5, 6 ], "C2 x D8", [ 16, 11 ], 7, -1 ], 
  [ 24, [ 4, 12, 5, 7 ], "C2 x D8", [ 16, 11 ], 7, -1 ], 
  [ 25, [ 4, 13, 1, 3 ], "C4 x C2", [ 8, 2 ], 4, 1 ], 
  [ 26, [ 4, 13, 2, 4 ], "C4 x C2", [ 8, 2 ], 4, -1 ], 
  [ 27, [ 4, 13, 3, 4 ], "D8", [ 8, 3 ], 6, -1 ], 
  [ 28, [ 4, 13, 4, 3 ], "D8", [ 8, 3 ], 5, 1 ], 
  [ 29, [ 4, 13, 5, 3 ], "C4 x C2 x C2", [ 16, 10 ], 4, 2 ], 
  [ 30, [ 4, 13, 6, 3 ], "C2 x D8", [ 16, 11 ], 7, 1 ], 
  [ 31, [ 4, 13, 7, 6 ], "C2 x D8", [ 16, 11 ], 6, 2 ], 
  [ 32, [ 4, 13, 7, 7 ], "C2 x D8", [ 16, 11 ], 7, -1 ], 
  [ 33, [ 4, 13, 7, 8 ], "C2 x D8", [ 16, 11 ], 7, -1 ], 
  [ 34, [ 4, 13, 8, 3 ], "C2 x D8", [ 16, 11 ], 7, -1 ], 
  [ 35, [ 4, 13, 8, 4 ], "C2 x D8", [ 16, 11 ], 5, 2 ], 
  [ 36, [ 4, 13, 9, 3 ], "C2 x D8", [ 16, 11 ], 7, -1 ], 
  [ 37, [ 4, 13, 10, 3 ], "C2 x C2 x D8", [ 32, 46 ], 7, 2 ], 
  [ 38, [ 4, 24, 1, 5 ], "A4", [ 12, 3 ], 8, 1 ], 
  [ 39, [ 4, 24, 2, 3 ], "C2 x A4", [ 24, 13 ], 10, -1 ], 
  [ 40, [ 4, 24, 2, 5 ], "C2 x A4", [ 24, 13 ], 9, -1 ], 
  [ 41, [ 4, 24, 3, 5 ], "S4", [ 24, 12 ], 13, 1 ], 
  [ 42, [ 4, 24, 4, 3 ], "S4", [ 24, 12 ], 12, -1 ], 
  [ 43, [ 4, 24, 4, 5 ], "S4", [ 24, 12 ], 11, -1 ], 
  [ 44, [ 4, 24, 5, 3 ], "C2 x S4", [ 48, 48 ], 15, -1 ], 
  [ 45, [ 4, 24, 5, 5 ], "C2 x S4", [ 48, 48 ], 14, -1 ], 
  [ 46, [ 4, 25, 1, 2 ], "C2 x A4", [ 24, 13 ], 10, 1 ], 
  [ 47, [ 4, 25, 1, 4 ], "C2 x A4", [ 24, 13 ], 9, 1 ], 
  [ 48, [ 4, 25, 2, 4 ], "C2 x A4", [ 24, 13 ], 8, 2 ], 
  [ 49, [ 4, 25, 3, 2 ], "S4", [ 24, 12 ], 12, 1 ], 
  [ 50, [ 4, 25, 3, 4 ], "S4", [ 24, 12 ], 11, 1 ], 
  [ 51, [ 4, 25, 4, 4 ], "S4", [ 24, 12 ], 13, -1 ], 
  [ 52, [ 4, 25, 5, 2 ], "C2 x C2 x A4", [ 48, 49 ], 10, 2 ], 
  [ 53, [ 4, 25, 5, 4 ], "C2 x C2 x A4", [ 48, 49 ], 9, 2 ], 
  [ 54, [ 4, 25, 6, 2 ], "C2 x S4", [ 48, 48 ], 15, -1 ], 
  [ 55, [ 4, 25, 6, 4 ], "C2 x S4", [ 48, 48 ], 14, -1 ], 
  [ 56, [ 4, 25, 7, 2 ], "C2 x S4", [ 48, 48 ], 15, 1 ], 
  [ 57, [ 4, 25, 7, 4 ], "C2 x S4", [ 48, 48 ], 14, 1 ], 
  [ 58, [ 4, 25, 8, 2 ], "C2 x S4", [ 48, 48 ], 15, -1 ], 
  [ 59, [ 4, 25, 8, 4 ], "C2 x S4", [ 48, 48 ], 14, -1 ], 
  [ 60, [ 4, 25, 9, 4 ], "C2 x S4", [ 48, 48 ], 13, 2 ], 
  [ 61, [ 4, 25, 10, 2 ], "C2 x S4", [ 48, 48 ], 12, 2 ], 
  [ 62, [ 4, 25, 10, 4 ], "C2 x S4", [ 48, 48 ], 11, 2 ], 
  [ 63, [ 4, 25, 11, 2 ], "C2 x C2 x S4", [ 96, 226 ], 15, 2 ], 
  [ 64, [ 4, 25, 11, 4 ], "C2 x C2 x S4", [ 96, 226 ], 14, 2 ] ]
\end{verbatim}
\end{example}

\bigskip
\begin{example}[The groups $N_{4,i}$ as in Table $8$]~\vspace*{-5mm}\\
\begin{verbatim}
gap> Read("BCAlgTori.gap");
gap> N4table:=List(N4,x->[Position(N4,x),x,
> StructureDescription(MatGroupZClass(x[1],x[2],x[3],x[4])),
> IdSmallGroup(MatGroupZClass(x[1],x[2],x[3],x[4]))]);
[ [ 1, [ 4, 5, 1, 12 ], "C2 x C2", [ 4, 2 ] ], 
  [ 2, [ 4, 5, 2, 5 ], "C2 x C2 x C2", [ 8, 5 ] ], 
  [ 3, [ 4, 5, 2, 8 ], "C2 x C2 x C2", [ 8, 5 ] ], 
  [ 4, [ 4, 5, 2, 9 ], "C2 x C2 x C2", [ 8, 5 ] ], 
  [ 5, [ 4, 6, 1, 6 ], "C2 x C2 x C2", [ 8, 5 ] ], 
  [ 6, [ 4, 6, 1, 11 ], "C2 x C2 x C2", [ 8, 5 ] ], 
  [ 7, [ 4, 6, 2, 6 ], "C2 x C2 x C2", [ 8, 5 ] ], 
  [ 8, [ 4, 6, 2, 10 ], "C2 x C2 x C2", [ 8, 5 ] ], 
  [ 9, [ 4, 6, 2, 12 ], "C2 x C2 x C2", [ 8, 5 ] ], 
  [ 10, [ 4, 6, 3, 4 ], "C2 x C2 x C2 x C2", [ 16, 14 ] ], 
  [ 11, [ 4, 6, 3, 7 ], "C2 x C2 x C2 x C2", [ 16, 14 ] ], 
  [ 12, [ 4, 6, 3, 8 ], "C2 x C2 x C2 x C2", [ 16, 14 ] ], 
  [ 13, [ 4, 12, 2, 5 ], "C4 x C2", [ 8, 2 ] ], 
  [ 14, [ 4, 12, 2, 6 ], "C4 x C2", [ 8, 2 ] ], 
  [ 15, [ 4, 12, 3, 11 ], "D8", [ 8, 3 ] ], 
  [ 16, [ 4, 12, 4, 10 ], "D8", [ 8, 3 ] ], 
  [ 17, [ 4, 12, 4, 11 ], "D8", [ 8, 3 ] ], 
  [ 18, [ 4, 12, 4, 12 ], "D8", [ 8, 3 ] ], 
  [ 19, [ 4, 12, 5, 8 ], "C2 x D8", [ 16, 11 ] ], 
  [ 20, [ 4, 12, 5, 9 ], "C2 x D8", [ 16, 11 ] ], 
  [ 21, [ 4, 12, 5, 10 ], "C2 x D8", [ 16, 11 ] ], 
  [ 22, [ 4, 12, 5, 11 ], "C2 x D8", [ 16, 11 ] ], 
  [ 23, [ 4, 13, 1, 5 ], "C4 x C2", [ 8, 2 ] ], 
  [ 24, [ 4, 13, 2, 5 ], "C4 x C2", [ 8, 2 ] ], 
  [ 25, [ 4, 13, 3, 5 ], "D8", [ 8, 3 ] ], 
  [ 26, [ 4, 13, 4, 5 ], "D8", [ 8, 3 ] ], 
  [ 27, [ 4, 13, 5, 4 ], "C4 x C2 x C2", [ 16, 10 ] ], 
  [ 28, [ 4, 13, 5, 5 ], "C4 x C2 x C2", [ 16, 10 ] ], 
  [ 29, [ 4, 13, 6, 5 ], "C2 x D8", [ 16, 11 ] ], 
  [ 30, [ 4, 13, 7, 9 ], "C2 x D8", [ 16, 11 ] ], 
  [ 31, [ 4, 13, 7, 10 ], "C2 x D8", [ 16, 11 ] ], 
  [ 32, [ 4, 13, 7, 11 ], "C2 x D8", [ 16, 11 ] ], 
  [ 33, [ 4, 13, 8, 5 ], "C2 x D8", [ 16, 11 ] ], 
  [ 34, [ 4, 13, 8, 6 ], "C2 x D8", [ 16, 11 ] ], 
  [ 35, [ 4, 13, 9, 4 ], "C2 x D8", [ 16, 11 ] ], 
  [ 36, [ 4, 13, 9, 5 ], "C2 x D8", [ 16, 11 ] ], 
  [ 37, [ 4, 13, 10, 4 ], "C2 x C2 x D8", [ 32, 46 ] ], 
  [ 38, [ 4, 13, 10, 5 ], "C2 x C2 x D8", [ 32, 46 ] ], 
  [ 39, [ 4, 18, 1, 3 ], "C4 x C2", [ 8, 2 ] ], 
  [ 40, [ 4, 18, 2, 4 ], "(C4 x C2) : C2", [ 16, 3 ] ], 
  [ 41, [ 4, 18, 2, 5 ], "(C4 x C2) : C2", [ 16, 3 ] ], 
  [ 42, [ 4, 18, 3, 5 ], "(C4 x C2) : C2", [ 16, 3 ] ], 
  [ 43, [ 4, 18, 3, 6 ], "(C4 x C2) : C2", [ 16, 3 ] ], 
  [ 44, [ 4, 18, 3, 7 ], "(C4 x C2) : C2", [ 16, 3 ] ], 
  [ 45, [ 4, 18, 4, 4 ], "C2 x D8", [ 16, 11 ] ], 
  [ 46, [ 4, 18, 4, 5 ], "C2 x D8", [ 16, 11 ] ], 
  [ 47, [ 4, 18, 5, 5 ], "(C2 x C2 x C2 x C2) : C2", [ 32, 27 ] ], 
  [ 48, [ 4, 18, 5, 6 ], "(C2 x C2 x C2 x C2) : C2", [ 32, 27 ] ], 
  [ 49, [ 4, 18, 5, 7 ], "(C2 x C2 x C2 x C2) : C2", [ 32, 27 ] ], 
  [ 50, [ 4, 19, 1, 2 ], "C4 : C4", [ 16, 4 ] ], 
  [ 51, [ 4, 19, 2, 2 ], "C4 x C4", [ 16, 2 ] ], 
  [ 52, [ 4, 19, 3, 2 ], "C4 x D8", [ 32, 25 ] ], 
  [ 53, [ 4, 19, 4, 3 ], "(C4 x C2 x C2) : C2", [ 32, 28 ] ], 
  [ 54, [ 4, 19, 4, 4 ], "(C4 x C2 x C2) : C2", [ 32, 28 ] ], 
  [ 55, [ 4, 19, 5, 2 ], "(C4 x C4) : C2", [ 32, 34 ] ], 
  [ 56, [ 4, 19, 6, 2 ], "D8 x D8", [ 64, 226 ] ], 
  [ 57, [ 4, 22, 1, 1 ], "C3 x C3", [ 9, 2 ] ], 
  [ 58, [ 4, 22, 2, 1 ], "C6 x C3", [ 18, 5 ] ], 
  [ 59, [ 4, 22, 3, 1 ], "C3 x S3", [ 18, 3 ] ], 
  [ 60, [ 4, 22, 4, 1 ], "C3 x S3", [ 18, 3 ] ], 
  [ 61, [ 4, 22, 5, 1 ], "(C3 x C3) : C2", [ 18, 4 ] ], 
  [ 62, [ 4, 22, 5, 2 ], "(C3 x C3) : C2", [ 18, 4 ] ], 
  [ 63, [ 4, 22, 6, 1 ], "C6 x S3", [ 36, 12 ] ], 
  [ 64, [ 4, 22, 7, 1 ], "C2 x ((C3 x C3) : C2)", [ 36, 13 ] ], 
  [ 65, [ 4, 22, 8, 1 ], "S3 x S3", [ 36, 10 ] ], 
  [ 66, [ 4, 22, 9, 1 ], "S3 x S3", [ 36, 10 ] ], 
  [ 67, [ 4, 22, 10, 1 ], "S3 x S3", [ 36, 10 ] ], 
  [ 68, [ 4, 22, 11, 1 ], "C2 x S3 x S3", [ 72, 46 ] ], 
  [ 69, [ 4, 24, 2, 4 ], "C2 x A4", [ 24, 13 ] ], 
  [ 70, [ 4, 24, 2, 6 ], "C2 x A4", [ 24, 13 ] ], 
  [ 71, [ 4, 24, 4, 4 ], "S4", [ 24, 12 ] ], 
  [ 72, [ 4, 24, 5, 4 ], "C2 x S4", [ 48, 48 ] ], 
  [ 73, [ 4, 24, 5, 6 ], "C2 x S4", [ 48, 48 ] ], 
  [ 74, [ 4, 25, 1, 3 ], "C2 x A4", [ 24, 13 ] ], 
  [ 75, [ 4, 25, 2, 3 ], "C2 x A4", [ 24, 13 ] ], 
  [ 76, [ 4, 25, 2, 5 ], "C2 x A4", [ 24, 13 ] ], 
  [ 77, [ 4, 25, 3, 3 ], "S4", [ 24, 12 ] ], 
  [ 78, [ 4, 25, 4, 3 ], "S4", [ 24, 12 ] ], 
  [ 79, [ 4, 25, 5, 3 ], "C2 x C2 x A4", [ 48, 49 ] ], 
  [ 80, [ 4, 25, 5, 5 ], "C2 x C2 x A4", [ 48, 49 ] ], 
  [ 81, [ 4, 25, 6, 3 ], "C2 x S4", [ 48, 48 ] ], 
  [ 82, [ 4, 25, 6, 5 ], "C2 x S4", [ 48, 48 ] ], 
  [ 83, [ 4, 25, 7, 3 ], "C2 x S4", [ 48, 48 ] ], 
  [ 84, [ 4, 25, 8, 3 ], "C2 x S4", [ 48, 48 ] ], 
  [ 85, [ 4, 25, 9, 3 ], "C2 x S4", [ 48, 48 ] ], 
  [ 86, [ 4, 25, 9, 5 ], "C2 x S4", [ 48, 48 ] ], 
  [ 87, [ 4, 25, 10, 3 ], "C2 x S4", [ 48, 48 ] ], 
  [ 88, [ 4, 25, 10, 5 ], "C2 x S4", [ 48, 48 ] ], 
  [ 89, [ 4, 25, 11, 3 ], "C2 x C2 x S4", [ 96, 226 ] ], 
  [ 90, [ 4, 25, 11, 5 ], "C2 x C2 x S4", [ 96, 226 ] ], 
  [ 91, [ 4, 29, 1, 1 ], "C3 x S3", [ 18, 3 ] ], 
  [ 92, [ 4, 29, 1, 2 ], "C3 x S3", [ 18, 3 ] ], 
  [ 93, [ 4, 29, 2, 1 ], "C6 x S3", [ 36, 12 ] ], 
  [ 94, [ 4, 29, 3, 1 ], "S3 x S3", [ 36, 10 ] ], 
  [ 95, [ 4, 29, 3, 2 ], "S3 x S3", [ 36, 10 ] ], 
  [ 96, [ 4, 29, 3, 3 ], "S3 x S3", [ 36, 10 ] ], 
  [ 97, [ 4, 29, 4, 1 ], "(C3 x C3) : C4", [ 36, 9 ] ], 
  [ 98, [ 4, 29, 4, 2 ], "(C3 x C3) : C4", [ 36, 9 ] ], 
  [ 99, [ 4, 29, 5, 1 ], "C2 x S3 x S3", [ 72, 46 ] ], 
  [ 100, [ 4, 29, 6, 1 ], "C2 x ((C3 x C3) : C4)", [ 72, 45 ] ], 
  [ 101, [ 4, 29, 7, 1 ], "(S3 x S3) : C2", [ 72, 40 ] ], 
  [ 102, [ 4, 29, 7, 2 ], "(S3 x S3) : C2", [ 72, 40 ] ], 
  [ 103, [ 4, 29, 8, 1 ], "(S3 x S3) : C2", [ 72, 40 ] ], 
  [ 104, [ 4, 29, 8, 2 ], "(S3 x S3) : C2", [ 72, 40 ] ], 
  [ 105, [ 4, 29, 9, 1 ], "C2 x ((S3 x S3) : C2)", [ 144, 186 ] ], 
  [ 106, [ 4, 32, 1, 2 ], "Q8", [ 8, 4 ] ], 
  [ 107, [ 4, 32, 2, 2 ], "C8 : C2", [ 16, 6 ] ], 
  [ 108, [ 4, 32, 3, 2 ], "QD16", [ 16, 8 ] ], 
  [ 109, [ 4, 32, 4, 2 ], "(C4 x C2) : C2", [ 16, 13 ] ], 
  [ 110, [ 4, 32, 5, 2 ], "SL(2,3)", [ 24, 3 ] ], 
  [ 111, [ 4, 32, 5, 3 ], "SL(2,3)", [ 24, 3 ] ], 
  [ 112, [ 4, 32, 6, 2 ], "C8 : (C2 x C2)", [ 32, 43 ] ], 
  [ 113, [ 4, 32, 7, 2 ], "(C8 : C2) : C2", [ 32, 7 ] ], 
  [ 114, [ 4, 32, 8, 2 ], "(C4 x C4) : C2", [ 32, 11 ] ], 
  [ 115, [ 4, 32, 9, 4 ], "(C2 x C2 x C2) : C4", [ 32, 6 ] ], 
  [ 116, [ 4, 32, 9, 5 ], "(C2 x C2 x C2) : C4", [ 32, 6 ] ], 
  [ 117, [ 4, 32, 10, 2 ], "(C2 x C2 x C2) : (C2 x C2)", [ 32, 49 ] ], 
  [ 118, [ 4, 32, 11, 2 ], "GL(2,3)", [ 48, 29 ] ], 
  [ 119, [ 4, 32, 11, 3 ], "GL(2,3)", [ 48, 29 ] ], 
  [ 120, [ 4, 32, 12, 2 ], "((C8 : C2) : C2) : C2", [ 64, 134 ] ], 
  [ 121, [ 4, 32, 13, 3 ], "((C8 : C2) : C2) : C2", [ 64, 32 ] ], 
  [ 122, [ 4, 32, 13, 4 ], "((C8 : C2) : C2) : C2", [ 64, 32 ] ], 
  [ 123, [ 4, 32, 14, 3 ], "((C2 x C2 x C2 x C2) : C2) : C2", [ 64, 138 ] ], 
  [ 124, [ 4, 32, 14, 4 ], "((C2 x C2 x C2 x C2) : C2) : C2", [ 64, 138 ] ], 
  [ 125, [ 4, 32, 15, 2 ], "((C2 x C2 x C2) : C4) : C2", [ 64, 34 ] ], 
  [ 126, [ 4, 32, 16, 2 ], "((C2 x C2 x C2) : (C2 x C2)) : C3", [ 96, 204 ] ], 
  [ 127, [ 4, 32, 16, 3 ], "((C2 x C2 x C2) : (C2 x C2)) : C3", [ 96, 204 ] ], 
  [ 128, [ 4, 32, 17, 2 ], "(D8 x D8) : C2", [ 128, 928 ] ], 
  [ 129, [ 4, 32, 18, 2 ], "(((C2 x C2 x C2 x C2) : C2) : C2) : C3", [ 192, 201 ] ], 
  [ 130, [ 4, 32, 18, 3 ], "(((C2 x C2 x C2 x C2) : C2) : C2) : C3", [ 192, 201 ] ], 
  [ 131, [ 4, 32, 19, 2 ], "(((C2 x C2 x C2) : (C2 x C2)) : C3) : C2", [ 192, 1493 ] ], 
  [ 132, [ 4, 32, 19, 3 ], "(((C2 x C2 x C2) : (C2 x C2)) : C3) : C2", [ 192, 1493 ] ], 
  [ 133, [ 4, 32, 20, 2 ], "(((C2 x C2 x C2) : (C2 x C2)) : C3) : C2", [ 192, 1494 ] ], 
  [ 134, [ 4, 32, 20, 3 ], "(((C2 x C2 x C2) : (C2 x C2)) : C3) : C2", [ 192, 1494 ] ], 
  [ 135, [ 4, 32, 21, 2 ], "((((C2 x C2 x C2) : (C2 x C2)) : C3) : C2) : C2", [ 384, 5602 ] ], 
  [ 136, [ 4, 32, 21, 3 ], "((((C2 x C2 x C2) : (C2 x C2)) : C3) : C2) : C2", [ 384, 5602 ] ], 
  [ 137, [ 4, 33, 1, 1 ], "C3 x Q8", [ 24, 11 ] ], 
  [ 138, [ 4, 33, 3, 1 ], "SL(2,3)", [ 24, 3 ] ], 
  [ 139, [ 4, 33, 4, 1 ], "(C3 x Q8) : C2", [ 48, 17 ] ], 
  [ 140, [ 4, 33, 5, 1 ], "((C4 x C2) : C2) : C3", [ 48, 33 ] ], 
  [ 141, [ 4, 33, 6, 1 ], "GL(2,3)", [ 48, 29 ] ], 
  [ 142, [ 4, 33, 7, 1 ], "C3 x SL(2,3)", [ 72, 25 ] ], 
  [ 143, [ 4, 33, 8, 1 ], "((C2 x C2 x C2) : (C2 x C2)) : C3", [ 96, 201 ] ],
  [ 144, [ 4, 33, 9, 1 ], "GL(2,3) : C2", [ 96, 193 ] ], 
  [ 145, [ 4, 33, 10, 1 ], "SL(2,3) : C4", [ 96, 67 ] ], 
  [ 146, [ 4, 33, 11, 1 ], "(C3 x SL(2,3)) : C2", [ 144, 125 ] ], 
  [ 147, [ 4, 33, 12, 1 ], "(GL(2,3) : C2) : C2", [ 192, 988 ] ], 
  [ 148, [ 4, 33, 13, 1 ], "((C2 x C2 x C2) : (C2 x C2)) : (C3 x C3)", [ 288, 860 ] ], 
  [ 149, [ 4, 33, 14, 1 ], "((((C2 x C2 x C2) : (C2 x C2)) : C3) : C2) : C3", [ 576, 8277 ] ], 
  [ 150, [ 4, 33, 14, 2 ], "((((C2 x C2 x C2) : (C2 x C2)) : C3) : C2) : C3", [ 576, 8277 ] ], 
  [ 151, [ 4, 33, 15, 1 ], "(((C2 x C2 x C2) : (C2 x C2)) : (C3 x C3)) : C2", [ 576, 8282 ] ], 
  [ 152, [ 4, 33, 16, 1 ], "((((C2 x C2 x C2) : (C2 x C2)) : (C3 x C3)) : C2) : C2", 
                           [ 1152, 157478 ] ] ]
\end{verbatim}
\end{example}

\bigskip
\begin{example}[The groups $I_{4,i}$ as in Table $9$]~\vspace*{-5mm}\\
\begin{verbatim}
gap> Read("BCAlgTori.gap");
gap> I4table:=List(I4,x->[Position(I4,x),x,
> StructureDescription(MatGroupZClass(x[1],x[2],x[3],x[4])),
> IdSmallGroup(MatGroupZClass(x[1],x[2],x[3],x[4]))]);
[ [ 1, [ 4, 31, 1, 3 ], "C5 : C4", [ 20, 3 ] ], 
  [ 2, [ 4, 31, 1, 4 ], "C5 : C4", [ 20, 3 ] ], 
  [ 3, [ 4, 31, 2, 2 ], "C2 x (C5 : C4)", [ 40, 12 ] ], 
  [ 4, [ 4, 31, 4, 2 ], "S5", [ 120, 34 ] ], 
  [ 5, [ 4, 31, 5, 2 ], "S5", [ 120, 34 ] ], 
  [ 6, [ 4, 31, 7, 2 ], "C2 x S5", [ 240, 189 ] ], 
  [ 7, [ 4, 33, 2, 1 ], "C3 : C8", [ 24, 1 ] ] ]
\end{verbatim}
\end{example}


\section{The subgroups $N_{3,j}\leq N_{3,i}$, $N_{31,j}\leq N_{31,i}$, $N_{31,l}, N_{4,j}, I_{4,m}\leq N_{4,i}$ and $I_{4,j}\leq I_{4,i}$}\label{S3}

We give all of the subgroups 
$N_{3,j}\leq N_{3,i}$ $(1\leq i,j\leq 15)$, 
$N_{31,j}\leq N_{31,i}$ $(1\leq i,j\leq 64)$, 
$N_{31,l}, N_{4,j}, I_{4,m}\leq N_{4,i}$ $(1\leq i,j\leq 152, 1\leq l\leq 64, 1\leq m\leq 7)$ and 
$I_{4,j}\leq I_{4,i}$ $(1\leq i,j\leq 7)$ 
in Table $10$, Table $11$ and Table $12$ respectively 
where $N_{3,i}$, $N_{31,i}$, $N_{4,i}$ and $I_{4,i}$ are 
given as in Definition \ref{defN3N4}.\\ 

Table $10$: The groups $N_{3,j}\leq N_{3,i}$ $(1\leq i,j\leq15)$ 
(see Theorem \ref{thmain2} for the ``bold'' cases)
{\small 
\begin{longtable}{ll}
$i$ & $N_{3,j}\leq N_{3,i}$\\\hline
\endfirsthead
$1$ & $\bm{N_{3,1}}$\\
$2$ & $\bm{N_{3,1}}, \bm{N_{3,2}}$\\
$3$ & $\bm{N_{3,3}}$\\
$4$ & $\bm{N_{3,4}}$\\
$5$ & $\bm{N_{3,1}}, \bm{N_{3,5}}$\\
$6$ & $\bm{N_{3,1}}, \bm{N_{3,6}}$\\
$7$ & $N_{3,1}, \bm{N_{3,2}}, \bm{N_{3,3}}, \bm{N_{3,4}}, \bm{N_{3,5}}, \bm{N_{3,6}}, \bm{N_{3,7}}$\\
$8$ & $N_{3,1}, \bm{N_{3,8}}$\\
$9$ & $N_{3,1}, N_{3,2}, N_{3,8}, \bm{N_{3,9}}$\\
$10$ & $N_{3,3}, \bm{N_{3,10}}$\\
$11$ & $N_{3,1}, N_{3,5}, N_{3,8}, \bm{N_{3,11}}$\\
$12$ & $N_{3,1}, N_{3,5}, \bm{N_{3,12}}$\\
$13$ & $N_{3,1}, N_{3,6}, N_{3,8}, \bm{N_{3,13}}$\\
$14$ & $N_{3,1}, N_{3,2}, N_{3,3}, N_{3,4}, N_{3,5}, N_{3,6}, N_{3,7}, N_{3,8}, N_{3,9}, N_{3,11}, N_{3,13}, \bm{N_{3,14}}$\\
$15$ & $N_{3,1}, N_{3,2}, N_{3,3}, N_{3,4}, N_{3,5}, N_{3,6}, N_{3,7}, \bm{N_{3,10}}, \bm{N_{3,12}}, \bm{N_{3,15}}$
\end{longtable}
}
\vspace*{2mm}

Table $11$: The groups $N_{31,j}\leq N_{31,i}$ $(1\leq i,j\leq 64)$
{\small 
\begin{longtable}{ll}
$i$ & $N_{31,j}\leq N_{31,i}$\\\hline
\endfirsthead
$1$ & $N_{31,1}$\\
$2$ & $N_{31,2}$\\
$3$ & $N_{31,1}, N_{31,3}$\\
$4$ & $N_{31,4}$\\
$5$ & $N_{31,5}$\\
$6$ & $N_{31,4}, N_{31,6}$\\
$7$ & $N_{31,7}, N_{31,}$\\
$8$ & $N_{31,4}, N_{31,8}$\\
\end{longtable}
}
\vspace*{2mm}

Table $12$: The groups $N_{31,l}, N_{4,j}, I_{4,m}\leq N_{4,i}$ 
$(1\leq i,j\leq 152, 1\leq l\leq 64, 1\leq m\leq 7)$ 
(see Theorem \ref{thmain4} for the ``bold'' cases)
{\small 
\begin{longtable}{ll}
$i$ & $N_{31,j}, N_{4,j}, I_{4,j}\leq N_{4,i}$\\\hline
\endfirsthead
$1$ & $\bm{N_{4,1}}$\\
$2$ & $\bm{N_{4,2}}$\\
$3$ & $\bm{N_{4,1}}, \bm{N_{4,3}}$\\
$4$ & $\bm{N_{4,4}}$\\
$5$ & $\bm{N_{4,5}}$\\
$6$ & $\bm{N_{4,1}}, \bm{N_{4,6}}$\\
$7$ & $\bm{N_{4,7}}$\\
$8$ & $\bm{N_{31,1}}, \bm{N_{4,1}}, \bm{N_{4,8}}$\\
\end{longtable}
}
\vspace*{2mm}

Table $13$: The groups $I_{4,j}\leq I_{4,i}$ $(1\leq i,j\leq 7)$ 
(see Theorem \ref{thmain5} for the ``bold'' cases)
{\small 
\begin{longtable}{ll}
$i$ & $I_{4,j}\leq I_{4,i}$\\\hline
\endfirsthead
$1$ & $\bm{I_{4,1}}$\\
$2$ & $\bm{I_{4,2}}$\\
$3$ & $\bm{I_{4,1}}, \bm{I_{4,2}}, \bm{I_{4,3}}$\\
$4$ & $I_{4,1}, \bm{I_{4,4}}$\\
$5$ & $I_{4,2}, \bm{I_{4,5}}$\\
$6$ & $I_{4,1}, I_{4,2}, I_{4,3}, \bm{I_{4,4}}, \bm{I_{4,5}}, \bm{I_{4,6}}$\\
$7$ & $\bm{I_{4,7}}$
\end{longtable}
}
\vspace*{2mm}
Tables $10, 11, 12, 13$ can be obtained by GAP computations as follows: 

\bigskip
\begin{example}[The groups $N_{3,j}\leq N_{3,i}$ as in Table 10]~\vspace*{-5mm}\\
\begin{verbatim}
gap> Read("BCAlgTori.gap");
gap> for i in [1..15] do
> x:=N3[i];
> G:=MatGroupZClass(x[1],x[2],x[3],x[4]);
> Gcs:=ConjugacyClassesSubgroups2(G);
> Gcsi:=List(Gcs,x->CrystCatZClass(Representative(x)));
> N3j:=Difference(Set(Gcsi,x->Position(N3,x)),[fail]);
> Print(i,"\t",N3j,"\n");
> od;
1	[ 1 ]
2	[ 1, 2 ]
3	[ 3 ]
4	[ 4 ]
5	[ 1, 5 ]
6	[ 1, 6 ]
7	[ 1, 2, 3, 4, 5, 6, 7 ]
8	[ 1, 8 ]
9	[ 1, 2, 8, 9 ]
10	[ 3, 10 ]
11	[ 1, 5, 8, 11 ]
12	[ 1, 5, 12 ]
13	[ 1, 6, 8, 13 ]
14	[ 1, 2, 3, 4, 5, 6, 7, 8, 9, 11, 13, 14 ]
15	[ 1, 2, 3, 4, 5, 6, 7, 10, 12, 15 ]
\end{verbatim}
\end{example}

\bigskip
\begin{example}[The groups $N_{31,j}\leq N_{31,i}$ as in Table 11]~\vspace*{-5mm}\\
\begin{verbatim}
gap> Read("BCAlgTori.gap");
gap> for i in [1..64] do
> x:=N31[i];
> G:=MatGroupZClass(x[1],x[2],x[3],x[4]);
> Gcs:=ConjugacyClassesSubgroups2(G);
> Gcsi:=List(Gcs,x->CrystCatZClass(Representative(x)));
> N31j:=Difference(Set(Gcsi,x->Position(N31,x)),[fail]);
> Print(i,"\t",N31j,"\n");
> od;
1	[ 1 ]
2	[ 2 ]
3	[ 1, 3 ]
4	[ 4 ]
5	[ 5 ]
6	[ 4, 6 ]
7	[ 7 ]
8	[ 4, 8 ]
9	[ 9 ]
10	[ 1, 10 ]
11	[ 1, 4, 11 ]
12	[ 2, 5, 7, 9, 12 ]
13	[ 1, 3, 4, 6, 8, 10, 11, 13 ]
14	[ 14 ]
15	[ 1, 15 ]
16	[ 1, 16 ]
17	[ 1, 2, 3, 14, 15, 16, 17 ]
18	[ 18 ]
19	[ 4, 19 ]
20	[ 1, 20 ]
21	[ 4, 21 ]
22	[ 1, 22 ]
23	[ 1, 3, 5, 18, 20, 22, 23 ]
24	[ 2, 4, 6, 18, 19, 21, 24 ]
25	[ 25 ]
26	[ 26 ]
27	[ 4, 27 ]
28	[ 4, 28 ]
29	[ 14, 18, 25, 26, 29 ]
30	[ 4, 7, 8, 19, 25, 28, 30 ]
31	[ 1, 4, 11, 15, 19, 20, 27, 31 ]
32	[ 4, 8, 9, 21, 26, 27, 32 ]
33	[ 1, 7, 10, 15, 22, 26, 33 ]
34	[ 1, 9, 10, 16, 20, 25, 34 ]
35	[ 1, 4, 11, 16, 21, 22, 28, 35 ]
36	[ 4, 5, 6, 14, 27, 28, 36 ]
37	[ 1, 2, 3, 4, 5, 6, 7, 8, 9, 10, 11, 12, 13, 14, 15, 16, 17, 18, 19, 20, 
     21, 22, 23, 24, 25, 26, 27, 28, 29, 30, 31, 32, 33, 34, 35, 36, 37 ]
38	[ 4, 38 ]
39	[ 5, 39 ]
40	[ 4, 6, 38, 40 ]
41	[ 4, 19, 38, 41 ]
42	[ 1, 22, 42 ]
43	[ 4, 21, 38, 43 ]
44	[ 1, 3, 5, 18, 20, 22, 23, 39, 42, 44 ]
45	[ 2, 4, 6, 18, 19, 21, 24, 38, 40, 41, 43, 45 ]
46	[ 7, 46 ]
47	[ 4, 8, 38, 47 ]
48	[ 1, 4, 11, 38, 48 ]
49	[ 4, 28, 49 ]
50	[ 4, 28, 38, 50 ]
51	[ 4, 27, 38, 51 ]
52	[ 2, 5, 7, 9, 12, 39, 46, 52 ]
53	[ 1, 3, 4, 6, 8, 10, 11, 13, 38, 40, 47, 48, 53 ]
54	[ 4, 5, 6, 14, 27, 28, 36, 39, 49, 54 ]
55	[ 4, 5, 6, 14, 27, 28, 36, 38, 40, 50, 51, 55 ]
56	[ 4, 7, 8, 19, 25, 28, 30, 46, 49, 56 ]
57	[ 4, 7, 8, 19, 25, 28, 30, 38, 41, 47, 50, 57 ]
58	[ 1, 7, 10, 15, 22, 26, 33, 42, 46, 58 ]
59	[ 4, 8, 9, 21, 26, 27, 32, 38, 43, 47, 51, 59 ]
60	[ 1, 4, 11, 15, 19, 20, 27, 31, 38, 41, 48, 51, 60 ]
61	[ 1, 4, 11, 16, 21, 22, 28, 35, 42, 49, 61 ]
62	[ 1, 4, 11, 16, 21, 22, 28, 35, 38, 43, 48, 50, 62 ]
63	[ 1, 2, 3, 4, 5, 6, 7, 8, 9, 10, 11, 12, 13, 14, 15, 16, 17, 18, 19, 20, 
     21, 22, 23, 24, 25, 26, 27, 28, 29, 30, 31, 32, 33, 34, 35, 36, 37, 39, 
     42, 44, 46, 49, 52, 54, 56, 58, 61, 63 ]
64	[ 1, 2, 3, 4, 5, 6, 7, 8, 9, 10, 11, 12, 13, 14, 15, 16, 17, 18, 19, 20, 
     21, 22, 23, 24, 25, 26, 27, 28, 29, 30, 31, 32, 33, 34, 35, 36, 37, 38, 
     40, 41, 43, 45, 47, 48, 50, 51, 53, 55, 57, 59, 60, 62, 64 ]
\end{verbatim}
\end{example}

\bigskip
\begin{example}[The groups $N_{31,j}, N_{4,j}, I_{4,j}\leq N_{4,i}$ as in Table 12]~\vspace*{-5mm}\\
\begin{verbatim}
gap> Read("BCAlgTori.gap");
gap> for i in [1..152] do
> x:=N4[i];
> G:=MatGroupZClass(x[1],x[2],x[3],x[4]);
> Gcs:=ConjugacyClassesSubgroups2(G);
> Gcsi:=List(Gcs,x->CrystCatZClass(Representative(x)));
> N31j:=Difference(Set(Gcsi,x->Position(N31,x)),[fail]);
> N4j:=Difference(Set(Gcsi,x->Position(N4,x)),[fail]);
> I4j:=Difference(Set(Gcsi,x->Position(I4,x)),[fail]);
> Print(i,"\t",N31j,"\t",N4j,"\t",I4j,"\n");
> od;
1	[  ]	[ 1 ]	[  ]
2	[  ]	[ 2 ]	[  ]
3	[  ]	[ 1, 3 ]	[  ]
4	[  ]	[ 4 ]	[  ]
5	[  ]	[ 5 ]	[  ]
6	[  ]	[ 1, 6 ]	[  ]
7	[  ]	[ 7 ]	[  ]
8	[ 1 ]	[ 1, 8 ]	[  ]
9	[ 1 ]	[ 9 ]	[  ]
10	[ 2 ]	[ 2, 5, 7, 10 ]	[  ]
11	[ 1, 2, 3 ]	[ 1, 3, 6, 8, 11 ]	[  ]
12	[ 1, 3 ]	[ 4, 9, 12 ]	[  ]
13	[  ]	[ 13 ]	[  ]
14	[  ]	[ 14 ]	[  ]
15	[  ]	[ 1, 15 ]	[  ]
16	[ 1 ]	[ 16 ]	[  ]
17	[  ]	[ 1, 17 ]	[  ]
18	[ 1 ]	[ 18 ]	[  ]
19	[ 1, 3 ]	[ 2, 13, 16, 19 ]	[  ]
20	[ 2 ]	[ 1, 3, 13, 15, 17, 20 ]	[  ]
21	[ 1, 3 ]	[ 1, 3, 14, 18, 21 ]	[  ]
22	[ 2 ]	[ 4, 14, 22 ]	[  ]
23	[  ]	[ 23 ]	[  ]
24	[  ]	[ 24 ]	[  ]
25	[  ]	[ 1, 25 ]	[  ]
26	[  ]	[ 1, 26 ]	[  ]
27	[ 2, 14 ]	[ 13, 23, 24, 27 ]	[  ]
28	[ 1, 3, 14 ]	[ 14, 28 ]	[  ]
29	[  ]	[ 1, 5, 6, 15, 23, 26, 29 ]	[  ]
30	[  ]	[ 1, 6, 7, 17, 24, 25, 30 ]	[  ]
31	[ 1, 15 ]	[ 1, 5, 8, 15, 16, 24, 25, 31 ]	[  ]
32	[ 1, 15 ]	[ 9, 18, 32 ]	[  ]
33	[ 1, 16 ]	[ 1, 7, 8, 16, 17, 23, 26, 33 ]	[  ]
34	[ 1, 16 ]	[ 9, 18, 34 ]	[  ]
35	[ 14 ]	[ 1, 2, 3, 25, 26, 35 ]	[  ]
36	[ 14 ]	[ 1, 3, 4, 36 ]	[  ]
37	[ 1, 2, 3, 14, 15, 16, 17 ]	[ 1, 2, 3, 5, 6, 7, 8, 10, 11, 13, 15, 16, 
     17, 19, 20, 23, 24, 25, 26, 27, 29, 30, 31, 33, 35, 37 ]	[  ]
38	[ 1, 2, 3, 14, 15, 16, 17 ]	[ 1, 3, 4, 6, 8, 9, 11, 12, 14, 18, 21, 22, 
     28, 32, 34, 36, 38 ]	[  ]
39	[  ]	[ 39 ]	[  ]
40	[ 1, 3 ]	[ 13, 39, 40 ]	[  ]
41	[ 2 ]	[ 14, 39, 41 ]	[  ]
42	[  ]	[ 4, 13, 42 ]	[  ]
43	[  ]	[ 1, 3, 13, 14, 43 ]	[  ]
44	[  ]	[ 2, 14, 44 ]	[  ]
45	[  ]	[ 1, 3, 39, 45 ]	[  ]
46	[  ]	[ 2, 4, 39, 46 ]	[  ]
47	[ 1, 3 ]	[ 2, 4, 9, 12, 13, 16, 19, 39, 40, 42, 46, 47 ]	[  ]
48	[ 1, 2, 3 ]	[ 1, 3, 6, 8, 11, 13, 14, 15, 17, 18, 20, 21, 39, 40, 41, 43, 
     45, 48 ]	[  ]
49	[ 2 ]	[ 2, 4, 5, 7, 10, 14, 22, 39, 41, 44, 46, 49 ]	[  ]
50	[ 14 ]	[ 13, 14, 50 ]	[  ]
51	[ 14 ]	[ 39, 51 ]	[  ]
52	[ 1, 2, 3, 14, 15, 16, 17 ]	[ 13, 14, 23, 24, 27, 28, 39, 40, 41, 50, 51, 
     52 ]	[  ]
53	[ 1, 3, 14 ]	[ 1, 2, 3, 13, 14, 16, 18, 19, 21, 25, 26, 28, 35, 43, 44, 
     50, 53 ]	[  ]
54	[ 2, 14 ]	[ 1, 3, 4, 13, 14, 15, 17, 20, 22, 23, 24, 27, 36, 42, 43, 50, 
     54 ]	[  ]
55	[ 14 ]	[ 1, 2, 3, 4, 25, 26, 35, 36, 39, 45, 46, 51, 55 ]	[  ]
56	[ 1, 2, 3, 14, 15, 16, 17 ]	[ 1, 2, 3, 4, 5, 6, 7, 8, 9, 10, 11, 12, 13, 
     14, 15, 16, 17, 18, 19, 20, 21, 22, 23, 24, 25, 26, 27, 28, 29, 30, 31, 
     32, 33, 34, 35, 36, 37, 38, 39, 40, 41, 42, 43, 44, 45, 46, 47, 48, 49, 
     50, 51, 52, 53, 54, 55, 56 ]	[  ]
57	[  ]	[ 57 ]	[  ]
58	[  ]	[ 57, 58 ]	[  ]
59	[  ]	[ 57, 59 ]	[  ]
60	[  ]	[ 57, 60 ]	[  ]
61	[  ]	[ 57, 61 ]	[  ]
62	[  ]	[ 57, 62 ]	[  ]
63	[  ]	[ 57, 58, 59, 60, 63 ]	[  ]
64	[  ]	[ 57, 58, 61, 62, 64 ]	[  ]
65	[  ]	[ 57, 59, 61, 65 ]	[  ]
66	[  ]	[ 57, 59, 60, 62, 66 ]	[  ]
67	[  ]	[ 57, 60, 61, 67 ]	[  ]
68	[  ]	[ 57, 58, 59, 60, 61, 62, 63, 64, 65, 66, 67, 68 ]	[  ]
69	[  ]	[ 2, 69 ]	[  ]
70	[  ]	[ 4, 70 ]	[  ]
71	[ 1 ]	[ 16, 71 ]	[  ]
72	[ 1, 3 ]	[ 2, 13, 16, 19, 69, 71, 72 ]	[  ]
73	[ 2 ]	[ 4, 14, 22, 70, 73 ]	[  ]
74	[  ]	[ 5, 74 ]	[  ]
75	[  ]	[ 7, 75 ]	[  ]
76	[ 1 ]	[ 9, 76 ]	[  ]
77	[  ]	[ 1, 26, 77 ]	[  ]
78	[  ]	[ 1, 25, 78 ]	[  ]
79	[ 2 ]	[ 2, 5, 7, 10, 69, 74, 75, 79 ]	[  ]
80	[ 1, 3 ]	[ 4, 9, 12, 70, 76, 80 ]	[  ]
81	[ 14 ]	[ 1, 2, 3, 25, 26, 35, 69, 77, 78, 81 ]	[  ]
82	[ 14 ]	[ 1, 3, 4, 36, 70, 82 ]	[  ]
83	[  ]	[ 1, 5, 6, 15, 23, 26, 29, 74, 77, 83 ]	[  ]
84	[ 1, 15 ]	[ 1, 5, 8, 15, 16, 24, 25, 31, 71, 74, 78, 84 ]	[  ]
85	[  ]	[ 1, 6, 7, 17, 24, 25, 30, 75, 78, 85 ]	[  ]
86	[ 1, 15 ]	[ 9, 18, 32, 76, 86 ]	[  ]
87	[ 1, 16 ]	[ 1, 7, 8, 16, 17, 23, 26, 33, 71, 75, 77, 87 ]	[  ]
88	[ 1, 16 ]	[ 9, 18, 34, 76, 88 ]	[  ]
89	[ 1, 2, 3, 14, 15, 16, 17 ]	[ 1, 2, 3, 5, 6, 7, 8, 10, 11, 13, 15, 16, 
     17, 19, 20, 23, 24, 25, 26, 27, 29, 30, 31, 33, 35, 37, 69, 71, 72, 74, 
     75, 77, 78, 79, 81, 83, 84, 85, 87, 89 ]	[  ]
90	[ 1, 2, 3, 14, 15, 16, 17 ]	[ 1, 3, 4, 6, 8, 9, 11, 12, 14, 18, 21, 22, 
     28, 32, 34, 36, 38, 70, 73, 76, 80, 82, 86, 88, 90 ]	[  ]
91	[  ]	[ 57, 91 ]	[  ]
92	[  ]	[ 57, 92 ]	[  ]
93	[  ]	[ 57, 58, 91, 92, 93 ]	[  ]
94	[  ]	[ 57, 61, 91, 94 ]	[  ]
95	[  ]	[ 57, 61, 92, 95 ]	[  ]
96	[  ]	[ 57, 62, 91, 92, 96 ]	[  ]
97	[  ]	[ 57, 61, 97 ]	[  ]
98	[  ]	[ 57, 61, 98 ]	[  ]
99	[  ]	[ 57, 58, 61, 62, 64, 91, 92, 93, 94, 95, 96, 99 ]	[  ]
100	[  ]	[ 57, 58, 61, 62, 64, 97, 98, 100 ]	[  ]
101	[  ]	[ 57, 59, 61, 65, 91, 94, 97, 101 ]	[  ]
102	[  ]	[ 57, 59, 61, 65, 92, 95, 98, 102 ]	[  ]
103	[  ]	[ 57, 60, 61, 67, 91, 94, 98, 103 ]	[  ]
104	[  ]	[ 57, 60, 61, 67, 92, 95, 97, 104 ]	[  ]
105	[  ]	[ 57, 58, 59, 60, 61, 62, 63, 64, 65, 66, 67, 68, 91, 92, 93, 94, 
      95, 96, 97, 98, 99, 100, 101, 102, 103, 104, 105 ]	[  ]
106	[  ]	[ 106 ]	[  ]
107	[  ]	[ 39, 107 ]	[  ]
108	[  ]	[ 106, 108 ]	[  ]
109	[  ]	[ 39, 106, 109 ]	[  ]
110	[  ]	[ 106, 110 ]	[  ]
111	[  ]	[ 106, 111 ]	[  ]
112	[  ]	[ 1, 3, 39, 45, 106, 107, 108, 109, 112 ]	[  ]
113	[  ]	[ 2, 4, 39, 46, 107, 113 ]	[  ]
114	[ 14 ]	[ 39, 51, 106, 107, 109, 114 ]	[  ]
115	[  ]	[ 2, 4, 13, 39, 42, 46, 115 ]	[  ]
116	[  ]	[ 2, 4, 14, 39, 44, 46, 116 ]	[  ]
117	[  ]	[ 2, 4, 39, 46, 106, 109, 117 ]	[  ]
118	[  ]	[ 106, 108, 110, 118 ]	[  ]
119	[  ]	[ 106, 108, 111, 119 ]	[  ]
120	[ 14 ]	[ 1, 2, 3, 4, 25, 26, 35, 36, 39, 45, 46, 51, 55, 106, 107, 108, 
      109, 112, 113, 114, 117, 120 ]	[  ]
121	[ 2 ]	[ 2, 4, 5, 7, 10, 13, 14, 22, 39, 41, 42, 44, 46, 49, 107, 113, 
      115, 121 ]	[  ]
122	[ 1, 3 ]	[ 2, 4, 9, 12, 13, 14, 16, 19, 39, 40, 42, 44, 46, 47, 107, 
      113, 116, 122 ]	[  ]
123	[ 1, 3 ]	[ 2, 4, 9, 12, 13, 16, 19, 39, 40, 42, 46, 47, 106, 109, 115, 
      117, 123 ]	[  ]
124	[ 2 ]	[ 2, 4, 5, 7, 10, 14, 22, 39, 41, 44, 46, 49, 106, 109, 116, 117, 
      124 ]	[  ]
125	[ 14 ]	[ 1, 2, 3, 4, 13, 14, 25, 26, 35, 36, 39, 42, 44, 45, 46, 51, 55, 
      115, 116, 125 ]	[  ]
126	[  ]	[ 2, 4, 39, 46, 69, 106, 109, 110, 117, 126 ]	[  ]
127	[  ]	[ 2, 4, 39, 46, 70, 106, 109, 111, 117, 127 ]	[  ]
128	[ 1, 2, 3, 14, 15, 16, 17 ]	[ 1, 2, 3, 4, 5, 6, 7, 8, 9, 10, 11, 12, 13, 
      14, 15, 16, 17, 18, 19, 20, 21, 22, 23, 24, 25, 26, 27, 28, 29, 30, 31, 
      32, 33, 34, 35, 36, 37, 38, 39, 40, 41, 42, 43, 44, 45, 46, 47, 48, 49, 
      50, 51, 52, 53, 54, 55, 56, 106, 107, 108, 109, 112, 113, 114, 115, 116, 
      117, 120, 121, 122, 123, 124, 125, 128 ]	[  ]
129	[ 1, 3 ]	[ 2, 4, 9, 12, 13, 16, 19, 39, 40, 42, 46, 47, 70, 76, 80, 106, 
      109, 111, 115, 117, 123, 127, 129 ]	[  ]
130	[ 2 ]	[ 2, 4, 5, 7, 10, 14, 22, 39, 41, 44, 46, 49, 69, 74, 75, 79, 106, 
      109, 110, 116, 117, 124, 126, 130 ]	[  ]
131	[ 1, 3 ]	[ 2, 4, 9, 12, 13, 16, 19, 39, 40, 42, 46, 47, 69, 71, 72, 106, 
      109, 110, 115, 117, 123, 126, 131 ]	[  ]
132	[ 2 ]	[ 2, 4, 5, 7, 10, 14, 22, 39, 41, 44, 46, 49, 70, 73, 106, 109, 
      111, 116, 117, 124, 127, 132 ]	[  ]
133	[ 14 ]	[ 1, 2, 3, 4, 25, 26, 35, 36, 39, 45, 46, 51, 55, 69, 77, 78, 81, 
      106, 107, 108, 109, 110, 112, 113, 114, 117, 118, 120, 126, 133 ]	[  ]
134	[ 14 ]	[ 1, 2, 3, 4, 25, 26, 35, 36, 39, 45, 46, 51, 55, 70, 82, 106, 
      107, 108, 109, 111, 112, 113, 114, 117, 119, 120, 127, 134 ]	[  ]
135	[ 1, 2, 3, 14, 15, 16, 17 ]	[ 1, 2, 3, 4, 5, 6, 7, 8, 9, 10, 11, 12, 13, 
      14, 15, 16, 17, 18, 19, 20, 21, 22, 23, 24, 25, 26, 27, 28, 29, 30, 31, 
      32, 33, 34, 35, 36, 37, 38, 39, 40, 41, 42, 43, 44, 45, 46, 47, 48, 49, 
      50, 51, 52, 53, 54, 55, 56, 70, 73, 76, 80, 82, 86, 88, 90, 106, 107, 108, 
      109, 111, 112, 113, 114, 115, 116, 117, 119, 120, 121, 122, 123, 124, 125, 
      127, 128, 129, 132, 134, 135 ]	[  ]
136	[ 1, 2, 3, 14, 15, 16, 17 ]	[ 1, 2, 3, 4, 5, 6, 7, 8, 9, 10, 11, 12, 13, 
      14, 15, 16, 17, 18, 19, 20, 21, 22, 23, 24, 25, 26, 27, 28, 29, 30, 31, 
      32, 33, 34, 35, 36, 37, 38, 39, 40, 41, 42, 43, 44, 45, 46, 47, 48, 49, 
      50, 51, 52, 53, 54, 55, 56, 69, 71, 72, 74, 75, 77, 78, 79, 81, 83, 84, 
      85, 87, 89, 106, 107, 108, 109, 110, 112, 113, 114, 115, 116, 117, 118, 
      120, 121, 122, 123, 124, 125, 126, 128, 130, 131, 133, 136 ]	[  ]
137	[  ]	[ 106, 137 ]	[  ]
138	[  ]	[ 106, 138 ]	[  ]
139	[  ]	[ 106, 108, 137, 139 ]	[ 7 ]
140	[  ]	[ 39, 106, 109, 138, 140 ]	[  ]
141	[  ]	[ 106, 108, 138, 141 ]	[  ]
142	[  ]	[ 57, 58, 106, 110, 111, 137, 138, 142 ]	[  ]
143	[  ]	[ 2, 4, 39, 46, 106, 109, 117, 137, 138, 140, 143 ]	[  ]
144	[  ]	[ 1, 3, 39, 45, 106, 107, 108, 109, 112, 138, 140, 141, 144 ]	[  ]
145	[ 14 ]	[ 39, 51, 106, 107, 109, 114, 138, 140, 145 ]	[ 7 ]
146	[  ]	[ 57, 58, 61, 62, 64, 106, 108, 110, 111, 118, 119, 137, 138, 139, 
     141, 142, 146 ]	[ 7 ]
147	[ 14 ]	[ 1, 2, 3, 4, 25, 26, 35, 36, 39, 45, 46, 51, 55, 106, 107, 108, 
      109, 112, 113, 114, 117, 120, 137, 138, 139, 140, 141, 143, 144, 145, 147 ]
    [ 7 ]
148	[  ]	[ 2, 4, 39, 46, 57, 58, 69, 70, 106, 109, 110, 111, 117, 126, 127, 
     137, 138, 140, 142, 143, 148 ]	[  ]
149	[ 1, 3 ]	[ 2, 4, 9, 12, 13, 16, 19, 39, 40, 42, 46, 47, 57, 58, 59, 60, 
      63, 69, 70, 71, 72, 76, 80, 106, 109, 110, 111, 115, 117, 123, 126, 127, 
      129, 131, 137, 138, 140, 142, 143, 148, 149 ]	[  ]
150	[ 2 ]	[ 2, 4, 5, 7, 10, 14, 22, 39, 41, 44, 46, 49, 57, 58, 59, 60, 63, 
      69, 70, 73, 74, 75, 79, 106, 109, 110, 111, 116, 117, 124, 126, 127, 130, 
      132, 137, 138, 140, 142, 143, 148, 150 ]	[  ]
151	[ 14 ]	[ 1, 2, 3, 4, 25, 26, 35, 36, 39, 45, 46, 51, 55, 57, 58, 61, 62, 
      64, 69, 70, 77, 78, 81, 82, 106, 107, 108, 109, 110, 111, 112, 113, 114, 
      117, 118, 119, 120, 126, 127, 133, 134, 137, 138, 139, 140, 141, 142, 143, 
      144, 145, 146, 147, 148, 151 ]	[ 7 ]
152	[ 1, 2, 3, 14, 15, 16, 17 ]	[ 1, 2, 3, 4, 5, 6, 7, 8, 9, 10, 11, 12, 13, 
      14, 15, 16, 17, 18, 19, 20, 21, 22, 23, 24, 25, 26, 27, 28, 29, 30, 31, 
      32, 33, 34, 35, 36, 37, 38, 39, 40, 41, 42, 43, 44, 45, 46, 47, 48, 49, 
      50, 51, 52, 53, 54, 55, 56, 57, 58, 59, 60, 61, 62, 63, 64, 65, 66, 67, 
      68, 69, 70, 71, 72, 73, 74, 75, 76, 77, 78, 79, 80, 81, 82, 83, 84, 85, 
      86, 87, 88, 89, 90, 106, 107, 108, 109, 110, 111, 112, 113, 114, 115, 116, 
      117, 118, 119, 120, 121, 122, 123, 124, 125, 126, 127, 128, 129, 130, 131, 
      132, 133, 134, 135, 136, 137, 138, 139, 140, 141, 142, 143, 144, 145, 146, 
      147, 148, 149, 150, 151, 152 ]	[ 7 ]
\end{verbatim}
\end{example}

\bigskip
\begin{example}[The groups $I_{4,j}\leq I_{4,i}$ as in Table 13]~\vspace*{-5mm}\\
\begin{verbatim}
gap> Read("BCAlgTori.gap");
gap> for i in [1..7] do
> x:=I4[i];
> G:=MatGroupZClass(x[1],x[2],x[3],x[4]);
> Gcs:=ConjugacyClassesSubgroups2(G);
> Gcsi:=List(Gcs,x->CrystCatZClass(Representative(x)));
> I4j:=Difference(Set(Gcsi,x->Position(I4,x)),[fail]);
> Print(i,"\t",I4j,"\n");
> od;
1	[ 1 ]
2	[ 2 ]
3	[ 1, 2, 3 ]
4	[ 1, 4 ]
5	[ 2, 5 ]
6	[ 1, 2, 3, 4, 5, 6 ]
7	[ 7 ]
\end{verbatim}
\end{example}
%

\section{Stably birational invariants: torus invariants}\label{S4}

Let $L/k$ (resp. $L^\prime/k$) be a finite Galois extension of fields 
with Galois group $G={\rm Gal}(L/k)$ 
(resp. $G^\prime={\rm Gal}(L^\prime/k))$. 
Let $M$ (resp. $M^\prime$) be a $G$-lattice (resp. $G^\prime$-lattice). 
Recall that $[M]^{fl}$ and $[M^\prime]^{fl}$ are 
said to be {\it weak stably $k$-equivalent}, 
denoted by $[M]^{fl}\sim [M^\prime]^{fl}$, if 
there exists a subdirect product $\widetilde{H}\leq G\times G^\prime$ 
of $G$ and $G^\prime$ 
with surjections 
$\varphi_1:\widetilde{H}\rightarrow G$ and 
$\varphi_2:\widetilde{H}\rightarrow G^\prime$ 
such that $[M]^{fl}=[M^\prime]^{fl}$ 
as $\widetilde{H}$-lattices 
where $\widetilde{H}$ acts on $M$ (resp. $M^\prime)$ 
through the surjection $\varphi_1$ (resp. $\varphi_2$) 
and 
two algebraic $k$-tori $T$ and $T^\prime$ are 
said to be {\it weak stably birationally $k$-equivalent}, 
denoted by $T\stackrel{\rm s.b.}{\sim} T^\prime$, 
if $[\widehat{T}]^{fl}\sim [\widehat{T}^\prime]^{fl}$ 
(see Definition \ref{d1.10}). 

In particular, 
if algebraic $k$-tori $T$ and $T^\prime$ are 
stably birationally $k$-equivalent, 
then $[\widehat{T}]^{fl}=[\widehat{T}^\prime]^{fl}$ as 
$H$-lattices for any $H\leq \widetilde{H}$ 
where $\widetilde{H}={\rm Gal}(\widetilde{L}/k)\leq G\times G^\prime$ 
is a subdirect product of 
$G={\rm Gal}(L/k)$ and $G^\prime={\rm Gal}(L^\prime/k)$, 
$\widetilde{L}=LL^\prime$, and $L$ and $L^\prime$ 
are the minimal splitting fields of $T$ and $T^\prime$ respectively.

Define 
\begin{align*}
{\Sha}^i_\omega (G,M):=
{\rm Ker}\left\{H^i(G,M)\xrightarrow{\rm res}\prod_{g\in G}
H^i(\langle g\rangle,M)\right\}\quad (i \geq 1).
\end{align*}
Then we see that 
\begin{align*}
{\Sha}^i_\omega (G,M)=
\bigcap_{H\leq G:\, {\rm maximal\, cyclic}\atop {\rm up\, to\, conjugacy}}{\rm Ker}\left\{H^i(G,M)\xrightarrow{\rm res}H^i(H,M)\right\}\quad (i \geq 1).
\end{align*}

Kunyavskii, Skorobogatov and Tsfasman 
\cite{KST89} found useful invariants of the flabby class $[M]^{fl}$ of $M$:
\begin{theorem}[Kunyavskii, Skorobogatov and Tsfasman 
{\cite[pages 26--27, 36--37]{KST89}, see also Sansuc \cite[Section 9]{San81} and Voskresenskii \cite[Section 12]{Vos98}}]
Let $M$ be a $G$-lattice.\\
{\rm (1)} Six groups 
\begin{align*}
&H^1(G,M), {\Sha}^1_\omega (G,M), {\Sha}^2_\omega (G,M),\\ 
&H^1(G,M^\circ), {\Sha}^1_\omega (G,M^\circ), {\Sha}^2_\omega (G,M^\circ)
\end{align*}
are invariants of the similarity class $[M]$ of $M$ 
where $M^\circ={\rm Hom}_\bZ(M,\bZ)$ is the dual $G$-lattice of $M$.\\
{\rm (2)} Three groups 
\begin{align*}
&\Sha^1_\omega(G,[M]^{fl})\simeq \Sha^2_\omega(G,M)\simeq H^1(G,[M]^{fl}),\\
&\Sha^2_\omega(G,([M]^{fl})^\circ)\simeq \Sha^1_\omega(G,M^\circ)\simeq 
H^1(G,(([M]^{fl})^\circ)^{fl}),\\
&\Sha^2_\omega(G,[M]^{fl})\simeq  \Sha^1_\omega(G,([M]^{fl})^{fl})\simeq 
H^1(G, ([M]^{fl})^{fl})
\end{align*}
are invariants of the flabby class $[M]^{fl}$ of $M$.
\end{theorem}

\begin{remark}\label{r4.3}
(1) 
We have $\Sha^2_\omega(G,\widehat{T})\simeq 
H^1(G,[\widehat{T}]^{fl})\simeq {\rm Br}(X)/{\rm Br}(k)$ 
where $X$ is a smooth $k$-compactification of an algebraic $k$-torus $T$ 
and ${\rm Br}(X)=H^2_{\rm \acute{e}t}(X,\bG_m)$ is the \'etale cohomological Brauer Group (the Brauer-Grothendieck group) of $X$ 
(it is the same as the Azumaya-Brauer group of $X$ for such $X$, 
see Colliot-Th\'{e}l\`{e}ne and Sansuc [CTS87, page 199]). 
We also see  
${\rm Br}_{\rm nr}(k(X)/k)\simeq {\rm Br}(X)\subset {\rm Br}(k(X))$ 
(see \cite[Theorem 5.11]{CTS07}, 
see also Saltman \cite[Proposition 10.5]{Sal99}).\\
(2) If $[M]^{fl}$ is invertible, then the 
three groups vanish: $\Sha^1_\omega(G,[M]^{fl})\simeq H^1(G,[M]^{fl})=0$, 
$\Sha^2_\omega(G,([M_G]^{fl})^\circ)\simeq 
H^1(G,(([M]^{fl})^\circ)^{fl})=0$, 
$\Sha^2_\omega(G,[M]^{fl})\simeq H^1(G,([M]^{fl})^{fl})=0$ 
because $[M]^{fl}$, $(([M]^{fl})^\circ)^{fl}$, $([M]^{fl})^{fl}$ 
are invertible and hence coflabby. 
\end{remark}

\begin{definition}\label{defTI}
Let $G\leq \GL(n,\bZ)$ and 
$M_G$ be the corresponding $G$-lattice of $\bZ$-rank $n$ 
as in Definition \ref{d2.2}. 
The {\it torus invariants} $TI_G=[l_1,l_2,l_3,l_4]$ of $[M_G]^{fl}$ 
are defined to be 
\begin{align*}
l_1=
\begin{cases}
0& \textrm{if}\ \ [M_G]^{fl}=0,\\
1& \textrm{if}\ \ [M_G]^{fl}\neq 0\ \textrm{but\ is\ invertible},\\
2& \textrm{if}\ \ [M_G]^{fl}\ \textrm{is\ not\ invertible}, 
\end{cases}
\end{align*}
$l_2$ (resp. $l_3$, $l_4$) 
is the abelian invariants of 
$\Sha^1_\omega(G,[M_G]^{fl})$ 
(resp. $\Sha^2_\omega(G,([M_G]^{fl})^\circ)$, 
$\Sha^2_\omega(G,[M_G]^{fl}$)). 
\end{definition}

By the definition, 
the torus invariants $TI_G$ of $[M_G]^{fl}$ are actually 
invariants of the flabby class $[M_G]^{fl}$ of $M_G$, 
i.e. $[M_G]^{fl}=[M_{G^\prime}]^{fl}\Rightarrow TI_G=TI_{G^\prime}$. 
We will use $TI_G$ 
in order to determine 
the weak stably birationally $k$-equivalent classes 
of algebraic $k$-tori $T$ up to dimension $4$ in 
Section \ref{S5} (see Theorem \ref{th5.1}).\\

We made the following GAP \cite{GAP} algorithms for computing 
the {\it torus invariants} of $[M_G]^{fl}$. 
They are also available as in \cite{BCAlgTori}.\\

Let $G\leq \GL(n,\bZ)$ 
and $M_G$ be the corresponding $G$-lattice of $\bZ$-rank $n$ 
as in Definition \ref{d2.2}.\\

\noindent 
{\tt H1($G$)} returns $H^1(G,M_G)$.\\

\noindent 
{\tt Sha1Omega($G$)} (resp. {\tt Sha1OmegaTr($G$)}) 
returns $\Sha^1_\omega(G,M_G)$ (resp. $\Sha^1_\omega(G,(M_G)^\circ)$).\\

\noindent 
{\tt ShaOmega($G,n$)} (resp. {\tt ShaOmegaFromGroup($M,n,G$)}) 
returns $\Sha^n_\omega(G,M_G)$ for $G$-lattice $M_G$ (resp. $G$-lattice $M$). 
This function needs HAP package \cite{HAP} in GAP.\\

\noindent 
{\tt TorusInvariants($G$)} 
(resp. {\tt TorusInvariantsHAP($G$)}) 
returns $TI_G=[l_1,l_2,l_3,l_4]$ 
where 
\[
l_1=\begin{cases}
0& \textrm{if}\ \ [M_G]^{fl}=0,\\
1& \textrm{if}\ \ [M_G]^{fl}\neq 0\ \textrm{but\ is\ invertible},\\
2& \textrm{if}\ \ [M_G]^{fl}\ \textrm{is\ not\ invertible},
\end{cases}\]
$l_2=H^1(G,[M_G]^{fl})\simeq \Sha^1_\omega(G,[M_G]^{fl})$, 
$l_3=\Sha^1_\omega(G,(M_G)^\circ)\simeq \Sha^2_\omega(G,([M_G]^{fl})^\circ)$, 
$l_4=H^1(G,([M_G]^{fl})^{fl})\simeq \Sha^2_\omega(G,[M_G]^{fl})$ 
via the command {\tt H1(}$G${\tt )} 
(resp. $l_4=\Sha^2_\omega(G,[M_G]^{fl})$ 
via the command\\
{\tt ShaOmegaFromGroup(}$[M_G]^{fl},2,G${\tt )}$)$. 
The function {\tt TorusInvariantsHAP} needs HAP package \cite{HAP} in GAP.\\

\begin{theorem}\label{th4.4}
For $G=N_{3,i}$ $(1\leq i\leq 15)$ as in Definition \ref{defN3N4}, 
three groups $\Sha^1_\omega(G,[M_G]^{fl})$, 
$\Sha^2_\omega(G,([M_G]^{fl})^\circ)$, 
$\Sha^2_\omega(G,[M_G]^{fl})$ 
are given as in {\rm Table} $6$. 
Equivalently, we have 
\begin{align*}
\Sha^1_\omega(G,[M_G]^{fl})&=
\begin{cases}
\bZ/2\bZ& i=1,8,\\
0& {\rm otherwise},
\end{cases}\\
\Sha^2_\omega(G,([M_G]^{fl})^\circ)&=
\begin{cases}
\bZ/2\bZ& i=1,8,\\
0& {\rm otherwise},
\end{cases}\\
\Sha^2_\omega(G,[M_G]^{fl})&=
\begin{cases}
\bZ/2\bZ& i=4,\\
0& {\rm otherwise}.
\end{cases}
\end{align*}
\end{theorem}
\begin{proof}
We may check the statement 
by using the GAP algorithm 
{\tt TorusInvariants} (or {\tt TorusInvariantsHAP}), 
see Example \ref{exTI-N3}. 
\end{proof}

\begin{theorem}\label{th4.5}
For $G=N_{31,i}$ $(1\leq i\leq 64)$ as in Definition \ref{defN3N4}, 
three groups $\Sha^1_\omega(G,[M_G]^{fl})$, 
$\Sha^2_\omega(G,([M_G]^{fl})^\circ)$, 
$\Sha^2_\omega(G,[M_G]^{fl})$ 
are given as in {\rm Table} $7$. 
Equivalently, we have 
\begin{align*}
\Sha^1_\omega(G,[M_G]^{fl})&=
\begin{cases}
\bZ/2\bZ& i=1,4,11,38,48,\\
0&{\rm otherwise},
\end{cases}\\
\Sha^2_\omega(G,([M_G]^{fl})^\circ)&=
\begin{cases}
\bZ/2\bZ& i=1,4,11,38,48,\\
0&{\rm otherwise},
\end{cases}\\
\Sha^2_\omega(G,[M_G]^{fl})&=
\begin{cases}
\bZ/2\bZ& i=14,18,25,26,29,\\
0&{\rm otherwise}.
\end{cases}
\end{align*}
\end{theorem}
\begin{proof}
The statement follows from Theorem \ref{th4.4} and Table $7$. 
We also compute them by using the GAP algorithm 
{\tt TorusInvariants} (or {\tt TorusInvariantsHAP}) for checking, 
see Example \ref{exTI-N31}. 
\end{proof}

\begin{theorem}\label{th4.6}
For $G=N_{4,i}$ $(1\leq i\leq 152)$ as in Definition \ref{defN3N4}, 
three groups $\Sha^1_\omega(G,[M_G]^{fl})$, 
$\Sha^2_\omega(G,([M_G]^{fl})^\circ)$, 
$\Sha^2_\omega(G,[M_G]^{fl})$ 
are given as in {\rm Table} $8$. 
Equivalently, we have 
\begin{align*}
\Sha^1_\omega(G,[M_G]^{fl})&=
\begin{cases}
\bZ/2\bZ& i=1,3,8,14,18,21,39,45,107,108,109,112,140,141,144,\\
(\bZ/2\bZ)^{\oplus 2}& i=106,138,\\
0&{\rm otherwise},
\end{cases}\\
\Sha^2_\omega(G,([M_G]^{fl})^\circ)&=
\begin{cases}
\bZ/2\bZ&\qquad i=1,9,18,76,\\
0&\qquad {\rm otherwise},
\end{cases}\\
\Sha^2_\omega(G,[M_G]^{fl})&=
\begin{cases}
\bZ/2\bZ& i=4,13,14,23,24,39,42,50,51,52,54,70,107,108,119,\\
(\bZ/2\bZ)^{\oplus 2}& i=27,106,111,\\
\bZ/3\bZ& i=57,60,61,67,92,95,97,104,\\
0&{\rm otherwise}.
\end{cases}
\end{align*}
\end{theorem}
\begin{proof}
We may check the statement 
by using the GAP algorithm 
{\tt TorusInvariants} (or {\tt TorusInvariantsHAP}), 
see Example \ref{exTI-N4}. 
\end{proof}

For $G=I_{4,i}$, we already see 
$\Sha^1_\omega(G,[M_G]^{fl})=0$, 
$\Sha^2_\omega(G,([M_G]^{fl})^\circ)=0$, $\Sha^2_\omega(G,[M_G]^{fl})=0$ 
(see Remark \ref{r4.3} (2)). 
We will compute them by using the GAP algorithm 
{\tt TorusInvariants} (or {\tt TorusInvariantsHAP}) for checking, 
see Example \ref{exTI-I4}.

\bigskip
\begin{example}[{Theorem \ref{th4.4}: three groups $\Sha^1_\omega(G,[M_G]^{fl})$, $\Sha^2_\omega(G,([M_G]^{fl})^\circ)$, $\Sha^2_\omega(G,[M_G]^{fl})$ for $G=N_{3,i}$ as in Table $6$}]\label{exTI-N3}~\vspace*{-5mm}\\
\begin{verbatim}
gap> Read("BCAlgTori.gap");
gap> N3g:=List(N3,x->MatGroupZClass(x[1],x[2],x[3],x[4]));;
gap> TorusInvariantsN3:=List(N3g,TorusInvariants);;
gap> for i in [1..15] do
> Print(i,"\t",TorusInvariantsN3[i],"\n");
> od;
1       [ 2, [ 2 ], [ 2 ], [  ] ]
2       [ 2, [  ], [  ], [  ] ]
3       [ 2, [  ], [  ], [  ] ]
4       [ 2, [  ], [  ], [ 2 ] ]
5       [ 2, [  ], [  ], [  ] ]
6       [ 2, [  ], [  ], [  ] ]
7       [ 2, [  ], [  ], [  ] ]
8       [ 2, [ 2 ], [ 2 ], [  ] ]
9       [ 2, [  ], [  ], [  ] ]
10      [ 2, [  ], [  ], [  ] ]
11      [ 2, [  ], [  ], [  ] ]
12      [ 2, [  ], [  ], [  ] ]
13      [ 2, [  ], [  ], [  ] ]
14      [ 2, [  ], [  ], [  ] ]
15      [ 2, [  ], [  ], [  ] ]

gap> List(N3g,TorusInvariants)=List(N3g,TorusInvariantsHAP);
true
\end{verbatim}
\end{example}

\bigskip
\begin{example}[{Theorem \ref{th4.5}: three groups $\Sha^1_\omega(G,[M_G]^{fl})$, $\Sha^2_\omega(G,([M_G]^{fl})^\circ)$, $\Sha^2_\omega(G,[M_G]^{fl})$ for $G=N_{31,i}$ as in Table $7$}]\label{exTI-N31}~\vspace*{-5mm}\\
\begin{verbatim}
gap> Read("BCAlgTori.gap");
gap> N31g:=List(N31,x->MatGroupZClass(x[1],x[2],x[3],x[4]));;
gap> TorusInvariantsN31:=List(N31g,TorusInvariants);;
gap> Read("BCAlgTori.gap");
gap> for i in [1..64] do
> Print(i,"\t",TorusInvariantsN31[i],"\n");
> od;
1       [ 2, [ 2 ], [ 2 ], [  ] ]
2       [ 2, [  ], [  ], [  ] ]
3       [ 2, [  ], [  ], [  ] ]
4       [ 2, [ 2 ], [ 2 ], [  ] ]
5       [ 2, [  ], [  ], [  ] ]
6       [ 2, [  ], [  ], [  ] ]
7       [ 2, [  ], [  ], [  ] ]
8       [ 2, [  ], [  ], [  ] ]
9       [ 2, [  ], [  ], [  ] ]
10      [ 2, [  ], [  ], [  ] ]
11      [ 2, [ 2 ], [ 2 ], [  ] ]
12      [ 2, [  ], [  ], [  ] ]
13      [ 2, [  ], [  ], [  ] ]
14      [ 2, [  ], [  ], [ 2 ] ]
15      [ 2, [  ], [  ], [  ] ]
16      [ 2, [  ], [  ], [  ] ]
17      [ 2, [  ], [  ], [  ] ]
18      [ 2, [  ], [  ], [ 2 ] ]
19      [ 2, [  ], [  ], [  ] ]
20      [ 2, [  ], [  ], [  ] ]
21      [ 2, [  ], [  ], [  ] ]
22      [ 2, [  ], [  ], [  ] ]
23      [ 2, [  ], [  ], [  ] ]
24      [ 2, [  ], [  ], [  ] ]
25      [ 2, [  ], [  ], [ 2 ] ]
26      [ 2, [  ], [  ], [ 2 ] ]
27      [ 2, [  ], [  ], [  ] ]
28      [ 2, [  ], [  ], [  ] ]
29      [ 2, [  ], [  ], [ 2 ] ]
30      [ 2, [  ], [  ], [  ] ]
31      [ 2, [  ], [  ], [  ] ]
32      [ 2, [  ], [  ], [  ] ]
33      [ 2, [  ], [  ], [  ] ]
34      [ 2, [  ], [  ], [  ] ]
35      [ 2, [  ], [  ], [  ] ]
36      [ 2, [  ], [  ], [  ] ]
37      [ 2, [  ], [  ], [  ] ]
38      [ 2, [ 2 ], [ 2 ], [  ] ]
39      [ 2, [  ], [  ], [  ] ]
40      [ 2, [  ], [  ], [  ] ]
41      [ 2, [  ], [  ], [  ] ]
42      [ 2, [  ], [  ], [  ] ]
43      [ 2, [  ], [  ], [  ] ]
44      [ 2, [  ], [  ], [  ] ]
45      [ 2, [  ], [  ], [  ] ]
46      [ 2, [  ], [  ], [  ] ]
47      [ 2, [  ], [  ], [  ] ]
48      [ 2, [ 2 ], [ 2 ], [  ] ]
49      [ 2, [  ], [  ], [  ] ]
50      [ 2, [  ], [  ], [  ] ]
51      [ 2, [  ], [  ], [  ] ]
52      [ 2, [  ], [  ], [  ] ]
53      [ 2, [  ], [  ], [  ] ]
54      [ 2, [  ], [  ], [  ] ]
55      [ 2, [  ], [  ], [  ] ]
56      [ 2, [  ], [  ], [  ] ]
57      [ 2, [  ], [  ], [  ] ]
58      [ 2, [  ], [  ], [  ] ]
59      [ 2, [  ], [  ], [  ] ]
60      [ 2, [  ], [  ], [  ] ]
61      [ 2, [  ], [  ], [  ] ]
62      [ 2, [  ], [  ], [  ] ]
63      [ 2, [  ], [  ], [  ] ]
64      [ 2, [  ], [  ], [  ] ]

gap> List(N31g,TorusInvariants)=List(N31g,TorusInvariantsHAP);
true
\end{verbatim}
\end{example}

\bigskip
\begin{example}[{Theorem \ref{th4.6}: three groups $\Sha^1_\omega(G,[M_G]^{fl})$, $\Sha^2_\omega(G,([M_G]^{fl})^\circ)$, $\Sha^2_\omega(G,[M_G]^{fl})$ for $G=N_{4,i}$ as in Table $8$}]\label{exTI-N4}~\vspace*{-5mm}\\
\begin{verbatim}
gap> Read("BCAlgTori.gap");
gap> N4g:=List(N4,x->MatGroupZClass(x[1],x[2],x[3],x[4]));;
gap> TorusInvariantsN4:=List(N4g,TorusInvariants);;
gap> for i in [1..152] do
> Print(i,"\t",TorusInvariantsN4[i],"\n");
> od;
1       [ 2, [ 2 ], [ 2 ], [  ] ]
2       [ 2, [  ], [  ], [  ] ]
3       [ 2, [ 2 ], [  ], [  ] ]
4       [ 2, [  ], [  ], [ 2 ] ]
5       [ 2, [  ], [  ], [  ] ]
6       [ 2, [  ], [  ], [  ] ]
7       [ 2, [  ], [  ], [  ] ]
8       [ 2, [ 2 ], [  ], [  ] ]
9       [ 2, [  ], [ 2 ], [  ] ]
10      [ 2, [  ], [  ], [  ] ]
11      [ 2, [  ], [  ], [  ] ]
12      [ 2, [  ], [  ], [  ] ]
13      [ 2, [  ], [  ], [ 2 ] ]
14      [ 2, [ 2 ], [  ], [ 2 ] ]
15      [ 2, [  ], [  ], [  ] ]
16      [ 2, [  ], [  ], [  ] ]
17      [ 2, [  ], [  ], [  ] ]
18      [ 2, [ 2 ], [ 2 ], [  ] ]
19      [ 2, [  ], [  ], [  ] ]
20      [ 2, [  ], [  ], [  ] ]
21      [ 2, [ 2 ], [  ], [  ] ]
22      [ 2, [  ], [  ], [  ] ]
23      [ 2, [  ], [  ], [ 2 ] ]
24      [ 2, [  ], [  ], [ 2 ] ]
25      [ 2, [  ], [  ], [  ] ]
26      [ 2, [  ], [  ], [  ] ]
27      [ 2, [  ], [  ], [ 2, 2 ] ]
28      [ 2, [  ], [  ], [  ] ]
29      [ 2, [  ], [  ], [  ] ]
30      [ 2, [  ], [  ], [  ] ]
31      [ 2, [  ], [  ], [  ] ]
32      [ 2, [  ], [  ], [  ] ]
33      [ 2, [  ], [  ], [  ] ]
34      [ 2, [  ], [  ], [  ] ]
35      [ 2, [  ], [  ], [  ] ]
36      [ 2, [  ], [  ], [  ] ]
37      [ 2, [  ], [  ], [  ] ]
38      [ 2, [  ], [  ], [  ] ]
39      [ 2, [ 2 ], [  ], [ 2 ] ]
40      [ 2, [  ], [  ], [  ] ]
41      [ 2, [  ], [  ], [  ] ]
42      [ 2, [  ], [  ], [ 2 ] ]
43      [ 2, [  ], [  ], [  ] ]
44      [ 2, [  ], [  ], [  ] ]
45      [ 2, [ 2 ], [  ], [  ] ]
46      [ 2, [  ], [  ], [  ] ]
47      [ 2, [  ], [  ], [  ] ]
48      [ 2, [  ], [  ], [  ] ]
49      [ 2, [  ], [  ], [  ] ]
50      [ 2, [  ], [  ], [ 2 ] ]
51      [ 2, [  ], [  ], [ 2 ] ]
52      [ 2, [  ], [  ], [ 2 ] ]
53      [ 2, [  ], [  ], [  ] ]
54      [ 2, [  ], [  ], [ 2 ] ]
55      [ 2, [  ], [  ], [  ] ]
56      [ 2, [  ], [  ], [  ] ]
57      [ 2, [  ], [  ], [ 3 ] ]
58      [ 2, [  ], [  ], [  ] ]
59      [ 2, [  ], [  ], [  ] ]
60      [ 2, [  ], [  ], [ 3 ] ]
61      [ 2, [  ], [  ], [ 3 ] ]
62      [ 2, [  ], [  ], [  ] ]
63      [ 2, [  ], [  ], [  ] ]
64      [ 2, [  ], [  ], [  ] ]
65      [ 2, [  ], [  ], [  ] ]
66      [ 2, [  ], [  ], [  ] ]
67      [ 2, [  ], [  ], [ 3 ] ]
68      [ 2, [  ], [  ], [  ] ]
69      [ 2, [  ], [  ], [  ] ]
70      [ 2, [  ], [  ], [ 2 ] ]
71      [ 2, [  ], [  ], [  ] ]
72      [ 2, [  ], [  ], [  ] ]
73      [ 2, [  ], [  ], [  ] ]
74      [ 2, [  ], [  ], [  ] ]
75      [ 2, [  ], [  ], [  ] ]
76      [ 2, [  ], [ 2 ], [  ] ]
77      [ 2, [  ], [  ], [  ] ]
78      [ 2, [  ], [  ], [  ] ]
79      [ 2, [  ], [  ], [  ] ]
80      [ 2, [  ], [  ], [  ] ]
81      [ 2, [  ], [  ], [  ] ]
82      [ 2, [  ], [  ], [  ] ]
83      [ 2, [  ], [  ], [  ] ]
84      [ 2, [  ], [  ], [  ] ]
85      [ 2, [  ], [  ], [  ] ]
86      [ 2, [  ], [  ], [  ] ]
87      [ 2, [  ], [  ], [  ] ]
88      [ 2, [  ], [  ], [  ] ]
89      [ 2, [  ], [  ], [  ] ]
90      [ 2, [  ], [  ], [  ] ]
91      [ 2, [  ], [  ], [  ] ]
92      [ 2, [  ], [  ], [ 3 ] ]
93      [ 2, [  ], [  ], [  ] ]
94      [ 2, [  ], [  ], [  ] ]
95      [ 2, [  ], [  ], [ 3 ] ]
96      [ 2, [  ], [  ], [  ] ]
97      [ 2, [  ], [  ], [ 3 ] ]
98      [ 2, [  ], [  ], [  ] ]
99      [ 2, [  ], [  ], [  ] ]
100     [ 2, [  ], [  ], [  ] ]
101     [ 2, [  ], [  ], [  ] ]
102     [ 2, [  ], [  ], [  ] ]
103     [ 2, [  ], [  ], [  ] ]
104     [ 2, [  ], [  ], [ 3 ] ]
105     [ 2, [  ], [  ], [  ] ]
106     [ 2, [ 2, 2 ], [  ], [ 2, 2 ] ]
107     [ 2, [ 2 ], [  ], [ 2 ] ]
108     [ 2, [ 2 ], [  ], [ 2 ] ]
109     [ 2, [ 2 ], [  ], [  ] ]
110     [ 2, [  ], [  ], [  ] ]
111     [ 2, [  ], [  ], [ 2, 2 ] ]
112     [ 2, [ 2 ], [  ], [  ] ]
113     [ 2, [  ], [  ], [  ] ]
114     [ 2, [  ], [  ], [  ] ]
115     [ 2, [  ], [  ], [  ] ]
116     [ 2, [  ], [  ], [  ] ]
117     [ 2, [  ], [  ], [  ] ]
118     [ 2, [  ], [  ], [  ] ]
119     [ 2, [  ], [  ], [ 2 ] ]
120     [ 2, [  ], [  ], [  ] ]
121     [ 2, [  ], [  ], [  ] ]
122     [ 2, [  ], [  ], [  ] ]
123     [ 2, [  ], [  ], [  ] ]
124     [ 2, [  ], [  ], [  ] ]
125     [ 2, [  ], [  ], [  ] ]
126     [ 2, [  ], [  ], [  ] ]
127     [ 2, [  ], [  ], [  ] ]
128     [ 2, [  ], [  ], [  ] ]
129     [ 2, [  ], [  ], [  ] ]
130     [ 2, [  ], [  ], [  ] ]
131     [ 2, [  ], [  ], [  ] ]
132     [ 2, [  ], [  ], [  ] ]
133     [ 2, [  ], [  ], [  ] ]
134     [ 2, [  ], [  ], [  ] ]
135     [ 2, [  ], [  ], [  ] ]
136     [ 2, [  ], [  ], [  ] ]
137     [ 2, [  ], [  ], [  ] ]
138     [ 2, [ 2, 2 ], [  ], [  ] ]
139     [ 2, [  ], [  ], [  ] ]
140     [ 2, [ 2 ], [  ], [  ] ]
141     [ 2, [ 2 ], [  ], [  ] ]
142     [ 2, [  ], [  ], [  ] ]
143     [ 2, [  ], [  ], [  ] ]
144     [ 2, [ 2 ], [  ], [  ] ]
145     [ 2, [  ], [  ], [  ] ]
146     [ 2, [  ], [  ], [  ] ]
147     [ 2, [  ], [  ], [  ] ]
148     [ 2, [  ], [  ], [  ] ]
149     [ 2, [  ], [  ], [  ] ]
150     [ 2, [  ], [  ], [  ] ]
151     [ 2, [  ], [  ], [  ] ]
152     [ 2, [  ], [  ], [  ] ]

gap> List(N4g,TorusInvariants)=List(N4g,TorusInvariantsHAP);
true
\end{verbatim}
\end{example}

\bigskip
\begin{example}[{Three groups $\Sha^1_\omega(G,[M_G]^{fl})$, $\Sha^2_\omega(G,([M_G]^{fl})^\circ)$, $\Sha^2_\omega(G,[M_G]^{fl})$ vanish for $G=I_{4,i}$ as in Table $9$}]\label{exTI-I4}~\vspace*{-5mm}\\
\begin{verbatim}
gap> Read("BCAlgTori.gap");
gap> I4g:=List(I4,x->MatGroupZClass(x[1],x[2],x[3],x[4]));;
gap> TorusInvariantsI4:=List(I4g,TorusInvariants);;
gap> for i in [1..7] do
> Print(i,"\t",TorusInvariantsI4[i],"\n");
> od;
1       [ 1, [  ], [  ], [  ] ]
2       [ 1, [  ], [  ], [  ] ]
3       [ 1, [  ], [  ], [  ] ]
4       [ 1, [  ], [  ], [  ] ]
5       [ 1, [  ], [  ], [  ] ]
6       [ 1, [  ], [  ], [  ] ]
7       [ 1, [  ], [  ], [  ] ]

gap> List(I4g,TorusInvariants)=List(I4g,TorusInvariantsHAP);
true
\end{verbatim}
\end{example}


\section{Weak stably $k$-equivalent classes via torus invariants}\label{S5}

From the arguments which are given in Section \ref{S4}, 
we obtain the following theorem which is crucial for 
the classification of 
the weak stably birationally $k$-equivalent classes 
of algebraic $k$-tori $T$. 
For the torus invariants $TI_G=[l_1,l_2,l_3,l_4]$ of $[M_G]^{fl}$, 
see Definition \ref{defTI}. 


\begin{theorem}\label{th5.1}
Let $G$ $($resp. $G^\prime$$)$ be a finite subgroup of 
$\GL(n_1,\bZ)$ $($resp. $\GL(n_2,\bZ)$$)$ and 
$M_{G}$ $($resp. $M_{G^\prime}$$)$ be a $G$-lattice 
$($resp. $G^\prime$-lattice$)$ as in Definition \ref{d2.2}. 
If $[M_G]^{fl}\sim [M_{G^\prime}]^{fl}$ as $\widetilde{H}$-lattices, 
then $TI_{\varphi_1(H)}=TI_{\varphi_2(H)}$ 
for any $H\leq \widetilde{H}$ 
where $\widetilde{H}\leq G\times G^\prime$ 
is a subdirect product of $G$ and $G^\prime$ 
which acts on $M_G$ 
$($resp. $M_{G^\prime}$$)$ 
through the surjection $\varphi_1: \widetilde{H} \rightarrow G$ 
$($resp. $\varphi_2: \widetilde{H} \rightarrow G^\prime$$)$. 
\end{theorem}

In order to use Theorem \ref{th5.1}, 
we made the following GAP \cite{GAP} algorithms 
for computing the subdirect products 
$\widetilde{H}\leq G\times G^\prime$ of $G$ and $G^\prime$ 
in $\GL(n_1,\bZ)\times \GL(n_2,\bZ)$ 
which satisfy $TI_{\varphi_1(H)}=TI_{\varphi_2(H)}$ 
for any $H\leq \widetilde{H}$. 
It is available as in \cite{BCAlgTori}.\\

\noindent 
{\tt ConjugacyClassesSubgroups2TorusInvariants($G$)} returns 
the records 
{\tt ConjugacyClassesSubgroups2} and\\ 
{\tt TorusInvariants} 
where 
{\tt ConjugacyClassesSubgroups2} 
is the list $[g_1,\ldots,g_m]$ 
of conjugacy classes of subgroups of $G\leq \GL(n,\bZ)$ 
with the fixed ordering via the function 
{\tt ConjugacyClassesSubgroups2(}$G${\tt )} 
(\cite[Section 4.1]{HY17}) and 
{\tt TorusInvariants} is the list 
$[{\tt TorusInvariants}(g_1),\ldots,{\tt TorusInvariants}(g_m)]$ 
via the function {\tt TorusInvariants($G$)} as in Section \ref{S4}.\\ 

\noindent 
{\tt PossibilityOfStablyEquivalentSubdirectProducts($G$,$G'$,}\\
\qquad {\tt ConjugacyClassesSubgroups2TorusInvariants($G$),}\\
\qquad {\tt ConjugacyClassesSubgroups2TorusInvariants($G'$))}\\
returns the list $l$ of the subdirect products 
$\widetilde{H}\leq G\times G^\prime$ of $G$ and $G^\prime$ 
up to $(\GL(n_1,\bZ)\times \GL(n_2,\bZ))$-conjugacy  
which satisfy $TI_{\varphi_1(H)}=TI_{\varphi_2(H)}$ 
for any $H\leq \widetilde{H}$ 
where $\widetilde{H}\leq G\times G^\prime$ 
is a subdirect product of $G$ and $G^\prime$ 
which acts on $M_G$ and $M_{G^\prime}$ 
through the surjections $\varphi_1: \widetilde{H} \rightarrow G$ 
and $\varphi_2: \widetilde{H} \rightarrow G^\prime$ respectively 
(indeed, this function computes it for $H$ up to conjugacy 
for the sake of saving time).

In particular, if the length of the list $l$ is zero, then 
we find that $[M_G]^{fl}$ and $[M_{G^\prime}]^{fl}$ 
are not weak stably $k$-equivalent.\\

Using the function {\tt PossibilityOfStablyEquivalentSubdirectProducts} 
above, we can obtain the following propositions which 
give a part of classification of the weak stably birationally 
$k$-equivalent classes of algebraic $k$-tori $T$ up to dimension $4$ 
(see Examples \ref{ex56}, \ref{ex57}, \ref{ex58}, \ref{ex59} 
with the size of the list $l$ 
(= output of the function above) respectively). 

\begin{proposition}\label{prop5.2}
Let $G=N_{3,i}$ and $G^\prime=N_{3,j}$ $(1\leq i<j\leq 15)$ 
as in Definition \ref{defN3N4} 
and $M_G$ and $M_{G^\prime}$ 
be the corresponding $G$-lattices as in Definition \ref{d2.2}. 
Then $[M_G]^{fl}$ and $[M_{G^\prime}]^{fl}$ 
are not weak stably $k$-equivalent 
except possibly for the cases $(i,j)\in J_{3,3}$ where 
$J_{3,3}=\{(5,6)$, $(11,13)\}$ with $G\simeq G^\prime$. 
In particular, if $G\not\simeq G^\prime$, 
then $[M_G]^{fl}\not\sim [M_{G^\prime}]^{fl}$. 
\end{proposition}
\begin{proof}
See Example \ref{ex56} for GAP computations. 
\end{proof}

\begin{proposition}\label{prop5.3}
Let $G=N_{3,i}$ and $G^\prime=N_{4,j}$ $(1\leq i\leq 15, 1\leq j\leq 152)$ 
as in Definition \ref{defN3N4} 
and $M_G$ and $M_{G^\prime}$ 
be the corresponding $G$-lattices as in Definition \ref{d2.2}. 
Then $[M_G]^{fl}$ and $[M_{G^\prime}]^{fl}$ 
are not weak stably $k$-equivalent 
except possibly for the cases $(i,j)\in J_{3,4}$ where 
$J_{3,4}=\{(1,1)$, $(2,6)$, $(3,2)$, $(3,5)$, $(3,7)$, $(4,13)$,
$(4,23)$, $(4,24)$, $(5,15)$, $(5,16)$, $(5,17)$, $(5,25)$, $(5,26)$,
$(6,15)$, $(6,16)$, $(6,17)$, $(6,25)$, $(6,26)$, $(7,19)$, $(7,29)$,
$(7,30)$, $(10,69)$, $(10,74)$, $(10,75)$, $(12,71)$, $(12,77)$,
$(12,78)$, $(15,72)$, $(15,83)$, $(15,85)\}$ with $G\simeq G^\prime$. 
In particular, if $G\not\simeq G^\prime$, 
then $[M_G]^{fl}\not\sim [M_{G^\prime}]^{fl}$. 
\end{proposition}
\begin{proof}
See Example \ref{ex57} for GAP computations. 
\end{proof}

\begin{proposition}\label{prop5.4}
Let $G=N_{4,i}$ and $G^\prime=N_{4,j}$ $(1\leq i<j\leq 152)$ 
as in Definition \ref{defN3N4} 
and $M_G$ and $M_{G^\prime}$ 
be the corresponding $G$-lattices as in Definition \ref{d2.2}.  
Then $[M_G]^{fl}$ and $[M_{G^\prime}]^{fl}$ 
are not weak stably $k$-equivalent 
except possibly for the cases $(i,j)\in J_{4,4}$ where 
$J_{4,4}=\{(2,5)$, $(2,7)$, $(3,8)$, $(5,7)$, $(13,23)$, 
$(13,24)$, $(14,39)$, $(15,16)$, $(15,17)$, 
$(15,25)$, $(15,26)$, $(16,17)$, $(16,25)$, 
$(16,26)$, $(17,25)$, $(17,26)$, $(19,29)$, 
$(19,30)$, $(20,31)$, $(20,33)$, $(20,35)$, 
$(22,46)$, $(23,24)$, $(25,26)$, $(29,30)$, 
$(31,33)$, $(31,35)$, $(32,34)$, $(33,35)$, 
$(41,44)$, $(59,91)$, $(60,92)$, $(63,93)$, 
$(65,94)$, $(66,96)$, $(67,95)$, $(68,99)$, 
$(69,74)$, $(69,75)$, $(71,77)$, $(71,78)$, 
$(72,83)$, $(72,85)$, $(74,75)$, $(77,78)$, 
$(81,84)$, $(81,87)$, $(83,85)$, 
$(84,87)$, $(86,88)$, $(102,103)\}$ 
with $G\simeq G^\prime$. 
In particular, if $G\not\simeq G^\prime$, 
then $[M_G]^{fl}\not\sim [M_{G^\prime}]^{fl}$. 
\end{proposition}
\begin{proof}
See Example \ref{ex58} for GAP computations. 
\end{proof}

\begin{proposition}\label{prop5.5}
Let $G=I_{4,i}$ and $G^\prime=I_{4,j}$ $(1\leq i<j\leq 7)$ 
as in Definition \ref{defN3N4}
and $M_G$ and $M_{G^\prime}$ 
be the corresponding $G$-lattices as in Definition \ref{d2.2}. 
Then $[M_G]^{fl}$ and $[M_{G^\prime}]^{fl}$ 
are not weak stably $k$-equivalent 
except possibly for the cases 
$(i,j)=(1,2)$, $(1,3)$, $(2,3)$, $(4,5)$, $(4,6)$, $(5,6)$. 
\end{proposition}
\begin{proof}
See Example \ref{ex59} for GAP computations. 
\end{proof}

\bigskip
\begin{example}[Proof of Proposition \ref{prop5.2} for $G=N_{3,i}$ and $G^\prime=N_{3,j}$]\label{ex56}~\vspace*{-5mm}\\
\begin{verbatim}
gap> Read("BCAlgTori.gap");
gap> N3g:=List(N3,x->MatGroupZClass(x[1],x[2],x[3],x[4]));;
gap> N3inv:=List(N3g,x->ConjugacyClassesSubgroups2TorusInvariants(x));;
gap> for i in Combinations([1..15],2) do
> pos:=PossibilityOfStablyEquivalentSubdirectProducts(
> N3g[i[1]],N3g[i[2]],N3inv[i[1]],N3inv[i[2]]);
> if pos<>[] then
> Print(i,"\t",Length(pos),"\n"); # [i,j] and the number of pos 
> fi;
> od;
[ 5, 6 ]	1
[ 11, 13 ]	1
\end{verbatim}
\end{example}

\bigskip
\begin{example}[Proof of Proposition \ref{prop5.3} for $G=N_{3,i}$ and $G^\prime=N_{4,j}$]\label{ex57}~\vspace*{-5mm}\\
\begin{verbatim}
gap> Read("BCAlgTori.gap");
gap> N3g:=List(N3,x->MatGroupZClass(x[1],x[2],x[3],x[4]));;
gap> N4g:=List(N4,x->MatGroupZClass(x[1],x[2],x[3],x[4]));;
gap> N3inv:=List(N3g,x->ConjugacyClassesSubgroups2TorusInvariants(x));;
gap> N4inv:=List(N4g,x->ConjugacyClassesSubgroups2TorusInvariants(x));;
gap> for i in [1..15] do
> for j in [1..152] do
> pos:=PossibilityOfStablyEquivalentSubdirectProducts(N3g[i],N4g[j],N3inv[i],N4inv[j]);
> if pos<>[] then
> Print([i,j],"\t",Length(pos),"\n"); # [i,j] and the number of pos 
> fi;
> od;od;
[ 1, 1 ]	1
[ 2, 6 ]	3
[ 3, 2 ]	3
[ 3, 5 ]	7
[ 3, 7 ]	7
[ 4, 13 ]	5
[ 4, 23 ]	10
[ 4, 24 ]	10
[ 5, 15 ]	1
[ 5, 16 ]	1
[ 5, 17 ]	1
[ 5, 25 ]	1
[ 5, 26 ]	1
[ 6, 15 ]	1
[ 6, 16 ]	1
[ 6, 17 ]	1
[ 6, 25 ]	1
[ 6, 26 ]	1
[ 7, 19 ]	2
[ 7, 29 ]	4
[ 7, 30 ]	4
[ 10, 69 ]	1
[ 10, 74 ]	1
[ 10, 75 ]	1
[ 12, 71 ]	1
[ 12, 77 ]	1
[ 12, 78 ]	1
[ 15, 72 ]	1
[ 15, 83 ]	1
[ 15, 85 ]	1
\end{verbatim}
\end{example}

\bigskip
\begin{example}[Proof of Proposition \ref{prop5.4} for $G=N_{4,i}$ and $G^\prime=N_{4,j}$]\label{ex58}~\vspace*{-5mm}\\
\begin{verbatim}
gap> Read("BCAlgTori.gap");
gap> N4g:=List(N4,x->MatGroupZClass(x[1],x[2],x[3],x[4]));;
gap> N4inv:=List(N4g,x->ConjugacyClassesSubgroups2TorusInvariants(x));;
gap> for i in Combinations([1..152],2) do
> pos:=PossibilityOfStablyEquivalentSubdirectProducts(
> N4g[i[1]],N4g[i[2]],N4inv[i[1]],N4inv[i[2]]);
> if pos<>[] then
> Print(i,"\t",Length(pos),"\n"); # [i,j] and the number of pos 
> fi;
> od;
[ 2, 5 ]	3
[ 2, 7 ]	3
[ 3, 8 ]	2
[ 5, 7 ]	7
[ 13, 23 ]	5
[ 13, 24 ]	5
[ 14, 39 ]	3
[ 15, 16 ]	1
[ 15, 17 ]	1
[ 15, 25 ]	1
[ 15, 26 ]	1
[ 16, 17 ]	1
[ 16, 25 ]	1
[ 16, 26 ]	1
[ 17, 25 ]	1
[ 17, 26 ]	1
[ 19, 29 ]	2
[ 19, 30 ]	2
[ 20, 31 ]	4
[ 20, 33 ]	4
[ 20, 35 ]	2
[ 22, 46 ]	1
[ 23, 24 ]	10
[ 25, 26 ]	1
[ 29, 30 ]	4
[ 31, 33 ]	8
[ 31, 35 ]	4
[ 32, 34 ]	1
[ 33, 35 ]	4
[ 41, 44 ]	1
[ 59, 91 ]	1
[ 60, 92 ]	1
[ 63, 93 ]	1
[ 65, 94 ]	1
[ 66, 96 ]	1
[ 67, 95 ]	1
[ 68, 99 ]	1
[ 69, 74 ]	1
[ 69, 75 ]	1
[ 71, 77 ]	1
[ 71, 78 ]	1
[ 72, 83 ]	1
[ 72, 85 ]	1
[ 74, 75 ]	1
[ 77, 78 ]	1
[ 81, 84 ]	2
[ 81, 87 ]	2
[ 83, 85 ]	1
[ 84, 87 ]	2
[ 86, 88 ]	1
[ 102, 103 ]	1
\end{verbatim}
\end{example}

\bigskip
\begin{example}[Proof of Proposition \ref{prop5.5} for $G=I_{4,i}$ and $G^\prime=I_{4,j}$]\label{ex59}~\vspace*{-5mm}\\
\begin{verbatim}
gap> Read("BCAlgTori.gap");
gap> I4g:=List(I4,x->MatGroupZClass(x[1],x[2],x[3],x[4]));;
gap> I4inv:=List(I4g,x->ConjugacyClassesSubgroups2TorusInvariants(x));;
gap> for i in Combinations([1..7],2) do
> pos:=PossibilityOfStablyEquivalentSubdirectProducts(
> I4g[i[1]],I4g[i[2]],I4inv[i[1]],I4inv[i[2]]);
> if pos<>[] then
> Print(i,"\t",Length(pos),"\n"); # [i,j] and the number of pos 
> fi;
> od;
[ 1, 2 ]	1
[ 1, 3 ]	1
[ 2, 3 ]	1
[ 4, 5 ]	1
[ 4, 6 ]	1
[ 5, 6 ]	1
\end{verbatim}
\end{example}


\section{Proofs of {Theorem \ref{thmain1}} $(N_{3,i})$ and {Theorem \ref{thmain3}} $(N_{31,i}, N_{4,i})$: weak stably $k$-equivalent classes}\label{S6} 

Let $G$ $($resp. $G^\prime$$)$ be a finite subgroup of 
$\GL(n_1,\bZ)$ $($resp. $\GL(n_2,\bZ)$$)$ 
with $n_1=3,4$ (resp. $n_2=3,4$) 
and $M_G$ $($resp. $M_{G^\prime}$$)$ be a $G$-lattice 
$($resp. $G^\prime$-lattice$)$ of $\bZ$-rank $n_1$ $($resp. $n_2)$ 
as in Definition \ref{d2.2}. 
Let $\widetilde{H}\leq G\times G^\prime$ be a subdirect product 
of $G$ and $G^\prime$ 
which acts on $M_G$ 
$($resp. $M_{G^\prime}$$)$ 
through the surjection $\varphi_1: \widetilde{H} \rightarrow G$ 
$($resp. $\varphi_2: \widetilde{H} \rightarrow G^\prime$$)$. 
Then $L(M_G\oplus M_{G^\prime})^{\widetilde{H}}$ 
(resp. $L(M_G)^G$, $L(M_{G^\prime})^{G^\prime}$) 
is the function field of an algebraic torus $T^{\prime\prime}$ 
(resp. $T$, $T^\prime$) of dimension $n_1+n_2$ (resp. $n_1$, $n_2$) 
where $L$ is the splitting field of $T^{\prime\prime}$. 

In order to prove Theorem \ref{thmain1} and Theorem \ref{thmain3}, 
we will prove the following strong versions with $\widetilde{H}\simeq G$: 

\begin{theorem}\label{th6.1}
Let $G=N_{3,i}$ and $G^\prime=N_{3,j}$ $(1\leq i<j\leq 15)$ 
as in Definition \ref{defN3N4} 
and $M_G$ and $M_{G^\prime}$ 
be the corresponding $G$-lattices as in Definition \ref{d2.2}. 
For the exceptional cases $(i,j)\in J_{3,3}=\{(5,6)$, $(11,13)\}$ 
with $G\simeq G^\prime$ 
as in Proposition \ref{prop5.2}, 
there exists a subdirect product 
$\widetilde{H}\simeq G\leq G\times G^\prime$ of $G$ and $G^\prime$ 
such that $[M_G]^{fl}=[M_{G^\prime}]^{fl}$ as $\widetilde{H}$-lattices. 
In particular, $[M_G]^{fl}\sim [M_{G^\prime}]^{fl}$.
\end{theorem}

\begin{theorem}\label{th6.2}
Let $G=N_{3,i}$ and $G^\prime=N_{4,j}$ $(1\leq i\leq 15, 1\leq j\leq 152)$ 
as in Definition \ref{defN3N4} 
and $M_G$ and $M_{G^\prime}$ 
be the corresponding $G$-lattices as in Definition \ref{d2.2}. 
For the exceptional cases $(i,j)\in J_{3,4}$ with $G\simeq G^\prime$ 
as in Proposition \ref{prop5.3}, 
there exists a subdirect product 
$\widetilde{H}\simeq G\leq G\times G^\prime$ of $G$ and $G^\prime$ 
such that $[M_G]^{fl}=[M_{G^\prime}]^{fl}$ as $\widetilde{H}$-lattices. 
In particular, $[M_G]^{fl}\sim [M_{G^\prime}]^{fl}$.
\end{theorem}

\begin{theorem}\label{th6.3}
Let $G=N_{4,i}$ and $G^\prime=N_{4,j}$ $(1\leq i<j\leq 152)$ 
as in Definition \ref{defN3N4} 
and $M_G$ and $M_{G^\prime}$ 
be the corresponding $G$-lattices as in Definition \ref{d2.2}. 
For the exceptional cases $(i,j)\in J_{4,4}$ with $G\simeq G^\prime$ 
as in Proposition \ref{prop5.4}, 
there exists a subdirect product 
$\widetilde{H}\simeq G\leq G\times G^\prime$ of $G$ and $G^\prime$ 
such that $[M_G]^{fl}=[M_{G^\prime}]^{fl}$ as $\widetilde{H}$-lattices. 
In particular, $[M_G]^{fl}\sim [M_{G^\prime}]^{fl}$.
\end{theorem}

\begin{theorem}\label{th6.4}
Let $G=I_{4,i}$ and $G^\prime=I_{4,j}$ $(1\leq i<j\leq 7)$ 
as in Definition \ref{defN3N4} 
and $M_G$ and $M_{G^\prime}$ 
be the corresponding $G$-lattices as in Definition \ref{d2.2}. 
For the exceptional cases 
$(i,j)=(1,2)$, $(1,3)$, $(2,3)$, $(4,5)$, $(4,6)$, $(5,6)$ 
as in Proposition \ref{prop5.5}, 
we have $[M_G]^{fl}\not\sim [M_{G^\prime}]^{fl}$. 
In particular, if $G\not\simeq G^\prime$, 
then $[M_G]^{fl}\not\sim [M_{G^\prime}]^{fl}$. 
\end{theorem}
%

For proofs of Theorem \ref{th6.1}, Theorem \ref{th6.2} and 
Theorem \ref{th6.3}, 
we will take an $\widetilde{H}$-lattice $F$ (resp. $F^\prime$)
with $[F]=[M_G]^{fl}$ (resp. $[F^\prime]=[M_{G^\prime}]^{fl}$) 
explicitly via the command \\
{\tt FlabbyResolutionLowRank($G$).actionF} 
(resp. {\tt FlabbyResolutionLowRank($G^\prime$).actionF}) (see below). 

We will give a necessary condition for $[F]=[F^\prime]$ 
as $\widetilde{H}$-lattices. 
This enables us to show that 
$[F]\neq [F^\prime]$ for some $\widetilde{H}$-lattices. 

Each isomorphism class of irreducible permutation $\widetilde{H}$-lattices 
corresponds to a conjugacy class of subgroup $H$ of $\widetilde{H}$ by
$H \leftrightarrow \bZ[\widetilde{H}/H]$. 
Let $H_1=\{1\},\ldots,H_r=\widetilde{H}$ 
be all conjugacy classes of subgroups of $\widetilde{H}$ 
whose ordering corresponds to the GAP function 
{\tt ConjugacyClassesSubgroups2($\widetilde{H}$)} 
(see \cite[Section 4.1, page 42]{HY17}).\\ 

We suppose that $[F]=[F^\prime]$ as $\widetilde{H}$-lattices. 
Then we have 
\begin{align}
\left(\bigoplus_{i=1}^r \bZ[\widetilde{H}/H_i]^{\oplus x_i}\right)\oplus F^{\oplus b_1}
\ \simeq\ 
\left(\bigoplus_{i=1}^r \bZ[\widetilde{H}/H_i]^{\oplus y_i}\right)\oplus 
F^{\prime\,\oplus b_1}\label{eqpos2-1}
\end{align}
where $b_1=1$. 
We write the equation (\ref{eqpos2-1}) as 
\begin{align}
\bigoplus_{i=1}^r \bZ[\widetilde{H}/H_i]^{\oplus a_i}\ \simeq\ (F-F^\prime)^{\oplus (-b_1)}\label{eqpos2}
\end{align}
formally where $a_i=x_i-y_i\in\bZ$. 
Then we may consider ``$F-F^\prime$\,'' formally in the sense of 
(\ref{eqpos2-1}). 
By computing some $\GL(n,\bZ)$-conjugacy class invariants, 
we will give a necessary condition for $[F]=[F^\prime]$. 

Let $\{c_1,\ldots,c_r\}$ be a set of complete representatives of 
the conjugacy classes of $\widetilde{H}$. 
Let $A_i(c_j)$ be the matrix representation of the factor coset 
action of $c_j\in \widetilde{H}$ on $\bZ[\widetilde{H}/H_i]$ and 
$B(c_j)$ be the matrix representation of the action of 
$c_j\in \widetilde{H}$ on $F-F^\prime$. 

By (\ref{eqpos2}), for each $c_j\in \widetilde{H}$, we have 
\begin{align}
\sum_{i=1}^r a_i\, {\rm tr}\, A_i(c_j)+ b_1\, {\rm tr}\, B(c_j)=0\label{eqp1}
\end{align}
where {\rm tr}\,$A$ is the trace of the matrix $A$. 
Similarly, we consider the rank of $H^0=\widehat{Z}^0$.  
For each $H_j$, we get
\begin{align}
\sum_{i=1}^r a_i\, \rank \widehat{Z}^0(H_j,\bZ[\widetilde{H}/H_i])
+b_1\, \rank \widehat{Z}^0(H_j,F-F^\prime)=0.\label{eqp2}
\end{align}
Finally, we compute $\widehat{H}^0$.
Let $\Syl_p(A)$ be a $p$-Sylow subgroup of an abelian group $A$.
$\Syl_p(A)$ can be written as a direct product of cyclic groups uniquely.
Let $n_{p,e}(\Syl_p(A))$ be the number of direct summands 
of cyclic groups of order $p^e$. 
For each $H_j,p,e$, we get 
\begin{align}
\sum_{i=1}^r a_i\, n_{p,e}(\Syl_p(\widehat{H}^0(H_j,\bZ[\widetilde{H}/H_i])))
+b_1\, n_{p,e}(\Syl_p(\widehat{H}^0(H_j,F-F^\prime)))=0.\label{eqp3}
\end{align}
By the equalities (\ref{eqp1}), (\ref{eqp2}) and (\ref{eqp3}), we may get a 
system of linear equations in $a_1,\dots,a_r,b_1$ over $\bZ$.
Namely, we have that 
$[F]=[F^\prime]$ as $\widetilde{H}$-lattices 
$\Longrightarrow$ there exist $a_1,\ldots,a_r\in\bZ$ and $b_1=\pm 1$ 
which satisfy (\ref{eqpos2}) $\Longrightarrow$ 
this system of linear equations has an integer solution 
in $a_1,\ldots,a_r$ with $b_1=\pm 1$. 

In particular, if this system of linear equations has no integer solutions, 
then we conclude that $[F]\neq [F^\prime]$ as $\widetilde{H}$-lattices. \\

We made the following GAP \cite{GAP} algorithms for computing above. 
It is available as in \cite{BCAlgTori}.\\

\noindent 
{\tt FlabbyResolutionLowRank($G$).actionF} returns 
the matrix representation of the action of $G$ on $F$ 
where $F$ is a suitable flabby class of $M_G$ $([F]=[M_G]^{fl})$ 
with low rank by using backtracking techniques 
(see \cite[Chapter 5]{HY17}, see also \cite[Algorithm 4.1 (3)]{HHY20}).\\

\noindent 
{\tt PossibilityOfStablyEquivalentFSubdirectProduct($\widetilde{H}$)} 
returns a basis 
$\mathcal{L}=\{l_1,\dots,l_s\}$ of the solution space 
$\{[a_1,\ldots,a_r,b_1]\mid a_i, b_1\in\bZ\}$ 
of the system of linear equations which is obtained by the equalities 
(\ref{eqp1}), (\ref{eqp2}) and (\ref{eqp3}) 
and gives all possibilities that establish the equation (\ref{eqpos2})
for a subdirect product 
$\widetilde{H}\leq G\times G^\prime$ of $G$ and $G^\prime$.\\

\noindent 
{\tt PossibilityOfStablyEquivalentMSubdirectProduct($\widetilde{H}$)} 
returns the same as\\ 
{\tt PossibilityOfStablyEquivalentFSubdirectProduct($\widetilde{H}$)} 
but with respect to $M_G$ and $M_{G^\prime}$ 
instead of $F$ and $F^\prime$.\\

\noindent 
{\tt PossibilityOfStablyEquivalentFSubdirectProduct($\widetilde{H}$:H2)} 
returns the same as\\ 
{\tt PossibilityOfStablyEquivalentFSubdirectProduct($\widetilde{H}$)} 
but using also the additional equality 
\begin{align}
\sum_{i=1}^r a_i\, n_{p,e}(\Syl_p(H^2(\widetilde{H},\bZ[\widetilde{H}/H_i])))+b_1\, n_{p,e}(\Syl_p(H^2(\widetilde{H},F-F^\prime)))=0\label{eqp4}
\end{align}
and the equalities (\ref{eqp1}), (\ref{eqp2}) and (\ref{eqp3}).\\

\noindent 
{\tt PossibilityOfStablyEquivalentMSubdirectProduct($\widetilde{H}$:H2)} 
returns the same as\\ 
{\tt PossibilityOfStablyEquivalentFSubdirectProduct($\widetilde{H}$:H2)} 
but with respect to $M_G$ and $M_{G^\prime}$ 
instead of $F$ and $F^\prime$.\\


In general, we will provide a method 
in order to confirm the isomorphism 
\begin{align}
\left(\bigoplus_{i=1}^r \bZ[\widetilde{H}/H_i]^{\oplus a_i}\right)\oplus F^{\oplus b_1}
\simeq
\left(\bigoplus_{i=1}^r \bZ[\widetilde{H}/H_i]^{\oplus a_i^{\prime}}\right)\oplus F^{\prime\,\oplus b_1^{\prime}}
\label{eqiso2}
\end{align} 
with $a_i,a_i^{\prime}\geq 0$, $b_1,b_1^\prime\geq 1$, 
although it is needed by trial and error. 

Let 
$G_1$ (resp. $G_2$) be 
the matrix representation group of the action of $\widetilde{H}$ 
on the left-hand side 
$(\oplus_{i=1}^r \bZ[\widetilde{H}/H_i]^{\oplus a_i})$ $\oplus$ 
$F^{\oplus b_1}$ 
(resp. the right-hand side 
$(\oplus_{i=1}^r \bZ[\widetilde{H}/H_i]^{\oplus a_i^{\prime}})$ $\oplus$ 
$F^{\oplus b_1^{\prime}}$) of the isomorphism (\ref{eqiso2}). 
Let $\mathcal{P}=\{P_1,\dots,P_m\}$ be a basis of the solution space 
of $G_1P=PG_2$ where $m={\rm rank}_\bZ$ ${\rm Hom}(G_1,G_2)$ $=$ 
${\rm rank}_\bZ$ ${\rm Hom}_{\widetilde{H}}(M_{G_1},M_{G_2})$. 
Our aim is to find the matrix $P$ which satisfies $G_1 P=P G_2$ 
by using computer effectively.
If we can get a matrix $P$ with det $P=\pm 1$, then 
$G_1$ and $G_2$ are $\GL(n,\bZ)$-conjugate where 
$n$ is the rank of both sides of (\ref{eqiso2}) 
and hence the isomorphism (\ref{eqiso2}) established. 
This implies that the flabby class $[F^{\oplus b_1}]=[F^{\prime\,\oplus b_1^{\prime}}]$ as $\widetilde{H}$-lattices.\\

We made the following GAP \cite{GAP} algorithms 
which may establish the isomorphism  (\ref{eqiso2}).\\

\noindent
{\tt StablyEquivalentFCheckPSubdirectProduct($\widetilde{H},l_1,l_2$)} returns 
a basis $\mathcal{P}=\{P_1,\dots,P_m\}$ 
of the solution space of $G_1P=PG_2$ 
where $m={\rm rank}_\bZ$ ${\rm Hom}(G_1,G_2)$ and 
$G_1$ (resp. $G_2$) is the matrix representation group of the action of 
$\widetilde{H}$ 
on $(\oplus_{i=1}^r \bZ[\widetilde{H}/H_i]^{\oplus a_i})\oplus F^{\oplus b_1}$ (resp. 
$(\oplus_{i=1}^r \bZ[\widetilde{H}/H_i]^{\oplus a_i^{\prime}})\oplus F^{\prime\,\oplus b_1^{\prime}}$) 
with the isomorphism (\ref{eqiso2}) 
for a subdirect product 
$\widetilde{H}\leq G\times G^\prime$ of $G$ and $G^\prime$, 
and 
lists $l_1=[a_1,\ldots,a_r,b_1]$, 
$l_2=[a_1^{\prime},\ldots,a_r^{\prime},b_1^{\prime}]$, if $P$ exists. 
If such $P$ does not exist, this returns {\tt [ ]}.\\

\noindent
{\tt StablyEquivalentMCheckPSubdirectProduct($\widetilde{H},l_1,l_2$)}
returns the same as\\ 
{\tt StablyEquivalentFCheckPSubdirectProduct($\widetilde{H},l_1,l_2$)}  
but with respect to $M_G$ and $M_{G^\prime}$ 
instead of $F$ and $F^\prime$.\\

\noindent
{\tt StablyEquivalentFCheckMatSubdirectProduct($\widetilde{H},l_1,l_2,P$)} 
returns 
{\tt true} if $G_1P=PG_2$ and det $P=\pm 1$ where $G_1$ (resp. $G_2$) 
is the matrix representation group of the action of $\widetilde{H}$ 
on $(\oplus_{i=1}^r \bZ[\widetilde{H}/H_i]^{\oplus a_i})\oplus F^{\oplus b_1}$ (resp. 
$(\oplus_{i=1}^r \bZ[\widetilde{H}/H_i]^{\oplus a_i^{\prime}})\oplus F^{\prime\,\oplus b_1^{\prime}}$) 
with the isomorphism (\ref{eqiso2})
for a subdirect product 
$\widetilde{H}\leq G\times G^\prime$ of $G$ and $G^\prime$, 
and 
lists $l_1=[a_1,\ldots,a_r,b_1]$, 
$l_2=[a_1^{\prime},\ldots,a_r^{\prime},b_1^{\prime}]$. 
If not, this returns {\tt false}.\\

\noindent 
{\tt StablyEquivalentMCheckMatSubdirectProduct($\widetilde{H},l_1,l_2,P$)}
returns the same as\\ 
{\tt StablyEquivalentFCheckMatSubdirectProduct($\widetilde{H},l_1,l_2,P$)}
but with respect to $M_G$ and $M_{G^\prime}$ 
instead of $F$ and $F^\prime$.\\

\noindent
{\tt StablyEquivalentFCheckGenSubdirectProduct($\widetilde{H},l_1,l_2$)} 
returns 
the list $[\mathcal{M}_1,\mathcal{M}_2]$ where 
$\mathcal{M}_1=[g_1,\ldots,g_t]$ (resp. $\mathcal{M}_2=[g_1^{\prime},\ldots,g_t^{\prime}]$) 
is a list of the generators of $G_1$ (resp. $G_2$) 
which is the matrix representation group of the action of $\widetilde{H}$ 
on $(\oplus_{i=1}^r \bZ[\widetilde{H}/H_i]^{\oplus a_i})\oplus F^{\oplus b_1}$ (resp. 
$(\oplus_{i=1}^r \bZ[\widetilde{H}/H_i]^{\oplus a_i^{\prime}})\oplus F^{\prime\,\oplus b_1^{\prime}}$) 
with the isomorphism (\ref{eqiso2})
for a subdirect product 
$\widetilde{H}\leq G\times G^\prime$ of $G$ and $G^\prime$, 
and 
lists $l_1=[a_1,\ldots,a_r,b_1]$, 
$l_2=[a_1^{\prime},\ldots,a_r^{\prime},b_1^{\prime}]$.\\

\noindent 
{\tt StablyEquivalentMCheckGenSubdirectProduct($\widetilde{H},l_1,l_2$)} 
returns the same as\\  
{\tt StablyEquivalentFCheckGenSubdirectProduct($\widetilde{H},l_1,l_2$)} 
but with respect to $M_G$ and $M_{G^\prime}$ 
instead of $F$ and $F^\prime$.\\

{\it Proof of Theorem \ref{th6.1}.}

Let $G=N_{3,i}$ and $G^\prime=N_{3,j}$ $(1\leq i<j\leq 15)$ 
as in Definition \ref{defN3N4}. 
Let $M_{G}$ $($resp. $M_{G^\prime}$$)$ be a $G$-lattice 
$($resp. $G^\prime$-lattice$)$ as in Definition \ref{d2.2}. 

We show that $[F]=[F^\prime]$ as $\widetilde{H}$-lattices 
with $[F]=[M_G]^{fl}$, $[F^\prime]=[M_{G^\prime}]^{fl}$ 
and $\widetilde{H}\simeq G$ 
where $F$ and $F^\prime$ are obtained via the command 
{\tt FlabbyResolutionLowRank($G$).actionF} 
and 
{\tt FlabbyResolutionLowRank($G^\prime$).actionF} 
for $(i,j)=(5,6)$, $(11,13)\in J_{3,3}$ with $G\simeq G^\prime$.\\ 

Case 1: $G=N_{3,5}\simeq D_4$ and $G^\prime=N_{3,6}\simeq D_4$.

Applying the function\\ 
{\tt 
PossibilityOfStablyEquivalentSubdirectProducts($G$,$G^\prime$,\\
ConjugacyClassesSubgroups2TorusInvariants($G$),\\
ConjugacyClassesSubgroups2TorusInvariants($G^\prime$))},\\  
we see that there exists the only one subdirect product 
$\widetilde{H}\simeq G\leq G\times G^\prime$ of $G$ and $G^\prime$ 
up to $(\GL(3,\bZ)\times \GL(3,\bZ))$-conjugacy 
which satisfy $TI_{\varphi_1(H)}=TI_{\varphi_2(H)}$ 
for any $H\leq \widetilde{H}$. 

For this $\widetilde{H}\simeq G$, 
by applying the function 
{\tt PossibilityOfStablyEquivalentFSubdirectProduct($\widetilde{H}$)}, 
we get a basis 
$\mathcal{L}=\{l_1\}$ with $l_1=[0,0,0,0,0,0,0,0,1]$ 
of the solution space 
$\{[a_1,\ldots,a_r,b_1]\mid a_i, b_1\in\bZ\}$ 
of the system of linear equations which is obtained by the equalities 
(\ref{eqp1}), (\ref{eqp2}) and (\ref{eqp3}) 
and gives all possibilities the equation (\ref{eqpos2}).

Because $a_i=0$ for $1\leq i\leq 8$ and 
$b_1=1$ in the equation (\ref{eqpos2-1}), 
we get the only possibility is $F-F^\prime=0$ 
in the sense of the equation (\ref{eqpos2}), 
i.e. $F\oplus Q\simeq F^\prime\oplus Q$ for some permutation 
$\widetilde{H}$-lattice $Q$, 
with ${\rm rank}_{\bZ}$ $F=$ ${\rm rank}_{\bZ}$ $F^\prime=7$. 

Applying the function 
{\tt StablyEquivalentFCheckPSubdirectProduct($\widetilde{H}$,LHSlist($l_1$),RHSlist($l_1$))}, we get a basis $\mathcal{P}=\{P_1,\dots,P_8\}$ 
of the solution space of $G_1P=PG_2$ 
with det $P_i=-1$ ($i=1,5,6$) and det $P_j=0$ ($j=2,3,4,7,8$). 
Hence we take the matrix $P=P_1\in \GL(7,\bZ)$ 
which satisfies $FP=PF^\prime$: 
\begin{align*}
P=\left(\begin{array}{ccccccc} 
1&0&0&0&0&0&0\\
0&1&0&1&0&0&-1\\
0&0&1&-1&0&0&1\\
0&0&0&0&0&0&1\\
0&0&0&0&1&0&0\\
0&0&0&0&0&1&0\\
0&0&0&1&0&0&0
\end{array}\right). 
\end{align*}
This implies that $F\simeq F^\prime$ as $\widetilde{H}$-lattices 
(see Example \ref{exN3N3}).\\ 

Case 2: $G=N_{3,11}\simeq S_4$ and $G^\prime=N_{3,13}\simeq S_4$. 

We can obtain the result by the same manner as in Case 1. 
We see that there exists the only one subdirect product 
$\widetilde{H}\simeq G\leq G\times G^\prime$ of $G$ and $G^\prime$ 
up to $(\GL(3,\bZ)\times \GL(3,\bZ))$-conjugacy 
which satisfy $TI_{\varphi_1(H)}=TI_{\varphi_2(H)}$ 
for any $H\leq \widetilde{H}$. 
Then we have a basis 
$\mathcal{L}=\{l_1,l_2\}$ with 
$l_1=[ 1, 0, -2, -1, -1, 0, 0, 2, 2, 1, -2, 0 ]$, 
$l_2=[ 0, 0, 0, 0, 0, 0, 0, 0, 0, 0, 0, 1 ]$ 
of the solution space 
$\{[a_1,\ldots,a_r,b_1]\mid a_i, b_1\in\bZ\}$ as in Case 1. 
It follows from the latter $l_2$ that there exists a possibility 
$F-F^\prime=0$ 
in the sense of the equation (\ref{eqpos2}), i.e. 
$F\oplus Q\simeq F^\prime\oplus Q$ for some permutation 
$\widetilde{H}$-lattice $Q$, 
with ${\rm rank}_{\bZ}$ $F=$ ${\rm rank}_{\bZ}$ $F^\prime=15$. 
Applying the function 
{\tt StablyEquivalentFCheckPSubdirectProduct($\widetilde{H}$,LHSlist($l_1$),RHSlist($l_1$))}, 
we get a basis $\mathcal{P}=\{P_1,\dots,P_{13}\}$ 
of the solution space of $G_1P=PG_2$ 
with det $P_2=-1$, det $P_4=2657205$ and det $P_i=0$ ($i\neq 2,4$). 
Namely, we may get $P=P_2\in \GL(15,\bZ)$ 
which satisfies $FP=PF^\prime$:
\begin{align*}
P=\left(\begin{array}{ccccccccccccccc} 
0&1&0&0&0&-1&1&2&-1&1&1&1&-1&2&1\\
0&0&0&0&0&0&0&0&0&0&0&0&0&-1&0\\
3&4&5&0&2&4&-3&0&3&-5&-4&-4&-2&-3&-1\\
-3&-4&-6&-1&-2&-5&4&1&-3&6&5&5&2&5&3\\
0&0&0&0&0&0&0&0&0&0&-1&0&0&0&0\\
-1&-2&-2&0&-1&-3&2&1&-1&4&3&2&1&3&1\\
0&0&0&0&0&0&-1&0&0&0&0&0&0&0&0\\
0&0&-1&0&0&0&0&0&0&0&0&0&0&0&0\\
0&-1&0&0&0&0&0&0&0&0&0&0&0&0&0\\
0&0&0&0&0&0&0&-1&0&0&0&0&0&0&0\\
0&0&0&0&0&0&0&0&-1&0&0&0&0&0&0\\
-1&0&0&0&0&0&0&0&0&0&0&0&0&0&0\\
1&2&2&0&1&0&0&1&1&0&-1&-1&-1&1&1\\
-2&-3&-3&0&-2&-2&2&-1&-1&3&3&2&2&1&0\\
-2&-3&-4&-1&-1&-4&3&2&-3&5&3&4&1&4&3
\end{array}\right). 
\end{align*}
This implies that $F\simeq F^\prime$ as $\widetilde{H}$-lattices 
(see Example \ref{exN3N3}).
\qed\\

By applying the function 
{\tt StablyEquivalentFCheckPSubdirectProduct($\widetilde{H}$,LHSlist($l_1$),RHSlist($l_1$))} as in the proof of Theorem \ref{th6.1}, 
we get a basis $\mathcal{P}=\{P_1,\dots,P_m\}$ 
of the solution space of $G_1P=PG_2$ with det $P_i=\pm 1$ 
for some $1\leq i\leq m$ 
where $G_1$ (resp. $G_2$) is the matrix representation group of the action of 
$\widetilde{H}$ 
on the left-hand side $(\oplus_{i=1}^r \bZ[\widetilde{H}/H_i]^{\oplus a_i})\oplus F^{\oplus b_1}$ 
(resp. the right-hand side 
$(\oplus_{i=1}^r \bZ[\widetilde{H}/H_i]^{\oplus a_i^{\prime}})\oplus F^{\prime\,\oplus b_1^{\prime}}$) 
of the isomorphism (\ref{eqiso2}) 
and $m={\rm rank}_\bZ$ ${\rm Hom}(G_1,G_2)$. 

However, in general, we have that det $P_i\neq\pm 1$ for any $1\leq i\leq m$. 
In the general case, we should seek a matrix $P$ with det $P=\pm 1$ 
which is given as a linear combination $P=\sum_{i=1}^m c_i P_i$. 
This task is important for us 
and not easy in general even if we use a computer.\\

We made the following GAP \cite{GAP} algorithms which 
may find a matrix $P=\sum_{i=1}^m c_i P_i$ 
with $G_1P=PG_2$ and det $P=\pm 1$ as a concatenation 
of suitable matrices with all invariant factors $1$.\\

{\it We will explain the algorithms below 
when the input $\mathcal{P}=\{P_1,\dots,P_m\}$ is obtained by\\  
{\tt StablyEquivalentFCheckPSubdirectProduct($\widetilde{H},l_1,l_2$)} 
although it works in more general situations.}\\

\noindent 
{\tt SearchPRowBlocks($\mathcal{P}$)}
returns the records {\tt bpBlocks} and {\tt rowBlocks} 
where 
{\tt bpBlocks} (resp. {\tt rowBlocks}) is 
the decomposition of the list $l=[1,\ldots,m]$ (resp. $l=[1,\ldots,n]$) 
with $m={\rm rank}_\bZ$ ${\rm Hom}(G_1,G_2)$ (resp. $n={\rm size}$ $G_1$)
according to the direct sum decomposition of $M_{G_1}$
for a basis $\mathcal{P}=\{P_1,\dots,P_m\}$ 
of the solution space of $G_1P=PG_2$ 
where $G_1$ (resp. $G_2$) is the matrix representation group of the action of 
$\widetilde{H}$ 
on the left-hand side $(\oplus_{i=1}^r \bZ[\widetilde{H}/H_i]^{\oplus a_i})\oplus F^{\oplus b_1}$ 
(resp. the right-hand side 
$(\oplus_{i=1}^r \bZ[\widetilde{H}/H_i]^{\oplus a_i^{\prime}})\oplus F^{\prime\,\oplus b_1^{\prime}}$) 
of the isomorphism (\ref{eqiso2}).\\ 

We write 
$B${\tt [}\,$t$\,{\tt ]} $=$ {\tt SearchPRowBlocks($\mathcal{P}$)}.{\tt bpBlocks[}\,$t$\,{\tt ]} and 
$R${\tt [}\,$t$\,{\tt ]} $=$ {\tt SearchPRowBlocks($\mathcal{P}$)}.{\tt rowBlocks[}\,$t$\,{\tt ]}.\\

\noindent 
{\tt SearchPFilterRowBlocks($\mathcal{P}$,$B${\tt [}\,$t$\,{\tt ]},$R${\tt [}\,$t$\,{\tt ]},$j$)} 
returns 
the lists $\{M_s\}$ where $M_s$ is the $n_t \times n$ matrix 
with all invariant factors $1$ which is of the form 
$M_s=\sum_{i\,\in\,B\mbox{\tt\scriptsize [}\,t\,\mbox{\tt\scriptsize ]}} c_i P_i^\prime$ $(c_i\in\{0,1\})$ 
at most $j$ non-zero $c_i$'s 
and 
$P_i^\prime$ is the submatrix of $P_i$ consists of 
$R${\tt [}\,$t$\,{\tt ]} rows 
with $n_t={\rm length}($$R${\tt [}\,$t$\,{\tt ]}$)$ 
for a basis $\mathcal{P}=\{P_1,\dots,P_m\}$ 
of the solution space of $G_1P=PG_2$ 
where $G_1$ (resp. $G_2$) is the matrix representation group of the action of 
$\widetilde{H}$ 
on the left-hand side $(\oplus_{i=1}^r \bZ[\widetilde{H}/H_i]^{\oplus a_i})\oplus F^{\oplus b_1}$ 
(resp. the right-hand side 
$(\oplus_{i=1}^r \bZ[\widetilde{H}/H_i]^{\oplus a_i^{\prime}})\oplus F^{\prime\,\oplus b_1^{\prime}}$) 
of the isomorphism (\ref{eqiso2}), 
$B${\tt [}\,$t$\,{\tt ]} $=$ {\tt SearchPRowBlocks($\mathcal{P}$)}.{\tt bpBlocks[}\,$t$\,{\tt ]}, 
$R${\tt [}\,$t$\,{\tt ]} $=$ {\tt SearchPRowBlocks($\mathcal{P}$)}.{\tt rowBlocks[}\,$t$\,{\tt ]}) 
and $j\geq 1$.\\

\noindent 
{\tt SearchPFilterRowBlocks($\mathcal{P}$,$B${\tt [}\,$t$\,{\tt ]},$R${\tt [}\,$t$\,{\tt ]},$j$,$C$)} returns the same as\\ 
{\tt SearchPFilterRowBlocks($\mathcal{P}$,$B${\tt [}\,$t$\,{\tt ]},$R${\tt [}\,$t$\,{\tt ]},$j$)} 
but with respect to $c_i\in C$ instead of $c_i\in\{0,1\}$ 
for the list $C$ of integers.\\

\noindent 
{\tt SearchPFilterRowBlocksRandomMT($\mathcal{P}$,$B${\tt [}\,$t$\,{\tt ]},$R${\tt [}\,$t$\,{\tt ]},$u$)} returns the same as\\ 
{\tt SearchPFilterRowBlocks($\mathcal{P}$,$B${\tt [}\,$t$\,{\tt ]},$R${\tt [}\,$t$\,{\tt ]},$j$)} 
but with respect to random $u$ $c_i$'s via Mersenne Twister 
instead of at most $j$ non-zero $c_i$'s 
for integer $u\geq 1$.\\

\noindent 
{\tt SearchPFilterRowBlocksRandomMT($\mathcal{P}$,$B${\tt [}\,$t$\,{\tt ]},$R${\tt [}\,$t$\,{\tt ]},$u$,$C$)} returns the same as\\
{\tt SearchPFilterRowBlocksRandomMT($\mathcal{P}$,$B${\tt [}\,$t$\,{\tt ]},$R${\tt [}\,$t$\,{\tt ]},$u$)} 
but with respect to $c_i\in C$ instead of $c_i\in\{0,1\}$ 
for the list $C$ of integers.\\

\noindent 
{\tt SearchPMergeRowBlock($m_1,m_2$)} 
returns all concatenations of the matrices $M_s$ and $M_t$ vertically 
with all invariant factors $1$ 
(resp. a concatenation of the matrices $M_s$ and $M_t$ vertically 
with determinant $\pm 1$) 
for $m_1=\{M_s\}$ and $m_2=\{M_t\}$ 
where $M_s$ are $n_1\times n$ matrices 
and $M_t$ are $n_2\times n$ matrices 
with $n_1+n_2<n$ (resp. $n_1+n_2=n$).\\

When there exists $t\in\bZ$ such that $R${\tt [}\,$t$\,{\tt ]} $=\{j\}$, 
we can use:\\

\noindent 
{\tt SearchPLinear($M,\mathcal{P}_1$)}
returns the list $\{{\rm det}(M+P_i)\}_{i\,\in\,B\mbox{\tt\scriptsize [}\,t\,\mbox{\tt\scriptsize ]}}$ 
of integers for an $n\times n$ matrix $M$ 
which is obtained by inserting the zero row into the $j$-th row of 
$(n-1)\times n$ matrix 
$M_s=\sum_{i\,\not\in\,B\mbox{\tt\scriptsize [}\,t\,\mbox{\tt\scriptsize ]}} c_i P_i^\prime$ 
with all invariant factors $1$ 
and $\mathcal{P}_1=\{P_i\}_{i\,\in\, B\mbox{\tt\scriptsize [}\,t\,\mbox{\tt\scriptsize ]}}$
where $B${\tt [}\,$t$\,{\tt ]} $=$ {\tt SearchPRowBlocks($\mathcal{P}$)}.{\tt bpBlocks[}\,$t$\,{\tt ]}, 
$P_i^\prime$ is the submatrix of $P_i$ deleting the $j$-th row 
and 
$\mathcal{P}=\{P_1,\ldots,P_m\}$ is obtained by 
{\tt StablyEquivalentFCheckPSubdirectProduct($\widetilde{H},l_1,l_2$)} 
under the assumption that 
there exists $t\in\bZ$ such that $R${\tt [}\,$t$\,{\tt ]} $=\{j\}$.\\

When there exist $t_1, t_2\in\bZ$ such that 
$R${\tt [}\,$t_1$\,{\tt ]} $=\{j_1\}$, 
$R${\tt [}\,$t_2$\,{\tt ]} $=\{j_2\}$, 
we can use:\\

\noindent 
{\tt SearchPBilinear($M,\mathcal{P}_1,\mathcal{P}_2$)}
returns the matrix 
$[{\rm det}(M+P_{i_1}+P_{i_2})]_{i_1\,\in\,B\mbox{\tt\scriptsize [}\,t_1\,\mbox{\tt\scriptsize ]},i_2\,\in\,B\mbox{\tt\scriptsize [}\,t_2\,\mbox{\tt\scriptsize ]}}$ 
for an $n\times n$ matrix $M$ 
which is obtained by inserting the two zero rows 
into the $j_1$-th row and the $j_2$-th row 
of $(n-2)\times n$ matrix 
$M_s=\sum_{i\,\not\in\,B\mbox{\tt\scriptsize [}\,t_1\,\mbox{\tt\scriptsize ]}
\,\cup\,B\mbox{\tt\scriptsize [}\,t_2\,\mbox{\tt\scriptsize ]}} 
c_i P_i^\prime$ 
with all invariant factors $1$, 
$\mathcal{P}_1=\{P_{i_1}\}_{i_1\,\in\, B\mbox{\tt\scriptsize [}\,t_1\,\mbox{\tt\scriptsize ]}}$ and 
$\mathcal{P}_2=\{P_{i_2}\}_{i_2\,\in\, B\mbox{\tt\scriptsize [}\,t_2\,\mbox{\tt\scriptsize ]}}$ 
where 
$B${\tt [}\,$t_1$\,{\tt ]} $=$ {\tt SearchPRowBlocks($\mathcal{P}$)}.{\tt bpBlocks[}\,$t_1$\,{\tt ]}, 
$B${\tt [}\,$t_2$\,{\tt ]} $=$ {\tt SearchPRowBlocks($\mathcal{P}$)}.{\tt bpBlocks[}\,$t_2$\,{\tt ]}, 
$P_i^\prime$ is the submatrix of $P_i$ 
deleting the $j_1$-th and the $j_2$-th rows 
and 
$\mathcal{P}=\{P_1,\ldots,P_m\}$ is obtained by\\
{\tt StablyEquivalentFCheckPSubdirectProduct($\widetilde{H},l_1,l_2$)} 
under the assumption that 
there exist $t_1, t_2\in\bZ$ such that 
$R${\tt [}\,$t_1$\,{\tt ]} $=\{j_1\}$ and 
$R${\tt [}\,$t_2$\,{\tt ]} $=\{j_2\}$.\\

When there exists $t\in\bZ$ such that 
$R${\tt [}\,$t$\,{\tt ]} $=\{j_1,j_2\}$, 
we can use:\\

\noindent 
{\tt SearchPQuadratic($M,\mathcal{P}_1$)}
returns the matrix 
$[\frac{1}{2}({\rm det}(M+P_{i_1}+P_{i_2})-{\rm det}(M+P_{i_1})-{\rm det}(M+P_{i_2}))]_{i_1,i_2\,\in\,B\mbox{\tt\scriptsize [}\,t\,\mbox{\tt\scriptsize ]}}$ 
for an $n\times n$ matrix $M$ 
which is obtained by inserting the two zero rows 
into the $j_1$-th row and the $j_2$-th row of 
$(n-2)\times n$ matrix 
$M_s=\sum_{i\,\not\in\,B\mbox{\tt\scriptsize [}\,t\,\mbox{\tt\scriptsize ]}} c_i P_i^\prime$ 
with all invariant factors $1$ and 
$\mathcal{P}_1=\{P_i\}_{i\,\in\, B\mbox{\tt\scriptsize [}\,t\,\mbox{\tt\scriptsize ]}}$
where 
$B${\tt [}\,$t$\,{\tt ]} $=$ {\tt SearchPRowBlocks($\mathcal{P}$)}.{\tt bpBlocks[}\,$t$\,{\tt ]}, 
$P_i^\prime$ is the submatrix of $P_i$ 
deleting the $j_1$-th and $j_2$-th rows and 
$\mathcal{P}=\{P_1,\ldots,P_m\}$ is obtained by 
{\tt StablyEquivalentFCheckPSubdirectProduct($\widetilde{H},l_1,l_2$)} 
under the assumption that 
there exists $t\in\bZ$ such that $R${\tt [}\,$t$\,{\tt ]} $=\{j_1,j_2\}$.\\

When $R${\tt [}\,$1$\,{\tt ]} $=\{1,\ldots,m\}$, 
we can use:\\

\noindent 
{\tt SearchP1($\mathcal{P}$)} 
returns a matrix $P=\sum_{i=1}^m c_i P_i$ 
with $c_i\in\{0,1\}$, $G_1P=PG_2$ and det $P=\pm 1$ 
where $G_1$ (resp. $G_2$) is the matrix representation group of the action of 
$\widetilde{H}$ 
on the left-hand side $(\oplus_{i=1}^r \bZ[\widetilde{H}/H_i]^{\oplus a_i})\oplus F^{\oplus b_1}$ 
(resp. the right-hand side 
$(\oplus_{i=1}^r \bZ[\widetilde{H}/H_i]^{\oplus a_i^{\prime}})\oplus F^{\prime\,\oplus b_1^{\prime}}$) 
of the isomorphism (\ref{eqiso2}) 
for $\mathcal{P}=\{P_1,\ldots,P_m\}$ which is obtained by 
{\tt StablyEquivalentFCheckPSubdirectProduct($\widetilde{H},l_1,l_2$)} 
under the assumption that 
$R${\tt [}\,$1$\,{\tt ]}$=\{1,\ldots,m\}$.\\

\noindent 
{\tt SearchP1($\mathcal{P}$,$C$)} returns the same as 
{\tt SearchP1($\mathcal{P}$)} 
but with respect to $c_i\in C$ instead of $c_i\in\{0,1\}$ 
for the list $C$ of integers.\\

{\it Proof of Theorem \ref{th6.2}.}


Let $G=N_{3,i}$ and $G^\prime=N_{4,j}$ $(1\leq i\leq 15, 1\leq j\leq 152)$ 
as in Definition \ref{defN3N4}. 
Let $M_{G}$ $($resp. $M_{G^\prime}$$)$ be a $G$-lattice 
$($resp. $G^\prime$-lattice$)$ as in Definition \ref{d2.2}. 
We show that $[F]=[F^\prime]$ as $\widetilde{H}$-lattices 
with $[F]=[M_G]^{fl}$, $[F^\prime]=[M_{G^\prime}]^{fl}$, 
$\widetilde{H}\simeq G$ 
for $(i,j)\in J_{3,4}=\{(1,1)$, $(2,6)$, $(3,2)$, $(3,5)$, $(3,7)$, $(4,13)$,
$(4,23)$, $(4,24)$, $(5,15)$, $(5,16)$, $(5,17)$, $(5,25)$, $(5,26)$,
$(6,15)$, $(6,16)$, $(6,17)$, $(6,25)$, $(6,26)$, $(7,19)$, $(7,29)$,
$(7,30)$, $(10,69)$, $(10,74)$, $(10,75)$, $(12,71)$, $(12,77)$,
$(12,78)$, $(15,72)$, $(15,83)$, $(15,85)\}$ with $G\simeq G^\prime$.\\

Case 1: $(i,j)=(1,1)$, $(2,6)$, $(3,5)$, $(5,15)$, $(5,26)$, 
$(10,74)$, $(12,77)$, $(15,83)$. 

By applying the function 
{\tt StablyEquivalentFCheckPSubdirectProduct($\widetilde{H}$,LHSlist($l_1$),RHSlist($l_1$))} as in the proof of Theorem \ref{th6.1}, 
we can get a basis $\mathcal{P}=\{P_1,\dots,P_m\}$ 
of the solution space of $G_1P=PG_2$ 
with det $P_s=\pm 1$ for some $1\leq s\leq m$. 
This implies that $[F]=[F^\prime]$ as $\widetilde{H}$-lattices 
(see Example \ref{exN3N4}).\\

Case 2: $(i,j)=(3,2)$, $(3,7)$, $(4,13)$, $(4,24)$, 
$(5,16)$, $(5,17)$, $(5,25)$, $(7,19)$, $(7,30)$, $(10,69)$, 
$(10,75)$, $(12,71)$, $(12,78)$, $(15,72)$, $(15,85)$. 

By applying the function 
{\tt StablyEquivalentFCheckPSubdirectProduct($\widetilde{H}$,LHSlist($l_1$),RHSlist($l_1$))} as in the proof of Theorem \ref{th6.1}, 
we obtain a basis 
$\mathcal{P}=\{P_1,\dots,P_m\}$ of the solution space of $G_1P=PG_2$. 
However, we find that det $P_s=0$ $(1\leq s\leq m)$. 

We will give a proof of the case $(i,j)=(3,2)$. 
The remaining cases can be proved by the same way. 

For $G=N_{3,3}\simeq C_2^3$ and $G^\prime=N_{4,2}\simeq C_2^3$, 
by applying the function {\tt SearchPRowBlocks}($\mathcal{P}$), 
we get 
$B=$ {\tt SearchPRowBlocks($\mathcal{P}$).bpBlocks} 
and 
$R=$ {\tt SearchPRowBlocks($\mathcal{P}$).rowBlocks} 
with $B=\{B_1,B_2\}$, $B_1=\{1,\ldots,5\}$, $B_2=\{6,\ldots,32\}$, 
$R=\{R_1,R_2\}$, $R_1=\{1\}$, $R_2=\{2,\ldots,12\}$ 
which correspond to 
$\mathcal{P}=\{P_1,\ldots,P_{32}\}$ with 
det $P_s=0$ $(1\leq s\leq 32)$ 
and $\bZ\oplus F\simeq F^\prime$
with rank $F=11$, rank $F^\prime=12$ 
in the equation (\ref{eqiso2}). 

Applying 
{\tt SearchPFilterRowBlocks}($\mathcal{P}$,$B_t$,$R_t$,$1$) 
$(t=1,2)$, 
we get $P_1,\ldots,P_5$ with an invariant factor $1$ 
and $P_{i_1},\ldots,P_{i_5}$ with all $11$ invariant factors $1$ 
with $2\leq i_1<\cdots<i_5\leq 12$. 
Then we use the function {\tt SearchPMergeRowBlock} to them, 
we get the following matrix $P$ with $G_1P=PG_2$ and det $P=-1$ 
where $G_1$ (resp. $G_2$) 
is the matrix representation group of the action of 
$\widetilde{H}$ 
on the left-hand side $\bZ\oplus F$ 
(resp. the right-hand side $F^\prime$) of the isomorphism (\ref{eqiso2}) 
(see Example \ref{exN3N4}): 
\begin{align*}
P=\left(\begin{array}{cccccccccccc} 
1&0&0&0&0&0&0&-1&1&1&1&1\\
0&1&0&0&0&0&0&0&0&0&0&0\\
0&0&0&0&0&0&0&0&0&0&0&-1\\
0&0&0&0&1&0&0&1&0&-1&-1&0\\
0&0&1&0&0&0&0&0&0&0&0&0\\
0&0&1&0&0&0&0&1&0&-1&0&-1\\
0&0&0&0&1&1&1&1&-1&-1&-1&0\\
0&1&0&1&0&1&0&1&-1&0&-1&-1\\
0&0&0&0&0&0&0&1&0&0&0&0\\
0&0&1&1&0&0&1&1&-1&-1&0&-1\\
0&1&0&0&1&1&0&1&-1&-1&-1&-1\\
0&0&0&0&1&0&0&0&0&0&0&0
\end{array}\right). 
\end{align*}
This implies that $\bZ\oplus F\simeq F^\prime$ 
as $\widetilde{H}$-lattices. 
In particular, we have $[F]=[F^\prime]$.\\ 

Case 3: $(i,j)=(4,23)$, $(7,29)$. 

By applying the function 
{\tt StablyEquivalentFCheckPSubdirectProduct($\widetilde{H}$,LHSlist($l_1$),RHSlist($l_1$))} as in the proof of Theorem \ref{th6.1}, 
we obtain a basis 
$\mathcal{P}=\{P_1,\dots,P_m\}$ of the solution space of $G_1P=PG_2$. 
However, we find that det $P_s\neq \pm 1$ $(1\leq s\leq m)$. 

We give a proof of the former case $(i,j)=(4,23)$ 
because the latter case $(i,j)=(7,29)$ can be proved by the same way. 

For $G=N_{3,4}\simeq C_4\times C_2$ and 
$G^\prime=N_{4,23}\simeq C_4\times C_2$, 
by applying the function {\tt SearchPRowBlocks}($\mathcal{P}$), 
we get 
$B=$ {\tt SearchPRowBlocks($\mathcal{P}$).bpBlocks} 
and 
$R=$ {\tt SearchPRowBlocks($\mathcal{P}$).rowBlocks} 
with $B=\{1,\ldots,19\}$, 
$R=\{1,\ldots,11\}$
which correspond to 
$\mathcal{P}=\{P_1,\ldots,P_{19}\}$ 
with det $P_s\in\{0,-3\}$ $(1\leq s\leq 19)$ 
and $F\simeq F^\prime$
with rank $F=$ rank $F^\prime=11$ 
in the equation (\ref{eqiso2}). 
In particular, we have 
$R${\tt [}\,$1$\,{\tt ]} $=\{1,\ldots,19\}$. 

Then by applying 
{\tt SearchP1}($\mathcal{P}$), 
we get the following matrix $P$ with $G_1P=PG_2$ and det $P=-1$ 
where $G_1$ (resp. $G_2$) 
is the matrix representation group of the action of 
$\widetilde{H}$ on the left-hand side $F$ 
(resp. the right-hand side $F^\prime$) of the isomorphism (\ref{eqiso2}) 
(see Example \ref{exN3N4}): 
\begin{align*}
P=\left(\begin{array}{ccccccccccc}
1&0&1&0&0&0&0&0&0&0&0\\
1&0&0&0&1&1&0&0&1&0&-1\\
0&0&0&1&0&0&1&0&-1&0&1\\
0&0&0&0&0&0&1&1&0&0&0\\
1&-1&0&0&0&1&0&1&0&0&0\\
0&1&0&1&0&-1&1&0&-1&0&1\\
1&0&1&0&1&1&-1&0&0&0&-1\\
1&-1&0&0&0&1&0&1&1&-1&0\\
0&0&0&0&1&0&0&0&0&1&0\\
1&0&1&1&0&0&0&0&-1&0&0\\
0&0&0&1&0&0&1&0&0&0&0
\end{array}\right).
\end{align*}
This implies that $F\simeq F^\prime$ 
as $\widetilde{H}$-lattices.\\ 

Case 4: $(i,j)=(6,15)$, $(6,16)$, $(6,17)$, $(6,25)$, $(6,26)$. 

Because we already see $[M_{N_{3,5}}]^{fl}=[M_{N_{3,6}}]^{fl}$ as 
$\widetilde{H}$-lattices with $\widetilde{H}\simeq G$ by Theorem \ref{th6.1}, 
it follows from the case $i=5$ in Cases 1 and 2 above 
that there exists a subdirect product 
$\widetilde{H}^\prime\simeq G\leq G\times G^\prime$ of $G$ and $G^\prime$ 
such that $[M_G]^{fl}=[M_{G^\prime}]^{fl}$ 
as $\widetilde{H}^\prime$-lattices.\qed\\

{\it Proof of Theorem \ref{th6.3}.}


Let $G=N_{4,i}$ and $G^\prime=N_{4,j}$ $(1\leq i<j\leq 152)$ 
as in Definition \ref{defN3N4}. 
Let $M_{G}$ $($resp. $M_{G^\prime}$$)$ be a $G$-lattice 
$($resp. $G^\prime$-lattice$)$ as in Definition \ref{d2.2}. 
We show that $[F]=[F^\prime]$ as $\widetilde{H}$-lattices 
with $[F]=[M_G]^{fl}$, $[F^\prime]=[M_{G^\prime}]^{fl}$, 
$\widetilde{H}\simeq G$ 
for $(i,j)\in J_{4,4}=\{(2,5)$, $(2,7)$, $(3,8)$, $(5,7)$, $(13,23)$, 
$(13,24)$, $(14,39)$, $(15,16)$, $(15,17)$, 
$(15,25)$, $(15,26)$, $(16,17)$, $(16,25)$, 
$(16,26)$, $(17,25)$, $(17,26)$, $(19,29)$, 
$(19,30)$, $(20,31)$, $(20,33)$, $(20,35)$, 
$(22,46)$, $(23,24)$, $(25,26)$, $(29,30)$, 
$(31,33)$, $(31,35)$, $(32,34)$, $(33,35)$, 
$(41,44)$, $(59,91)$, $(60,92)$, $(63,93)$, 
$(65,94)$, $(66,96)$, 
$(67,95)$, $(68,99)$, 
$(69,74)$, $(69,75)$, $(71,77)$, $(71,78)$, 
$(72,83)$, $(72,85)$, $(74,75)$, $(77,78)$, 
$(81,84)$, $(81,87)$, $(83,85)$, 
$(84,87)$, $(86,88)$, $(102,103)\}$ with $G\simeq G^\prime$.\\

Case 1: $(i,j)=(2,5)$, $(2,7)$, $(5,7)$, $(13,23)$, 
$(13,24)$,  $(15,16)$, $(15,17)$, 
$(15,25)$, $(15,26)$, $(16,17)$, $(16,25)$, 
$(16,26)$, $(17,25)$, $(17,26)$, $(19,29)$, 
$(19,30)$, $(23,24)$, $(25,26)$, $(29,30)$, 
$(69,74)$, $(69,75)$, $(71,77)$, $(71,78)$, 
$(72,83)$, $(72,85)$, $(74,75)$, $(77,78)$, $(83,85)$. 

It follows from Theorem \ref{th6.2} that for some $1\leq i\leq 15$, 
$[M_G]^{fl}=[M_{N_{3,i}}]^{fl}=[M_{G^\prime}]^{fl}$ 
as $\widetilde{H}^\prime$-lattices for some subdirect product 
$\widetilde{H}^\prime\simeq G\leq G\times G^\prime$ of $G$ and $G^\prime$.\\ 


Case 2: $(i,j)=(3,8)$, $(14,39)$, $(20,31)$, $(20,33)$, $(20,35)$, 
$(32,34)$, $(41,44)$, $(81,84)$, $(81,87)$. 

By applying the function 
{\tt StablyEquivalentFCheckPSubdirectProduct($\widetilde{H}$,LHSlist($l_1$),RHSlist($l_1$))} as in the proof of Theorem \ref{th6.1}, 
we get a basis $\mathcal{P}=\{P_1,\dots,P_m\}$ 
of the solution space of $G_1P=PG_2$ with det $P_s=\pm 1$ 
for some $1\leq s\leq m$ 
(see Example \ref{exN4N4}).\\

Case 3: $(i,j)=(86,88)$. 

By applying the function 
{\tt StablyEquivalentFCheckPSubdirectProduct($\widetilde{H}$,LHSlist($l_1$),RHSlist($l_1$))} as in the proof of Theorem \ref{th6.1}, 
we obtain a basis 
$\mathcal{P}=\{P_1,\dots,P_m\}$ of the solution space of $G_1P=PG_2$. 
However, we find that det $P_s=0$ $(1\leq s\leq m)$. 

By using the function {\tt SearchPMergeRowBlock}, 
we can prove the result as in Case 2 of the proof of Theorem \ref{th6.2} 
(see Example \ref{exN4N4}).\\

Case 4: $(i,j)=(63,93)$, $(65,94)$, $(67,95)$, $(68,99)$. 

By applying the function 
{\tt StablyEquivalentFCheckPSubdirectProduct($\widetilde{H}$,LHSlist($l_1$),RHSlist($l_1$))} as in the proof of Theorem \ref{th6.1}, 
we obtain a basis 
$\mathcal{P}=\{P_1,\dots,P_m\}$ of the solution space of $G_1P=PG_2$. 
However, we find that det $P_s=0$ $(1\leq s\leq m)$. 

We give a proof of the first case $(i,j)=(63,93)$ 
because the remaining cases can be proved by the same way. 
(For the cases $(67,95)$, $(68,99)$, we use the function 
{\tt SearchPFilterRowBlocksRandomMT($\mathcal{P}$,$B${\tt [}\,$t$\,{\tt ]},$R${\tt [}\,$t$\,{\tt ]},$u$)} via Mersenne Twister 
(cf. Matsumoto and Nishimura \cite{MN98})
instead of {\tt SearchPRowBlocks}($\mathcal{P}$) (see Example \ref{exN3N4}).)

For $G=N_{4,63}\simeq S_3\times C_6$ and 
$G^\prime=N_{4,93}\simeq S_3\times C_6$, 
by applying the function {\tt SearchPRowBlocks}($\mathcal{P}$), 
we get 
$B=$ {\tt SearchPRowBlocks($\mathcal{P}$).bpBlocks} 
and 
$R=$ {\tt SearchPRowBlocks($\mathcal{P}$).rowBlocks} 
with $B=\{B_1,\ldots,B_5\}$, 
$B_1=\{1,\ldots,13\}$, $B_2=\{14,\ldots,27\}$, 
$B_3=\{28,\ldots,34\}$, $B_4=\{35,\ldots,41\}$, 
$B_5=\{42,\ldots,91\}$, 
$R=\{R_1,\ldots,R_5\}$, 
$R_1=\{1,\ldots,6\}$, $R_2=\{7,\ldots,10\}$,
$R_3=\{11\}$, $R_4=\{12\}$, $R_5=\{13,\ldots,38\}$ 
which correspond to 
$\mathcal{P}=\{P_1,\ldots,P_{91}\}$ with 
det $P_s=0$ $(1\leq s\leq 91)$ and 
$F=[M_G]^{fl}$, $F^\prime=[M_{G^\prime}]^{fl}$ 
with rank $F=$ rank $F^\prime=26$. 

Because $R_3=\{11\}$, $R_4=\{12\}$, 
by applying 
{\tt SearchPBilinear($M$,$\{P_{i_1}\}_{i_1\,\in\, B_3}$,$\{P_{i_2}\}_{i_2\,\in\, B_4}$)}, 
we get a $38\times 38$ matrix $P$ with $G_1P=PG_2$ and det $P=-1$ 
where $G_1$ (resp. $G_2$) 
is the matrix representation group of the action of 
$\widetilde{H}$ on the left-hand side 
$\bZ[\widetilde{H}/H_{11}]\oplus \bZ[\widetilde{H}/H_{16}]\oplus 
\bZ\oplus \bZ\oplus F$ with rank $6+4+1+1+26=38$ 
(resp. the right-hand side $
\bZ[\widetilde{H}/H_{17}]\oplus 
\bZ[\widetilde{H}/H_{18}]\oplus 
\bZ[\widetilde{H}/H_{19}]\oplus 
\bZ[\widetilde{H}/H_{20}]\oplus 
\bZ[\widetilde{H}/H_{21}]\oplus 
F^\prime$ with rank $3+3+2+2+2+26=38$) 
of the isomorphism (\ref{eqiso2}) where 
$H_{11}\simeq C_6$, 
$H_{16}\simeq C_3^2$ 
(resp. 
$H_{17}\simeq C_6\times C_2$, 
$H_{18}\simeq D_6$, 
$H_{19}\simeq C_6\times C_3$, 
$H_{20}\simeq S_3\times C_3$, 
$H_{21}\simeq S_3\times C_3$). 
Hence we have $[F]=[F^\prime]$ as 
$\widetilde{H}$-lattices with $\widetilde{H}\simeq G$ 
(see Example \ref{exN3N4}).\\

Case 5: $(i,j)=(102,103)$. 

By applying the function 
{\tt StablyEquivalentFCheckPSubdirectProduct($\widetilde{H}$,LHSlist($l_1$),RHSlist($l_1$))} as in the proof of Theorem \ref{th6.1}, 
we obtain a basis 
$\mathcal{P}=\{P_1,\dots,P_m\}$ of the solution space of $G_1P=PG_2$. 
However, we find that det $P_s=0$ $(1\leq s\leq m)$. 

For $G=N_{4,102}\simeq S_3^2\rtimes C_2$ and 
$G^\prime=N_{4,103}\simeq S_3^2\rtimes C_2$, 
by applying the function {\tt SearchPRowBlocks}($\mathcal{P}$), 
we get 
$B=$ {\tt SearchPRowBlocks($\mathcal{P}$).bpBlocks} 
and 
$R=$ {\tt SearchPRowBlocks($\mathcal{P}$).rowBlocks} 
with $B=\{B_1,\ldots,B_4\}$, 
$B_1=\{1,\ldots,9\}$, $B_2=\{10,\ldots,18\}$, 
$B_3=\{19,\ldots,26\}$, $B_4=\{27,\ldots,56\}$, 
$R=\{R_1,\ldots,R_4\}$, 
$R_1=\{1,\ldots,6\}$, $R_2=\{7,\ldots,10\}$,
$R_3=\{11,12\}$, $R_4=\{13,\ldots,38\}$ 
which correspond to 
$\mathcal{P}=\{P_1,\ldots,P_{56}\}$ with 
det $P_s=0$ $(1\leq s\leq 56)$ and 
$F=[M_G]^{fl}$, $F^\prime=[M_{G^\prime}]^{fl}$ 
with rank $F=$ rank $F^\prime=26$. 

Because $R_3=\{11,12\}$, 
by applying 
{\tt SearchPQuadratic($M$,$\{P_i\}_{i\,\in\, B_3}$)}, 
we get a $38\times 38$ matrix $P$ with $G_1P=PG_2$ and det $P=1$ 
where $G_1$ (resp. $G_2$) 
is the matrix representation group of the action of 
$\widetilde{H}$ on the left-hand side 
$\bZ[\widetilde{H}/H_{19}]\oplus \bZ[\widetilde{H}/H_{21}]\oplus 
\bZ[\widetilde{H}/H_{25}]\oplus F$ with rank $6+4+2+26=38$ 
(resp. the right-hand side $
\bZ[\widetilde{H}/H_{18}]\oplus 
\bZ[\widetilde{H}/H_{22}]\oplus 
\bZ[\widetilde{H}/H_{23}]\oplus 
F^\prime$ with rank $6+4+2+26=38$) of the isomorphism (\ref{eqiso2}) 
where 
$H_{19}\simeq D_6$, 
$H_{21}\simeq S_3\times C_3$, 
$H_{25}\simeq S_3^2$ 
(resp. 
$H_{18}\simeq D_6$, 
$H_{22}\simeq S_3\times C_3$, 
$H_{23}\simeq S_3^2$). 
Hence we have $[F]=[F^\prime]$ as 
$\widetilde{H}$-lattices with $\widetilde{H}\simeq G$ 
(see Example \ref{exN3N4}).\\

Case 6: $(i,j)=(22,46)$, $(59,91)$, $(60,92)$, $(66,96)$. 

By applying the function 
{\tt StablyEquivalentFCheckPSubdirectProduct($\widetilde{H}$,LHSlist($l_1$),RHSlist($l_1$))} as in the proof of Theorem \ref{th6.1}, 
we obtain a basis 
$\mathcal{P}=\{P_1,\dots,P_m\}$ of the solution space of $G_1P=PG_2$. 
However, we find that det $P_s=0$ $(1\leq s\leq m)$. 

By applying the function {\tt SearchP1}($\mathcal{P}$), 
we get the following matrix $P$ with $G_1P=PG_2$ and det $P=\pm 1$ 
as in Case 3 of the proof of Theorem \ref{th6.2} 
(see Example \ref{exN4N4}).\\

Case 7: $(i,j)=(31,33)$, $(31,35)$, $(33,35)$, $(84,87)$. 
For these cases, 
there exists $N_{4,m}$ ($1\leq m\leq 152$) which appears in 
Cases 2, 3, 4 above 
such that 
$[M_G]^{fl}=[M_{N_{4,m}}]^{fl}=[M_{G^\prime}]^{fl}$ 
as $\widetilde{H}^\prime$-lattices for some subdirect product 
$\widetilde{H}^\prime\simeq G\leq G\times G^\prime$ 
of $G$ and $G^\prime$. 
Hence these cases can be reduced to Cases 2, 3, 4 above.\qed\\

{\it Proof of Theorem \ref{th6.4}.}

%
%
For each $(i,j)$, by applying the function 
{\tt PossibilityOfStablyEquivalentFSubdirectProduct($\widetilde{H}$:H2)}, 
we may confirm that 
$[M_G]^{fl}\not\sim [M_{G^\prime}]^{fl}$ 
(see Example \ref{exI4I4}).\qed\\

{\it Proof of Theorem \ref{thmain1}.}

The statement follows from 
Proposition \ref{prop5.2} and Theorem \ref{th6.1}.\qed\\

{\it Proof of Theorem \ref{thmain3}.}

By Proposition \ref{prop5.3}, Proposition \ref{prop5.4}, 
Theorem \ref{th6.2}, Theorem \ref{th6.3} and Theorem \ref{th6.4}, 
it is enough to check the cases $N_{31,i}$. 
For $G^\prime=N_{31,j}$ $(1\leq j\leq 64)$, 
we can find $G=N_{3,i}$ ($1\leq i\leq 15$) 
such that $[M_G]^{fl}\sim [M_{G^\prime}]^{fl}$ by Table $7$. 
The last statement also follows from 
Theorem \ref{th6.2}.\qed

\bigskip
\begin{example}[Theorem \ref{th6.1} for $G=N_{3,i}$ and $G^\prime=N_{3,j}$ with $(i,j)=(5,6), (11,13)$]\label{exN3N3}~\vspace*{-5mm}\\
\begin{verbatim}
gap> Read("BCAlgTori.gap");
gap> N3g:=List(N3,x->MatGroupZClass(x[1],x[2],x[3],x[4]));;

gap> N3inv5:=ConjugacyClassesSubgroups2TorusInvariants(N3g[5]);;
gap> N3inv6:=ConjugacyClassesSubgroups2TorusInvariants(N3g[6]);;
gap> pos:=PossibilityOfStablyEquivalentSubdirectProducts(N3g[5],N3g[6],N3inv5,N3inv6);;
gap> Length(pos); # the number of pos 
1
gap> List([N3g[5],N3g[6]],StructureDescription);
[ "D8", "D8" ]
gap> List(pos,StructureDescription); # N3g[5]=N3g[6]=pos[1]=D4
[ "D8" ]
gap> PossibilityOfStablyEquivalentFSubdirectProduct(pos[1]);
[ [ 0, 0, 0, 0, 0, 0, 0, 0, 1 ] ]
gap> l:=last[1]; # possibility for F=F' in the sense of the equation (8)
[ 0, 0, 0, 0, 0, 0, 0, 0, 1 ]
gap> LHSlist(l); # the list l1 for the left-hand side of (8)
[ 0, 0, 0, 0, 0, 0, 0, 0, 1 ]
gap> RHSlist(l); # the list l2 for the right-hand side of (8)
[ 0, 0, 0, 0, 0, 0, 0, 0, 1 ]
gap> bp:=StablyEquivalentFCheckPSubdirectProduct(pos[1],LHSlist(l),RHSlist(l));;
gap> List(bp,Determinant); # det(bp[1])=-1 
[ -1, 0, 0, 0, -1, -1, 0, 0 ]
gap> bp[1]; # the matrix P=bp[1] with det(P)=-1, FP=PF' and rank(F)=rank(F')=7
[ [ 1, 0, 0, 0, 0, 0, 0 ], 
  [ 0, 1, 0, 1, 0, 0, -1 ], 
  [ 0, 0, 1, -1, 0, 0, 1 ],
  [ 0, 0, 0, 0, 0, 0, 1 ], 
  [ 0, 0, 0, 0, 1, 0, 0 ], 
  [ 0, 0, 0, 0, 0, 1, 0 ], 
  [ 0, 0, 0, 1, 0, 0, 0 ] ]
gap> StablyEquivalentFCheckMatSubdirectProduct(pos[1],LHSlist(l),RHSlist(l),bp[1]);
true

gap> N3inv11:=ConjugacyClassesSubgroups2TorusInvariants(N3g[11]);;
gap> N3inv13:=ConjugacyClassesSubgroups2TorusInvariants(N3g[13]);;
gap> pos:=PossibilityOfStablyEquivalentSubdirectProducts(N3g[11],N3g[13],N3inv11,N3inv13);;
gap> Length(pos); # the number of pos
1
gap> List([N3g[11],N3g[13]],StructureDescription);
[ "S4", "S4" ]
gap> List(pos,StructureDescription); # N3g[11]=N3g[13]=pos[1]=S4
[ "S4" ]
gap> PossibilityOfStablyEquivalentFSubdirectProduct(pos[1]);
[ [ 1, 0, -2, -1, -1, 0, 0, 2, 2, 1, -2, 0 ], 
  [ 0, 0, 0, 0, 0, 0, 0, 0, 0, 0, 0, 1 ] ]
gap> l:=last[2]; # possibility for F=F' in the sense of the equation (8)
[ 0, 0, 0, 0, 0, 0, 0, 0, 0, 0, 0, 1 ]
gap> LHSlist(l); # the list l1 for the left-hand side of (8)
[ 0, 0, 0, 0, 0, 0, 0, 0, 0, 0, 0, 1 ]
gap> RHSlist(l); # the list l2 for the right-hand side of (8)
[ 0, 0, 0, 0, 0, 0, 0, 0, 0, 0, 0, 1 ]
gap> bp:=StablyEquivalentFCheckPSubdirectProduct(pos[1],LHSlist(l),RHSlist(l));;
gap> List(bp,Determinant); # det(bp[2])=-1 
[ 0, -1, 0, 2657205, 0, 0, 0, 0, 0, 0, 0, 0, 0 ]
gap> bp[2]; # the matrix P=bp[2] with det(P)=-1, FP=PF' and rank(F)=rank(F')=15
[ [ 0, 1, 0, 0, 0, -1, 1, 2, -1, 1, 1, 1, -1, 2, 1 ], 
  [ 0, 0, 0, 0, 0, 0, 0, 0, 0, 0, 0, 0, 0, -1, 0 ], 
  [ 3, 4, 5, 0, 2, 4, -3, 0, 3, -5, -4, -4, -2, -3, -1 ], 
  [ -3, -4, -6, -1, -2, -5, 4, 1, -3, 6, 5, 5, 2, 5, 3 ], 
  [ 0, 0, 0, 0, 0, 0, 0, 0, 0, 0, -1, 0, 0, 0, 0 ], 
  [ -1, -2, -2, 0, -1, -3, 2, 1, -1, 4, 3, 2, 1, 3, 1 ], 
  [ 0, 0, 0, 0, 0, 0, -1, 0, 0, 0, 0, 0, 0, 0, 0 ], 
  [ 0, 0, -1, 0, 0, 0, 0, 0, 0, 0, 0, 0, 0, 0, 0 ], 
  [ 0, -1, 0, 0, 0, 0, 0, 0, 0, 0, 0, 0, 0, 0, 0 ], 
  [ 0, 0, 0, 0, 0, 0, 0, -1, 0, 0, 0, 0, 0, 0, 0 ], 
  [ 0, 0, 0, 0, 0, 0, 0, 0, -1, 0, 0, 0, 0, 0, 0 ], 
  [ -1, 0, 0, 0, 0, 0, 0, 0, 0, 0, 0, 0, 0, 0, 0 ], 
  [ 1, 2, 2, 0, 1, 0, 0, 1, 1, 0, -1, -1, -1, 1, 1 ], 
  [ -2, -3, -3, 0, -2, -2, 2, -1, -1, 3, 3, 2, 2, 1, 0 ], 
  [ -2, -3, -4, -1, -1, -4, 3, 2, -3, 5, 3, 4, 1, 4, 3 ] ]
gap> StablyEquivalentFCheckMatSubdirectProduct(pos[1],LHSlist(l),RHSlist(l),bp[2]);
true
\end{verbatim}
\end{example}

\bigskip
\begin{example}[Theorem \ref{th6.2} for $G=N_{3,i}$ and $G^\prime=N_{4,j}$]\label{exN3N4}~\vspace*{-5mm}\\

\end{example}

\bigskip
\begin{example}[Theorem \ref{th6.4} for $G=I_{4,i}$ and $G^\prime=I_{4,j}$]\label{exI4I4}~\vspace*{-5mm}\\
\begin{verbatim}
gap> Read("BCAlgTori.gap");
gap> I4g:=List(I4,x->MatGroupZClass(x[1],x[2],x[3],x[4]));;

gap> I4inv1:=ConjugacyClassesSubgroups2TorusInvariants(I4g[1]);;
gap> I4inv2:=ConjugacyClassesSubgroups2TorusInvariants(I4g[2]);;
gap> pos:=PossibilityOfStablyEquivalentSubdirectProducts(I4g[1],I4g[2],I4inv1,I4inv2);;
gap> Length(pos); # the number of pos
1
gap> List([I4g[1],I4g[2]],StructureDescription);
[ "C5 : C4", "C5 : C4" ]
gap> List(pos,StructureDescription); # I4g[1]=N4g[2]=pos[1]=F20
[ "C5 : C4" ]
gap> PossibilityOfStablyEquivalentFSubdirectProduct(pos[1]);
# possibility for F=F' in the sense of the equation (8)
[ [ 0, 0, 0, 0, 0, 0, 1 ] ]
gap> PossibilityOfStablyEquivalentFSubdirectProduct(pos[1]:H2);
# impossible in the sense of the equation (8) with H2 option
[  ]

gap> I4inv1:=ConjugacyClassesSubgroups2TorusInvariants(I4g[1]);;
gap> I4inv3:=ConjugacyClassesSubgroups2TorusInvariants(I4g[3]);;
gap> pos:=PossibilityOfStablyEquivalentSubdirectProducts(I4g[1],I4g[3],I4inv1,I4inv3);;
gap> Length(pos); # the number of pos
1
gap> List([I4g[1],I4g[3]],StructureDescription);
[ "C5 : C4", "C2 x (C5 : C4)" ]
gap> List(pos,StructureDescription); # I4g[1]=F20, N4g[3]=pos[1]=C2xF20
[ "C2 x (C5 : C4)" ]
gap> PossibilityOfStablyEquivalentFSubdirectProduct(pos[1]);
# possibility in the sense of the equation (8)
[ [ 1, 0, 0, -2, 1, -1, -1, -1, 2, 0, 2, 0, -1, 1, 1, -2, -2 ], 
  [ 0, 1, 0, 0, -1, -1, -1, 0, 2, -1, 0, 0, 1, 1, 1, -2, 0 ], 
  [ 0, 0, 1, -1, 0, 0, 0, 0, 0, 0, 1, -1, 0, 0, 0, 0, -2 ] ]
gap> List(last,x->x[Length(x)]);
[ -2, 0, -2 ]
gap> Gcd(last); # impossible in the sense of the equation (8) 
2

gap> I4inv2:=ConjugacyClassesSubgroups2TorusInvariants(I4g[2]);;
gap> I4inv3:=ConjugacyClassesSubgroups2TorusInvariants(I4g[3]);;
gap> pos:=PossibilityOfStablyEquivalentSubdirectProducts(I4g[2],I4g[3],I4inv2,I4inv3);;
gap> Length(pos); # the number of pos
1
gap> List([I4g[2],I4g[3]],StructureDescription);
[ "C5 : C4", "C2 x (C5 : C4)" ]
gap> List(pos,StructureDescription); # I4g[2]=F20, N4g[3]=pos[1]=C2xF20
[ "C2 x (C5 : C4)" ]
gap> PossibilityOfStablyEquivalentFSubdirectProduct(pos[1]);
# possibility in the sense of the equation (8)
[ [ 1, 0, 0, -2, 1, -1, -1, -1, 2, 0, 2, 0, -1, 1, 1, -2, -2 ], 
  [ 0, 1, 0, 0, -1, -1, -1, 0, 2, -1, 0, 0, 1, 1, 1, -2, 0 ], 
  [ 0, 0, 1, -1, 0, 0, 0, 0, 0, 0, 1, -1, 0, 0, 0, 0, -2 ] ]
ap> List(last,x->x[Length(x)]);
[ -2, 0, -2 ]
gap> Gcd(last); # impossible in the sense of the equation (8) 
2

gap> I4inv4:=ConjugacyClassesSubgroups2TorusInvariants(I4g[4]);;
gap> I4inv5:=ConjugacyClassesSubgroups2TorusInvariants(I4g[5]);;
gap> pos:=PossibilityOfStablyEquivalentSubdirectProducts(I4g[4],I4g[5],I4inv4,I4inv5);;
gap> Length(pos); # the number of pos
1
gap> List([I4g[4],I4g[5]],StructureDescription);
[ "S5", "S5" ]
gap> List(pos,StructureDescription); # I4g[4]=N4g[5]=pos[1]=S5
[ "S5" ]
gap> PossibilityOfStablyEquivalentFSubdirectProduct(pos[1]);
# possibility in the sense of the equation (8)
[ [ 1, 0, 0, -1, 0, -4, 0, 0, 2, -2, 1, 0, -1, -1, 0, 4, 4, 1, -4, 0 ], 
  [ 0, 1, 0, 0, 0, -1, -1, 0, 0, -1, 0, 0, 0, 0, 1, 1, 1, 0, -1, 0 ], 
  [ 0, 0, 1, 0, 0, -2, 0, 0, 1, -1, 0, 0, -1, -1, 0, 2, 2, 1, -2, 0 ], 
  [ 0, 0, 0, 0, 1, -2, 2, 0, 2, -2, 1, -2, -1, -2, -2, 2, 4, 1, -2, 0 ], 
  [ 0, 0, 0, 0, 0, 0, 0, 0, 0, 0, 0, 0, 0, 0, 0, 0, 0, 0, 0, 1 ] ]
gap> PossibilityOfStablyEquivalentFSubdirectProduct(pos[1]:H2);
# possibility in the sense of the equation (8) with H2 option
[ [ 1, 0, 0, -1, 0, -4, 0, 0, 2, -2, 1, 0, -1, -1, 0, 4, 4, 1, -4, 0 ], 
  [ 0, 1, 0, 0, 0, -1, -1, 0, 0, -1, 0, 0, 0, 0, 1, 1, 1, 0, -1, 0 ], 
  [ 0, 0, 1, 0, 0, -2, 0, 0, 1, -1, 0, 0, -1, -1, 0, 2, 2, 1, -2, 0 ], 
  [ 0, 0, 0, 0, 1, -2, 2, 0, 2, -2, 1, -2, -1, -2, -2, 2, 4, 1, -2, 0 ] ]
gap> List(last,x->x[Length(x)]); # impossible in the sense of the equation (8) 
[ 0, 0, 0, 0 ]

gap> I4inv4:=ConjugacyClassesSubgroups2TorusInvariants(I4g[4]);;
gap> I4inv6:=ConjugacyClassesSubgroups2TorusInvariants(I4g[6]);;
gap> pos:=PossibilityOfStablyEquivalentSubdirectProducts(I4g[4],I4g[6],I4inv4,I4inv6);;
gap> Length(pos); # the number of pos
1
gap> List([I4g[4],I4g[6]],StructureDescription);
[ "S5", "C2 x S5" ]
gap> List(pos,StructureDescription); # I4g[4]=S5, N4g[6]=pos[1]=C2xS5
[ "C2 x S5" ]
gap> PossibilityOfStablyEquivalentFSubdirectProduct(pos[1]);;
# possibility in the sense of the equation (8)
gap> List(last,x->x[Length(x)]);
[ -2, 0, 0, 0, -2, 0, -2, 0, 0, 0, 0, 0, 0, 0, -2, 0, 0, 0 ]
gap> Gcd(last); # impossible in the sense of the equation (8) 
2

gap> I4inv5:=ConjugacyClassesSubgroups2TorusInvariants(I4g[5]);;
gap> I4inv6:=ConjugacyClassesSubgroups2TorusInvariants(I4g[6]);;
gap> pos:=PossibilityOfStablyEquivalentSubdirectProducts(I4g[5],I4g[6],I4inv5,I4inv6);;
gap> Length(pos); # the number of pos
1
gap> List([I4g[5],I4g[6]],StructureDescription);
[ "S5", "C2 x S5" ]
gap> List(pos,StructureDescription); # I4g[5]=S5, N4g[6]=pos[1]=C2xS5
[ "C2 x S5" ]
gap> PossibilityOfStablyEquivalentFSubdirectProduct(pos[1]);;
# possibility in the sense of the equation (8)
gap> List(last,x->x[Length(x)]);
[ -2, 0, 0, 0, 0, -2, -2, 0, 0, 0, 0, 0, 0, 0, -2, 0, 0, 0 ]
gap> Gcd(last); # impossible in the sense of the equation (8) 
2
\end{verbatim}
\end{example}

\section{The $p$-Part of the flabby class $[M_G]^{fl}$ as a $\bZ_p[\Syl_p(G)]$-lattice}\label{S7}

Let $G$ be a finite group. 
Let $p$ be a prime number, 
$\bZ_{(p)}$ be the localization of $\bZ$ at the prime ideal $(p)$ 
and $\bZ_p$ be the ring of $p$-adic integers. 
It is known that 
Krull-Schmidt-Azumaya theorem holds for $\bZ_p[G]$-lattices, 
i.e. $\bZ_p[G]$-lattices have the unique decomposition into 
indecomposable ones up to isomorphism and numbering 
(see Azumaya \cite[Theorem 1]{Azu50}, Borevich and Faddeev \cite{BF59}, 
Swan \cite[Proposition 6.1 and Remark]{Swa60}, 
Reiner \cite[Theorem 1]{Rei61}, 
Curtis and Reiner \cite[Theorem 6.12]{CR81}, 
see also Reiner \cite[Section 5]{Rei70}) 
although it does not hold for $G$-lattices in general 
(see Swan \cite[Section 10]{Swa60}, Curtis and Reiner \cite[Section 36]{CR81}, 
Hoshi and Yamasaki \cite[page 9 and Chapter 4]{HY17}). 
If $G$ is a $p$-group with $p\neq 2$, then 
Krull-Schmidt-Azumaya theorem holds for $\bZ_{(p)}[G]$-lattices 
(see Jones \cite[{Theorem 2 with $q=1$}]{Jon65}, 
Curtis and Reiner \cite[Theorem 36.1]{CR81}). 
It follows from Swan's theorem \cite[Theorem 8.1]{Swa60} that 
the $G$-lattice 
$\bZ[G/H]$ is indecomposable for any $H\leq G$ 
(see Curtis and Reiner \cite[Theorem 32.14]{CR81}). 
By Jacobinski \cite[Theorem 3.3 and Corollary 3.5]{Jac68}, 
for $p$-group $G$ and $G$-lattice $M$, 
$M$ is an indecomposable $G$-lattice 
if and only if 
$M_p=M\otimes_\bZ\bZ_p$ (resp. $M_{(p)}=M\otimes_\bZ\bZ_{(p)}$) 
is an indecomposable $\bZ_p[G]$-lattice (resp. $\bZ_{(p)}[G]$-lattice). 
In particular, if $G$ is a $p$-group, then 
the $\bZ_p[G]$-lattice $\bZ_p[G/H]$ 
(resp. $\bZ_{(p)}[G]$-lattice $\bZ_{(p)}[G/H]$) 
is indecomposable for any $H\leq G$. 

\begin{definition}\label{def7.1}
Let $G$ be a finite group, 
$M$ be a $G$-lattice and  
\begin{align*}
0 \rightarrow M\rightarrow P\rightarrow F\rightarrow 0
\end{align*}
be a flabby resolution of $M$ where $P$ is permutation and $F$ is flabby. 

(1) ({Kunyavskii \cite[page 15]{Kun90}}) 
Take a $\bZ_2[G]$-lattice 
$\widetilde{M}=M\otimes_\bZ\bZ_2$. 
By tensoring with $\bZ_2$, 
we also get a flabby resolution of $\widetilde{M}$: 
\begin{align*}
0 \rightarrow \widetilde{M} \rightarrow \widetilde{P} \rightarrow \widetilde{F} \rightarrow 0
\end{align*}
as $\bZ_2[G]$-lattices 
where $\widetilde{P}=P\otimes_\bZ\bZ_2$ is permutation 
and $\widetilde{F}=F\otimes_\bZ\bZ_2$ is flabby. 
We can take a direct sum decomposition 
$\widetilde{F}\simeq \widetilde{N}\oplus \widetilde{Q}$ where 
$\widetilde{N}$ is almost indecomposable, i.e. 
does not contain a permutation direct summand, 
and $\widetilde{Q}$ is permutation. 
We call $\widetilde{N}$ {\it an almost indecomposable part of $[M]^{fl}$ 
as a $\bZ_2[G]$-lattice}. 

(2) 
Let $\Syl_p(G)$ be a $p$-Sylow subgroup of $G$. 
Then we obtain a flabby resolution of $M\mid_{\Syl_p(G)}$ 
which is a $\Syl_p(G)$-lattice obtained 
by restricting the action of $G$ on $M$ to $\Syl_p(G)$: 
\begin{align*}
0 \rightarrow M\mid_{\Syl_p(G)}\rightarrow P\mid_{\Syl_p(G)}
\rightarrow F\mid_{\Syl_p(G)}\rightarrow 0. 
\end{align*}
We also obtain a flabby resolution of $M_p=M\mid_{\Syl_p(G)}$: 
\begin{align*}
0 \rightarrow M_p\rightarrow P_p\rightarrow F_p\rightarrow 0
\end{align*}
as $\bZ[\Syl_p(G)]$-lattices 
where $P_p$ is permutation and $F_p$ is flabby. 
Then we see that $[F_p]=[F\mid_{\Syl_p(G)}]$ 
and 
we can take the direct sum decomposition 
$\widetilde{F}_p=F_p\otimes_\bZ \bZ_p\simeq 
\widetilde{N}_p \oplus \widetilde{Q}_p$ 
as $\bZ_p[\Syl_p(G)]$-lattices 
where $\widetilde{N}_p$ does not contain a permutation direct summand 
and $\widetilde{Q}_p$ is permutation. 
We call $\widetilde{N}_p$ {\it the $p$-part of $[M]^{fl}$ 
as a $\bZ_p[\Syl_p(G)]$-lattice}. 
\end{definition}
\begin{remark}\label{rem7.2}
(1) The $p$-part $\widetilde{N}_p$ of $[M]^{fl}$ 
as $\bZ_p[\Syl_p(G)]$-lattices 
is uniquely determined 
(see the first paragraph of this section) 
and does not depend on a choice of $F$ with $[F]=[M]^{fl}$.\\ 
(2) If $G$ is a $2$-group, then we find that 
$\widetilde{N}\simeq \widetilde{N}_2$ as $\bZ_2[G]$-lattices. 
But an almost indecomposable part $\widetilde{N}$ of $[M]^{fl}$ 
is not uniquely determined in general when $G$ is not a $2$-group 
(see Example \ref{ex7.3}). 
\end{remark}
\begin{example}\label{ex7.3}
Let $G=I_{4,1}\leq {\rm GL}(4,\bZ)$ be as in Definition \ref{defN3N4} 
with $G\simeq F_{20}$ 
and $M_G$ be the corresponding $G$-lattice with rank $M_G=4$ 
as in Definition \ref{d2.2}. 
We take $\widetilde{M}=M_G\otimes \bZ_2$ where $\bZ_2$ is the ring of $2$-adic integers. 
Then we see that 
$\widetilde{M}$ is an indecomposable $\bZ_2[G]$-lattice and 
$\widetilde{M}\oplus \bZ_2\simeq \bZ_2[G/C_4]$ as $\bZ_2[G]$-lattices. 
In particular, $\widetilde{M}^{\oplus r}$ is almost indecomposable for any $r\geq 1$. 
We get that $F=[M_G]^{fl}$ with rank $F=16$ 
and $\widetilde{F}=F\otimes_\bZ \bZ_2\simeq 
\widetilde{M}\oplus \widetilde{M}\oplus \widetilde{M}\oplus \bZ_2[G/C_5]$ as $\bZ_2[G]$-lattices 
and $\widetilde{N}\simeq \widetilde{M}\oplus \widetilde{M}\oplus \widetilde{M}$ is an 
almost indecomposable part of $[M]^{fl}$. 
On the other hand, we also see that 
$\widetilde{F}\oplus\bZ_2\oplus\bZ_2\oplus\bZ_2\simeq 
\bZ_2[G/C_4]\oplus\bZ_2[G/C_4]\oplus\bZ_2[G/C_4]\oplus\bZ_2[G/C_5]$ 
and hence $\widetilde{N}=0$ is an 
almost indecomposable part of $[M]^{fl}$. 
Hence we see that an 
almost indecomposable part of $[M]^{fl}$ is not uniquely determined 
in general when $G$ is not a $2$-group 
(see Example \ref{ex7.3b} for GAP computations). 
\end{example}

We immediately obtain the following lemma
from the definitions of the weak stably $k$-equivalence 
and of the $p$-part of $[M]^{fl}$ as a $\bZ_p[\Syl_p(G)]$-lattice 
(Definition \ref{d1.10} and Definition \ref{def7.1}). 
\begin{lemma}\label{lem7.2b}
Let $M$ $($resp. $M^\prime$$)$ be a $G$-lattice $($resp. $G^\prime$-lattice$)$. 
If $[M]^{fl}\sim [M^\prime]^{fl}$, 
i.e. there exists a subdirect product $\widetilde{H}\leq G\times G^\prime$ 
of $G$ and $G^\prime$ which acts on $M$ $($resp. $M^\prime)$ 
through the surjection 
$\varphi_1:\widetilde{H}\rightarrow G$ 
$($resp. $\varphi_2:\widetilde{H}\rightarrow G^\prime$$)$ 
such that $[M]^{fl}=[M^\prime]^{fl}$ as $\widetilde{H}$-lattices, 
then 
$\widetilde{N}_p\simeq \widetilde{N}^\prime_p$ 
as $\bZ_p[\Syl_p(\widetilde{H})]$-lattices 
where $\widetilde{N}_p$ $($resp. $\widetilde{N}^\prime_p$$)$ 
is the $p$-part of $[M]^{fl}$ $($resp. $[M^\prime]^{fl}$$)$ 
as a $\bZ_p[\varphi_1(\Syl_p(\widetilde{H}))]$-lattice 
$($resp. $\bZ_p[\varphi_2(\Syl_p(\widetilde{H}))]$-lattice$)$ 
with $\bZ_p[\varphi_1(\Syl_p(\widetilde{H}))]\simeq \bZ_p[\Syl_p(G)]$ 
$($resp. 
$\bZ_p[\varphi_2(\Syl_p(\widetilde{H}))]\simeq \bZ_p[\Syl_p(G^\prime)]$$)$. 
In particular, the indecomposability and 
the $\bZ_p$-rank of $\widetilde{N}_p$ are invariants for the 
weak stably $k$-equivalent classes of $[M]^{fl}$. 
\end{lemma}

Retract non-rationality of an algebraic $k$-torus $T$ 
can be detected by the non-vanishing of $\widetilde{N}_p$:

\begin{lemma}\label{lem7.3b}
Let $M$ be a $G$-lattice and 
$\widetilde{N}_p$ be the $p$-part of $[M]^{fl}$ as a $\bZ_p[\Syl_p(G)]$-lattice. 
If $\widetilde{N}_p\neq 0$, then $[M]^{fl}$ is not invertible. 
In particular, the corresponding torus $T$ 
with $\widehat{T}=M$ is not retract $k$-rational. 
\end{lemma}
\begin{proof}
Suppose that $[F]=[M]^{fl}$ is invertible. 
Then there exists $F^\prime$ such that $F\oplus F^\prime$ 
is permutation. 
This implies that 
$[F\oplus F^\prime]=[F]\oplus [F^\prime]=0$. 
Taking a $p$-Sylow subgroup and tensoring $\bZ_p$, 
we get $[\widetilde{N}_p]\oplus [\widetilde{F^\prime}_p]=0$ 
as $\bZ_p[\Syl_p(G)]$-lattices. 
Because Krull-Schmidt-Azumaya theorem holds for $\bZ_p[G]$-lattices 
and $\bZ_p[\Syl_p(G)/H]$ is indecomposable for any $H\leq \Syl_p(G)$ 
(see the beginning of this section), 
we find that $\widetilde{N}_p$ is permutation. 
Contradiction. 
\end{proof}

For a $G$-lattice $M$ with $[F]=[M]^{fl}$, 
if $[F]\neq 0$ and $\widetilde{F}_p=F \otimes_\bZ \bZ_p$
is an indecomposable $\bZ_p[\Syl_p(G)]$-lattice, 
then $\widetilde{N}_p=\widetilde{F}_p$. 
In order to deal with indecomposable $G$-lattices, 
we recall some basic terminology and lemmas in ring theory 
(see e.g. Karpilovsky \cite[Theorem 7.5 and Corollary 7.6]{Kar87} for Krull-Schmidt-Azumaya theorem for $R$-modules of finite length). 
\begin{definition}[{Karpilovsky \cite[Chapter 1, page 22, page 36]{Kar87}}]
Let $R$ be a ring with $1$ and $V$ be a left $R$-module.\\
{\rm (1)} 
The {\it Jacobson radical} $J(V)$ of $V$ 
is the intersection of all maximal $R$-submodules of $V$.
Because $R$ itself is a left $R$-module, 
we can take $J(R)$ and call it {\it the Jacobson radical of $R$}.\\
{\rm (2)} 
A ring $R$ is said to be {\it local} if $R/J(R)$ is a division ring.
\end{definition}

\begin{lemma}[{Karpilovsky \cite[Chapter 1, Corollary 6.17]{Kar87}}]
\label{lem7.5}
Let $R$ be a ring with $1$ and $I$ be an ideal of $R$.\\
{\rm (1)}
$J(R/I) \supset (J(R)+I)/I$.\\
{\rm (2)}
If $I \subset J(R)$,
then $J(R/I)=J(R)/I$.
In particular, $J(R/J(R))=0$.\\
{\rm (3)}
If $J(R/I)=0$, then $J(R) \subset I$.\\
{\rm (4)}
$I=J(R)$ if and only if $I \subset J(R)$ and $J(R/I)=0$.
\end{lemma}

\begin{lemma}[{Karpilovsky \cite[Chapter 1, Proposition 6.18, Proposition 6.19, Corollary 6.21, Proposition 6.22]{Kar87}}]\label{lem7.6}
Let $R$ be a ring with $1$.\\
{\rm (1)} 
Let $x$ be an element of $R$. 
Then $x \in J(R)$ if and only if for all $r \in R$, $1-rx$ is a left unit. 
In particular, $J(R)$ contains no nonzero idempotents.\\ 
{\rm (2)}
$J(R)$ is the unique largest ideal $I$ of $R$
such that $1-rx$ is a unit for all $r \in R, x \in I$.\\
{\rm (3)} If $I$ is left ideal of $R$ whose elements are all nilpotent, 
then $I\subset J(R)$.\\
{\rm (4)} If $R$ is artinian, then $J(R)$ is nilpotent.
\end{lemma}

\begin{lemma}[{Karpilovsky \cite[Chapter 1, Lemma 7.1]{Kar87}}]
\label{lem7.7}
Let $R$ be a ring with $1$ 
and $U(R)$ be the group of units of $R$.
Then the following conditions are equivalent:\\
{\rm (1)} $R$ is local;\\
{\rm (2)} $R$ has a  unique maximal left ideal;\\
{\rm (3)} $J(R)=R\setminus U(R)$;\\
{\rm (4)} $R\setminus U(R)$ is a left ideal of $R$.
\end{lemma}

\begin{lemma}[{Karpilovsky \cite[Chapter 1, Corollary 7.2, Lemma 7.3]{Kar87}}]
\label{lem7.8}
Let $R$ be a ring with $1$ and $V$ be an $R$-module.\\
{\rm (1)} If $R$ is local, then the only idempotents of $R$ are $0$ and $1$.\\
{\rm (2)} $V$ is indecomposable if and only if the only idempotents of 
$\End_R(V)$ are $0$ and $1$. 
In particular, if $\End_R(V)$ is local, then $V$ is indecomposable. 
\end{lemma}

\begin{lemma}[{Karpilovsky \cite[Chapter 1, Lemma 7.4]{Kar87}}]
\label{lem7.9}
Let $R$ be a ring with $1$ 
and $V$ be an $R$-module of finite length.\\
{\rm (1)} If $V$ is indecomposable, then every $f \in \End_R(V)$ is
either a unit or nilpotent.\\
{\rm (2)} $V$ is indecomposable if and only if $\End_R(V)$ is a local ring.
\end{lemma}

We will apply the following lemmas which are useful to certify 
the indecomposability of a $G$-lattice. 

\begin{lemma}\label{lem7.10}
Let $G$ be a finite group and $M$ be a $G$-lattice. 
If $M \otimes_\bZ \bZ/p\bZ$ is an indecomposable 
$(\bZ/p\bZ)[G]$-lattice, then 
$\widetilde{M}_p=M \otimes_\bZ \bZ_p$ 
is an indecomposable $\bZ_p[G]$-lattice. 
In particular, $M$ is an indecomposable $G$-lattice. 
\end{lemma}
\begin{proof}
We may regard $\bZ_p/p\bZ_p$ as $\bZ/p\bZ$ and hence 
$\widetilde{M}_p$ is an indecomposable $\bZ_p[G]$-lattice. 
The last statement follows from $\bZ\subset \bZ_p$. 
\end{proof}

\begin{lemma}\label{lem7.11}
Let $G$ be a finite group and $M$ be a $G$-lattice. 
Let $E=\End_{\bZ[G]}(M)$ be the endomorphism ring of $M$. 
Let $p$ be a prime number. 
If there exists a nilpotent left ideal $I$ of $E/pE$ 
with codimension $1$ as a $(\bZ/p\bZ)$-vector space, 
then $\widetilde{M}_p=M \otimes_\bZ \bZ_p$ 
is an indecomposable $\bZ_p[G]$-lattice 
and hence $M$ is an indecomposable $G$-lattice. 
In particular, if we take $M=\bZ[G]$ for a $p$-group $G$, 
then $\widetilde{M}_p=\bZ_p[G]$ is an indecomposable $\bZ_p[G]$-lattice 
and hence $\bZ[G]$ is an indecomposable $G$-lattice. 
\end{lemma}
\begin{proof}
We take $\widetilde{E}=\End_{\bZ_p[G]}(M\otimes\bZ_p)\supset E$. 
For $x \in p\widetilde{E}$, $1-rx$ is a unit for all $r\in \widetilde{E}$ 
because $(1-rx)^{-1}=\sum_{i=0}^{\infty}(rx)^i$. 
It follows from Lemma \ref{lem7.6} (1) that 
$p\widetilde{E} \subset J(\widetilde{E})$. 
By Lemma \ref{lem7.5} (2), we have 
$J(\widetilde{E}/p\widetilde{E})=J(\widetilde{E})/p\widetilde{E}$. 
Then we investigate $\widetilde{E}/p\widetilde{E}\simeq 
\End_{(\bZ/p\bZ)[G]}(M \otimes_\bZ \bZ/p\bZ)$ 
instead of $\widetilde{E}$. 
Note that $\widetilde{E}/p\widetilde{E}$ is a finite dimensional vector space 
over $\bZ/p\bZ$. 
From the assumption, we have a nilpotent left ideal 
$\widetilde{I}=I\otimes_\bZ \bZ_p$ of $\widetilde{E}/p\widetilde{E}$ 
with codimension $1$ as a $(\bZ/p\bZ)$-vector space, 
then it follows from Lemma \ref{lem7.6} (3) and Lemma \ref{lem7.7} that 
$\widetilde{I}=J(\widetilde{E}/p\widetilde{E})$ 
and $\widetilde{E}/p\widetilde{E}$ is a local ring. 
By Lemma \ref{lem7.9} (2), we find that 
$M \otimes_\bZ \bZ/p\bZ$ is an indecomposable 
$(\bZ/p\bZ)[G]$-lattice. 
Hence the statement follows from Lemma \ref{lem7.10}. 

The last statement follows from that 
$E=\End_{\bZ[G]}(\bZ[G])\simeq \bZ[G]$, 
$E/pE\simeq \widetilde{E}/p\widetilde{E}\simeq (\bZ/p\bZ)[G]$ 
and $\widetilde{E}\simeq \bZ_p[G]$ is a local ring with the maximal ideal 
$\widetilde{I}_G+p\bZ_p[G]$ where 
$\widetilde{I}_G={\rm Ker}(\varepsilon)$ and 
$\varepsilon : \bZ_p[G]\rightarrow \bZ_p$ is the augmentation map 
(Shizuo Endo pointed out this to the authors). 
\end{proof}

We made the following GAP \cite{GAP} algorithms 
which enable us to confirm the indecomposability of $[M_G]^{fl}$ 
and to get the $p$-part $\widetilde{N}_p$ of $[M_G]^{fl}$ 
as a $\bZ_p[\Syl_p(G)]$-lattice as in Definition \ref{def7.1}. 
It is available as in \cite{BCAlgTori}.\\

Let $G\leq \GL(n,\bZ)$ 
and $M_G$ be the corresponding $G$-lattice of $\bZ$-rank $n$ 
as in Definition \ref{d2.2}.\\

\noindent 
{\tt Endomorphismring(}$G${\tt )} returns a $\bZ$-basis of
$\End_{\bZ[G]}(M_G)$ for a finite subgroup $G$ of $\GL(n,\bZ)$.\\

\noindent 
{\tt IsCodimJacobsonEnd1(}$G,p${\tt )} returns {\tt true} (resp. {\tt false}) 
if ${\rm dim}_{\bZ/p\bZ} (E/pE)\big/J(E/pE)=1$ (resp. $\neq 1$) 
where $E=\End_{\bZ[G]}(M_G)$ 
for a finite subgroup $G$ of $\GL(n,\bZ)$ and prime number $p$. 
If this returns {\tt true}, then $M_G \otimes_\bZ \bZ_p$ 
is an indecomposable $\bZ_p[G]$-lattice. 
In particular, $M_G$ is an indecomposable $G$-lattice 
(see Lemma \ref{lem7.11}).\\

\noindent 
{\tt IdempotentsModp(}$B,p${\tt )} returns all idempotents of $R/pR$ 
for a $\bZ$-basis $B$ of a subring $R$ of $n\times n$ matrices 
$M(n,\bZ)$ over $\bZ$ and prime number $p$. 
If this returns only 
the zero and the identity matrices when $R=\End_{\bZ[G]}(M_G)$, then 
$M_G \otimes_\bZ \bZ_p$ 
is an indecomposable $\bZ_p[G]$-lattice. 
In particular, $M_G$ is an indecomposable $G$-lattice 
(see Lemma \ref{lem7.10}).\\
\begin{theorem}[{see Kunyavskii \cite[Section 5]{Kun90}}]\label{th7.12}
Let $G=N_{3,i}$ $(1\leq i\leq 15)$ be groups as in Definition \ref{defN3N4} 
and $M_G$ be the corresponding $G$-lattice as in Definition \ref{d2.2}. 
The $2$-part $\widetilde{N}_2$ of $[M_G]^{fl}$ as a $\bZ_p[\Syl_p(G)]$-lattice 
is a faithful and indecomposable $\bZ_2[\Syl_2(G)]$-lattice 
and the $\bZ_2$-rank of $\widetilde{N}_2$ 
is given as in {\rm Table} $14$. 
\end{theorem}
\begin{proof}
This theorem is due to Kunyavskii \cite[Section 5]{Kun90} 
except for the indecomposability of $\widetilde{N}_2$ 
(Kunyavskii actually 
showed that the almost indecomposability of $\widetilde{N}_2$). 
We should take a $2$-Sylow subgroup $\Syl_2(G)$ of $G=N_{3,i}$ 
and $7$ cases of them with $1\leq i\leq 7$ are $2$-groups. 
We may check the theorem by applying the GAP algorithms 
{\tt FlabbyResolutionLowRank($G$).actionF} (see Section \ref{S6}), 
{\tt Endomorphismring} and {\tt IsCodimJacobsonEnd1} 
(see Example \ref{ex7.14} for GAP computations). 
\end{proof}

Table $14$: 
The $\bZ_2$-rank of $\widetilde{N}_2$ for $G=N_{3,i}$ $(1\leq i\leq 15)$ with $\Syl_2(G)=N_{3,j}$
{\small 
\begin{longtable}{c|rrrrrrrrrrrrrrr}
$i$: $G=N_{3,i}$ & $1$ & $2$ & $3$ & $4$ & $5$ & $6$ & $7$ & $8$ & $9$ & $10$ & $11$ & $12$ & $13$ & $14$ & $15$\\\hline
\vspace*{-3mm}\\
${\rm rank}_{\bZ_2}\, \widetilde{N}_2$ & $5$ & $9$ & $11$ & $11$ & $7$ & $7$ & $11$ & $5$ & $9$ & $11$ & $7$ & $7$ & $7$ & $11$ & $11$\\
\vspace*{-3mm}\\
$j$: $\Syl_2(G)=N_{3,j}$ & $i$ & $i$ & $i$ & $i$ & $i$ & $i$ & $i$ & $1$ & $2$ & $3$ & $5$ & $5$ & $6$ & $7$ & $7$\\
\vspace*{-3mm}\\
${\rm rank}_{\bZ_3}\, \widetilde{N}_3$ &  &  &  &  &  &  &  & $0$ & $0$ & $0$ & $0$ & $0$ & $0$ & $0$ & $0$\\
\vspace*{-3mm}\\
$j$: $\Syl_3(G)=N_{3,j}$ &  &  &  &  &  &  &  & -- & -- & -- & -- & -- & -- & -- & --
\end{longtable}
}

\begin{theorem}\label{th7.13}
Let $G=N_{4,i}$ $(1\leq i\leq 152)$ be groups as in Definition \ref{defN3N4} 
and $M_G$ be the corresponding $G$-lattice as in Definition \ref{d2.2}. 
The $2$-part $\widetilde{N}_2$ $($resp. $3$-part $\widetilde{N}_3$$)$ 
of $[M_G]^{fl}$ as a $\bZ_p[\Syl_p(G)]$-lattice 
is a faithful and indecomposable $\bZ_2[\Syl_2(G)]$-lattice 
$($resp. $\bZ_3[\Syl_3(G)]$-lattice$)$ unless it vanishes 
and the $\bZ_2$-rank of 
$\widetilde{N}_2$ $($resp. the $\bZ_3$-rank of $\widetilde{N}_3$$)$ 
is given as in {\rm Table} $15$. 
\end{theorem}
\begin{proof}
As in the proof of Theorem \ref{th7.12}, 
we apply the GAP algorithms 
{\tt FlabbyResolutionLowRank($G$).actionF} (see Section \ref{S6}), 
{\tt Endomorphismring} and  
{\tt IsCodimJacobsonEnd1}. 
We should take a $p$-Sylow subgroup $\Syl_p(G)$ $(p=2,3)$ of $G=N_{4,i}$, 
and $73$ cases of them with $1\leq i\leq 56$, $106\leq i\leq 109$, 
$112\leq i\leq 117$, $120\leq i\leq 125$, $i=128$ are $2$-groups 
and the only $1$ case of them with $i=57$ is a $3$-group. 
Furthermore, 
we may check that there exists 
no integer $j$ such that $\Syl_p(G)$ is $\GL(4,\bZ)$-conjugate 
to $N_{31,j}$ for $p=2,3$. 
Then for $2$-groups and a $3$-group cases of $G$ 
we obtain that $\widetilde{N}_2$ and $\widetilde{N}_3$ are zero or 
indecomposable except for $\widetilde{N}_2$ of $G=N_{4,i}$ where 
$i\in \{2, 7, 13, 16, 17, 19, 24, 25, 30, 106\}$. 
By Theorem \ref{thmain3} (see also Table $4$), 
we already see the following weak stably $k$-equivalences
\begin{align*}
&[N_{4,2}]^{fl}\sim [N_{4,7}]^{fl}\sim [N_{4,5}]^{fl},\\
&[N_{4,13}]^{fl}\sim [N_{4,24}]^{fl}\sim [N_{4,23}]^{fl},\\
&[N_{4,16}]^{fl}\sim [N_{4,17}]^{fl}\sim [N_{4,25}]^{fl}\sim [N_{4,15}]^{fl},\\
&[N_{4,19}]^{fl}\sim [N_{4,30}]^{fl}\sim [N_{4,29}]^{fl}. 
\end{align*}
Then apply Lemma \ref{lem7.2b}. 
We find that $\widetilde{N}_2$ is indecomposable except for 
the last case $i=106$. 
For $G=N_{4,106}\simeq Q_8$, by using the function 
{\tt IdempotentsModp(EndomorphismRingBasis($F$),2)} 
where $F$ is obtained via {\tt FlabbyResolutionLowRank($G$).actionF} 
(see Section \ref{S6}) with $[F]=[M_G]^{fl}$ and ${\rm rank}_\bZ$ $F=12$, 
we confirm that $\widetilde{F}_2$ is an indecomposable $\bZ_2[G]$-lattice 
with ${\rm rank}_{\bZ_2}$ $\widetilde{F}_2=12$
and hence $\widetilde{N}_2=\widetilde{F}_2$. 
Finally, we may also confirm that the $p$-part $\widetilde{N}_p$ $(p=2,3)$ 
of $[M_G]^{fl}$ as a $\bZ_p[\Syl_p(G)]$-lattice is a faithful 
$\bZ_p[G]$-lattice (see Example \ref{ex7.15} for GAP computations). 
\end{proof}

\newpage
Table $15$: 
The $\bZ_p$-rank of $\widetilde{N}_p$ $(p=2,3)$ for $G=N_{4,i}$ $(1\leq i\leq 152)$ with $\Syl_p(G)=N_{4,j}$ $(1\leq j\leq 152)$
{\small 
\begin{longtable}{c|rrrrrrrrrrrrrrrrrrrr}
$i$: $G=N_{4,i}$ & $1$ & $2$ & $3$ & $4$ & $5$ & $6$ & $7$ & $8$ & $9$ & $10$ & $11$ & $12$ & $13$ & $14$ & $15$ & $16$ & $17$ & $18$ & $19$ & $20$\\\hline
\vspace*{-3mm}\\
${\rm rank}_{\bZ_2}\, \widetilde{N}_2$ & $5$ & $11$ & $12$ & $20$ & $11$ & $9$ & $11$ & $12$ & $16$ & $20$ & $20$ & $20$ & $11$ & $12$ & $7$ & $7$ & $7$ & $8$ & $11$ & $20$\\
\vspace*{-3mm}\\
$j$: $\Syl_2(G)=N_{4,j}$ & $i$ & $i$ & $i$ & $i$ & $i$ & $i$ & $i$ & $i$ & $i$ & $i$ & $i$ & $i$ & $i$ & $i$ & $i$ & $i$ & $i$ & $i$ & $i$ & $i$\\
\vspace*{-3mm}\\
${\rm rank}_{\bZ_3}\, \widetilde{N}_3$ & $$ & $$ & $$ & $$ & $$ & $$ & $$ & $$ & $$ & $$ & $$ & $$ & $$ & $$ & $$\\
\vspace*{-3mm}\\
$j$: $\Syl_3(G)=N_{4,j}$ & $$ & $$ & $$ & $$ & $$ & $$ & $$ & $$ & $$ & $$ & $$ & $$ & $$ & $$ & $$
\end{longtable}

\begin{longtable}{rrrrrrrrrrrrrrrrrrrrr}
& $21$ & $22$ & $23$ & $24$ & $25$ & $26$ & $27$ & $28$ & $29$ & $30$ & $31$ & $32$ & $33$ & $34$ & $35$ & $36$ & $37$ & $38$ & $39$ & $40$\\\hline
\vspace*{-3mm}\\
& $12$ & $20$ & $11$ & $11$ & $7$ & $7$ & $20$ & $20$ & $11$ & $11$ & $20$ & $16$ & $20$ & $16$ & $20$ & $20$ & $20$ & $20$ & $12$ & $12$\\
\vspace*{-3mm}\\
& $i$ & $i$ & $i$ & $i$ & $i$ & $i$ & $i$ & $i$ & $i$ & $i$ & $i$ & $i$ & $i$ & $i$ & $i$ & $i$ & $i$ & $i$ & $i$ & $i$\\
\vspace*{-3mm}\\
& $$ & $$ & $$ & $$ & $$ & $$ & $$ & $$ & $$ & $$ & $$ & $$ & $$ & $$ & $$ & $$ & $$ & $$ & $$ & $$\\
\vspace*{-3mm}\\
& $$ & $$ & $$ & $$ & $$ & $$ & $$ & $$ & $$ & $$ & $$ & $$ & $$ & $$ & $$ & $$ & $$ & $$ & $$ & $$
\vspace*{5mm}\\
& $41$ & $42$ & $43$ & $44$ & $45$ & $46$ & $47$ & $48$ & $49$ & $50$ & $51$ & $52$ & $53$ & $54$ & $55$ & $56$ & $57$ & $58$ & $59$ & $60$\\\hline
\vspace*{-3mm}\\
& $20$ & $20$ & $20$ & $20$ & $20$ & $20$ & $20$ & $20$ & $20$ & $20$ & $20$ & $20$ & $20$ & $20$ & $20$ & $20$ &  & $0$ & $0$ & $0$\\
\vspace*{-3mm}\\
& $i$ & $i$ & $i$ & $i$ & $i$ & $i$ & $i$ & $i$ & $i$ & $i$ & $i$ & $i$ & $i$ & $i$ & $i$ & $i$ &  & -- & -- & --\\
\vspace*{-3mm}\\
& $$ & $$ & $$ & $$ & $$ & $$ & $$ & $$ & $$ & $$ & $$ & $$ & $$ & $$ & $$ & $$ & $11$ & $11$ & $11$ & $11$\\
\vspace*{-3mm}\\
& $$ & $$ & $$ & $$ & $$ & $$ & $$ & $$ & $$ & $$ & $$ & $$ & $$ & $$ & $$ & $$ & $i$ & $57$ & $57$ & $57$
\vspace*{5mm}\\
& $61$ & $62$ & $63$ & $64$ & $65$ & $66$ & $67$ & $68$ & $69$ & $70$ & $71$ & $72$ & $73$ & $74$ & $75$ & $76$ & $77$ & $78$ & $79$ & $80$\\\hline
\vspace*{-3mm}\\
& $0$ & $0$ & $0$ & $0$ & $0$ & $0$ & $0$ & $0$ & $11$ & $20$ & $7$ & $11$ & $20$ & $11$ & $11$ & $16$ & $7$ & $7$ & $20$ & $20$\\
\vspace*{-3mm}\\
& -- & -- & -- & -- & -- & -- & -- & -- & $2$ & $4$ & $16$ & $19$ & $22$ & $5$ & $7$ & $9$ & $26$ & $25$ & $10$ & $12$\\
\vspace*{-3mm}\\
& $11$ & $11$ & $11$ & $11$ & $11$ & $11$ & $11$ & $11$ & $0$ & $0$ & $0$ & $0$ & $0$ & $0$ & $0$ & $0$ & $0$ & $0$ & $0$ & $0$\\
\vspace*{-3mm}\\
& $57$ & $57$ & $57$ & $57$ & $57$ & $57$ & $57$ & $57$ & -- & -- & -- & -- & -- & -- & -- & -- & -- & -- & -- & --
\vspace*{5mm}\\
& $81$ & $82$ & $83$ & $84$ & $85$ & $86$ & $87$ & $88$ & $89$ & $90$ & $91$ & $92$ & $93$ & $94$ & $95$ & $96$ & $97$ & $98$ & $99$ & $100$\\\hline
\vspace*{-3mm}\\
& $20$ & $20$ & $11$ & $20$ & $11$ & $16$ & $20$ & $16$ & $20$ & $20$ & $0$ & $0$ & $0$ & $0$ & $0$ & $0$ & $0$ & $0$ & $0$ & $0$\\
\vspace*{-3mm}\\
& $35$ & $36$ & $29$ & $31$ & $30$ & $32$ & $33$ & $34$ & $37$ & $38$ & -- & -- & -- & -- & -- & -- & -- & -- & -- & --\\
\vspace*{-3mm}\\
& $0$ & $0$ & $0$ & $0$ & $0$ & $0$ & $0$ & $0$ & $0$ & $0$ & $11$ & $11$ & $11$ & $11$ & $11$ & $11$ & $11$ & $11$ & $11$ & $11$\\
\vspace*{-3mm}\\
& -- & -- & -- & -- & -- & -- & -- & -- & -- & -- & $57$ & $57$ & $57$ & $57$ & $57$ & $57$ & $57$ & $57$ & $57$ & $57$
\vspace*{5mm}\\
& $101$ & $102$ & $103$ & $104$ & $105$ & $106$ & $107$ & $108$ & $109$ & $110$ & $111$ & $112$ & $113$ & $114$ & $115$ & $116$ & $117$ & $118$ & $119$ & $120$\\\hline
\vspace*{-3mm}\\
& $0$ & $0$ & $0$ & $0$ & $0$ & $12$ & $20$ & $12$ & $20$ & $12$ & $12$ & $20$ & $20$ & $20$ & $20$ & $20$ & $20$ & $12$ & $12$ & $20$\\
\vspace*{-3mm}\\
& -- & -- & -- & -- & -- & $i$ & $i$ & $i$ & $i$ & $106$ & $106$ & $i$ & $i$ & $i$ & $i$ & $i$ & $i$ & $108$ & $108$ & $i$\\
\vspace*{-3mm}\\
& $11$ & $11$ & $11$ & $11$ & $11$ &  &  &  &  & $0$ & $0$ &  &  &  &  &  &  & $0$ & $0$ &  \\
\vspace*{-3mm}\\
& $57$ & $57$ & $57$ & $57$ & $57$ &  &  &  &  & -- & -- &  &  &  &  &  &  & -- & -- & 
\vspace*{5mm}\\
& $121$ & $122$ & $123$ & $124$ & $125$ & $126$ & $127$ & $128$ & $129$ & $130$ & $131$ & $132$ & $133$ & $134$ & $135$ & $136$ & $137$ & $138$ & $139$ & $140$\\\hline
\vspace*{-3mm}\\
& $20$ & $20$ & $20$ & $20$ & $20$ & $20$ & $20$ & $20$ & $20$ & $20$ & $20$ & $20$ & $20$ & $20$ & $20$ & $20$ & $12$ & $12$ & $12$ & $20$\\
\vspace*{-3mm}\\
& $i$ & $i$ & $i$ & $i$ & $i$ & $117$ & $117$ & $i$ & $123$ & $124$ & $123$ & $124$ & $120$ & $120$ & $128$ & $128$ & $106$ & $106$ & $108$ & $109$\\
\vspace*{-3mm}\\
&  &  &  &  &  & $0$ & $0$ &  & $0$ & $0$ & $0$ & $0$ & $0$ & $0$ & $0$ & $0$ & $0$ & $0$ & $0$ & $0$\\
\vspace*{-3mm}\\
&  &  &  &  &  & -- & -- &  & -- & -- & -- & -- & -- & -- & -- & -- & -- & -- & -- & --
\vspace*{5mm}\\
& $141$ & $142$ & $143$ & $144$ & $145$ & $146$ & $147$ & $148$ & $149$ & $150$ & $151$ & $152$\\\hline
\vspace*{-3mm}\\
& $12$ & $12$ & $20$ & $20$ & $20$ & $12$ & $20$ & $20$ & $20$ & $20$ & $20$ & $20$\\
\vspace*{-3mm}\\
& $108$ & $106$ & $117$ & $112$ & $114$ & $108$ & $120$ & $117$ & $123$ & $124$ & $120$ & $128$\\
\vspace*{-3mm}\\
& $0$ & $11$ & $0$ & $0$ & $0$ & $11$ & $0$ & $11$ & $11$ & $11$ & $11$ & $11$\\
\vspace*{-3mm}\\
& -- & $57$ & -- & -- & -- & $57$ & -- & $57$ & $57$ & $57$ & $57$ & $57$
\end{longtable}
}

\newpage
\bigskip
\begin{example}[Computations of Example \ref{ex7.3}]\label{ex7.3b}~\vspace*{-5mm}\\
\begin{verbatim}
gap> Read("BCAlgTori.gap");
gap> I4g:=List(I4,x->MatGroupZClass(x[1],x[2],x[3],x[4]));;
gap> G:=I4g[1];
MatGroupZClass( 4, 31, 1, 3 )
gap> PermutationLatticeWithRank1:=Group([[[1]],[[1]]]);
Group([ [ [ 1 ] ], [ [ 1 ] ] ])
gap> PermutationLatticeWithRank2:=Group(CosetRepresentation(G,Group(G.1^2,G.2)));
Group([ [ [ 0, 1 ], [ 1, 0 ] ], [ [ 1, 0 ], [ 0, 1 ] ] ])
gap> PermutationLatticeWithRank4:=Group(CosetRepresentation(G,Group(G.2)));
<matrix group with 2 generators>
gap> PermutationLatticeWithRank5:=Group(CosetRepresentation(G,Group(G.1)));
<matrix group with 2 generators>
gap> GPermutationLatticeWithRank1:=DirectSumMatrixGroup([G,PermutationLatticeWithRank1]);
<matrix group with 2 generators>
gap> transmat1:=TransformationMat(GeneratorsOfGroup(GPermutationLatticeWithRank1),
> GeneratorsOfGroup(PermutationLatticeWithRank5));
[ [ [ 1, 1, 1, 1, -4 ], 
    [ 1, 1, 1, -4, 1 ], 
    [ 1, 1, -4, 1, 1 ], 
    [ 1, -4, 1, 1, 1 ], 
    [ 0, 0, 0, 0, 0 ] ], 
  [ [ 0, 0, 0, 0, 0 ], 
    [ 0, 0, 0, 0, 0 ], 
    [ 0, 0, 0, 0, 0 ], 
    [ 0, 0, 0, 0, 0 ], 
    [ 1, 1, 1, 1, 1 ] ] ]
gap> List(transmat1,Determinant);
[ 0, 0 ]
gap> p1:=Sum(transmat1);
[ [ 1, 1, 1, 1, -4 ], 
  [ 1, 1, 1, -4, 1 ], 
  [ 1, 1, -4, 1, 1 ], 
  [ 1, -4, 1, 1, 1 ], 
  [ 1, 1, 1, 1, 1 ] ]
gap> Determinant(p1);
625
gap> Factors(Determinant(p1));
[ 5, 5, 5, 5 ]
gap> GPermutationLatticeWithRank1^p1=PermutationLatticeWithRank5;
true
gap> F:=FlabbyResolutionLowRank(G).actionF;
<matrix group with 2 generators>
gap> List(F,x->[Order(x),Trace(x)]);
[ [ 4, 0 ], [ 4, 0 ], [ 5, 1 ], [ 4, 0 ], [ 5, 1 ], [ 5, 1 ], [ 2, 0 ], 
  [ 2, 0 ], [ 4, 0 ], [ 2, 0 ], [ 4, 0 ], [ 4, 0 ], [ 5, 1 ], [ 2, 0 ], 
  [ 4, 0 ], [ 4, 0 ], [ 4, 0 ], [ 4, 0 ], [ 1, 16 ], [ 2, 0 ] ]
gap> Collected(last);
[ [ [ 1, 16 ], 1 ], [ [ 2, 0 ], 5 ], [ [ 4, 0 ], 10 ], [ [ 5, 1 ], 4 ] ]
gap> Irr(G);
[ Character( CharacterTable( MatGroupZClass( 4, 31, 1, 3 ) ),
  [ 1, 1, 1, 1, 1 ] ), 
  Character( CharacterTable( MatGroupZClass( 4, 31, 1, 3 ) ),
  [ 1, -1, -1, 1, 1 ] ), 
  Character( CharacterTable( MatGroupZClass( 4, 31, 1, 3 ) ),
  [ 1, -E(4), E(4), 1, -1 ] ), 
  Character( CharacterTable( MatGroupZClass( 4, 31, 1, 3 ) ),
  [ 1, E(4), -E(4), 1, -1 ] ), 
  Character( CharacterTable( MatGroupZClass( 4, 31, 1, 3 ) ),
  [ 4, 0, 0, -1, 0 ] ) ]
gap> List(G,x->[Order(x),Trace(x)]);
[ [ 1, 4 ], [ 5, -1 ], [ 5, -1 ], [ 5, -1 ], [ 5, -1 ], [ 2, 0 ], [ 2, 0 ], 
  [ 2, 0 ], [ 2, 0 ], [ 2, 0 ], [ 4, 0 ], [ 4, 0 ], [ 4, 0 ], [ 4, 0 ], 
  [ 4, 0 ], [ 4, 0 ], [ 4, 0 ], [ 4, 0 ], [ 4, 0 ], [ 4, 0 ] ]
gap> Collected(last); 
# M~ is indecomposable with Z_2-rank 4 which corresponds to Irr(G)[5]
[ [ [ 1, 4 ], 1 ], [ [ 2, 0 ], 5 ], [ [ 4, 0 ], 10 ], [ [ 5, -1 ], 4 ] ]
gap> List(PermutationLatticeWithRank1,x->[Order(x),Trace(x)]);
[ [ 1, 1 ] ]
gap> List(PermutationLatticeWithRank2,x->[Order(x),Trace(x)]);
[ [ 2, 0 ], [ 1, 2 ] ]
gap> List(PermutationLatticeWithRank4,x->[Order(x),Trace(x)]);
[ [ 2, 0 ], [ 4, 0 ], [ 4, 0 ], [ 1, 4 ] ]
gap> List(PermutationLatticeWithRank5,x->[Order(x),Trace(x)]);
[ [ 1, 5 ], [ 5, 0 ], [ 5, 0 ], [ 5, 0 ], [ 5, 0 ], [ 2, 1 ], [ 2, 1 ], 
  [ 2, 1 ], [ 2, 1 ], [ 2, 1 ], [ 4, 1 ], [ 4, 1 ], [ 4, 1 ], [ 4, 1 ], 
  [ 4, 1 ], [ 4, 1 ], [ 4, 1 ], [ 4, 1 ], [ 4, 1 ], [ 4, 1 ] ]
gap> Collected(last);
[ [ [ 1, 5 ], 1 ], [ [ 2, 1 ], 5 ], [ [ 4, 1 ], 10 ], [ [ 5, 0 ], 4 ] ]
gap> Fx:=DirectSumMatrixGroup([G,G,G,PermutationLatticeWithRank4]);
<matrix group with 2 generators>
gap> transmat2:=TransformationMat(GeneratorsOfGroup(F),GeneratorsOfGroup(Fx));;
gap> Filtered(transmat2,x->Determinant(x)<>0);
[  ]
gap> Length(transmat2);
13
gap> Filtered(List(Combinations([1..13],2),x->Sum(transmat2{x})),y->Determinant(y)<>0);
[  ]
gap> Filtered(List(Combinations([1..13],3),x->Sum(transmat2{x})),y->Determinant(y)<>0);
[  ]
gap> p2:=Filtered(List(Combinations([1..13],4),
> x->Sum(transmat2{x})),y->Determinant(y)<>0)[1];
[ [ 0, 1, 0, 0, 0, 0, 1, 0, 0, 0, 0, 1, 1, 0, 0, 0 ], 
  [ 0, 0, 1, 0, -1, -1, -1, -1, 0, 1, 0, 0, 0, 1, 0, 0 ], 
  [ 0, 0, 0, 1, 0, 1, 0, 0, -1, -1, -1, -1, 0, 0, 1, 0 ], 
  [ -1, -1, -1, -1, 0, 0, 0, 1, 0, 0, 1, 0, 0, 0, 0, 1 ], 
  [ 0, 0, 0, 1, 1, 0, 0, 0, 0, 0, 1, 0, 0, 1, 0, 0 ], 
  [ -1, -1, -1, -1, 0, 0, 1, 0, 1, 0, 0, 0, 0, 0, 1, 0 ], 
  [ 1, 0, 0, 0, -1, -1, -1, -1, 0, 0, 0, 1, 0, 0, 0, 1 ], 
  [ 1, 0, 0, 0, 0, 0, 0, 1, 0, 1, 0, 0, 0, 0, 1, 0 ], 
  [ 0, 1, 0, 0, 1, 0, 0, 0, -1, -1, -1, -1, 0, 0, 0, 1 ], 
  [ 0, 0, 1, 0, 0, 1, 0, 0, 1, 0, 0, 0, 0, 0, 0, 1 ], 
  [ 1, 0, 0, 0, 0, 1, 0, 0, 0, 0, 1, 0, 1, 0, 0, 0 ], 
  [ -1, -1, -1, -1, 1, 0, 0, 0, 0, 1, 0, 0, 1, 0, 0, 0 ], 
  [ 0, 1, 0, 0, 0, 0, 0, 1, 1, 0, 0, 0, 0, 1, 0, 0 ], 
  [ 0, 0, 0, 1, -1, -1, -1, -1, 1, 0, 0, 0, 1, 0, 0, 0 ], 
  [ 1, 0, 0, 0, 0, 0, 1, 0, -1, -1, -1, -1, 0, 1, 0, 0 ], 
  [ 0, 0, 1, 0, 1, 0, 0, 0, 0, 0, 0, 1, 0, 0, 1, 0 ] ]
gap> Determinant(p2);
-15625
gap> Factors(last);
[ -5, 5, 5, 5, 5, 5 ]
gap> F^p2=Fx;
true
\end{verbatim}
\end{example}

\bigskip
\begin{example}[Theorem \ref{th7.12}: the group $\widetilde{N}_2$ for $G=N_{3,i}$ $(1\leq i\leq 15)$]\label{ex7.14}~\vspace*{-5mm}\\
\begin{verbatim}
gap> Read("BCAlgTori.gap");
gap> N3g:=List(N3,x->MatGroupZClass(x[1],x[2],x[3],x[4])); # G=N_{3,i}
[ MatGroupZClass( 3, 3, 1, 3 ), MatGroupZClass( 3, 3, 3, 3 ), 
  MatGroupZClass( 3, 3, 3, 4 ), MatGroupZClass( 3, 4, 3, 2 ), 
  MatGroupZClass( 3, 4, 4, 2 ), MatGroupZClass( 3, 4, 6, 3 ), 
  MatGroupZClass( 3, 4, 7, 2 ), MatGroupZClass( 3, 7, 1, 2 ), 
  MatGroupZClass( 3, 7, 2, 2 ), MatGroupZClass( 3, 7, 2, 3 ), 
  MatGroupZClass( 3, 7, 3, 2 ), MatGroupZClass( 3, 7, 3, 3 ), 
  MatGroupZClass( 3, 7, 4, 2 ), MatGroupZClass( 3, 7, 5, 2 ), 
  MatGroupZClass( 3, 7, 5, 3 ) ]
gap> Length(N3g); # the number of N_{3,i}
15
gap> N3g2:=Filtered(N3g,x->Set(Factors(Order(x)))=[2]); # 2-groups in N_{3,i}
[ MatGroupZClass( 3, 3, 1, 3 ), MatGroupZClass( 3, 3, 3, 3 ), 
  MatGroupZClass( 3, 3, 3, 4 ), MatGroupZClass( 3, 4, 3, 2 ), 
  MatGroupZClass( 3, 4, 4, 2 ), MatGroupZClass( 3, 4, 6, 3 ), 
  MatGroupZClass( 3, 4, 7, 2 ) ]
gap> Length(N3g2); # the number of 2-groups in N_{3,i}
7
gap> List(N3g2,x->Position(N3,CrystCatZClass(x)));
[ 1, 2, 3, 4, 5, 6, 7 ]
gap> List(N3g,x->Position(N3,CrystCatZClass(SylowSubgroup(x,2))));
[ 1, 2, 3, 4, 5, 6, 7, 1, 2, 3, 5, 5, 6, 7, 7 ]
gap> N3g2F:=List(N3g2,x->FlabbyResolutionLowRank(x).actionF);; 
# flabby resolutions of 2-groups in N_{3,i}

gap> List([1..Length(N3g2)],x->Order(N3g2[x])=Order(N3g2F[x]));
[ true, true, true, true, true, true, true ]
gap> Collected(last); # N2~ are faithful
[ [ true, 7 ] ]
gap> List(N3g2F,x->IsCodimJacobsonEnd1(x,2)); # N2~ are indecomposable
[ true, true, true, true, true, true, true ]
gap> List(N3g2F,x->Rank(x.1)); # rank of N2~
[ 5, 9, 11, 11, 7, 7, 11 ]
\end{verbatim}
\end{example}

\bigskip
\begin{example}[Theorem \ref{th7.13}: the groups $\widetilde{N}_2$ and $\widetilde{N}_3$ for $G=N_{4,i}$ $(1\leq i\leq 152)$]\label{ex7.15}~\vspace*{-5mm}\\
\begin{verbatim}
gap> Read("BCAlgTori.gap");
gap> N4g:=List(N4,x->MatGroupZClass(x[1],x[2],x[3],x[4]));; # G=N_{4,i}
gap> Length(N4g); # the number of N_{4,i}
152
gap> N4g2:=Filtered(N4g,x->Set(Factors(Order(x)))=[2]); # 2-groups in N_{4,i}
[ MatGroupZClass( 4, 5, 1, 12 ), MatGroupZClass( 4, 5, 2, 5 ), 
  MatGroupZClass( 4, 5, 2, 8 ), MatGroupZClass( 4, 5, 2, 9 ), 
  MatGroupZClass( 4, 6, 1, 6 ), MatGroupZClass( 4, 6, 1, 11 ), 
  MatGroupZClass( 4, 6, 2, 6 ), MatGroupZClass( 4, 6, 2, 10 ), 
  MatGroupZClass( 4, 6, 2, 12 ), MatGroupZClass( 4, 6, 3, 4 ), 
  MatGroupZClass( 4, 6, 3, 7 ), MatGroupZClass( 4, 6, 3, 8 ), 
  MatGroupZClass( 4, 12, 2, 5 ), MatGroupZClass( 4, 12, 2, 6 ), 
  MatGroupZClass( 4, 12, 3, 11 ), MatGroupZClass( 4, 12, 4, 10 ), 
  MatGroupZClass( 4, 12, 4, 11 ), MatGroupZClass( 4, 12, 4, 12 ), 
  MatGroupZClass( 4, 12, 5, 8 ), MatGroupZClass( 4, 12, 5, 9 ), 
  MatGroupZClass( 4, 12, 5, 10 ), MatGroupZClass( 4, 12, 5, 11 ), 
  MatGroupZClass( 4, 13, 1, 5 ), MatGroupZClass( 4, 13, 2, 5 ), 
  MatGroupZClass( 4, 13, 3, 5 ), MatGroupZClass( 4, 13, 4, 5 ), 
  MatGroupZClass( 4, 13, 5, 4 ), MatGroupZClass( 4, 13, 5, 5 ), 
  MatGroupZClass( 4, 13, 6, 5 ), MatGroupZClass( 4, 13, 7, 9 ), 
  MatGroupZClass( 4, 13, 7, 10 ), MatGroupZClass( 4, 13, 7, 11 ), 
  MatGroupZClass( 4, 13, 8, 5 ), MatGroupZClass( 4, 13, 8, 6 ), 
  MatGroupZClass( 4, 13, 9, 4 ), MatGroupZClass( 4, 13, 9, 5 ), 
  MatGroupZClass( 4, 13, 10, 4 ), MatGroupZClass( 4, 13, 10, 5 ), 
  MatGroupZClass( 4, 18, 1, 3 ), MatGroupZClass( 4, 18, 2, 4 ), 
  MatGroupZClass( 4, 18, 2, 5 ), MatGroupZClass( 4, 18, 3, 5 ), 
  MatGroupZClass( 4, 18, 3, 6 ), MatGroupZClass( 4, 18, 3, 7 ), 
  MatGroupZClass( 4, 18, 4, 4 ), MatGroupZClass( 4, 18, 4, 5 ), 
  MatGroupZClass( 4, 18, 5, 5 ), MatGroupZClass( 4, 18, 5, 6 ), 
  MatGroupZClass( 4, 18, 5, 7 ), MatGroupZClass( 4, 19, 1, 2 ), 
  MatGroupZClass( 4, 19, 2, 2 ), MatGroupZClass( 4, 19, 3, 2 ), 
  MatGroupZClass( 4, 19, 4, 3 ), MatGroupZClass( 4, 19, 4, 4 ), 
  MatGroupZClass( 4, 19, 5, 2 ), MatGroupZClass( 4, 19, 6, 2 ), 
  MatGroupZClass( 4, 32, 1, 2 ), MatGroupZClass( 4, 32, 2, 2 ), 
  MatGroupZClass( 4, 32, 3, 2 ), MatGroupZClass( 4, 32, 4, 2 ), 
  MatGroupZClass( 4, 32, 6, 2 ), MatGroupZClass( 4, 32, 7, 2 ), 
  MatGroupZClass( 4, 32, 8, 2 ), MatGroupZClass( 4, 32, 9, 4 ), 
  MatGroupZClass( 4, 32, 9, 5 ), MatGroupZClass( 4, 32, 10, 2 ), 
  MatGroupZClass( 4, 32, 12, 2 ), MatGroupZClass( 4, 32, 13, 3 ), 
  MatGroupZClass( 4, 32, 13, 4 ), MatGroupZClass( 4, 32, 14, 3 ), 
  MatGroupZClass( 4, 32, 14, 4 ), MatGroupZClass( 4, 32, 15, 2 ), 
  MatGroupZClass( 4, 32, 17, 2 ) ]
gap> Length(N4g2); # the number of 2-groups in N_{4,i}
73
gap> List(N4g2,x->Position(N4,CrystCatZClass(x)));
[ 1, 2, 3, 4, 5, 6, 7, 8, 9, 10, 11, 12, 13, 14, 15, 16, 17, 18, 19, 20, 21, 
  22, 23, 24, 25, 26, 27, 28, 29, 30, 31, 32, 33, 34, 35, 36, 37, 38, 39, 40, 
  41, 42, 43, 44, 45, 46, 47, 48, 49, 50, 51, 52, 53, 54, 55, 56, 106, 107, 
  108, 109, 112, 113, 114, 115, 116, 117, 120, 121, 122, 123, 124, 125, 128 ]
gap> List(N4g,x->Position(N4,CrystCatZClass(SylowSubgroup(x,2))));
[ 1, 2, 3, 4, 5, 6, 7, 8, 9, 10, 11, 12, 13, 14, 15, 16, 17, 18, 19, 20, 21, 
  22, 23, 24, 25, 26, 27, 28, 29, 30, 31, 32, 33, 34, 35, 36, 37, 38, 39, 40, 
  41, 42, 43, 44, 45, 46, 47, 48, 49, 50, 51, 52, 53, 54, 55, 56, fail, fail, 
  fail, fail, fail, fail, fail, fail, fail, fail, fail, fail, 2, 4, 16, 19, 
  22, 5, 7, 9, 26, 25, 10, 12, 35, 36, 29, 31, 30, 32, 33, 34, 37, 38, fail, 
  fail, fail, fail, fail, fail, fail, fail, fail, fail, fail, fail, fail, 
  fail, fail, 106, 107, 108, 109, 106, 106, 112, 113, 114, 115, 116, 117, 
  108, 108, 120, 121, 122, 123, 124, 125, 117, 117, 128, 123, 124, 123, 124, 
  120, 120, 128, 128, 106, 106, 108, 109, 108, 106, 117, 112, 114, 108, 120, 
  117, 123, 124, 120, 128 ]
gap> for i in [1..152] do
> if last[i]<>i then Print(i,"\t",last[i],"\n"); fi; od;
57	fail
58	fail
59	fail
60	fail
61	fail
62	fail
63	fail
64	fail
65	fail
66	fail
67	fail
68	fail
69	2
70	4
71	16
72	19
73	22
74	5
75	7
76	9
77	26
78	25
79	10
80	12
81	35
82	36
83	29
84	31
85	30
86	32
87	33
88	34
89	37
90	38
91	fail
92	fail
93	fail
94	fail
95	fail
96	fail
97	fail
98	fail
99	fail
100	fail
101	fail
102	fail
103	fail
104	fail
105	fail
110	106
111	106
118	108
119	108
126	117
127	117
129	123
130	124
131	123
132	124
133	120
134	120
135	128
136	128
137	106
138	106
139	108
140	109
141	108
142	106
143	117
144	112
145	114
146	108
147	120
148	117
149	123
150	124
151	120
152	128
gap> List(N4g,x->Position(N31,CrystCatZClass(SylowSubgroup(x,2))));
[ fail, fail, fail, fail, fail, fail, fail, fail, fail, fail, fail, fail, 
  fail, fail, fail, fail, fail, fail, fail, fail, fail, fail, fail, fail, 
  fail, fail, fail, fail, fail, fail, fail, fail, fail, fail, fail, fail, 
  fail, fail, fail, fail, fail, fail, fail, fail, fail, fail, fail, fail, 
  fail, fail, fail, fail, fail, fail, fail, fail, fail, fail, fail, fail, 
  fail, fail, fail, fail, fail, fail, fail, fail, fail, fail, fail, fail, 
  fail, fail, fail, fail, fail, fail, fail, fail, fail, fail, fail, fail, 
  fail, fail, fail, fail, fail, fail, fail, fail, fail, fail, fail, fail, 
  fail, fail, fail, fail, fail, fail, fail, fail, fail, fail, fail, fail, 
  fail, fail, fail, fail, fail, fail, fail, fail, fail, fail, fail, fail, 
  fail, fail, fail, fail, fail, fail, fail, fail, fail, fail, fail, fail, 
  fail, fail, fail, fail, fail, fail, fail, fail, fail, fail, fail, fail, 
  fail, fail, fail, fail, fail, fail, fail, fail ]
gap> Collected(last);
[ [ fail, 152 ] ]

gap> N4g2F:=List(N4g2,x->FlabbyResolutionLowRank(x).actionF);;
# flabby resolutions of 2-groups in N_{4,i}

gap> List([1..Length(N4g2)],x->Order(N4g2[x])=Order(N4g2F[x]));
[ true, true, true, true, true, true, true, true, true, true, true, true, 
  true, true, true, true, true, true, true, true, true, true, true, true, 
  true, true, true, true, true, true, true, true, true, true, true, true, 
  true, true, true, true, true, true, true, true, true, true, true, true, 
  true, true, true, true, true, true, true, true, true, true, true, true, 
  true, true, true, true, true, true, true, true, true, true, true, true, 
  true ]
gap> Collected(last); # N2~ are faithful
[ [ true, 73 ] ]
gap> List(N4g2F,x->Rank(x.1)); # rank of N2~
[ 5, 12, 12, 20, 11, 9, 12, 12, 16, 20, 20, 20, 12, 12, 7, 8, 8, 8, 12, 20, 
  12, 20, 11, 12, 8, 7, 20, 20, 11, 12, 20, 16, 20, 16, 20, 20, 20, 20, 12, 
  12, 20, 20, 20, 20, 20, 20, 20, 20, 20, 20, 20, 20, 20, 20, 20, 20, 12, 20, 
  12, 20, 20, 20, 20, 20, 20, 20, 20, 20, 20, 20, 20, 20, 20 ]
gap> List(N4g2F,x->IsCodimJacobsonEnd1(x,2));
[ true, false, true, true, true, true, false, true, true, true, true, true, 
  false, true, true, false, false, true, false, true, true, true, true, 
  false, false, true, true, true, true, false, true, true, true, true, true, 
  true, true, true, true, true, true, true, true, true, true, true, true, 
  true, true, true, true, true, true, true, true, true, false, true, true, 
  true, true, true, true, true, true, true, true, true, true, true, true, 
  true, true ]
gap> Filtered([1..73],x->last[x]=false);
[ 2, 7, 13, 16, 17, 19, 24, 25, 30, 57 ]
gap> List(last,x->Position(N4,CrystCatZClass(N4g2[x])));
[ 2, 7, 13, 16, 17, 19, 24, 25, 30, 106 ]

gap> eb:=EndomorphismRingBasis(N4g2F[57]);; # basis of End(N2~)
gap> Length(eb);
20
gap> IdempotentsModp(eb,2); # idempotens of End(N2~) mod 2 with rank(N2~)=12
[ [ [ 0, 0, 0, 0, 0, 0, 0, 0, 0, 0, 0, 0 ], 
      [ 0, 0, 0, 0, 0, 0, 0, 0, 0, 0, 0, 0 ], 
      [ 0, 0, 0, 0, 0, 0, 0, 0, 0, 0, 0, 0 ], 
      [ 0, 0, 0, 0, 0, 0, 0, 0, 0, 0, 0, 0 ], 
      [ 0, 0, 0, 0, 0, 0, 0, 0, 0, 0, 0, 0 ], 
      [ 0, 0, 0, 0, 0, 0, 0, 0, 0, 0, 0, 0 ], 
      [ 0, 0, 0, 0, 0, 0, 0, 0, 0, 0, 0, 0 ], 
      [ 0, 0, 0, 0, 0, 0, 0, 0, 0, 0, 0, 0 ], 
      [ 0, 0, 0, 0, 0, 0, 0, 0, 0, 0, 0, 0 ], 
      [ 0, 0, 0, 0, 0, 0, 0, 0, 0, 0, 0, 0 ], 
      [ 0, 0, 0, 0, 0, 0, 0, 0, 0, 0, 0, 0 ], 
      [ 0, 0, 0, 0, 0, 0, 0, 0, 0, 0, 0, 0 ] ], 
  [ [ 1, 0, 0, 0, 0, 0, 0, 0, 0, 0, 0, 0 ], 
      [ 0, 1, 0, 0, 0, 0, 0, 0, 0, 0, 0, 0 ], 
      [ 0, 0, 1, 0, 0, 0, 0, 0, 0, 0, 0, 0 ], 
      [ 0, 0, 0, 1, 0, 0, 0, 0, 0, 0, 0, 0 ], 
      [ 0, 0, 0, 0, 1, 0, 0, 0, 0, 0, 0, 0 ], 
      [ 0, 0, 0, 0, 0, 1, 0, 0, 0, 0, 0, 0 ], 
      [ 0, 0, 0, 0, 0, 0, 1, 0, 0, 0, 0, 0 ], 
      [ 0, 0, 0, 0, 0, 0, 0, 1, 0, 0, 0, 0 ], 
      [ 0, 0, 0, 0, 0, 0, 0, 0, 1, 0, 0, 0 ], 
      [ 0, 0, 0, 0, 0, 0, 0, 0, 0, 1, 0, 0 ], 
      [ 0, 0, 0, 0, 0, 0, 0, 0, 0, 0, 1, 0 ], 
      [ 0, 0, 0, 0, 0, 0, 0, 0, 0, 0, 0, 1 ] ] ]
gap> Length(last); # idempotents of End(N2~) are only the zero and the identity matrices 
2

gap> N4g3:=Filtered(N4g,x->Set(Factors(Order(x)))=[3]); # 3-groups in N_{4,i}
[ MatGroupZClass( 4, 22, 1, 1 ) ]
gap> Length(N4g3); # the number of 3-groups in N_{4,i}
1
gap> List(N4g3,x->Position(N4,CrystCatZClass(x)));
[ 57 ]
gap> List(N4g,x->Position(N4,CrystCatZClass(SylowSubgroup(x,3))));
[ fail, fail, fail, fail, fail, fail, fail, fail, fail, fail, fail, fail, 
  fail, fail, fail, fail, fail, fail, fail, fail, fail, fail, fail, fail, 
  fail, fail, fail, fail, fail, fail, fail, fail, fail, fail, fail, fail, 
  fail, fail, fail, fail, fail, fail, fail, fail, fail, fail, fail, fail, 
  fail, fail, fail, fail, fail, fail, fail, fail, 57, 57, 57, 57, 57, 57, 
  57, 57, 57, 57, 57, 57, fail, fail, fail, fail, fail, fail, fail, fail, 
  fail, fail, fail, fail, fail, fail, fail, fail, fail, fail, fail, fail, 
  fail, fail, 57, 57, 57, 57, 57, 57, 57, 57, 57, 57, 57, 57, 57, 57, 57, 
  fail, fail, fail, fail, fail, fail, fail, fail, fail, fail, fail, fail, 
  fail, fail, fail, fail, fail, fail, fail, fail, fail, fail, fail, fail, 
  fail, fail, fail, fail, fail, fail, fail, fail, fail, fail, fail, fail, 
  57, fail, fail, fail, 57, fail, 57, 57, 57, 57, 57 ]
gap> Filtered([1..152],x->last[x]=57);
[ 57, 58, 59, 60, 61, 62, 63, 64, 65, 66, 67, 68, 91, 92, 93, 94, 95, 96, 
  97, 98, 99, 100, 101, 102, 103, 104, 105, 142, 146, 148, 149, 150, 151, 152 ]
gap> N4g3F:=List(N4g3,x->FlabbyResolutionLowRank(x).actionF);;
# flabby resolution of a 3-group in N_{3,i}

gap> Order(N4g3[1])=Order(N4g3F[1]); # N3~ is faithful
true
gap> List(N4g3F,x->Rank(x.1)); # rank of N3~
[ 11 ]
gap> List(N4g3F,x->IsCodimJacobsonEnd1(x,3)); # N3~ is indecomposable
[ true ]
\end{verbatim}
\end{example}


\section{Proofs of ${\rm (1)}\Rightarrow {\rm (2)}$ of {Theorem \ref{thmain2}} $(N_{3,i})$, {Theorem \ref{thmain4}} $(N_{4,i})$ and {Theorem \ref{thmain5}} $(I_{4,i})$: stably $k$-equivalent classes}\label{S8}


We will prove the following theorems 
which also give proofs of 
${\rm (1)}\Rightarrow {\rm (2)}$ of {Theorem \ref{thmain2}}, {Theorem \ref{thmain4}} and {Theorem \ref{thmain5}. 
Note that the third condition of {\rm (2)}, i.e. 
$T_i\times_k K$ and $T_j^\prime\times_k K$ 
are weak stably birationally $K$-equivalent 
for any $k\subset K\subset L_i$, of the theorems 
follows from {\rm (1)} $T_i\stackrel{\rm s.b.}{\approx} T_j^\prime$ 
and the second condition of {\rm (2)}, i.e. $L_i=L_j^\prime$, 
which claim that $[M_G]^{fl}=[M_{G^\prime}]^{fl}$ as $\widetilde{H}$-lattices 
and $\widetilde{H}\simeq G\simeq G^\prime$ 
where $\widetilde{H}$ is a subdirect product of $G$ and $G^\prime$ 
(and restricting it for any $H\leq \widetilde{H}$). 

\begin{theorem}\label{th8.1}
Let $T_i$ and $T_j^\prime$ $(1\leq i,j\leq 15)$ 
be algebraic $k$-tori of dimension $3$ 
with the minimal splitting fields $L_i$ and $L^\prime_j$ 
and the character modules 
$\widehat{T}_i=M_G$ and $\widehat{T}^\prime_j=M_{G^\prime}$ 
which satisfy that 
$G$ and $G^\prime$ are 
$\GL(3,\bZ)$-conjugate to $N_{3,i}$ and $N_{3,j}$ respectively. 
For $1\leq i,j\leq 15$, 
if $T_i$ and $T_j^\prime$ are stably birationally $k$-equivalent, then 
$G\simeq G^\prime$, $L_i=L_j^\prime$.
\end{theorem}

\begin{theorem}\label{th8.2}
Let $T_i$ and $T_j^\prime$ $(1\leq i,j\leq 152)$ 
be algebraic $k$-tori of dimension $4$ 
with the minimal splitting fields $L_i$ and $L^\prime_j$ 
and the character modules 
$\widehat{T}_i=M_G$ and $\widehat{T}^\prime_j=M_{G^\prime}$ 
which satisfy that 
$G$ and $G^\prime$ are 
$\GL(4,\bZ)$-conjugate to $N_{4,i}$ and $N_{4,j}$ respectively. 
For $1\leq i,j\leq 152$, 
if $T_i$ and $T_j^\prime$ are stably birationally $k$-equivalent, then 
$G\simeq G^\prime$, $L_i=L_j^\prime$.
\end{theorem}

\begin{theorem}\label{th8.3}
Let $T_i$ and $T_j^\prime$ $(1\leq i,j\leq 7)$ 
be algebraic $k$-tori of dimension $4$ 
with the minimal splitting fields $L_i$ and $L^\prime_j$ 
and the character modules 
$\widehat{T}_i=M_G$ and $\widehat{T}^\prime_j=M_{G^\prime}$ 
which satisfy that 
$G$ and $G^\prime$ are 
$\GL(4,\bZ)$-conjugate to $I_{4,i}$ and $I_{4,j}$ respectively. 
For $1\leq i,j\leq 7$, 
if $T_i$ and $T_j^\prime$ are stably birationally $k$-equivalent, then 
$G\simeq G^\prime$, $L_i=L_j^\prime$.
\end{theorem}
\begin{remark}\label{rem8.4}
In general, $T_i\stackrel{\rm s.b.}{\approx} T_j^\prime$ 
does not imply $G\simeq G^\prime$, $L_i=L_j^\prime$ 
(see Theorem \ref{thmain3}, Table $3$  and Remark \ref{rem1.41} (2), (3)). 
\end{remark}

In order to prove Theorem \ref{th8.1} and Theorem \ref{th8.2}, 
we use the following lemma which is a similar result of Lemma \ref{lem7.2b} 
for the stably birationally $k$-equivalent $k$-tori 
via Theorem \ref{thVos70} (ii) 
(see also the condition ${\rm (ii)}^\prime$ 
in the second paragraph before Definition \ref{d1.10}). 
\begin{lemma}\label{lem8.5}
Let $T$ $($resp. $T^\prime$$)$ be an algebraic $k$-torus with 
the minimal splitting fields $L$ $($resp. $L^\prime$$)$ 
with $G={\rm Gal}(L/k)$ $($resp. $G^\prime={\rm Gal}(L^\prime/k)$$)$. 
Let $M$ $($resp. $M^\prime$$)$ be the corresponding $G$-lattice 
$($resp. $G^\prime$-lattice$)$ with $M\simeq \widehat{T}$ 
$($resp. $M^\prime\simeq \widehat{T}^\prime$$)$. 
If $T$ and $T^\prime$ are stably birationally $k$-equivalent, 
i.e. $[M]^{fl}=[M^\prime]^{fl}$ as $\widetilde{H}$-lattices 
where 
$\widetilde{L}=LL^\prime$ and 
$\widetilde{H}={\rm Gal}(\widetilde{L}/k)\leq G\times G^\prime$
is a subdirect product of $G$ and $G^\prime$ 
which acts on $M$ $($resp. $M^\prime)$ 
through the surjection 
$\varphi_1:\widetilde{H}\rightarrow G$ 
$($resp. $\varphi_2:\widetilde{H}\rightarrow G^\prime$$)$, 
then 
$\widetilde{N}_p\simeq \widetilde{N}^\prime_p$ 
as $\bZ_p[\Syl_p(\widetilde{H})]$-lattices 
where $\widetilde{N}_p$ $($resp. $\widetilde{N}^\prime_p$$)$ 
is the $p$-part of $[M]^{fl}$ $($resp. $[M^\prime]^{fl}$$)$ 
as a $\bZ_p[\varphi_1(\Syl_p(\widetilde{H}))]$-lattice 
$($resp. $\bZ_p[\varphi_2(\Syl_p(\widetilde{H}))]$-lattice$)$ 
with $\bZ_p[\varphi_1(\Syl_p(\widetilde{H}))]\simeq \bZ_p[\Syl_p(G)]$ 
$($resp. 
$\bZ_p[\varphi_2(\Syl_p(\widetilde{H}))]\simeq \bZ_p[\Syl_p(G^\prime)]$$)$. 
In particular, the indecomposability and 
the $\bZ_p$-rank of $\widetilde{N}_p$ are invariants for the 
stably birationally $k$-equivalent classes of algebraic $k$-tori. 
\end{lemma}

We made the following GAP \cite{GAP} algorithms 
in order to prove 
Theorem \ref{th8.1}, Theorem \ref{th8.2} and Theorem \ref{th8.3}. 
They are available as in \cite{BCAlgTori}.\\

\noindent 
{\tt ConjugacyClassesSubgroups2WSEC(}$G${\tt )} 
returns the records 
{\tt ConjugacyClassesSubgroups2} and 
{\tt WSEC}
where\\ 
{\tt ConjugacyClassesSubgroups2} 
is the list $[g_1,\ldots,g_m]$ 
of conjugacy classes of subgroups of $G\leq \GL(n,\bZ)$ $(n=3,4)$ 
with the fixed ordering via the function 
{\tt ConjugacyClassesSubgroups2(}$G${\tt )} 
(see \cite[Section 4.1]{HY17}) and 
{\tt WSEC} 
is the list $[w_1,\ldots,w_m]$ where 
$g_i$ is in the $w_i$-th weak stably $k$-equivalent class 
${\rm WSEC}_{w_i}$ in dimension $n$.\\

\noindent
{\tt MaximalInvariantNormalSubgroup(}$G$,
{\tt ConjugacyClassesSubgroups2WSEC(}$G${\tt ))}
returns the maximal normal subgroup $N$ of $G$
which satisfy that $\pi(H_1)=\pi(H_2)$ implies $\psi(H_1)=\psi(H_2)$ 
for any $H_1,H_2\leq G$ 
where $\pi:G\rightarrow G/N$ is the natural homomorphism, 
$\psi : H_i\mapsto w_i$,  
and $H_i$ is in the $w_i$-th weak stably $k$-equivalent class 
${\rm WSEC}_{w_i}$ in dimension $n$. 
If this returns {\tt Group([}\ \ {\tt])}, then $N=1$ 
where 
$\widetilde{H}\leq G\times G$ is a 
subdirect product of $G$ and $G$ with 
surjections $\varphi_1:\widetilde{H}\rightarrow G$, 
$\varphi_2:\widetilde{H}\rightarrow G$ and 
$N_1=\varphi_1({\rm Ker}(\varphi_2))\lhd N\lhd G$.\\

We use the following algorithm 
which is a variant of the computations as in Section \ref{S5} 
with the weak stably $k$-equivalent classes ${\rm WSEC}_r$ 
instead of the torus invariants $TI_G$.\\

\noindent 
{\tt PossibilityOfStablyEquivalentSubdirectProducts(}$G,G^\prime,$\\
\qquad {\tt ConjugacyClassesSubgroups2WSEC(}$G${\tt )},\\
\qquad {\tt ConjugacyClassesSubgroups2WSEC(}$G^\prime${\tt )},{\tt ["WSEC"])}\\
returns the list $l$ of the subdirect products 
$\widetilde{H}\leq G\times G^\prime$ of $G$ and $G^\prime$ 
up to $(\GL(n_1,\bZ)\times \GL(n_2,\bZ))$-conjugacy  
which satisfy $w_1=w_2$ 
for any $H\leq \widetilde{H}$ 
where $\varphi_i(H)$ is in the $w_i$-th weak stably $k$-equivalent class 
${\rm WSEC}_{w_i}$ in dimension $n$ $(n=3,4)$
and 
$\widetilde{H}\leq G\times G^\prime$ 
is a subdirect product of $G$ and $G^\prime$ 
which acts on $M_G$ and $M_{G^\prime}$ 
through the surjections $\varphi_1: \widetilde{H} \rightarrow G$ 
and $\varphi_2: \widetilde{H} \rightarrow G^\prime$ respectively 
(indeed, this function computes it for $H$ up to conjugacy 
for the sake of saving time).\\

{\it Proof of Theorem \ref{th8.1}.} 

Let $T_i$ and $T_j^\prime$ $(1\leq i,j\leq 15)$ 
be algebraic $k$-tori of dimension $3$ with the minimal splitting fields 
$L_i$ and $L_j^\prime$ 
with Galois groups 
$G={\rm Gal}(L_i/k)$ and $G^\prime={\rm Gal}(L_j^\prime/k)$ respectively. 
Assume that 
$\widehat{T}_i\simeq M_G$ and $\widehat{T}^\prime_j\simeq M_{G^\prime}$
and $G$ and $G^\prime$ are 
$\GL(3,\bZ)$-conjugate to $N_{3,i}$ and $N_{3,j}$ respectively 
where $N_{3,i}$ and $N_{3,j}$ are groups as in Definition \ref{defN3N4} 
and $M_G$ and $M_{G^\prime} $ are the corresponding $G$-lattice 
and $G^\prime$-lattice as in Definition \ref{d2.2}. 

It follows from the assumption $T_i\stackrel{\rm s.b.}{\approx} T_j^\prime$ 
and Theorem \ref{thVos70} (ii) that 
$[M_G]^{fl}=[M_{G^\prime}]^{fl}$ as $\widetilde{H}$-lattices where 
$\widetilde{L}=L_iL_j^\prime$ 
and $\widetilde{H}={\rm Gal}(\widetilde{L}/k)$. 
The group $\widetilde{H}\leq G\times G^\prime$ becomes a 
subdirect product of $G={\rm Gal}(L_i/k)$ and 
$G^\prime={\rm Gal}(L^\prime_j/k)$ with 
surjections $\varphi_1:\widetilde{H}\rightarrow G$ and 
$\varphi_2:\widetilde{H}\rightarrow G^\prime$. 

We should show that 
$\varphi_1:\widetilde{H}\xrightarrow{\sim} G$ and 
$\varphi_2:\widetilde{H}\xrightarrow{\sim} G^\prime$ 
are isomorphisms, i.e. 
$\widetilde{L}=L_i=L_j^\prime$.\\

(1) The case where $i=j$. 
We have $\widetilde{H}\simeq G\simeq G^\prime$, 
$\widetilde{L}=L_i=L_j^\prime$.\\

Step 1. 
We take a $p$-Sylow subgroup $\Syl_p(\widetilde{H})$ of $\widetilde{H}$. 
Then $\varphi_1(\Syl_p(\widetilde{H}))$ 
(resp. $\varphi_2(\Syl_p(\widetilde{H}))$) 
becomes a $p$-Sylow subgroup of $G$ (resp. $G^\prime$). 

We define 
$N_1:=\varphi_1({\rm Ker}(\varphi_2))\lhd G$, 
$N_2:=\varphi_2({\rm Ker}(\varphi_1))\lhd G^\prime$, 
$\pi:G\rightarrow G/N_1$, 
$\pi^\prime:G^\prime\rightarrow G^\prime/N_2$. 
Then we have 
$\overline{\varphi}_1=\pi\varphi_1: \widetilde{H}\rightarrow G/N_1$, 
$\overline{\varphi}_2=\pi^\prime\varphi_2: 
\widetilde{H}\rightarrow G^\prime/N_2$, 
and $(\overline{\varphi}_2)(\overline{\varphi}_1)^{-1}:
G/N_1\xrightarrow{\sim}G^\prime/N_2$. 

Conversely, 
we find that a subdirect product 
$\widetilde{H}\leq G\times G^\prime$ of $G$ and $G^\prime$ 
with surjections $\varphi_1: \widetilde{H}\to G$, $\varphi_2: \widetilde{H}\to G^\prime$ 
is given by 
\begin{align*}
\widetilde{H}=\{(g_1,g_2)\in G\times G^\prime\mid \overline{\varphi}(\pi_1(g_1))=\pi_2(g_2), 
\overline{\varphi}=(\overline{\varphi}_2)(\overline{\varphi}_1)^{-1}:
G/N_1\xrightarrow{\sim}G^\prime/N_2\}
\end{align*}
and hence there exists a one-to-one correspondence 
between the set of all subdirect products $\widetilde{H}$ of $G$, $G^\prime$ 
with surjections $\varphi_1: \widetilde{H}\to G$, $\varphi_2: \widetilde{H}\to G^\prime$ and 
\begin{align*}
\{(N_1,N_2,\overline{\varphi})\mid N_1\lhd G, N_2\lhd G^\prime, 
\overline{\varphi}=(\overline{\varphi}_2)(\overline{\varphi}_1)^{-1}:
G/N_1\xrightarrow{\sim}G^\prime/N_2\}. 
\end{align*}

It follows from the definitions of $N_1$ and $\pi$ that 
for $H_1,H_2\leq G$ 
$\pi(H_1)=\pi(H_2)$ 
if and only if 
$\varphi_2\varphi_1^{-1}(H_1)=\varphi_2\varphi_1^{-1}(H_2)$. 
Because $[M_G]^{fl}=[M_{G^\prime}]^{fl}$ as $\widetilde{H}$-lattices 
and $\varphi_1^{-1}(H_m)\leq \widetilde{H}$ $(m=1,2)$ 
is a subdirect product of 
$H_m\leq G$ and $\varphi_2\varphi_1^{-1}(H_m)\leq G^\prime$, 
we have $[M_{H_m}]^{fl}=[M_{\varphi_2\varphi_1^{-1}(H_m)}]^{fl}$ 
as $\varphi_1^{-1}(H_m)$-lattices by restricting the action of 
$\widetilde{H}$ to $\varphi_1^{-1}(H_m)$. 
Hence $[M_{H_m}]^{fl}\sim [M_{\varphi_2\varphi_1^{-1}(H_m)}]^{fl}$ $(m=1,2)$. 
Then for any $H_1,H_2\leq G$, 
$\pi(H_1)=\pi(H_2)$, i.e. 
$\varphi_2\varphi_1^{-1}(H_1)=\varphi_2\varphi_1^{-1}(H_2)$, 
implies $[M_{H_1}]^{fl}\sim [M_{H_2}]^{fl}$. 

In summary, for any $H_1,H_2\leq G$, if $\pi(H_1)=\pi(H_2)$, then $w_1=w_2$ 
where $[M_{H_m}]^{fl}$ is in the $w_m$-th weak stably $k$-equivalent class 
${\rm WSEC}_{w_m}$ in dimension $3$ $(m=1,2)$. 

We see that there exists the maximal $N\lhd G$ with $N_1\lhd N$ such that 
$\pi(H_1)=\pi(H_2)$, i.e. $H_1N=H_2N$,  
implies $[M_{H_1}]^{fl}\sim [M_{H_2}]^{fl}$ 
where $\pi: G\to G/N$ because the following reason: 
If (i) $H_1N=H_2N$ implies $[M_{H_1}]^{fl}\sim [M_{H_2}]^{fl}$ and 
(ii) $H_1N^\prime=H_2N^\prime$ implies $[M_{H_1}]^{fl}\sim [M_{H_2}]^{fl}$, 
then (iii) $H_1NN^\prime=H_2NN^\prime$ implies $[M_{H_1}]^{fl}\sim [M_{H_2}]^{fl}$, 
because  
by (i), 
$[M_{H_1N}]^{fl}\sim [M_{H_1}]^{fl}$, 
$[M_{H_2N}]^{fl}\sim [M_{H_2}]^{fl}$, 
and by (ii), 
$[M_{H_1NN^\prime}]^{fl}\sim [M_{H_1N}]^{fl}$, 
$[M_{H_2NN^\prime}]^{fl}\sim [M_{H_2N}]^{fl}$, 
then we have that 
$[M_{H_1NN^\prime}]^{fl}\sim [M_{H_2NN^\prime}]^{fl}$ implies 
$[M_{H_1}]^{fl}\sim [M_{H_2}]^{fl}$. 

By applying the function 
{\tt MaximalInvariantNormalSubgroup(}$G$,
{\tt ConjugacyClassesSubgroups2WSEC(}$G${\tt ))}, 
we can obtain the maximal subgroup $N\lhd G$ as above with 
$N_1\lhd N=1$ (resp. $N_1\lhd N\simeq C_2$) 
if $G=N_{3,i}$ $(i\neq 4)$ (resp. $G=N_{3,4}$). 
Hence we conclude that $\widetilde{H}\simeq G\simeq G^\prime$, 
$L_i=L^\prime_i$ except for the one case $i=4$ 
(see Example \ref{ex8.7} for GAP computations).\\ 

Step 2. 
We treat the exceptional case $G=N_{3,4}\simeq G^\prime\simeq C_4\times C_2$. 
By Step 1, we have $N_1\leq N\simeq C_2$. 
Suppose that $N_1\simeq C_2$. 
Then $N_2\simeq N_1\simeq C_2$. 
Hence $\widetilde{H}$ is of order $16$ because 
$|\widetilde{H}|=|G|\times |N_2|=|N_1|\times |G^\prime|$. 
We take the $2$-part $\widetilde{N}_2$ of 
$[M_{\varphi_1(\Syl_2(\widetilde{H}))}]^{fl}$ 
$($resp. $\widetilde{N}^\prime_2$ of 
$[M_{\varphi_2(\Syl_2(\widetilde{H}))}]^{fl}$$)$ 
as a $\bZ_2[\varphi_1(\Syl_2(\widetilde{H}))]$-lattice 
$($resp. $\bZ_2[\varphi_2(\Syl_2(\widetilde{H}))]$-lattice$)$ 
with $\bZ_2[\varphi_1(\Syl_2(\widetilde{H}))]\simeq \bZ_2[\Syl_2(G)]$ 
$($resp. 
$\bZ_2[\varphi_2(\Syl_2(\widetilde{H}))]\simeq \bZ_2[\Syl_2(G^\prime)]$$)$ 
as in Definition \ref{def7.1}. 
By the definition, 
we see that 
${\rm Ker}(\varphi_2)\leq \widetilde{H}$ acts on 
$\widetilde{N}^\prime_2$ trivially. 
It follows from Lemma \ref{lem8.5} that 
${\rm Ker}(\varphi_2)\leq \widetilde{H}$ acts also on 
$\widetilde{N}_2$ trivially. 
Similarly, ${\rm Ker}(\varphi_1)\leq \widetilde{H}$ 
acts on both $\widetilde{N}_2$ and $\widetilde{N}^\prime_2$ trivially. 
Note that ${\rm Ker}(\varphi_2)\cap {\rm Ker}(\varphi_1)=1$. 
Hence $C_2\times C_2\leq 
{\rm Ker}(\varphi_2){\rm Ker}(\varphi_1)\leq\widetilde{H}$ 
acts on $\widetilde{N}_2$ trivially. 
However, by Theorem \ref{th7.12} (Table $14$), 
we know that $\widetilde{N}_2$ is an indecomposable 
faithful $\bZ_2[\Syl_2(G)]$-lattice with $\bZ_2$-rank $11$. 
This yields a contradiction. 
We conclude that $N_1=1$ and hence $L_4=L^\prime_4$.\\

(2) The case where $i\neq j$. 
We may assume that $(i,j)=(5,6), (11,13)\in J_{3,3}$ and 
$G\simeq G^\prime$ by Theorem \ref{thmain1} 
(see also Proposition \ref{prop5.2} and Theorem \ref{th6.1}).\\

It follows from Theorem \ref{th6.1} that there exists a subdirect product 
$G_{i,j}\leq G\times G^\prime$ with $G_{i,j}\simeq G\simeq G^\prime$ 
such that $[M_G]^{fl}=[M_{G^\prime}]^{fl}$ as $G_{i,j}$-lattices 
where $G_{i,j}$ acts on $M_G$ and $M_{G^\prime}$ 
through the isomorphisms $\lambda_1: G_{i,j} \xrightarrow{\sim} G$ 
and $\lambda_2: G_{i,j} \xrightarrow{\sim} G^\prime$ respectively. 
Then we obtain that 
$[M_G]^{fl}=[M_{G^\prime}]^{fl}$ as $G$-lattices 
where $G$ acts on $M_{G^\prime}$ 
through $G^\prime=\lambda_2\lambda_1^{-1}(G)$. 

We can take a subdirect product 
$\widetilde{H}_{i,j}:=\{(h,x)\in \widetilde{H}\times G_{i,j}\mid \varphi_2(h)=\lambda_2(x)\}\leq \widetilde{H}\times G_{i,j}$ 
which satisfies 
$[M_G]^{fl}=[M_{G^\prime}]^{fl}$ as $\widetilde{H}_{i,j}$-lattices 
with surjections 
$\theta_1: \widetilde{H}_{i,j}\rightarrow \widetilde{H}, (h,x)\mapsto h$ and 
$\theta_2: \widetilde{H}_{i,j}\rightarrow G_{i,j}, (h,x)\mapsto x$, 
and the condition 
$\varphi_2\theta_1=\lambda_2\theta_2: 
\widetilde{H}_{i,j}\rightarrow G^\prime$ 
where $\widetilde{H}_{i,j}$ 
acts on $M_G$ in two ways 
through the surjections 
$\varphi_1\theta_1: \widetilde{H}_{i,j}\rightarrow G$
and 
$\lambda_1\theta_2: \widetilde{H}_{i,j}\rightarrow G$, 
and on $M_{G^\prime}$ 
through the surjection 
$\varphi_2\theta_1=\lambda_2\theta_2: 
\widetilde{H}_{i,j}\rightarrow G^\prime$. 

Then we consider a subdirect product 
$\widetilde{G}:=\{(\varphi_1\theta_1(x),\lambda_1\theta_2(x))\in G\times G
\mid x\in \widetilde{H}_{i,j}\}\leq \varphi_1\theta_1(\widetilde{H}_{i,j})
\times 
\lambda_1\theta_2(\widetilde{H}_{i,j})$ of 
$\varphi_1\theta_1(\widetilde{H}_{i,j})=G$ 
and $\lambda_1\theta_2(\widetilde{H}_{i,j})=G$ 
with surjections 
$\mu_1: \widetilde{G}\rightarrow \varphi_1\theta_1(\widetilde{H}_{i,j})$ 
and 
$\mu_2: \widetilde{G}\rightarrow \lambda_1\theta_2(\widetilde{H}_{i,j})$ 
which satisfies 
$[M_{\varphi_1\theta_1(\widetilde{H}_{i,j})}]^{fl}
=[M_{\lambda_1\theta_2(\widetilde{H}_{i,j})}]^{fl}$ 
as $\widetilde{G}$-lattices 
where $\widetilde{G}$ acts on 
$M_{\varphi_1\theta_1(\widetilde{H}_{i,j})}$ 
and 
$M_{\lambda_1\theta_2(\widetilde{H}_{i,j})}$ 
through the surjections $\mu_1$ and $\mu_2$. 

Because $\varphi_1\theta_1(\widetilde{H}_{i,j})=G$ 
and $\lambda_1\theta_2(\widetilde{H}_{i,j})=G$, 
it follows from (1) the case where $i=j$ above that 
$\mu_1: \widetilde{G}\xrightarrow{\sim} \varphi_1\theta_1(\widetilde{H}_{i,j})=G$ 
and 
$\mu_2: \widetilde{G}\xrightarrow{\sim} \lambda_1\theta_2(\widetilde{H}_{i,j})=G$ are isomorphisms. 
Hence 
${\rm Ker}(\varphi_1\theta_1)
={\rm Ker}(\lambda_1\theta_2)$. 
Because $\lambda_1$ and $\lambda_2$ are isomorphisms, 
${\rm Ker}(\lambda_1\theta_2)=
{\rm Ker}(\lambda_2\theta_2)=
{\rm Ker}(\varphi_2\theta_1)$ 
and hence ${\rm Ker}(\varphi_1\theta_1)=
{\rm Ker}(\varphi_2\theta_1)$. 
Because $\theta_1$ is surjective, 
we have 
${\rm Ker}(\varphi_1)={\rm Ker}(\varphi_2)$. 
This implies that 
$\varphi_1:\widetilde{H}\xrightarrow{\sim} G$ and 
$\varphi_2:\widetilde{H}\xrightarrow{\sim} G^\prime$ 
are isomorphisms. 
This implies that $\widetilde{L}=L_i=L_j^\prime$. 
\qed\\ 

{\it Proof of Theorem \ref{th8.2}.} 

We can prove the theorem by the same way as 
in the proof of Theorem \ref{th8.1}. 

Let $T_i$ and $T_j^\prime$ $(1\leq i,j\leq 152)$ 
be algebraic $k$-tori of dimension $4$ with the minimal splitting fields 
$L_i$ and $L_j^\prime$ 
with Galois groups 
$G={\rm Gal}(L_i/k)$ and $G^\prime={\rm Gal}(L_j^\prime/k)$ respectively. 
Assume that 
$\widehat{T}_i\simeq M_G$ and $\widehat{T}^\prime_j\simeq M_{G^\prime}$
and $G$ and $G^\prime$ are 
$\GL(4,\bZ)$-conjugate to $N_{4,i}$ and $N_{4,j}$ respectively 
where $N_{4,i}$ and $N_{4,j}$ are groups as in Definition \ref{defN3N4} 
and $M_G$ and $M_{G^\prime} $ are the corresponding $G$-lattice 
and $G^\prime$-lattice as in Definition \ref{d2.2}. 

It follows from the assumption $T_i\stackrel{\rm s.b.}{\approx} T_j^\prime$ 
and Theorem \ref{thVos70} (ii) that 
$[M_G]^{fl}=[M_{G^\prime}]^{fl}$ as $\widetilde{H}$-lattices where 
$\widetilde{L}=L_iL_j^\prime$ 
and $\widetilde{H}={\rm Gal}(\widetilde{L}/k)$. 
The group $\widetilde{H}\leq G\times G^\prime$ becomes a 
subdirect product of $G={\rm Gal}(L_i/k)$ and 
$G^\prime={\rm Gal}(L^\prime_j/k)$ with 
surjections $\varphi_1:\widetilde{H}\rightarrow G$ and 
$\varphi_2:\widetilde{H}\rightarrow G^\prime$. 

We should show that 
$\varphi_1:\widetilde{H}\xrightarrow{\sim} G$ and 
$\varphi_2:\widetilde{H}\xrightarrow{\sim} G^\prime$ 
are isomorphisms, i.e. $\widetilde{H}\simeq G\simeq G^\prime$, 
$\widetilde{L}=L_i=L_j^\prime$.\\

(1) The case where $i=j$. 
We have $\widetilde{H}\simeq G\simeq G^\prime$, 
$\widetilde{L}=L_i=L_j^\prime$.\\

Step 1. We can apply the GAP function\\ 
{\tt MaximalInvariantNormalSubgroup(}$G$,
{\tt ConjugacyClassesSubgroups2WSEC(}$G${\tt ))} 
as the same as in 
(1) Step 1 of the proof of Theorem \ref{th8.1}. 
Then we can obtain the maximal subgroup $N\lhd G$ with $N_1\lhd N=1$ 
(resp. $N_1\lhd N\simeq C_2$) 
if $G=N_{4,i}$ $(i\not\in I)$ (resp. $G=N_{4,i}$ $(i\in I)$) where 
$I=\{13, 14, 23, 24, 39, 106, 107, 108, 109, 110, 111, 118, 119, 137$, $138$, 
$139, 140, 141\}$. 
Hence we conclude that $\widetilde{H}\simeq G\simeq G^\prime$, 
$L_i=L^\prime_i$ except for the case $i\in I$ 
(see Example \ref{ex8.8}).\\

Step 2. We should treat the exceptional cases $G=N_{4,i}$ $(i\in I)$. 
By Step 1, we have $N_1\leq N\simeq C_2$ for each cases. 
Suppose that $N_1\simeq C_2$. 
Then $N_2\simeq N_1\simeq C_2$. 
This yields a contradiction via Lemma \ref{lem8.5} 
as the same as in Step 2 of the proof of Theorem \ref{th8.1} because 
by Theorem \ref{th7.13} (Table $15$), we obtain that 
$\widetilde{N}_2$ is an indecomposable faithful 
$\bZ_2[\Syl_2(G)]$-lattice with $\bZ_2$-rank 
$11$, $12$, $11$, $11$, $12$, $12$, $20$, $12$, $20$, $12$, $12$, 
$12$, $12$, $12$, $12$, $12$, $20$, $12$ 
for $G=N_{4,i}$ $(i\in I$) respectively. 
We conclude that $N_1=1$ and hence $L_i=L^\prime_i$ for $i\in I$.\\

(2) The case where $i\neq j$. 
We may assume that $(i,j)\in J_{4,4}$ and 
$G\simeq G^\prime$ by Theorem \ref{thmain3} 
(see also Proposition \ref{prop5.4} and Theorem \ref{th6.3}). 
We can prove this case by the same way as 
in the proof of Theorem \ref{th8.1} (2).\qed\\ 

{\it Proof of Theorem \ref{th8.3}.} 

We can prove the theorem by the same way as 
in the proof of Theorem \ref{th8.1} via 
$\widetilde{H}\simeq G\simeq G^\prime$, 
$\widetilde{L}=L_i=L_j^\prime$. 
By Theorem \ref{thmain3}, we have $i=j$.\\ 

Step 1. We can apply the GAP function\\ 
{\tt MaximalInvariantNormalSubgroup(}$G$,
{\tt ConjugacyClassesSubgroups2WSEC(}$G${\tt ))} 
as the same as in 
(1) Step 1 of the proof of Theorem \ref{th8.1}. 
Then we can obtain the maximal subgroup $N\lhd G$ with $N_1\lhd N=1$ 
(resp. $N_1\lhd N\simeq C_4$)  
if $G=I_{4,i}$ $(i\neq 7)$ (resp. $G=N_{4,7}$). 
Hence we conclude that $\widetilde{H}\simeq G\simeq G^\prime$, 
$L_i=L^\prime_i$ except for the case $i=7$ 
(see Example \ref{ex8.9}).\\

Step 2. We should treat the exceptional case 
$G=I_{4,7}\simeq C_3\rtimes C_8$. 
By Step 1, we have $N_1\leq N\simeq C_4$. 
We can not use the same method as in Step 2 of 
the proof of Theorem \ref{th8.1} (Theorem \ref{th8.2}) 
because $\widetilde{N}_2=0$. 

By applying the function 
{\tt PossibilityOfStablyEquivalentSubdirectProducts(}$G,G^\prime,$\\
\qquad {\tt ConjugacyClassesSubgroups2WSEC(}$G${\tt )},\\
\qquad {\tt ConjugacyClassesSubgroups2WSEC(}$G^\prime${\tt ))},{\tt ["WSEC"])},\\
we see that 
there exist $4$ possibilities of $[M_G]^{fl}=[M_{G^\prime}]^{fl}$ as 
$\widetilde{H}$-lattices where 
$\widetilde{H}={\rm Gal}(\widetilde{L}_7/k)$ 
and $\widetilde{L}_7=L_7L_7^\prime$ 
with $|\widetilde{H}|= 96, 48, 24, 24$. 
For the first two $\widetilde{H}$ with $|\widetilde{H}|= 96, 48$, 
by applying the function\\
{\tt PossibilityOfStablyEquivalentFSubdirectProduct($\widetilde{H}$)} 
as in Section \ref{S6}, 
we may confirm that 
$[M_G]^{fl}\neq [M_{G^\prime}]^{fl}$ as $\widetilde{H}$-lattices. 
Hence we conclude that 
$|\widetilde{H}|=[\widetilde{L}_7:k]=24=|G|=|G^\prime|$, 
$N_1\simeq N_2=1$ and $L_7=L_7^\prime$ (see Example \ref{ex8.9} for GAP computations). \qed

\bigskip
\begin{example}[Step 1 in the proof of Theorem \ref{th8.1}]\label{ex8.7}~\vspace*{-5mm}\\
\begin{verbatim}
gap> Read("BCAlgTori.gap");
gap> N3g:=List(N3,x->MatGroupZClass(x[1],x[2],x[3],x[4]));;
gap> Length(N3g);
15
gap> N3wsec:=List(N3g,ConjugacyClassesSubgroups2WSEC);;
gap> N3m:=List([1..15],x->MaximalInvariantNormalSubgroup(N3g[x],N3wsec[x]));;
gap> List(N3m,Order); # |N_1|=1 or 2
[ 1, 1, 1, 2, 1, 1, 1, 1, 1, 1, 1, 1, 1, 1, 1 ]
\end{verbatim}
\end{example}

\bigskip
\begin{example}[Step 1 in the proof of Theorem \ref{th8.2}]\label{ex8.8}~\vspace*{-5mm}\\
\begin{verbatim}
gap> Read("BCAlgTori.gap");
gap> N4g:=List(N4,x->MatGroupZClass(x[1],x[2],x[3],x[4]));;
gap> Length(N4g);
152
gap> N4wsec:=List(N4g,ConjugacyClassesSubgroups2WSEC);;
gap> N4m:=List([1..152],x->MaximalInvariantNormalSubgroup(N4g[x],N4wsec[x]));;
gap> List(N4m,Order); # |N_1|=1 or 2
[ 1, 1, 1, 1, 1, 1, 1, 1, 1, 1, 1, 1, 2, 2, 1, 1, 1, 1, 1, 1, 1, 1, 2, 2, 1, 
  1, 1, 1, 1, 1, 1, 1, 1, 1, 1, 1, 1, 1, 2, 1, 1, 1, 1, 1, 1, 1, 1, 1, 1, 1, 
  1, 1, 1, 1, 1, 1, 1, 1, 1, 1, 1, 1, 1, 1, 1, 1, 1, 1, 1, 1, 1, 1, 1, 1, 1, 
  1, 1, 1, 1, 1, 1, 1, 1, 1, 1, 1, 1, 1, 1, 1, 1, 1, 1, 1, 1, 1, 1, 1, 1, 1, 
  1, 1, 1, 1, 1, 2, 2, 2, 2, 2, 2, 1, 1, 1, 1, 1, 1, 2, 2, 1, 1, 1, 1, 1, 1, 
  1, 1, 1, 1, 1, 1, 1, 1, 1, 1, 1, 2, 2, 2, 2, 2, 1, 1, 1, 1, 1, 1, 1, 1, 1, 
  1, 1 ]
gap> m2:=Filtered([1..152],x->Order(N4m[x])>1);
[ 13, 14, 23, 24, 39, 106, 107, 108, 109, 110, 111, 118, 119, 137, 138, 139, 
  140, 141 ]
gap> List(m2,x->Order(N4m[x]));
[ 2, 2, 2, 2, 2, 2, 2, 2, 2, 2, 2, 2, 2, 2, 2, 2, 2, 2 ]
\end{verbatim}
\end{example}

\bigskip
\begin{example}[Step 1 and Step 2 in the proof of Theorem \ref{th8.3}]\label{ex8.9}~\vspace*{-5mm}\\
\begin{verbatim}
gap> Read("BCAlgTori.gap");
gap> I4g:=List(I4,x->MatGroupZClass(x[1],x[2],x[3],x[4]));;
gap> Length(I4g);
7
gap> I4wsec:=List(I4g,ConjugacyClassesSubgroups2WSEC);;
gap> I4m:=List([1..7],x->MaximalInvariantNormalSubgroup(I4g[x],I4wsec[x]));;
gap> List(I4m,Order); # |N_1|=1 or 2 or 4
[ 1, 1, 1, 1, 1, 1, 4 ]

gap> Order(I4g[7]);
24
gap> IdSmallGroup(I4g[7]);
[ 24, 1 ]
gap> IdSmallGroup(I4m[7]);
[ 4, 1 ]
gap> StructureDescription(SmallGroup(24,1));
"C3 : C8"
gap> StructureDescription(SmallGroup(4,1));
"C4"
gap> pos:=PossibilityOfStablyEquivalentSubdirectProducts(I4g[7],I4g[7],I4wsec[7],I4wsec[7]);
# there exists 4 possibilities 
[ <pc group of size 96 with 6 generators>, 
  <pc group of size 48 with 5 generators>, 
  <pc group of size 24 with 4 generators>, 
  <pc group of size 24 with 4 generators> ]
gap> PossibilityOfStablyEquivalentFSubdirectProduct(pos[1]);
[ [ 0, 0, 0, 0, 0, 1, 0, 0, -1, 0, 0, 0, 0, 0, 0, 1, 0, -1, 0, 0, 0, 0, 0, 1, 
    0, 0, -1, 0, 0, 0, 0, 0, -1, 0, 0, 0, 0, 0, 1, 0, 0, 0, 0, 0, -2 ], 
  [ 0, 0, 0, 0, 0, 0, 0, 0, 0, 0, 1, 0, 0, 0, 0, 0, 0, 0, 0, -1, -1, 0, -1, 
    0, 0, 0, 0, 0, 0, 1, -1, -1, 0, 0, 1, 1, 0, 0, 0, 2, -1, 1, 1, -2, 0 ], 
  [ 0, 0, 0, 0, 0, 0, 0, 0, 0, 0, 0, 1, 0, 0, 0, 0, -1, 0, 0, 0, 0, -1, 0, 0, 
    0, 0, 0, 0, -1, -1, 1, -1, 0, 0, 0, 0, 1, 1, 0, 2, 1, -1, 1, -2, 0 ], 
  [ 0, 0, 0, 0, 0, 0, 0, 0, 0, 0, 0, 0, 0, 0, 0, 0, 0, 0, 1, 0, 0, 0, 0, 0, 
    0, 0, 0, 0, 0, -1, -1, -1, 0, -1, 0, 0, 0, 0, 0, 2, 1, 1, 1, -2, 0 ] ]
gap> List(last,x->x[Length(x)]); # pos[1] is impossible 
[ -2, 0, 0, 0 ]
gap> PossibilityOfStablyEquivalentFSubdirectProduct(pos[2]);
[ [ 0, 1, -1, 0, 0, 0, 0, 0, 1, -1, 0, 0, 0, 0, 0, 0, 0, 0, 0, 0, 0, 0, -2 ], 
  [ 0, 0, 0, 0, 0, 0, 1, 0, 0, 0, 0, -1, -1, -1, -1, 0, 0, 2, 1, 1, 1, -2, 0 ] ]
gap> List(last,x->x[Length(x)]); # pos[2] is impossible 
[ -2, 0 ]
\end{verbatim}
\end{example}

\section{Proofs of ${\rm (2)}\Rightarrow {\rm (1)}$ of {Theorem \ref{thmain2}} $(N_{3,i})$, {Theorem \ref{thmain4}} $(N_{4,i})$ and {Theorem \ref{thmain5}} $(I_{4,i})$: stably $k$-equivalent classes}\label{S9}

Let $T_i$ and $T_j^\prime$ 
be algebraic $k$-tori of dimension $3$ or $4$ 
with the minimal splitting fields $L_i$ and $L^\prime_j$ 
and the character modules 
$\widehat{T}_i=M_G$ and $\widehat{T}^\prime_j=M_{G^\prime}$ 
which satisfy that 
{\rm (i)} $G$ and $G^\prime$ are 
$\GL(3,\bZ)$-conjugate to $N_{3,i}$ and $N_{3,j}$ $(1\leq i,j\leq 15)$; 
{\rm (ii)} $G$ and $G^\prime$ are 
$\GL(4,\bZ)$-conjugate to $N_{4,i}$ and $N_{4,j}$ $(1\leq i,j\leq 152)$; 
or 
{\rm (iii)} $G$ and $G^\prime$ are 
$\GL(4,\bZ)$-conjugate to $I_{4,i}$ and $I_{4,j}$ $(1\leq i,j\leq 7)$, 
respectively. 

We assume that the condition {\rm (2)} 
of Theorem \ref{thmain2}, Theorem \ref{thmain4} and Theorem \ref{thmain5}:\\

\noindent 
(2) $G\simeq G^\prime$, $L_i=L_j^\prime$, 
$T_i\times_k K$ and $T_j^\prime\times_k K$ 
are weak stably birationally $K$-equivalent 
for any $k\subset K\subset L_i$.\\

By the assumption {\rm (2)}, it follows from 
{Theorem \ref{thmain1}} and {Theorem \ref{thmain3} that 
$G\simeq  G^\prime\simeq {\rm Gal}(L_i/k)={\rm Gal}(L_j^\prime/k)$ 
(see also Remark \ref{rem1.25}), 
and for any $H={\rm Gal}(L_i/K)\leq G$,  
$[M_H]^{fl}=[M_{H^\prime}]^{fl}$ as $\widetilde{H}$-lattices 
where $\widetilde{H}\leq H\times H^\prime$ is a subdirect product 
of $H\leq G$ and $H^\prime={\rm Gal}(L_j^\prime/K)\leq G^\prime$ 
which acts on $M_H$ and $M_{H^\prime}$ 
through the surjections $\varphi_1: \widetilde{H} \rightarrow H$ 
and $\varphi_2: \widetilde{H} \rightarrow H^\prime$ respectively. 
In particular, we have $[M_H]^{fl}\sim [M_{H^\prime}]^{fl}$. 

When $i\neq j$, 
by using Proposition \ref{prop5.2}, Proposition \ref{prop5.4}, 
Proposition \ref{prop5.5}, Theorem \ref{th6.4}, 
we see that\\ $(i,j)\in J_{3,3}=\{(5,6)$, $(11,13)\}$ when 
{\rm (i)} $G$ and $G^\prime$ are 
$\GL(3,\bZ)$-conjugate to $N_{3,i}$ and $N_{3,j}$ $(1\leq i,j\leq 15)$,\\ 
$(i,j)\in J_{4,4}$ where 
$J_{4,4}=\{(2,5)$, $(2,7)$, $(3,8)$, $(5,7)$, $(13,23)$, 
$(13,24)$, $(14,39)$, $(15,16)$, $(15,17)$, 
$(15,25)$, $(15,26)$, $(16,17)$, $(16,25)$, 
$(16,26)$, $(17,25)$, $(17,26)$, $(19,29)$, 
$(19,30)$, $(20,31)$, $(20,33)$, $(20,35)$, 
$(22,46)$, $(23,24)$, $(25,26)$, $(29,30)$, 
$(31,33)$, $(31,35)$, $(32,34)$, $(33,35)$, 
$(41,44)$, $(59,91)$, $(60,92)$, $(63,93)$, 
$(65,94)$, $(66,96)$, $(67,95)$, $(68,99)$, 
$(69,74)$, $(69,75)$, $(71,77)$, $(71,78)$, 
$(72,83)$, $(72,85)$, $(74,75)$, $(77,78)$, 
$(81,84)$, $(81,87)$, $(83,85)$, 
$(84,87)$, $(86,88)$, $(102,103)\}$ when 
{\rm (ii)} $G$ and $G^\prime$ are 
$\GL(4,\bZ)$-conjugate to $N_{4,i}$ and $N_{4,j}$ $(1\leq i,j\leq 152)$,\\ 
and it does not occur when 
{\rm (iii)} $G$ and $G^\prime$ are 
$\GL(4,\bZ)$-conjugate to $I_{4,i}$ and $I_{4,j}$ $(1\leq i,j\leq 7)$. 

By the assumption (2) $G\simeq G^\prime$, $L_i=L_j^\prime$, 
we have 
$\varphi_1: {\rm Gal}(L_i/k)\xrightarrow{\sim} G\leq {\rm GL}(n,\bZ)$, 
$f\mapsto \varphi_1(f)$, 
$\varphi_2: {\rm Gal}(L_j^\prime/k)\xrightarrow{\sim} G^\prime\leq {\rm GL}(n,\bZ)$, 
$f\mapsto \varphi_2(f)$ and a subdirect product 
$\widetilde{H}=\{(\varphi_1(f),\varphi_2(f))\mid f\in {\rm Gal}(L_i/k)={\rm Gal}(L_j^\prime/k)\}\leq G\times G^\prime$ of $G$, $G^\prime$ with $\widetilde{H}\simeq G$. 

On the other hand, 
by Theorem \ref{th6.1} and Theorem \ref{th6.3}, 
we have  $[M_G]^{fl}=[M_{G^\prime}]^{fl}\ {\rm as}\ 
\widetilde{H}^\prime\textrm{-lattices}$ where $\widetilde{H}^\prime\simeq \{(g,g^\prime)=(\psi_1(\widetilde{h}),\psi_2(\widetilde{h})\mid \widetilde{h}\in \widetilde{H}^\prime\}\simeq G$ with surjections 
$\psi_1:\widetilde{H}^\prime\xrightarrow{\sim} G$, 
$\psi_2:\widetilde{H}^\prime\xrightarrow{\sim} G^\prime$. 

Then 
there exists $\tau\in {\rm Aut}(G)$ such that 
$(\psi_1\psi_2^{-1})(\varphi_2\varphi_1^{-1})(g)=g^\tau$ 
for any $g\in G$ where 
${\rm Aut}(G)$ is the group of automorphisms on $G$.
\\

We consider the following subgroups of ${\rm Aut}(G)$ for 
$G\leq {\rm GL}(n\,\bZ)$: 
\begin{align*}
{\rm Inn}(G)\leq X\leq Y\leq Z\leq {\rm Aut}(G),
\end{align*}
\begin{align*}
X&={\rm Aut}_{\GL(n,\bZ)}(G)=
\{\sigma\in{\rm Aut}(G)\mid 
{\rm there}\ {\rm exists}\ u\in {\rm GL}(n,\bZ)\ {\rm such}\ {\rm that}\ 
u^{-1}gu=g^\sigma\ {\rm for}\ {\rm any}\ g\in G\}\\
&\simeq N_{\GL(n,\bZ)}(G)/Z_{\GL(n,\bZ)}(G),\\
Y&=\{\sigma\in{\rm Aut}(G)\mid [M_G]^{fl}=[M_{G^\sigma}]^{fl}\ {\rm as}\ 
\widetilde{H}\textrm{-lattices}\ {\rm where}\ \widetilde{H}=\{(g,g^\sigma)\mid g\in G\}\simeq G\},\\
Z&=\{\sigma\in{\rm Aut}(G)\mid [M_H]^{fl}\sim [M_{H^\sigma}]^{fl}\ {\rm for}\ {\rm any}\ H\leq G\}
\end{align*} 
where 
${\rm Inn}(G)$ is the group of inner automorphisms on $G$, 
$N_{\GL(n,\bZ)}(G)$ is the normalizer of $G$ in $\GL(n,\bZ)$ and 
$Z_{\GL(n,\bZ)}(G)$ is the centralizer of $G$ in $\GL(n,\bZ)$.\\ 

Then we see that 
$[M_G]^{fl}=[M_{G^\prime}]^{fl}$ as $\widetilde{H}$-lattices 
if and only if $\tau \in Y$. 

We also see that 
$[M_H]^{fl}\sim [M_{H^\prime}]^{fl}$ 
as $\widetilde{H}\mid_{\varphi_1^{-1}(H)}$-lattices 
for any $H\leq G$ 
if and only if $\tau\in Z$ 
where $H^\prime=\varphi_2\varphi_1^{-1}(H)$
and 
$\widetilde{H}\mid_{\varphi_1^{-1}(H)}
=\{\widetilde{h}\in\widetilde{H}\mid \varphi_1(\widetilde{h})\in H\}\leq 
H\times H^\prime$ 
with $\varphi_1^{-1}(H)=\varphi_2^{-1}(H^\prime)$.\\

Then we find that ${\rm (2)}\Rightarrow {\rm (1)}$ of 
{Theorem \ref{thmain2}}, {Theorem \ref{thmain4}} and {Theorem \ref{thmain5}} 
is equivalent to $\tau\in Z\Rightarrow \tau \in Y$ 
for any $\tau\in {\rm Aut}(G)$ 
and hence 
we may assume that $i=j$, 
i.e. $G$ and $G^\prime$ are $\GL(n,\bZ)$-conjugate 
to the same $N_{3,i}$, $N_{4,i}$ or $I_{4,i}$, 
via $\tau\in {\rm Aut}(G)$ with 
$(\psi_1\psi_2^{-1})(\varphi_2\varphi_1^{-1})(g)=g^\tau$ 
for any $g\in G$. 
Moreover, we may take the smallest $i$ of 
$N_{3,i}$, $N_{4,i}$ or $I_{4,i}$ 
in the weak stably $k$-equivalent class ${\rm WSEC}_r$ $(1\leq r\leq 128)$.\\

From now on, we assume that $i=j$. 

In order to prove 
${\rm (2)}\Rightarrow {\rm (1)}$ of 
{Theorem \ref{thmain2}}, {Theorem \ref{thmain4}} and {Theorem \ref{thmain5}}, 
under the condition $L_i=L^\prime_i$, 
we should show that $Y=Z$ 
where 
\begin{align*}
Y&=\{\sigma\in{\rm Aut}(G)\mid [M_G]^{fl}=[M_{G^\sigma}]^{fl}\ {\rm as}\ 
\widetilde{H}\textrm{-lattices}\ {\rm where}\ \widetilde{H}=\{(g,g^\sigma)\mid g\in G\}\simeq G\},\\
Z&=\{\sigma\in{\rm Aut}(G)\mid [M_H]^{fl}\sim [M_{H^\sigma}]^{fl}\ {\rm for}\ {\rm any}\ H\leq G\}
\end{align*}
because the set $Y$ (resp. $Z$) can be regarded as the set 
of possibilities of $G^\prime=G^\sigma$ 
which satisfies the condition (1) (resp. the third condition of (2), i.e. 
$T_i\times_k K$ and $T_j^\prime\times_k K$ 
are weak stably birationally $K$-equivalent 
for any $k\subset K\subset L_i$). 
Indeed, we will show that $Y=Z$ 
except for the cases 
when $i=j=137,139,145,147$ (resp. $i=j=7$) 
of $N_{4,i}$ and $N_{4,j}$ (resp. $I_{4,i}$ and $I_{4,j}$). 

If $\sigma\in X\simeq N_{\GL(n,\bZ)}(G)/Z_{\GL(n,\bZ)}(G)$, then 
there exists $u\in {\rm GL}(n,\bZ)$ such that
$u^{-1}gu=g^\sigma$ for any $g\in G$. 
Hence, in order to get $Y=Z$, we should show that 
$\sigma_1,\ldots,\sigma_s\in Y$ 
for $\sigma_1,\ldots,\sigma_s\in Z$ 
where $\{\sigma_1,\ldots,\sigma_s\}$ 
is a complete set of representatives 
of the double coset $X\backslash Z/X$. 
Furthermore, we should check that $\sigma_1,\ldots,\sigma_t\in Y$ 
for some generators $\sigma_1,\ldots,\sigma_t\in Z$ 
of the double coset $X\backslash Z/X$, i.e. 
elements $\sigma_1,\ldots,\sigma_t\in Z$ 
which satisfy $\langle \sigma_1,\ldots,\sigma_t,x\mid x\in X\rangle=Z$. 
Note that the left coset (resp. the right coset) of $G$ under $X$ 
corresponds to the surjection $\varphi_1:\widetilde{H}\rightarrow G$ 
(resp. $\varphi_2:\widetilde{H}\rightarrow G^\sigma$).\\

For the exceptional cases $i=j=137,139,145,147$ (resp. $i=j=7$), 
the implications ${\rm (2)}\Rightarrow {\rm (1)}$ of 
{Theorem \ref{thmain4}} and {Theorem \ref{thmain5}} 
does not hold. 
Indeed, we will show that 
$X=Y\lhd Z$ with $Z/Y\simeq C_2,C_2^2,C_2,C_2$ (resp. $C_2$). 
Note that under the assumption $L_i=L^\prime_j$, i.e. 
$\widetilde{H}\simeq G\simeq G^\prime$, 
${\rm (1)}\Leftrightarrow {\rm (2)}$ of Theorem \ref{thmain4} 
is equivalent to $Y=Z$.\\

We made the following GAP \cite{GAP} algorithms 
in order to confirm that $Y=Z$.  
It is available as in \cite{BCAlgTori}.\\

Let $G\leq \GL(n,\bZ)$ 
and $M_G$ be the corresponding $G$-lattice of $\bZ$-rank $n$ 
as in Definition \ref{d2.2}.\\

\noindent 
{\tt IsomorphismFromSubdirectProduct($\widetilde{H}$}{\tt)} 
returns the isomorphism $\sigma: G/N\rightarrow G^\prime/N^\prime$ 
which satisfies $\sigma(\varphi_1(h)N)$ $=$ $\varphi_2(h)N^\prime$ 
for any $h\in\widetilde{H}$ 
where $N=\varphi_1({\rm Ker}(\varphi_2))$ and 
$N^\prime=\varphi_2({\rm Ker}(\varphi_1))$ 
for a subdirect product 
$\widetilde{H}\leq G\times G^\prime$ of $G$ and $G^\prime$ 
with surjections $\varphi_1: \widetilde{H} \rightarrow G$ 
and $\varphi_2: \widetilde{H} \rightarrow G^\prime$.\\

\noindent 
{\tt AutGSubdirectProductsWSECInvariant}($G$)\tt } 
returns subdirect products 
$\widetilde{H}_m$ $=$ 
$\{(g,g^{\sigma_m})\mid g\in G, g^{\sigma_m}\in G^{\sigma_m}\}$ 
$(1\leq m\leq s)$ of $G$ and $G^{\sigma_m}$ 
where 
$\{\sigma_1,\ldots,\sigma_s\}$ is a complete set of representatives 
of the double coset $X\backslash Z/X$,  
\begin{align*}
{\rm Inn}(G)\leq X\leq Y\leq Z\leq {\rm Aut}(G), 
\end{align*}
\begin{align*}
X&={\rm Aut}_{\GL(n,\bZ)}(G)=
\{\sigma\in{\rm Aut}(G)\mid 
{\rm there}\ {\rm exists}\ u\in {\rm GL}(n,\bZ)\ {\rm such}\ {\rm that}\ 
u^{-1}gu=g^\sigma\ {\rm for}\ {\rm any}\ g\in G\}\\
&\simeq N_{\GL(n,\bZ)}(G)/Z_{\GL(n,\bZ)}(G),\\
Y&=\{\sigma\in{\rm Aut}(G)\mid [M_G]^{fl}=[M_{G^\sigma}]^{fl}\ {\rm as}\ 
\widetilde{H}\textrm{-lattices}\ {\rm where}\ \widetilde{H}=\{(g,g^\sigma)\mid g\in G\}\simeq G\},\\
Z&=\{\sigma\in{\rm Aut}(G)\mid [M_H]^{fl}\sim [M_{H^\sigma}]^{fl}\ {\rm for}\ {\rm any}\ H\leq G\},
\end{align*} 
${\rm Inn}(G)$ is the group of inner automorphisms on $G$, 
${\rm Aut}(G)$ is the group of automorphisms on $G$,  
$N_{\GL(n,\bZ)}(G)$ is the normalizer of $G$ in $\GL(n,\bZ)$ and 
$Z_{\GL(n,\bZ)}(G)$ is the centralizer of $G$ in $\GL(n,\bZ)$.\\

\noindent 
{\tt AutGSubdirectProductsWSECInvariantGen(}$G${\tt)}
returns the same as 
{\tt AutGSubdirectProductsWSECInvariant}($G$)\tt } 
but with respect to 
$\{\sigma_1,\ldots,\sigma_t\}$ 
where $\sigma_1,\ldots,\sigma_t\in Z$ are some minimal number of generators 
of the double cosets of $X\backslash Z/X$, i.e. 
minimal number of elements $\sigma_1,\ldots,\sigma_t\in Z$ 
which satisfy $\langle \sigma_1,\ldots,\sigma_t,x\mid x\in X\rangle=Z$, 
instead of a complete set of representatives 
of the double coset $X\backslash Z/X$. 
If this returns {\tt []}, then we get $X=Y=Z$.\\

\noindent 
{\tt AutGLnZ(}$G${\tt )} 
returns 
\begin{align*}
X&={\rm Aut}_{\GL(n,\bZ)}(G)=
\{\sigma\in{\rm Aut}(G)\mid 
{\rm there}\ {\rm exists}\ u\in {\rm GL}(n,\bZ)\ {\rm such}\ {\rm that}\ 
u^{-1}gu=g^\sigma\ {\rm for}\ {\rm any}\ g\in G\}\\
&\simeq N_{\GL(n,\bZ)}(G)/Z_{\GL(n,\bZ)}(G),
\end{align*}

\noindent 
{\tt N3WSECMembersTable[}$r${\tt ][}$i${\tt ]} 
returns an integer $j$ which satisfies that $N_{3,j}$ is 
the $i$-th group in the weak stably $k$-equivalent class ${\rm WSEC}_r$.\\

\noindent 
{\tt N4WSECMembersTable[}$r${\tt ][}$i${\tt ]} 
is the same as 
{\tt N3WSECMembersTable[}$r${\tt ][}$i${\tt ]} 
but using $N_{4,j}$ instead of $N_{3,j}$.\\

\noindent 
{\tt I4WSECMembersTable[}$r${\tt ][}$i${\tt ]} 
is the same as 
{\tt N3WSECMembersTable[}$r${\tt ][}$i${\tt ]} 
but using $I_{4,j}$ instead of $N_{3,j}$.\\

\noindent 
{\tt AutWSEC(}$G${\tt )} 
returns 
\begin{align*}
Z=\{\sigma\in{\rm Aut}(G)\mid [M_H]^{fl}\sim [M_{H^\sigma}]^{fl}\ {\rm for}\ {\rm any}\ H\leq G\}. 
\end{align*}

\noindent 
{\tt IdCoset(}$G,H${\tt )} 
returns $[d,m]$ when the action of $G$ on $G/H$ may be regarded 
as $G\simeq dTm\leq S_d$ with $d=[G:H]>1$. 
If $d=1$, this returns ${\tt [}\ 1\ {\tt ]}$.\\

{\it Proof of ${\rm (2)}\Rightarrow {\rm (1)}$ of {Theorem \ref{thmain2}}.}

%
We should check the case of the first group $G=N_{3,i}$ 
$(i=1,2,3,4,5,7,8,9,10,11,12,14,15)$ in each weak stably $k$-equivalent class 
${\rm WSEC}_r$ $(1\leq r\leq 13)$ (see Table $2$). 

Applying the function 
{\tt AutGSubdirectProductsWSECInvariantGen(}$G${\tt)}, 
we get some minimal number of generators 
$\sigma_1,\ldots,\sigma_t\in Z$ of the double coset $X\backslash Z/X$, i.e. 
minimal number of elements $\sigma_1,\ldots,\sigma_t\in Z$ which satisfy 
$\langle \sigma_1,\ldots,\sigma_t,x\mid x\in X\rangle=Z$. 
We will show that $Y=Z$ for each case.\\

Case 1: $i=1,5,8,9,10,11,12,15$ with 
$[M_G]^{fl}\in {\rm WSEC}_r$ $(r=1,5,7,8,9,10,11,13)$. 
We see that the function {\tt AutGSubdirectProductsWSECInvariantGen(}$G${\tt)} 
returns {\tt []} and hence we get $X=Y=Z$ (see Example \ref{ex9.1} for GAP computations).\\

Case 2: $i=2,3,4,7$ with 
$[M_G]^{fl}\in {\rm WSEC}_r$ $(r=2,3,4,6)$. 
By using the same technique as in the proof of Theorem \ref{th6.1}, 
by applying the function 
{\tt PossibilityOfStablyEquivalentFSubdirectProduct($\widetilde{H}$)} 
we can get a basis $\mathcal{P}=\{P_1,\dots,P_m\}$ 
of the solution space of $G_1P=PG_2$ 
with det $P_s=\pm 1$ for some $1\leq s\leq m$  
where $G_1$ (resp. $G_2$) 
is the matrix representation group of the action of 
$\widetilde{H}$ on the left-hand side 
(resp. the right-hand side) of the isomorphism (\ref{eqiso2}) 
and 
we may confirm that $F\simeq F^\prime$ as $\widetilde{H}$-lattices 
where $F=[M_G]^{fl}$ and $F^\prime=[M_{G^{\sigma_m}}]^{fl}$ 
for any $1\leq m\leq t$ with $t=1,1,2,2$ when $i=2,3,4,7$ respectively 
(see Example \ref{ex9.1} for GAP computations).\\

Case 3: $i=14$ with $[M_G]^{fl}\in {\rm WSEC}_r$ $(r=12)$. 
By using the same technique as in Case 3 of the proof of Theorem \ref{th6.2}, 
by applying the function 
{\tt PossibilityOfStablyEquivalentFSubdirectProduct($\widetilde{H}$)}  
and {\tt SearchP1}($\mathcal{P}$), 
we may confirm that $F\simeq F^\prime$ as $\widetilde{H}$-lattices 
where $F=[M_G]^{fl}$ and $F^\prime=[M_{G^{\sigma_m}}]^{fl}$ 
for any $1\leq m\leq t$ with $t=1$ (see Example \ref{ex9.1} for GAP computations).\\

This implies that $Y=Z$ and the proof is completed.\qed\\

{\it Proof of the last statement of {Theorem \ref{thmain2}}.}

Because we already have $(1)\Leftrightarrow (2)$ of {Theorem \ref{thmain2}}, 
we get 
\begin{align*}
{\rm WSEC}_r=\coprod_{L/k\atop {\rm Gal}(L/k)\simeq {\rm G}_r}
{\rm WSEC}_{r,L},\ 
{\rm WSEC}_{r,L}=
\coprod_{t=1}^{\lambda_{r}} {\rm SEC}_{r,L,t}
\end{align*}
modulo stably birationally $k$-equivalence $\stackrel{\rm s.b.}{\approx}$ 
where 
${\rm SEC}_{r,L,t}$ $(1\leq t\leq \lambda_{r})$ is the $t$-th 
stably $k$-equivalent class of $T$ of dimension $3$ in 
${\rm WSEC}_{r,L}$ which corresponds to the fixed minimal splitting field $L$ 
in 
${\rm WSEC}_r=\coprod_{L/k\atop {\rm Gal}(L/k)\simeq {\rm G}_r} {\rm WSEC}_{r,L}$ 
with 
$[\widehat{T}]^{fl}\in {\rm WSEC}_{r,L}
=\coprod_{t=1}^{\lambda_r} {\rm SEC}_{r,L,t}$, 
${\rm Gal}(L/k)\simeq {\rm G}_r\simeq N_{3,i}$ 
$(1\leq r\leq 13)$ and 
$\lambda_r=|{\rm WSEC}_{r,L}|$. 

We see that $\lambda_r=|{\rm WSEC}_{r,L}|=|Y\backslash {\rm Aut}(G)|$ 
because we have one to one correspondence 
$\overline{\sigma}=Y \sigma\in Y\backslash {\rm Aut}(G)$ $\leftrightarrow$ 
$[M_{G^\sigma}]^{fl}\in{\rm WSEC}_{r,L}$. 
We can get $\lambda_r=|{\rm WSEC}_{r,L}|$ as 
$\lambda_r=|Y\backslash {\rm Aut}(G)|$ 
by using the GAP functions {\tt AutWSEC(}$G${\tt )}, 
{\tt AutomorphismGroup(}$G${\tt )}, {\tt IdCoset(}$G$,$H${\tt )} as 
$Y=$ {\tt AutWSEC(}$G${\tt )}, 
${\rm Aut}(G)=$ {\tt AutomorphismGroup(}$G${\tt )}, 
$\lambda_r=$ {\tt IdCoset}$({\rm Aut}(G),Y){\tt [}$\,$1$\,${\tt ]}$ 
which is given as in Table $2$ (see Example \ref{ex9.2} for GAP computations). 
\qed\\

~\\
{\it Proof of ${\rm (2)}\Rightarrow {\rm (1)}$ of {Theorem \ref{thmain4}}.}

%
The case $G=N_{4,i}$ with $[M_G]^{fl}\in {\rm WSEC}_r$ $(1\leq r\leq 13)$ 
can be reduced to Theorem \ref{thmain2}. 
Hence we should check the case of the first group $G=N_{4,i}$ 
in each weak stably $k$-equivalent class 
${\rm WSEC}_r$ $(14\leq r\leq 121)$ (see Table $4$). 

Applying the function 
{\tt AutGSubdirectProductsWSECInvariantGen(}$G${\tt)}, 
we get some minimal number of generators 
$\sigma_1,\ldots,\sigma_t\in Z$ of the double coset $X\backslash Z/X$, i.e. 
minimal number of elements $\sigma_1,\ldots,\sigma_t\in Z$ which satisfy 
$\langle \sigma_1,\ldots,\sigma_t,x\mid x\in X\rangle=Z$. 
We will show that $Y=Z$ except for the cases 
when $i=137,139,145,147$ of $N_{4,i}$. \\

Case 1: $i=9,12,18,21,32,40,42,43,45,47,59,60,63,65,66,67,68,70,76,80,86,102,106,107,108,109,110$, $111,112,113,115,116,117,122,123,126,127,129,131,138,140,141,142,143,144,146,148,149,150,151,152$ 
with\\ 
$[M_G]^{fl}\in {\rm WSEC}_r$ 
$(r=16,19,21,23,27,31,33,34,35,36,48,49,52,54,55,56,57,58,60,62,65,72,75,76,77,78,79$,\\
$80,81,82,84,85,86,91,92,95,96,98,100,107,109,110,111,112,113,115,117,118,119,120,121)$. 
The function\\ {\tt AutGSubdirectProductsWSECInvariantGen(}$G${\tt)} 
returns {\tt []} and hence we get $X=Y=Z$ (see Example \ref{ex9.3} for GAP computations).\\

Case 2: $i=3,4,10,11,14,20,22,27,28,36,37,38,41,48,49,51,52,54,55,56,89$ 
with 
$[M_G]^{fl}\in {\rm WSEC}_r$ 
$(r=14,15,17,18,20,22,24,25,26,28,29,30,32,37,38,40,41,43,44,45,66)$. 
By using the same technique as in the proof of Theorem \ref{th6.1}, 
by applying the function 
{\tt PossibilityOfStablyEquivalentFSubdirectProduct($\widetilde{H}$)} 
we can get a basis $\mathcal{P}=\{P_1,\dots,P_m\}$ 
of the solution space of $G_1P=PG_2$ 
with det $P_s=\pm 1$ for some $1\leq s\leq m$ 
where $G_1$ (resp. $G_2$) 
is the matrix representation group of the action of 
$\widetilde{H}$ on the left-hand side 
(resp. the right-hand side) of the isomorphism (\ref{eqiso2}) 
and 
we may confirm that $F\simeq F^\prime$ as $\widetilde{H}$-lattices 
where $F=[M_G]^{fl}$ and $F^\prime=[M_{G^{\sigma_m}}]^{fl}$ 
for any $1\leq m\leq t$ with $t=1,1,1,1,1,2,2,2,1,1,2,1,2,1,1,1,2,2,1,1,2$ when $i=3,4,10,11,14,20,22,27,28,36,37,38,41,48,49,51,52,54$, 
$55,56,89$ respectively (see Example \ref{ex9.3} for GAP computations).\\

Case 3: $i=58,61,64,73,82,90,97,100,101,104,105$ 
with 
$[M_G]^{fl}\in {\rm WSEC}_r$ 
$(r=47,50,53,59,64,67,68,70$, $71,73,74)$. 
We have $t=1$, i.e. $\langle \sigma_1,x\mid x\in X\rangle=Z$. 
As in Case 2 of the proof of Theorem \ref{th6.2}, 
by applying 
{\tt PossibilityOfStablyEquivalentFSubdirectProduct($\widetilde{H}$)}, 
{\tt SearchPRowBlocks}($\mathcal{P}$) 
and {\tt SearchPMergeRowBlock}, 
we get a matrix $P$ with $G_1P=PG_2$ and det $P=\pm 1$ 
where $G_1$ (resp. $G_2$) 
is the matrix representation group of the action of 
$\widetilde{H}$ on the left-hand side 
(resp. the right-hand side) of the isomorphism (\ref{eqiso2}).  
Then we have $[F]=[F^\prime]$ as $\widetilde{H}$-lattices 
where $F=[M_G]^{fl}$ and $F^\prime=[M_{G^{\sigma_1}}]^{fl}$ 
(see Example \ref{ex9.3}). 

For example, when $i=100$, we see that there exists a possibility 
$\bZ[\widetilde{H}/H_{19}]\oplus F\simeq \bZ[\widetilde{H}/H_{18}]\oplus F^\prime$ with ${\rm rank}_\bZ \bZ[\widetilde{H}/H_{18}]={\rm rank}_\bZ \bZ[\widetilde{H}/H_{19}]=6$ and ${\rm rank}_\bZ F={\rm rank}_\bZ F^\prime=26$. 
We do not know whether this isomorphism indeed occurs. 
However, after adding $\bZ[\widetilde{H}/H_{20}]$ with 
${\rm rank}_\bZ \bZ[\widetilde{H}/H_{20}]=4$ to both sides, 
we may establish the isomorphism 
$\bZ[\widetilde{H}/H_{19}]\oplus\bZ[\widetilde{H}/H_{20}]\oplus F
\simeq 
\bZ[\widetilde{H}/H_{18}]\oplus\bZ[\widetilde{H}/H_{20}]\oplus F^\prime$. 
The remaining cases can be obtained by using the same manner 
(see Example \ref{ex9.3} for GAP computations).\\

Case 4: $i=50,53,57,62,79,81$ 
with 
$[M_G]^{fl}\in {\rm WSEC}_r$ 
$(r=39,42,46,51,61,63)$. 
As in Case 3 of the proof of Theorem \ref{th6.2}, 
by applying the functions 
{\tt PossibilityOfStablyEquivalentFSubdirectProduct($\widetilde{H}$)}, 
{\tt SearchPRowBlocks}($\mathcal{P}$) 
and 
{\tt SearchP1}($\mathcal{P}$), 
we get a matrix $P$ with $G_1P=PG_2$ and det $P=\pm 1$ 
where $G_1$ (resp. $G_2$) 
is the matrix representation group of the action of 
$\widetilde{H}$ on the left-hand side 
(resp. the right-hand side) of the isomorphism (\ref{eqiso2}).  
This implies that $F\simeq F^\prime$ as $\widetilde{H}$-lattices 
where $F=[M_G]^{fl}$ and $F^\prime=[M_{G^{\sigma_m}}]^{fl}$ 
for any $1\leq m\leq t$ with $t=1,1,1,1,2,1$ 
when $i=50,53,57,62,79,81$ respectively (see Example \ref{ex9.3} for GAP computations).\\

Case 5: $i=98,114,118,119,120,121,124,125,128,130,132,133,134,135,136$ 
with 
$[M_G]^{fl}\in {\rm WSEC}_r$ 
$(r=69,83,87,88$, $89,90,93,94,97,99,101,102,103,104,105)$. 
We have $t=1$, i.e. $\langle \sigma_1,x\mid x\in X\rangle=Z$. 
By applying the functions 
{\tt PossibilityOfStablyEquivalentFSubdirectProduct($\widetilde{H}$)}, 
{\tt SearchPRowBlocks}($\mathcal{P}$) 
and {\tt SearchPLinear($M,\mathcal{P}_1$)} (see Section \ref{S6}), 
we get a matrix $P$ with $G_1P=PG_2$ and det $P=\pm 1$ 
where $G_1$ (resp. $G_2$) 
is the matrix representation group of the action of 
$\widetilde{H}$ on the left-hand side 
(resp. the right-hand side) of the isomorphism (\ref{eqiso2}). 
Hence we have $[F]=[F^\prime]$ as 
$\widetilde{H}$-lattices 
where $F=[M_G]^{fl}$ and $F^\prime=[M_{G^{\sigma_1}}]^{fl}$ 
(see Example \ref{ex9.3}). 

For example, when $i=98$, 
we see that there exists a possibility 
$\bZ[\widetilde{H}/H_7]\oplus F\simeq \bZ[\widetilde{H}/H_6]\oplus F^\prime$ 
with ${\rm rank}_\bZ \bZ[\widetilde{H}/H_6]={\rm rank}_\bZ \bZ[\widetilde{H}/H_7]=6$ and ${\rm rank}_\bZ F={\rm rank}_\bZ F^\prime=26$. 
We do not know whether this isomorphism indeed occurs. 
However, after adding $\bZ[\widetilde{H}/H_8]\oplus\bZ$ with 
${\rm rank}_\bZ \bZ[\widetilde{H}/H_8]=4$ to both sides, 
we may establish the isomorphism 
$\bZ[\widetilde{H}/H_7]\oplus\bZ[\widetilde{H}/H_8]\oplus\bZ\oplus F
\simeq 
\bZ[\widetilde{H}/H_6]\oplus\bZ[\widetilde{H}/H_8]\oplus\bZ\oplus F^\prime$. 
The remaining cases can be obtained by using the same manner 
(see Example \ref{ex9.3} for GAP computations).\\

We conclude that $Y=Z$ except for the cases $i=137,139,145,147$ 
and the proof is completed.\qed\\

{\it Proof of ${\rm (2)}\Rightarrow {\rm (1)}$ of {Theorem \ref{thmain5}}.}

As in Case 1 of the proof of ${\rm (2)}\Rightarrow {\rm (1)}$ of 
{Theorem \ref{thmain2}} and {Theorem \ref{thmain4}}, 
we find that 
the function\\ {\tt AutGSubdirectProductsWSECInvariantGen(}$G${\tt)} 
returns {\tt []} and hence we get $X=Y=Z$ 
(see Example \ref{ex9.4} for GAP computations).\qed\\

We will show {Theorem \ref{thmain4}} and {Theorem \ref{thmain5}} 
for the exceptional cases $i=j=137,139,145,147$ (resp. $i=j=7$) below. 
We should show that 
$X=Y\lhd Z$ with $Z/Y\simeq C_2,C_2^2,C_2,C_2$ (resp. $C_2$).\\

{\it Proof of {Theorem \ref{thmain4}} for the exceptional cases $i=j=137,139,145,147$.} 

For $i=j=137,139,145,147$, we will show that 
$X=Y\lhd Z$ with $Z/Y\simeq C_2,C_2^2,C_2,C_2$ respectively.\\

The case where  $i=j=137$ with $[M_G]^{fl}\in {\rm WSEC}_{106}$: $G\simeq Q_8\times C_3$.

We obtain that $Z(G)\simeq C_6$, 
${\rm Inn}(G)\simeq G/Z(G)\simeq C_2^2\leq 
X={\rm Aut}_{\GL(4,\bZ)}(G)\simeq S_4\leq 
Y\leq Z\simeq {\rm Aut}(G)\simeq S_4\times C_2$ 
with $[{\rm Aut}(G):X]=2$. 
Applying the function 
{\tt AutGSubdirectProductsWSECInvariantGen(}$G${\tt)}, 
we see that $t=1$, 
i.e. $\langle \sigma_1,x\mid x\in X\rangle=Z$. 
As in the proof of Theorem \ref{th6.1}, by applying the function 
{\tt PossibilityOfStablyEquivalentFSubdirectProduct($\widetilde{H}$)}, 
we get the only possibility is $F-F^\prime=0$ 
in the sense of the equation (\ref{eqpos2}) in Section \ref{S6}, 
i.e. $F\oplus Q\simeq F^\prime\oplus Q$ for some permutation 
$\widetilde{H}$-lattice $Q$, 
where $F=[M_G]^{fl}$ and $F^\prime=[M_{G^{\sigma_1}}]^{fl}$ 
with ${\rm rank}_\bZ F={\rm rank}_\bZ F^\prime=20$. 

Applying the function 
{\tt StablyEquivalentFCheckPSubdirectProduct($\widetilde{H}$,LHSlist($l_1$),RHSlist($l_1$))}, 
we obtain a basis 
$\mathcal{P}=\{P_1,\dots,P_{18}\}$ of the solution space of $G_1P=PG_2$. 

Suppose that $F\oplus Q\simeq F^\prime\oplus Q$ for some permutation 
$\widetilde{H}$-lattice $Q$. 
Then we have $F_p\simeq F^\prime_p$ with 
${\rm rank}_{\bF_p} F={\rm rank}_{\bF_p} F^\prime=20$
where $F_p=F\otimes_\bZ \bF_p$ and $F^\prime_p=F^\prime\otimes_\bZ \bF_p$ 
because Krull-Schmidt-Azumaya theorem holds for $\bF_p[G]$-lattices 
(Curtis and Reiner \cite[Theorem 6.12]{CR81}, see also Section \ref{S7}). 
We should have 
${\rm rank}_{\bF_p}$ 
$\sum_{f\in {\rm Hom}_{\bF_p[G]}(F_p,F^\prime_p)}{\rm Im}(f)$ $=$ 
${\rm rank}_{\bF_p} ({}^tP_1|\cdots|{}^tP_{18})\otimes_\bZ \bF_p=20$ 
where ${}^tP_i$ is 
the transposed matrix of $P_i$. 
However, we can check that 
${\rm rank}_{\bF_2} ({}^tP_1|\cdots|{}^tP_{18})\otimes_\bZ \bF_2=18$. 
This yields a contradiction. 

We conclude that $[F]\neq [F^\prime]$ as $\widetilde{H}$-lattices 
($\bF_2[\widetilde{H}]$-lattices). 
This implies that $X=Y<Z$ with $[Z:X]=2$. 

We can apply the same method for $M_G$ and $M_{G^{\sigma_1}}$ 
instead of $F$ and $F^\prime$. 
Then we get that $\mathcal{P}=\{P_1,P_2\}$ and 
${\rm rank}_{\bF_2} ({}^tP_1|{}^tP_2)\otimes_\bZ \bF_2=2<4={\rm rank}_{\bF_2} M_G\otimes \bF_2$. 
Hence we have 
$M_G\otimes_\bZ \bF_2\not\simeq M_{G^{\sigma_1}}\otimes_\bZ \bF_2$ 
as $\bF_2[\widetilde{H}]$-lattices 
(see Example \ref{ex9.5} for GAP computations).\\

The case where $i=j=139$ with $[M_G]^{fl}\in {\rm WSEC}_{108}$: $G\simeq (Q_8\times C_3)\rtimes C_2$. 

We obtain that $Z(G)\simeq C_2$, 
${\rm Inn}(G)\simeq G/Z(G)\simeq (C_6\times C_2)\rtimes C_2\simeq 
X={\rm Aut}_{\GL(4,\bZ)}(G)\leq 
Y\leq Z\simeq {\rm Aut}(G)\simeq S_3\times D_4\times C_2$ 
with $[{\rm Aut}(G):X]=4$. 
Applying the function 
{\tt AutGSubdirectProductsWSECInvariant(}$G${\tt)}, 
we see that $s=4$, i.e. 
$\{\sigma_1,\sigma_2,\sigma_3,\sigma_4\}$ is a complete set 
of representatives of the double coset $X\backslash Z/X$. 
We may assume that $\sigma_1\in X$. 

For $m=2,3,4$, 
we have $F=[M_{G^{\sigma_1}}]^{fl}$ 
and $F^\prime=[M_{G^{\sigma_m}}]^{fl}$ 
with ${\rm rank}_\bZ F={\rm rank}_\bZ F^\prime=20$. 

As in the proof of Theorem \ref{th6.1}, we see that 
{\tt PossibilityOfStablyEquivalentFSubdirectProduct($\widetilde{H}$)} 
returns a basis $\mathcal{L}=\{l_1,l_2,l_3,l_4\}$ of the solution space 
(see Section 5) where 
\begin{align*}
l_1&=[ 1, 0, -2, -1, 0, 0, 0, 0, 2, -1, 1, -1, 0, 0, 0, 2, 1, -1, 1, -2, 0 ],\\
l_2&=[ 0, 1, 0, 0, 0, 0, -2, -1, 0, -1, 1, -1, 0, 0, 2, 2, 1, -1, 1, -2, 0 ],\\
l_3&=[ 0, 0, 0, 0, 0, 1, 0, 0, 0, -1, -1, -1, 0, -1, 0, 2, 1, 1, 1, -2, 0 ],\\
l_4&=[ 0, 0, 0, 0, 0, 0, 0, 0, 0, 0, 0, 0, 0, 0, 0, 0, 0, 0, 0, 0, 1 ].
\end{align*}
Hence we consider the possibilities 
\begin{align}
&\bZ[\widetilde{H}]\oplus \bZ[\widetilde{H}/H_9]^{\oplus 2}\oplus \bZ[\widetilde{H}/H_{11}]\oplus \bZ[\widetilde{H}/H_{16}]^{\oplus 2}\oplus \bZ[\widetilde{H}/H_{17}]\oplus \bZ[\widetilde{H}/H_{19}]\label{iso1}\\
&\simeq \bZ[\widetilde{H}/H_3]^{\oplus 2}\oplus \bZ[\widetilde{H}/H_4]\oplus \bZ[\widetilde{H}/H_{10}]\oplus \bZ[\widetilde{H}/H_{12}]\oplus \bZ[\widetilde{H}/H_{18}]^{\oplus 2}\oplus \bZ^{\oplus 2},\nonumber\\
&\bZ[\widetilde{H}/H_2]\oplus \bZ[\widetilde{H}/H_{11}]\oplus \bZ[\widetilde{H}/H_{15}]^{\oplus 2}\oplus \bZ[\widetilde{H}/H_{16}]^{\oplus 2}\oplus \bZ[\widetilde{H}/H_{17}]\oplus \bZ[\widetilde{H}/H_{19}]\label{iso2}\\
&\simeq \bZ[\widetilde{H}/H_7]^{\oplus 2}\oplus \bZ[\widetilde{H}/H_8]\oplus \bZ[\widetilde{H}/H_{10}]\oplus \bZ[\widetilde{H}/H_{12}]\oplus \bZ[\widetilde{H}/H_{18}]\oplus \bZ^{\oplus 2},\nonumber\\
&\bZ[\widetilde{H}/H_6]\oplus \bZ[\widetilde{H}/H_{16}]^{\oplus 2}\oplus \bZ[\widetilde{H}/H_{17}]\oplus \bZ[\widetilde{H}/H_{18}]\oplus \bZ[\widetilde{H}/H_{19}]\label{iso3}\\
&\simeq \bZ[\widetilde{H}/H_{10}]\oplus \bZ[\widetilde{H}/H_{11}]\oplus \bZ[\widetilde{H}/H_{12}]\oplus \bZ[\widetilde{H}/H_{14}]\oplus \bZ^{\oplus 2},\nonumber\\
&F\simeq F^\prime\nonumber
\end{align}
in the sense of the equation (\ref{eqpos2-1}) in Section \ref{S6}. 

For $u=1,2,3$, applying 
{\tt StablyEquivalentFCheckPSubdirectProduct($\widetilde{H}$,LHSlist($l_u$),RHSlist($l_u$))}, 
we can confirm that the isomorphisms 
(\ref{iso1}), (\ref{iso2}), (\ref{iso3}) indeed hold 
after tensoring $\bF_2$ (resp. $\bF_3$). 

Also applying 
{\tt StablyEquivalentFCheckPSubdirectProduct($\widetilde{H}$,LHSlist($l_4$),RHSlist($l_4$))}, 
we obtain a basis 
$\mathcal{P}=\{P_1,\dots,P_{10}\}$ of the solution space of $G_1P=PG_2$. 

Suppose that $[F]=[F^\prime]$ as $\widetilde{H}$-lattices. 
Then for $p=2,3$ we have that 
$F_p\simeq F^\prime_p$ with 
${\rm rank}_{\bF_p} F={\rm rank}_{\bF_p} F^\prime=20$ 
where $F_p=F\otimes_\bZ \bF_p$ and $F^\prime_p=F^\prime\otimes_\bZ \bF_p$ 
because Krull-Schmidt-Azumaya theorem holds for $\bF_p[G]$-lattices 
(Curtis and Reiner \cite[Theorem 6.12]{CR81}, see also Section \ref{S7}). 
We should have 
${\rm rank}_{\bF_p}$ 
$\sum_{f\in {\rm Hom}_{\bF_p[G]}(F_p,F^\prime_p)}{\rm Im}(f)$ $=$ 
${\rm rank}_{\bF_p} ({}^tP_1|\cdots|{}^tP_{10})\otimes_\bZ \bF_p=20$ 
where ${}^tP_i$ is 
the transposed matrix of $P_i$. 
However, we can check that 
${\rm rank}_{\bF_2} ({}^tP_1|\cdots|{}^tP_{10})\otimes_\bZ \bF_2=20$ 
(resp. $18$, $18$, $20$)
and 
${\rm rank}_{\bF_3} ({}^tP_1|\cdots|{}^tP_{10})\otimes_\bZ \bF_3=20$ 
(resp. $20$, $18$, $18$) 
for $m=1$ (resp. $2$, $3$, $4$). 
This yields a contradiction. 
We conclude that $[F]\neq [F^\prime]$ as $\widetilde{H}$-lattices 
($\bF_2[\widetilde{H}]$-lattices or $\bF_3[\widetilde{H}]$-lattices). 
We also obtain that $X=Y<Z$ with $Z/Y\simeq C_2^2$. 

We can apply the same method for $M_{G^{\sigma_1}}$ and $M_{G^{\sigma_m}}$ 
instead of $F$ and $F^\prime$. 
Then we get that $\mathcal{P}=\{P_1\}$ and 
${\rm rank}_{\bF_2} ({}^tP_1)\otimes_\bZ \bF_2=4$ (resp. $2$, $2$, $4$) and 
${\rm rank}_{\bF_2} ({}^tP_1)\otimes_\bZ \bF_2=4$ (resp. $4$, $2$, $2$)  
for $m=1$ (resp. $2$, $3$, $4$). 
Hence we have 
$M_{G^{\sigma_1}}\otimes_\bZ \bF_2\not\simeq M_{G^{\sigma_m}}\otimes_\bZ \bF_2$ as $\bF_2[\widetilde{H}]$-lattices 
for $m=2,3$
and 
$M_{G^{\sigma_1}}\otimes_\bZ \bF_3\not\simeq M_{G^{\sigma_m}}\otimes_\bZ \bF_3$ $\bF_3[\widetilde{H}]$-lattices 
for $m=3,4$ 
(see Example \ref{ex9.5} for GAP computations).\\

The case where $i=j=145$ with $[M_G]^{fl}\in {\rm WSEC}_{114}$: $G\simeq \SL(2,\bF_3)\rtimes C_4$. 

We obtain that $Z(G)\simeq C_4$, 
${\rm Inn}(G)\simeq G/Z(G)\simeq S_4\leq 
X={\rm Aut}_{\GL(4,\bZ)}(G)\leq 
Y\leq Z\simeq {\rm Aut}(G)\simeq S_4\times C_2^2$ 
with $[{\rm Aut}(G):X]=2$. 
Applying the function 
{\tt AutGSubdirectProductsWSECInvariantGen(}$G${\tt)}, 
we see that $t=1$, 
i.e. $\langle \sigma_1,x\mid x\in X\rangle=Z$. 

We have $F=[M_G]^{fl}$ 
and $F^\prime=[M_{G^{\sigma_1}}]^{fl}$ 
with ${\rm rank}_\bZ F={\rm rank}_\bZ F^\prime=20$.

As in the proof of Theorem \ref{th6.1}, we see that 
{\tt PossibilityOfStablyEquivalentFSubdirectProduct($\widetilde{H}$)} 
returns a basis $\mathcal{L}=\{l_1,l_2\}$ of the solution space 
(see Section 5) where 
\begin{align*}
l_1&=[ 0, 0, 0, 0, 0, 0, 1, 0, 0, 0, 0, 0, 0, 0, -2, -1, -1, 0, 0, 0, 2, 2, 1, -2, 0 ],\\
l_2&=[ 0, 0, 0, 0, 0, 0, 0, 0, 0, 0, 0, 0, 0, 0, 0, 0, 0, 0, 0, 0, 0, 0, 0, 0, 1 ]. 
\end{align*}
Hence we consider the possibilities 
\begin{align}
&\bZ[\widetilde{H}/H_7]\oplus\bZ[\widetilde{H}/H_{21}]^{\oplus 2}\oplus\bZ[\widetilde{H}/H_{22}]^{\oplus 2}\oplus\bZ[\widetilde{H}/H_{23}]
\simeq \bZ[\widetilde{H}/H_{15}]^{\oplus 2}\oplus\bZ[\widetilde{H}/H_{16}]\oplus\bZ[\widetilde{H}/H_{17}]\oplus\bZ^{\oplus 2}\label{iso4}\\
&F\simeq F^\prime\nonumber
\end{align}
in the sense of the equation (\ref{eqpos2-1}) in Section \ref{S6}. 

Applying 
{\tt StablyEquivalentFCheckPSubdirectProduct($\widetilde{H}$,LHSlist($l_1$),RHSlist($l_1$))}, 
we obtain a basis 
$\mathcal{P}=\{P_1,\dots,P_{90}\}$ of the solution space of $G_1P=PG_2$ 
and we can confirm that the isomorphism (\ref{iso4}) indeed holds 
after tensoring $\bF_3$. 

Suppose that $[F]=[F^\prime]$ as $\widetilde{H}$-lattices. 
Then we get  
$F_3\simeq F^\prime_3$ with 
${\rm rank}_{\bF_3} F={\rm rank}_{\bF_3} F^\prime=20$ 
where $F_3=F\otimes_\bZ \bF_3$ and $F^\prime_3=F^\prime\otimes_\bZ \bF_3$ 
because Krull-Schmidt-Azumaya theorem holds for $\bF_p[G]$-lattices 
(Curtis and Reiner \cite[Theorem 6.12]{CR81}, see also Section \ref{S7}). 
We should have 
${\rm rank}_{\bF_3}$ 
$\sum_{f\in {\rm Hom}_{\bF_p[G]}(F_3,F^\prime_3)}{\rm Im}(f)$ $=$ 
${\rm rank}_{\bF_3} ({}^tP_1|\cdots|{}^tP_{10})\otimes_\bZ \bF_3=20$. 
However, we can check that 
${\rm rank}_{\bF_3} ({}^tP_1|\cdots|{}^tP_{10})\otimes_\bZ \bF_3=16$. 
 
This yields a contradiction. 
We conclude that $[F]\neq [F^\prime]$ as $\widetilde{H}$-lattices 
($\bF_3[\widetilde{H}]$-lattices). 
This implies that $X=Y<Z$ with $[Z:X]=2$. 

We can apply the same method for $M_G$ and $M_{G^{\sigma_1}}$ 
instead of $F$ and $F^\prime$. 
Then we find that the function 
{\tt PossibilityOfStablyEquivalentFSubdirectProduct($\widetilde{H}$)} 
returns {\tt []}. 
This implies that 
$M_G\not\simeq M_{G^{\sigma_1}}$ as $\widetilde{H}$-lattices and 
$M_G\otimes_\bZ \bF_3\not\simeq M_{G^{\sigma_1}}\otimes_\bZ \bF_3$ 
as $\bF_3[\widetilde{H}]$-lattices 
(see Example \ref{ex9.5} for GAP computations).\\

The case where $i=j=147$ with $[M_G]^{fl}\in {\rm WSEC}_{116}$: $G\simeq (\GL(2,\bF_3)\rtimes C_2)\rtimes C_2\simeq 
(\SL(2,\bF_3)\rtimes C_4)\rtimes C_2$. 

We obtain that $Z(G)\simeq C_2$, 
${\rm Inn}(G)\simeq G/Z(G)\simeq (A_4\times C_2^2)\rtimes C_2
\leq X={\rm Aut}_{\GL(4,\bZ)}(G)\leq 
Y\leq Z\leq {\rm Aut}(G)\simeq S_4\times D_4\times C_2$ 
with ${\rm Aut}(G)/X\simeq C_2^2$. 

We find that $H\lhd G$ with $[G:H]=2$ 
which satisfies $H\simeq N_{4,145}$ with $[N_{4,145}]^{fl}\in 
{\rm WSEC}_{114}$. 

Applying the function 
{\tt AutGSubdirectProductsWSECInvariantGen(}$G${\tt)}, 
we see that $t=1$, 
i.e. $\langle \sigma_1,x\mid x\in X\rangle=Z$. 
Then we see that $Z=\langle X,\sigma_1\rangle$ is of order $192$ 
with $[Z:X]=2$ and $[{\rm Aut}(G):Z]=2$. 

Applying the function 
{\tt IsomorphismFromSubdirectProduct($\widetilde{H}$}{\tt)}, 
we obtain the automorphism $\sigma_1\in {\rm Aut}(G)$ 
and for this $\sigma_1$ we can check that 
$[M_H]^{fl}\neq [M_{H^{\sigma_1|_H}}]^{fl}$. 
This implies that $[M_G]^{fl}\neq [M_{G^{\sigma_1}}]^{fl}$. 
Hence we have $X=Y<Z$ with $Z/Y\simeq C_2$ and ${\rm Aut}(G)/Z\simeq C_2$. 

It follows from the case 
where $i=j=145$ with $[M_G]^{fl}\in {\rm WSEC}_{114}$ that 
$M_G\otimes_\bZ \bF_3\not\simeq M_{G^{\sigma_1}}\otimes_\bZ \bF_3$ 
as $\bF_3[\widetilde{H}]$-lattices 
(see Example \ref{ex9.5} for GAP computations).\qed\\

{\it Proof of {Theorem \ref{thmain5}} for the exceptional case $i=j=7$.}

We will show that 
$X=Y\lhd Z$ with $Z/Y\simeq C_2$. 

We obtain that $Z(G)\simeq C_4$, 
${\rm Inn}(G)\simeq G/Z(G)\simeq S_3
\leq X={\rm Aut}_{\GL(4,\bZ)}(G)\simeq D_6\leq 
Y\leq Z\simeq {\rm Aut}(G)\simeq S_3\times C_2^2$ 
with $[{\rm Aut}(G):X]=2$.

Applying 
{\tt AutGSubdirectProductsWSECInvariantGen(}$G${\tt)}, 
we see that $t=1$, 
i.e. $\langle \sigma_1,x\mid x\in X\rangle=Z$. 
As in the proof of Theorem \ref{th6.1}, by applying the function 
{\tt PossibilityOfStablyEquivalentFSubdirectProduct($\widetilde{H}$)}, 
we get the only possibility is $F-F^\prime=0$ 
in the sense of the equation (\ref{eqpos2}) in Section \ref{S6}, 
i.e. $F\oplus Q\simeq F^\prime\oplus Q$ for some permutation 
$\widetilde{H}$-lattice $Q$, 
where $F=[M_G]^{fl}$ and $F^\prime=[M_{G^{\sigma_1}}]^{fl}$ 
with ${\rm rank}_\bZ F={\rm rank}_\bZ F^\prime=20$. 

Applying the function 
{\tt StablyEquivalentFCheckPSubdirectProduct($\widetilde{H}$,LHSlist($l_1$),RHSlist($l_1$))}, 
we obtain a basis 
$\mathcal{P}=\{P_1,\dots,P_{18}\}$ of the solution space of $G_1P=PG_2$. 

As in the case of $i=j=137$ of the proof of {Theorem \ref{thmain4}}, 
we get that 
${\rm rank}_{\bF_3} ({}^tP_1|\cdots|{}^tP_{18})\otimes_\bZ \bF_3=18$ 
and hence 
$[F]\neq [F^\prime]$ as $\widetilde{H}$-lattices 
($\bF_3[\widetilde{H}]$-lattices). 
Thus we conclude that $X=Y\lhd Z$ with $Z/Y\simeq C_2$. 

For the last assertion, 
we can apply the same method for $M_G$ and $M_{G^{\sigma_1}}$ 
instead of $F$ and $F^\prime$. 
Then we get that $\mathcal{P}=\{P_1,P_2\}$ and 
${\rm rank}_{\bF_3} ({}^tP_1|{}^tP_2)\otimes_\bZ \bF_3=2<4={\rm rank}_{\bF_3} M_G\otimes \bF_3$. 
Hence we have 
$M_G\otimes_\bZ \bF_3\not\simeq M_{G^{\sigma_1}}\otimes_\bZ \bF_3$ 
as $\bF_3[\widetilde{H}]$-lattices 
(see Example \ref{ex9.6} for GAP computations).\qed\\

{\it Proof of the last statement of Theorem \ref{thmain4} and Theorem \ref{thmain5}.}

Because we already have $(1)\Leftrightarrow (2)$ of {Theorem \ref{thmain4}} 
$($resp. Theorem \ref{thmain5}$)$, 
we get 
\begin{align*}
{\rm WSEC}_r=\coprod_{L/k\atop {\rm Gal}(L/k)\simeq {\rm G}_r}
{\rm WSEC}_{r,L},\ 
{\rm WSEC}_{r,L}=
\coprod_{t=1}^{\lambda_{r}} {\rm SEC}_{r,L,t}
\end{align*}
modulo stably birationally $k$-equivalence $\stackrel{\rm s.b.}{\approx}$ 
where 
${\rm SEC}_{r,L,t}$ $(1\leq t\leq \lambda_{r})$ is the $t$-th 
stably $k$-equivalent class of $T$ of dimension $4$ in 
${\rm WSEC}_{r,L}$ which corresponds to the fixed minimal splitting field $L$ 
in 
${\rm WSEC}_r=\coprod_{L/k\atop {\rm Gal}(L/k)\simeq {\rm G}_r} {\rm WSEC}_{r,L}$ 
with 
$[\widehat{T}]^{fl}\in {\rm WSEC}_{r,L}
=\coprod_{t=1}^{\lambda_r} {\rm SEC}_{r,L,t}$, 
${\rm Gal}(L/k)\simeq {\rm G}_r\simeq N_{4,i}$ $($resp. $I_{4,i}$$)$
$(1\leq r\leq 121, r\neq 7,8,10,12$$)$ 
$($resp. $(122\leq r\leq 128)$$)$ and 
$\lambda_r=|{\rm WSEC}_{r,L}|$. 

We get $\lambda_r=|{\rm WSEC}_{r,L}|=|Y\backslash {\rm Aut}(G)|$ 
via one to one correspondence 
$\overline{\sigma}=Y \sigma\in Y\backslash {\rm Aut}(G)$ $\leftrightarrow$ 
$[M_{G^\sigma}]^{fl}\in{\rm WSEC}_{r,L}$. 
Then we get 
$\lambda_r=|Y\backslash {\rm Aut}(G)|$ 
as $Y=$ {\tt AutWSEC(}$G${\tt )}, 
${\rm Aut}(G)$ $=$ ${\tt AutomorphismGroup(}$G${\tt )}$, 
$\lambda_r=$ {\tt IdCoset}$(${\tt Aut}$(G),${\tt Y}$){\tt [}$\,$1$\,${\tt ]}$
which is given as in Table $4$ $($resp. Table $5$$)$ 
$($see Example \ref{ex9.8} $($resp. Example \ref{ex9.9}$)$$)$. 
\qed

\bigskip
\begin{example}[Proof of ${\rm (2)}\Rightarrow {\rm (1)}$ of {Theorem \ref{thmain2}}]\label{ex9.1}~\vspace*{-5mm}\\
\begin{verbatim}
gap> Read("BCAlgTori.gap");
gap> N3g:=List(N3,x->MatGroupZClass(x[1],x[2],x[3],x[4]));;

gap> N3WSECMembersTable[1][1]; # N_{3,1} in WSEC_1
1
gap> N3autWSEC1invgen:=AutGSubdirectProductsWSECInvariantGen(N3g[1]); # checking X=Y=Z
[  ]

gap> N3WSECMembersTable[2][1]; # N_{3,2} in WSEC_2
2
gap> N3autWSEC2invgen:=AutGSubdirectProductsWSECInvariantGen(N3g[2]);
[ <pc group with 3 generators> ]
gap> Length(N3autWSEC2invgen); # the number of generators of X\Z/X
1
gap> PossibilityOfStablyEquivalentFSubdirectProduct(N3autWSEC2invgen[1]);
[ [ 0, 0, 0, 0, 0, 0, 0, 0, 0, 0, 0, 0, 0, 0, 0, 0, 1 ] ]
gap> l:=last[1]; # possibility for F=F' in the sense of the equation (8)
[ 0, 0, 0, 0, 0, 0, 0, 0, 0, 0, 0, 0, 0, 0, 0, 0, 1 ]
gap> bp:=StablyEquivalentFCheckPSubdirectProduct(
> N3autWSEC2invgen[1],LHSlist(l),RHSlist(l));;
gap> List(bp,Determinant); # the matrix P=bp[1] with det(P)=1 and FP=PF'
[ 1, -3, 3, 0, 0, 0, 0, 0, 0, 0, 0, 0, 0, 0, 0 ]
gap> StablyEquivalentFCheckMatSubdirectProduct( 
> N3autWSEC2invgen[1],LHSlist(l),RHSlist(l),bp[1]); 
true

gap> N3WSECMembersTable[3][1]; # N_{3,3} in WSEC_3
3
gap> N3autWSEC3invgen:=AutGSubdirectProductsWSECInvariantGen(N3g[3]);
[ <pc group with 3 generators> ]
gap> Length(N3autWSEC3invgen); # the number of generators of X\Z/X
1
gap> PossibilityOfStablyEquivalentFSubdirectProduct(N3autWSEC3invgen[1]);
[ [ 0, 0, 0, 0, 0, 0, 0, 0, 0, 0, 0, 0, 0, 0, 0, 0, 1 ] ]
gap> l:=last[1]; # possibility for F=F' in the sense of the equation (8)
[ 0, 0, 0, 0, 0, 0, 0, 0, 0, 0, 0, 0, 0, 0, 0, 0, 1 ]
gap> bp:=StablyEquivalentFCheckPSubdirectProduct(
> N3autWSEC3invgen[1],LHSlist(l),RHSlist(l));;
gap> List(bp,Determinant); # the matrix P=bp[1] with det(P)=-1 and FP=PF'
[ -1, 0, 0, 1, 0, -1, -1, 1, 1, 0, -1, 0, 0, 0, 0, 0, 0, 0, 0, 0, 0, 0, 0 ]
gap> StablyEquivalentFCheckMatSubdirectProduct(
> N3autWSEC3invgen[1],LHSlist(l),RHSlist(l),bp[1]);
true

gap> N3WSECMembersTable[4][1]; # N_{3,4} in WSEC_4
4
gap> N3autWSEC4invgen:=AutGSubdirectProductsWSECInvariantGen(N3g[4]);
[ <pc group with 3 generators>, <pc group with 3 generators> ]
gap> Length(N3autWSEC4invgen); # the number of generators of X\Z/X
2
gap> PossibilityOfStablyEquivalentFSubdirectProduct(N3autWSEC4invgen[1]);
[ [ 0, 0, 0, 0, 0, 0, 0, 0, 1 ] ]
gap> l:=last[1]; # possibility for F=F' in the sense of the equation (8)
[ 0, 0, 0, 0, 0, 0, 0, 0, 1 ]
gap> bp:=StablyEquivalentFCheckPSubdirectProduct(
> N3autWSEC4invgen[1],LHSlist(l),RHSlist(l));;
gap> List(bp,Determinant); # the matrix P=bp[1] with det(P)=-1 and FP=PF'
[ -1, 0, 0, -1, 0, -1, -1, -1, -1, 0, -1, 0, 0, 0, 0, 0, 0, 0, 0 ]
gap> StablyEquivalentFCheckMatSubdirectProduct(
> N3autWSEC4invgen[1],LHSlist(l),RHSlist(l),bp[1]);
true
gap> PossibilityOfStablyEquivalentFSubdirectProduct(N3autWSEC4invgen[2]);
[ [ 0, 0, 0, 0, 0, 0, 0, 0, 1 ] ]
gap> l:=last[1]; # possibility for F=F' in the sense of the equation (8)
[ 0, 0, 0, 0, 0, 0, 0, 0, 1 ]
gap> bp:=StablyEquivalentFCheckPSubdirectProduct(
> N3autWSEC4invgen[2],LHSlist(l),RHSlist(l));;
gap> List(bp,Determinant); # the matrix P=bp[1] with det(P)=-1 and FP=PF' 
[ -1, 0, 0, -1, 0, -1, -1, -1, -1, 0, -1, 0, 0, 0, 0, 0, 0, 0, 0 ]
gap> StablyEquivalentFCheckMatSubdirectProduct(
> N3autWSEC4invgen[2],LHSlist(l),RHSlist(l),bp[1]);
true

gap> N3WSECMembersTable[5][1]; # N_{3,5} in WSEC_5
5
gap> N3autWSEC5invgen:=AutGSubdirectProductsWSECInvariantGen(N3g[5]); # checking X=Y=Z
[  ]

gap> N3WSECMembersTable[5][2]; # N_{3,6} in WSEC_5
6
gap> N3autWSEC5_2invgen:=AutGSubdirectProductsWSECInvariantGen(N3g[6]); # checking X=Y=Z
[  ]

gap> N3WSECMembersTable[6][1]; # N_{3,7} in WSEC_6
7
gap> N3autWSEC6invgen:=AutGSubdirectProductsWSECInvariantGen(N3g[7]);
[ <pc group with 4 generators>, <pc group with 4 generators> ]
gap> Length(N3autWSEC6invgen); # the number of generators of X\Z/X
2
gap> PossibilityOfStablyEquivalentFSubdirectProduct(N3autWSEC6invgen[1]);
[ [ 0, 0, 0, 0, 0, 0, 0, 0, 0, 0, 0, 0, 0, 0, 0, 0, 0, 0, 0, 0, 0, 0, 0, 0, 0, 0, 0, 1 ] ]
gap> l:=last[1]; # possibility for F=F' in the sense of the equation (8)
[ 0, 0, 0, 0, 0, 0, 0, 0, 0, 0, 0, 0, 0, 0, 0, 0, 0, 0, 0, 0, 0, 0, 0, 0, 0, 0, 0, 1 ]
gap> bp:=StablyEquivalentFCheckPSubdirectProduct(
> N3autWSEC6invgen[1],LHSlist(l),RHSlist(l));;
gap> List(bp,Determinant); # the matrix P=bp[4] with det(P)=-1 and FP=PF'
[ 0, 0, 0, -1, 0, -1, -1, -1, 0, 0, 0, 0, 0, 0, 0 ]
gap> StablyEquivalentFCheckMatSubdirectProduct(
> N3autWSEC6invgen[1],LHSlist(l),RHSlist(l),bp[4]);
true
gap> PossibilityOfStablyEquivalentFSubdirectProduct(N3autWSEC6invgen[2]);
[ [ 0, 0, 0, 0, 0, 0, 0, 0, 0, 0, 0, 0, 0, 0, 0, 0, 0, 0, 0, 0, 0, 0, 0, 0, 0, 0, 0, 1 ] ]
gap> l:=last[1]; # possibility for F=F' in the sense of the equation (8)
[ 0, 0, 0, 0, 0, 0, 0, 0, 0, 0, 0, 0, 0, 0, 0, 0, 0, 0, 0, 0, 0, 0, 0, 0, 0, 0, 0, 1 ]
gap> bp:=StablyEquivalentFCheckPSubdirectProduct(
> N3autWSEC6invgen[2],LHSlist(l),RHSlist(l));;
gap> List(bp,Determinant); # the matrix P=bp[1] with det(P)=-1 and FP=PF'
[ -1, 0, 0, 0, 0, -1, -1, 0, 0, 0, 0, 0, 0, 0, 0 ]
gap> StablyEquivalentFCheckMatSubdirectProduct(
> N3autWSEC6invgen[2],LHSlist(l),RHSlist(l),bp[1]);
true

gap> N3WSECMembersTable[7][1]; # N_{3,8} in WSEC_7
8
gap> N3autWSEC7invgen:=AutGSubdirectProductsWSECInvariantGen(N3g[8]); # checking X=Y=Z
[  ]

gap> N3WSECMembersTable[8][1]; # N_{3,9} in WSEC_8
9
gap> N3autWSEC8invgen:=AutGSubdirectProductsWSECInvariantGen(N3g[9]); # checking X=Y=Z
[  ]

gap> N3WSECMembersTable[9][1]; # N_{3,10} in WSEC_9
10
gap> N3autWSEC9invgen:=AutGSubdirectProductsWSECInvariantGen(N3g[10]); # checking X=Y=Z
[  ]

gap> N3WSECMembersTable[10][1]; # N_{3,11} in WSEC_{10}
11
gap> N3autWSEC10invgen:=AutGSubdirectProductsWSECInvariantGen(N3g[11]); # checking X=Y=Z
[  ]

gap> N3WSECMembersTable[10][2]; # N_{3,13} in WSEC_{10}
13
gap> N3autWSEC10_2invgen:=AutGSubdirectProductsWSECInvariantGen(N3g[13]); # checking X=Y=Z
[  ]

gap> N3WSECMembersTable[11][1]; # N_{3,12} in WSEC_{11}
12
gap> N3autWSEC11invgen:=AutGSubdirectProductsWSECInvariantGen(N3g[12]); # checking X=Y=Z
[  ]

gap> N3WSECMembersTable[12][1]; # N_{3,14} in WSEC_{12}
14
gap> N3autWSEC12invgen:=AutGSubdirectProductsWSECInvariantGen(N3g[14]);
[ <pc group with 5 generators> ]
gap> Length(N3autWSEC12invgen); # the number of generators of X\Z/X
1
gap> PossibilityOfStablyEquivalentFSubdirectProduct(N3autWSEC12invgen[1]);;
gap> l:=last[Length(last)]; # possibility for F=F' in the sense of the equation (8)
[ 0, 0, 0, 0, 0, 0, 0, 0, 0, 0, 0, 0, 0, 0, 0, 0, 0, 0, 0, 0, 0, 0, 0, 0, 0, 
  0, 0, 0, 0, 0, 0, 0, 0, 1 ]
gap> bp:=StablyEquivalentFCheckPSubdirectProduct(
> N3autWSEC12invgen[1],LHSlist(l),RHSlist(l));;
gap> List(bp,Determinant); # |det(bp[s])|<>1
[ 5, 0, 0, 0, 65536, 0, 0, 0, 0, 0, 0 ]
gap> SearchPRowBlocks(bp);
rec( bpBlocks := [ [ 1, 2, 3, 4, 5, 6, 7, 8, 9, 10, 11 ] ], 
  rowBlocks := [ [ 1, 2, 3, 4, 5, 6, 7, 8, 9, 10, 11, 12, 13, 14, 15 ] ] )
gap> P:=SearchP1(bp); # the matrix P with det(P)=1, FP=PF' and rank(F)=rank(F')=15
[ [ 1, 1, 1, 0, 0, 1, 0, 1, -1, 0, 0, 1, -1, -1, 0 ], 
  [ 0, 0, 0, 0, 0, 0, 0, 0, 0, 0, 1, 0, 0, 0, 0 ], 
  [ 0, 0, 0, 0, 0, 0, 1, 0, 0, 0, 0, 0, 0, 0, 0 ], 
  [ 0, 1, -1, 2, 1, -1, 0, 0, 2, -1, 1, -1, 1, 0, -1 ], 
  [ 0, 0, 0, 0, 0, 0, 0, 0, 0, 1, 0, 0, 0, 0, 0 ], 
  [ 0, -1, 1, -1, 0, 1, 0, 0, 0, 1, -1, 1, 0, 1, 1 ], 
  [ 0, 0, 1, 0, 0, 0, 0, 0, 0, 0, 0, 0, 0, 0, 0 ], 
  [ 0, 1, 0, 2, 1, 0, 0, -1, 1, -1, 1, 0, 0, -1, -2 ], 
  [ 0, 0, 0, 0, 0, 0, 0, 0, 0, 0, 0, 0, 0, 1, 0 ], 
  [ 0, 0, 0, 0, 1, 0, 0, 0, 0, 0, 0, 0, 0, 0, 0 ], 
  [ 0, 1, 0, 0, 0, 0, 0, 0, 0, 0, 0, 0, 0, 0, 0 ], 
  [ 0, 0, 1, 0, 1, 2, -1, 0, -1, 1, 0, 1, -2, -1, 0 ], 
  [ 0, -1, 1, -1, 1, 2, 0, 0, -1, 2, -1, 1, -1, 0, 1 ], 
  [ 0, 0, 0, 0, 0, 0, 0, 0, 1, 0, 0, 0, 0, 0, 0 ], 
  [ 0, 1, -1, 1, 0, -1, 0, 0, 1, 0, 1, -1, 1, 1, 0 ] ]
gap> StablyEquivalentFCheckMatSubdirectProduct(
> N3autWSEC12invgen[1],LHSlist(l),RHSlist(l),P);
true

gap> N3WSECMembersTable[13][1]; # N_{3,15} in WSEC_{13}
15
gap> N3autWSEC13invgen:=AutGSubdirectProductsWSECInvariantGen(N3g[15]); # checking X=Y=Z
[  ]
\end{verbatim}
\end{example}

\bigskip
\begin{example}[Proof of {Theorem \ref{thmain2}}: the computation of $\lambda_r=|{\rm WSEC}_{r,L}|=|Y\backslash {\rm Aut}(G)|$ as in Table $2$]\label{ex9.2}~\vspace*{-5mm}\\
\begin{verbatim}
gap> Read("BCAlgTori.gap");
gap> N3g:=List(N3,x->MatGroupZClass(x[1],x[2],x[3],x[4]));;
gap> for i in [1..15] do
> Y:=AutWSEC(N3g[i]);
> AutG:=AutomorphismGroup(N3g[i]);
> d:=IdCoset(AutG,Y);
> if d=[1] then
> Print(i,"\t",1,"\n");
> else
> Print(i,"\t",d,"\t",StructureDescription(TransitiveGroup(d[1],d[2])),"\n");
> fi;
> od;
1	1
2	[ 7, 5 ]	PSL(3,2)
3	1
4	1
5	[ 2, 1 ]	C2
6	[ 2, 1 ]	C2
7	[ 4, 3 ]	D8
8	1
9	1
10	1
11	1
12	1
13	1
14	1
15	[ 2, 1 ]	C2
\end{verbatim}
\end{example}

\bigskip
\begin{example}[Proof of ${\rm (2)}\Rightarrow {\rm (1)}$ of {Theorem \ref{thmain4}}]\label{ex9.3}~\vspace*{-5mm}\\

\end{example}

\bigskip
\begin{example}[Proof of ${\rm (2)}\Rightarrow {\rm (1)}$ of {Theorem \ref{thmain5}}]\label{ex9.4}~\vspace*{-5mm}\\
\begin{verbatim}
gap> Read("BCAlgTori.gap");
gap> I4g:=List(I4,x->MatGroupZClass(x[1],x[2],x[3],x[4]));;

gap> I4WSECMembersTable[1][1]; # I_{4,1} in WSEC_{122}
1
gap> N4autWSEC122invgen:=AutGSubdirectProductsWSECInvariantGen(I4g[1]); # checking X=Y=Z
[  ]

gap> I4WSECMembersTable[2][1]; # I_{4,2} in WSEC_{123}
2
gap> N4autWSEC123invgen:=AutGSubdirectProductsWSECInvariantGen(I4g[2]); # checking X=Y=Z
[  ]

gap> I4WSECMembersTable[3][1]; # I_{4,3} in WSEC_{124}
3
gap> N4autWSEC124invgen:=AutGSubdirectProductsWSECInvariantGen(I4g[3]); # checking X=Y=Z
[  ]

gap> I4WSECMembersTable[4][1]; # I_{4,4} in WSEC_{125}
4
gap> N4autWSEC125invgen:=AutGSubdirectProductsWSECInvariantGen(I4g[4]); # checking X=Y=Z
[  ]

gap> I4WSECMembersTable[5][1]; # I_{4,5} in WSEC_{126}
5
gap> N4autWSEC126invgen:=AutGSubdirectProductsWSECInvariantGen(I4g[5]); # checking X=Y=Z
[  ]

gap> I4WSECMembersTable[6][1]; # I_{4,6} in WSEC_{127}
6
gap> N4autWSEC127invgen:=AutGSubdirectProductsWSECInvariantGen(I4g[6]); # checking X=Y=Z
[  ]
\end{verbatim}
\end{example}

\bigskip
\begin{example}[Proof of {Theorem \ref{thmain4}} for the exceptional cases $i=j=137,139,145,147$]\label{ex9.5}~\vspace*{-5mm}\\

\end{example}

\bigskip
\begin{example}[Proof of {Theorem \ref{thmain5}} for the exceptional case $i=j=7$]\label{ex9.6}~\vspace*{-5mm}\\
\begin{verbatim}
gap> Read("BCAlgTori.gap");
gap> I4g:=List(I4,x->MatGroupZClass(x[1],x[2],x[3],x[4]));;

gap> I4WSECMembersTable[7][1]; # I_{4,7} in WSEC_{128}
7
gap> G:=I4g[7];
MatGroupZClass( 4, 33, 2, 1 )
gap> StructureDescription(G);
"C3 : C8"
gap> StructureDescription(Centre(G)); # Z(G)=C4
"C4"
gap> StructureDescription(G/Centre(G)); # Inn(G)=G/Z(G)=S3, |Inn(G)|=6
"S3"
gap> StructureDescription(AutomorphismGroup(G)); # Aut(G)=C2xC2xS3, |Aut(G)|=24
"C2 x C2 x S3"
gap> Order(AutGLnZ(G)); # |X|=|AutGLnZ(G)|=12 
12
gap> Order(Centre(AutGLnZ(G))); # X=AutGLnZ(G)=D6
2
gap> I4autWSEC128invgen:=AutGSubdirectProductsWSECInvariantGen(G);
[ <pc group with 4 generators> ]
gap> Length(I4autWSEC128invgen); # the number of generators of X\Z/X, X<Z=Aut(G)
1

gap> GeneratorsOfGroup(G);
[ [ [ 0, 0, -1, 0 ], [ -1, 0, 0, 0 ], [ 1, 1, 1, -2 ], [ 0, 1, 0, -1 ] ], 
  [ [ 0, 1, 0, -1 ], [ 0, 0, -1, 1 ], [ -1, 0, 0, 1 ], [ 0, 1, 0, 0 ] ] ]
gap> List(GeneratorsOfGroup(G),x->Image(Projection(I4autWSEC128invgen[1],2),
> PreImage(Projection(I4autWSEC128invgen[1],1),x)));
[ [ [ 0, 0, 1, 0 ], [ 1, 0, 0, 0 ], [ -1, -1, -1, 2 ], [ 0, -1, 0, 1 ] ], 
  [ [ 0, 1, 0, -1 ], [ 0, 0, -1, 1 ], [ -1, 0, 0, 1 ], [ 0, 1, 0, 0 ] ] ]

gap> PossibilityOfStablyEquivalentFSubdirectProduct(I4autWSEC128invgen[1]);
[ [ 0, 0, 0, 0, 0, 0, 0, 0, 1 ] ]
gap> l:=last[1]; # there exists a possibility for F=F' 
[ 0, 0, 0, 0, 0, 0, 0, 0, 1 ]
gap> bp:=StablyEquivalentFCheckPSubdirectProduct(
> I4autWSEC128invgen[1],LHSlist(l),RHSlist(l));;
gap> SearchPRowBlocks(bp);
rec( bpBlocks := 
    [ [ 1, 2, 3, 4, 5, 6, 7, 8, 9, 10, 11, 12, 13, 14, 15, 16, 17, 18 ] ], 
  rowBlocks := 
    [ [ 1, 2, 3, 4, 5, 6, 7, 8, 9, 10, 11, 12, 13, 14, 15, 16, 17, 18, 19, 20 ] ] )
gap> Rank(Concatenation(bp)*Z(3)^0);  # [F]=[F'] is impossible over F_3 (over Z)
# X=Y<Z, Z/Y=C2
18

gap> StablyEquivalentMCheckPSubdirectProduct(I4autWSEC128invgen[1],LHSlist(l),RHSlist(l));
[ [ [ 1, 2, 2, -2 ], [ -2, -1, 0, 2 ], [ 0, 0, 1, -2 ], [ 0, 0, 2, -1 ] ], 
  [ [ 0, 3, 1, -2 ], [ -2, -1, -1, 4 ], [ -2, -1, 1, 0 ], [ -1, 1, 2, 0 ] ] ]
gap> Rank(Concatenation(last)*Z(3)^0); # M_G=M_G' is impossible over F_3 (over Z)
2
\end{verbatim}
\end{example}

\bigskip
\begin{example}[Computations of $X=Y=Z$ for $G=N_{31,i}$ as in Table $3$]\label{ex9.7}~\vspace*{-5mm}\\
\begin{verbatim}
gap> Read("BCAlgTori.gap");
gap> N31g:=List(N31,x->MatGroupZClass(x[1],x[2],x[3],x[4]));;

gap> N31WSECMembersTable[1]; # WSEC_1
[ 1, 4, 11 ]
gap> N31WSECMembersTable[1][1]; # N_{31,1} in WSEC_1
1
gap> N31autWSEC1invgen:=AutGSubdirectProductsWSECInvariantGen(N31g[1]);
[ <pc group with 2 generators> ]
gap> N31WSECMembersTable[1][2]; # N_{31,4} in WSEC_1
4
gap> N31autWSEC1_2invgen:=AutGSubdirectProductsWSECInvariantGen(N31g[4]); # checking X=Y=Z
[  ]
gap> N31WSECMembersTable[1][3]; # N_{31,11} in WSEC_1
11
gap> N31autWSEC1_3invgen:=AutGSubdirectProductsWSECInvariantGen(N31g[11]);
[ <pc group with 3 generators> ]

gap> N31WSECMembersTable[2]; # WSEC_2
[ 3, 6, 8, 10, 13 ]
gap> N31WSECMembersTable[2][1]; # N_{31,3} in WSEC_2
3
gap> N31autWSEC2invgen:=AutGSubdirectProductsWSECInvariantGen(N31g[3]);
[ <pc group with 3 generators> ]
gap> N31WSECMembersTable[2][2]; # N_{31,6} in WSEC_2
6
gap> N31autWSEC2_2invgen:=AutGSubdirectProductsWSECInvariantGen(N31g[6]);
[ <pc group with 3 generators> ]
gap> N31WSECMembersTable[2][3]; # N_{31,8} in WSEC_2
8
gap> N31autWSEC2_3invgen:=AutGSubdirectProductsWSECInvariantGen(N31g[8]);
[ <pc group with 3 generators> ]
gap> N31WSECMembersTable[2][4]; # N_{31,10} in WSEC_2
10
gap> N31autWSEC2_4invgen:=AutGSubdirectProductsWSECInvariantGen(N31g[10]);
[ <pc group with 3 generators> ]
gap> N31WSECMembersTable[2][5]; # N_{31,13} in WSEC_2
13
gap> N31autWSEC2_5invgen:=AutGSubdirectProductsWSECInvariantGen(N31g[13]);
[ <pc group with 4 generators> ]

gap> N31WSECMembersTable[3]; # WSEC_3
[ 2, 5, 7, 9, 12 ]
gap> N31WSECMembersTable[3][1]; # N_{31,2} in WSEC_3
2
gap> N31autWSEC3invgen:=AutGSubdirectProductsWSECInvariantGen(N31g[2]);
[ <pc group with 3 generators> ]
gap> N31WSECMembersTable[3][2]; # N_{31,5} in WSEC_3
5
gap> N31autWSEC3_2invgen:=AutGSubdirectProductsWSECInvariantGen(N31g[5]);
[ <pc group with 3 generators> ]
gap> N31WSECMembersTable[3][3]; # N_{31,7} in WSEC_3
7
gap> N31autWSEC3_3invgen:=AutGSubdirectProductsWSECInvariantGen(N31g[7]);
[ <pc group with 3 generators> ]
gap> N31WSECMembersTable[3][4]; # N_{31,9} in WSEC_3
9
gap> N31autWSEC3_4invgen:=AutGSubdirectProductsWSECInvariantGen(N31g[9]);
[ <pc group with 3 generators> ]
gap> N31WSECMembersTable[3][5]; # N_{31,12} in WSEC_3
12
gap> N31autWSEC3_5invgen:=AutGSubdirectProductsWSECInvariantGen(N31g[12]);
[ <pc group with 4 generators> ]

gap> N31WSECMembersTable[4]; # WSEC_4
[ 14, 18, 25, 26, 29 ]
gap> N31WSECMembersTable[4][1]; # N_{31,14} in WSEC_4
14
gap> N31autWSEC4invgen:=AutGSubdirectProductsWSECInvariantGen(N31g[14]);
[ <pc group with 3 generators>, <pc group with 3 generators> ]
gap> N31WSECMembersTable[4][2]; # N_{31,18} in WSEC_4
18
gap> N31autWSEC4_2invgen:=AutGSubdirectProductsWSECInvariantGen(N31g[18]);
[ <pc group with 3 generators>, <pc group with 3 generators> ]
gap> N31WSECMembersTable[4][3]; # N_{31,25} in WSEC_4
25
gap> N31autWSEC4_3invgen:=AutGSubdirectProductsWSECInvariantGen(N31g[25]);
[ <pc group with 3 generators>, <pc group with 3 generators> ]
gap> N31WSECMembersTable[4][4]; # N_{31,26} in WSEC_4
26
gap> N31autWSEC4_4invgen:=AutGSubdirectProductsWSECInvariantGen(N31g[26]);
[ <pc group with 3 generators>, <pc group with 3 generators> ]
gap> N31WSECMembersTable[4][5]; # N_{31,29} in WSEC_4
29
gap> N31autWSEC4_5invgen:=AutGSubdirectProductsWSECInvariantGen(N31g[29]);
[ <pc group with 4 generators>, <pc group with 4 generators>, 
  <pc group with 4 generators> ]

gap> N31WSECMembersTable[5]; # WSEC_5
[ 15, 16, 19, 20, 21, 22, 27, 28, 31, 35 ]
gap> N31WSECMembersTable[5][1]; # N_{31,15} in WSEC_5
15
gap> N31autWSEC5invgen:=AutGSubdirectProductsWSECInvariantGen(N31g[15]); # checking X=Y=Z
[  ]
gap> N31WSECMembersTable[5][2]; # N_{31,16} in WSEC_5
16
gap> N31autWSEC5_2invgen:=AutGSubdirectProductsWSECInvariantGen(N31g[16]); # checking X=Y=Z
[  ]
gap> N31WSECMembersTable[5][3]; # N_{31,19} in WSEC_5
19
gap> N31autWSEC5_3invgen:=AutGSubdirectProductsWSECInvariantGen(N31g[19]); # checking X=Y=Z
[  ]
gap> N31WSECMembersTable[5][4]; # N_{31,20} in WSEC_5
20
gap> N31autWSEC5_4invgen:=AutGSubdirectProductsWSECInvariantGen(N31g[20]); # checking X=Y=Z
[  ]
gap> N31WSECMembersTable[5][5]; # N_{31,21} in WSEC_5
21
gap> N31autWSEC5_5invgen:=AutGSubdirectProductsWSECInvariantGen(N31g[21]); # checking X=Y=Z
[  ]
gap> N31WSECMembersTable[5][6]; # N_{31,22} in WSEC_5
22
gap> N31autWSEC5_6invgen:=AutGSubdirectProductsWSECInvariantGen(N31g[22]); # checking X=Y=Z
[  ]
gap> N31WSECMembersTable[5][7]; # N_{31,27} in WSEC_5
27
gap> N31autWSEC5_7invgen:=AutGSubdirectProductsWSECInvariantGen(N31g[27]); # checking X=Y=Z
[  ]
gap> N31WSECMembersTable[5][8]; # N_{31,28} in WSEC_5
28
gap> N31autWSEC5_8invgen:=AutGSubdirectProductsWSECInvariantGen(N31g[28]); # checking X=Y=Z
[  ]
gap> N31WSECMembersTable[5][9]; # N_{31,31} in WSEC_5
31
gap> N31autWSEC5_9invgen:=AutGSubdirectProductsWSECInvariantGen(N31g[31]);
[ <pc group with 4 generators>, <pc group with 4 generators> ]
gap> N31WSECMembersTable[5][10]; # N_{31,35} in WSEC_5
35
gap> N31autWSEC5_10invgen:=AutGSubdirectProductsWSECInvariantGen(N31g[35]);
[ <pc group with 4 generators>, <pc group with 4 generators> ]

gap> N31WSECMembersTable[6]; # WSEC_6
[ 17, 23, 24, 30, 32, 33, 34, 36, 37 ]
gap> N31WSECMembersTable[6][1]; # N_{31,17} in WSEC_6
17
gap> N31autWSEC6invgen:=AutGSubdirectProductsWSECInvariantGen(N31g[17]);
[ <pc group with 4 generators>, <pc group with 4 generators> ]
gap> N31WSECMembersTable[6][2]; # N_{31,23} in WSEC_6
23
gap> N31autWSEC6_2invgen:=AutGSubdirectProductsWSECInvariantGen(N31g[23]);
[ <pc group with 4 generators>, <pc group with 4 generators> ]
gap> N31WSECMembersTable[6][3]; # N_{31,24} in WSEC_6
24
gap> N31autWSEC6_3invgen:=AutGSubdirectProductsWSECInvariantGen(N31g[24]);
[ <pc group with 4 generators>, <pc group with 4 generators> ]
gap> N31WSECMembersTable[6][4]; # N_{31,30} in WSEC_6
30
gap> N31autWSEC6_4invgen:=AutGSubdirectProductsWSECInvariantGen(N31g[30]);
[ <pc group with 4 generators>, <pc group with 4 generators> ]
gap> N31WSECMembersTable[6][5]; # N_{31,32} in WSEC_6
32
gap> N31autWSEC6_5invgen:=AutGSubdirectProductsWSECInvariantGen(N31g[32]);
[ <pc group with 4 generators>, <pc group with 4 generators> ]
gap> N31WSECMembersTable[6][6]; # N_{31,33} in WSEC_6
33
gap> N31autWSEC6_6invgen:=AutGSubdirectProductsWSECInvariantGen(N31g[33]);
[ <pc group with 4 generators>, <pc group with 4 generators> ]
gap> N31WSECMembersTable[6][7]; # N_{31,34} in WSEC_6
34
gap> N31autWSEC6_7invgen:=AutGSubdirectProductsWSECInvariantGen(N31g[34]);
[ <pc group with 4 generators>, <pc group with 4 generators> ]
gap> N31WSECMembersTable[6][8]; # N_{31,36} in WSEC_6
36
gap> N31autWSEC6_8invgen:=AutGSubdirectProductsWSECInvariantGen(N31g[36]);
[ <pc group with 4 generators>, <pc group with 4 generators> ]
gap> N31WSECMembersTable[6][9]; # N_{31,37} in WSEC_6
37
gap> N31autWSEC6_9invgen:=AutGSubdirectProductsWSECInvariantGen(N31g[37]);
[ <pc group with 5 generators>, <pc group with 5 generators>, 
  <pc group with 5 generators>, <pc group with 5 generators> ]

gap> N31WSECMembersTable[7]; # WSEC_7
[ 38, 48 ]
gap> N31WSECMembersTable[7][1]; # N_{31,38} in WSEC_7
38
gap> N31autWSEC7invgen:=AutGSubdirectProductsWSECInvariantGen(N31g[38]); # checking X=Y=Z
[  ]
gap> N31WSECMembersTable[7][2]; # N_{31,48} in WSEC_7
48
gap> N31autWSEC7_2invgen:=AutGSubdirectProductsWSECInvariantGen(N31g[48]); # checking X=Y=Z
[  ]

gap> N31WSECMembersTable[8]; # WSEC_8
[ 40, 47, 53 ]
gap> N31WSECMembersTable[8][1]; # N_{31,40} in WSEC_8
40
gap> N31autWSEC8invgen:=AutGSubdirectProductsWSECInvariantGen(N31g[40]); # checking X=Y=Z
[  ]
gap> N31WSECMembersTable[8][2]; # N_{31,47} in WSEC_8
47
gap> N31autWSEC8_2invgen:=AutGSubdirectProductsWSECInvariantGen(N31g[47]); # checking X=Y=Z
[  ]
gap> N31WSECMembersTable[8][3]; # N_{31,53} in WSEC_8
53
gap> N31autWSEC8_3invgen:=AutGSubdirectProductsWSECInvariantGen(N31g[53]);
[ <pc group with 5 generators> ]

gap> N31WSECMembersTable[9]; # WSEC_9
[ 39, 46, 52 ]
gap> N31WSECMembersTable[9][1]; # N_{31,39} in WSEC_9
39
gap> N31autWSEC9invgen:=AutGSubdirectProductsWSECInvariantGen(N31g[39]); # checking X=Y=Z
[  ]
gap> N31WSECMembersTable[9][2]; # N_{31,46} in WSEC_9
46
gap> N31autWSEC9_2invgen:=AutGSubdirectProductsWSECInvariantGen(N31g[46]); # checking X=Y=Z
[  ]
gap> N31WSECMembersTable[9][3]; # N_{31,52} in WSEC_9
52
gap> N31autWSEC9_3invgen:=AutGSubdirectProductsWSECInvariantGen(N31g[52]);
[ <pc group with 5 generators> ]

gap> N31WSECMembersTable[10]; # WSEC_{10}
[ 41, 43, 50, 51, 60, 62 ]
gap> N31WSECMembersTable[10][1]; # N_{31,41} in WSEC_{10}
41
gap> N31autWSEC10invgen:=AutGSubdirectProductsWSECInvariantGen(N31g[41]); # checking X=Y=Z
[  ]
gap> N31WSECMembersTable[10][2]; # N_{31,43} in WSEC_{10}
43
gap> N31autWSEC10_2invgen:=AutGSubdirectProductsWSECInvariantGen(N31g[43]); # checking X=Y=Z
[  ]
gap> N31WSECMembersTable[10][3]; # N_{31,50} in WSEC_{10}
50
gap> N31autWSEC10_3invgen:=AutGSubdirectProductsWSECInvariantGen(N31g[50]); # checking X=Y=Z
[  ]
gap> N31WSECMembersTable[10][4]; # N_{31,51} in WSEC_{10}
51
gap> N31autWSEC10_4invgen:=AutGSubdirectProductsWSECInvariantGen(N31g[51]); # checking X=Y=Z
[  ]
gap> N31WSECMembersTable[10][5]; # N_{31,60} in WSEC_{10}
60
gap> N31autWSEC10_5invgen:=AutGSubdirectProductsWSECInvariantGen(N31g[60]);
[ <pc group with 5 generators> ]
gap> N31WSECMembersTable[10][6]; # N_{31,62} in WSEC_{10}
62
gap> N31autWSEC10_6invgen:=AutGSubdirectProductsWSECInvariantGen(N31g[62]);
[ <pc group with 5 generators> ]

gap> N31WSECMembersTable[11]; # WSEC_{11}
[ 42, 49, 61 ]
gap> N31WSECMembersTable[11][1]; # N_{31,42} in WSEC_{11}
42
gap> N31autWSEC11invgen:=AutGSubdirectProductsWSECInvariantGen(N31g[42]); # checking X=Y=Z
[  ]
gap> N31WSECMembersTable[11][2]; # N_{31,49} in WSEC_{11}
49
gap> N31autWSEC11_2invgen:=AutGSubdirectProductsWSECInvariantGen(N31g[49]); # checking X=Y=Z
[  ]
gap> N31WSECMembersTable[11][3]; # N_{31,61} in WSEC_{11}
61
gap> N31autWSEC11_3invgen:=AutGSubdirectProductsWSECInvariantGen(N31g[61]);
[ <pc group with 5 generators> ]

gap> N31WSECMembersTable[12]; # WSEC_{12}
[ 45, 55, 57, 59, 64 ]
gap> N31WSECMembersTable[12][1]; # N_{31,45} in WSEC_{12}
45
gap> N31autWSEC12invgen:=AutGSubdirectProductsWSECInvariantGen(N31g[45]);
[ <pc group with 5 generators> ]
gap> N31WSECMembersTable[12][2]; # N_{31,55} in WSEC_{12}
55
gap> N31autWSEC12_2invgen:=AutGSubdirectProductsWSECInvariantGen(N31g[55]);
[ <pc group with 5 generators> ]
gap> N31WSECMembersTable[12][3]; # N_{31,57} in WSEC_{12}
57
gap> N31autWSEC12_3invgen:=AutGSubdirectProductsWSECInvariantGen(N31g[57]);
[ <pc group with 5 generators> ]
gap> N31WSECMembersTable[12][4]; # N_{31,59} in WSEC_{12}
59
gap> N31autWSEC12_4invgen:=AutGSubdirectProductsWSECInvariantGen(N31g[59]);
[ <pc group with 5 generators> ]
gap> N31WSECMembersTable[12][5]; # N_{31,64} in WSEC_{12}
64
gap> N31autWSEC12_5invgen:=AutGSubdirectProductsWSECInvariantGen(N31g[64]);
[ <pc group with 6 generators>, <pc group with 6 generators> ]

gap> N31WSECMembersTable[13]; # WSEC_{13}
[ 44, 54, 56, 58, 63 ]
gap> N31WSECMembersTable[13][1]; # N_{31,44} in WSEC_{13}
44
gap> N31autWSEC13invgen:=AutGSubdirectProductsWSECInvariantGen(N31g[44]); # checking X=Y=Z
[  ]
gap> N31WSECMembersTable[13][2]; # N_{31,54} in WSEC_{13}
54
gap> N31autWSEC13_2invgen:=AutGSubdirectProductsWSECInvariantGen(N31g[54]); # checking X=Y=Z
[  ]
gap> N31WSECMembersTable[13][3]; # N_{31,56} in WSEC_{13}
56
gap> N31autWSEC13_3invgen:=AutGSubdirectProductsWSECInvariantGen(N31g[56]); # checking X=Y=Z
[  ]
gap> N31WSECMembersTable[13][4]; # N_{31,58} in WSEC_{13}
58
gap> N31autWSEC13_4invgen:=AutGSubdirectProductsWSECInvariantGen(N31g[58]); # checking X=Y=Z
[  ]
gap> N31WSECMembersTable[13][5]; # N_{31,63} in WSEC_{13}
63
gap> N31autWSEC13_5invgen:=AutGSubdirectProductsWSECInvariantGen(N31g[63]);
[ <pc group with 6 generators>, <pc group with 6 generators> ]
\end{verbatim}
\end{example}

\bigskip
\begin{example}[Proof of {Theorem \ref{thmain4}}: the computation of $\lambda_r=|{\rm WSEC}_{r,L}|=|Y\backslash {\rm Aut}(G)|$ as in Table $4$]\label{ex9.8}~\vspace*{-5mm}\\
\begin{verbatim}
gap> Read("BCAlgTori.gap");
gap> N4g:=List(N4,x->MatGroupZClass(x[1],x[2],x[3],x[4]));;
gap> for i in [1..152] do
> if i in [137,139,145,147] then
> Y:=AutGLnZ(N4g[i]);
> else
> Y:=AutWSEC(N4g[i]);
> fi;
> AutG:=AutomorphismGroup(N4g[i]);
> if Order(AutG)/Order(Y)>30 then
> YN:=Intersection(ConjugateSubgroups(AutG,Y));
> Print(i,"\t",Order(AutG)/Order(Y),"\t",StructureDescription(AutG/YN),"\n");
> else
> d:=IdCoset(AutG,Y);
> if d=[1] then
> Print(i,"\t",1,"\n");
> else
> Print(i,"\t",d,"\t",StructureDescription(TransitiveGroup(d[1],d[2])),"\n");
> fi;
> fi;
> od;
1	1
2	1
3	[ 21, 14 ]	PSL(3,2)
4	1
5	1
6	[ 7, 5 ]	PSL(3,2)
7	1
8	[ 21, 14 ]	PSL(3,2)
9	[ 28, 32 ]	PSL(3,2)
10	1
11	420	A8
12	840	A8
13	1
14	1
15	[ 2, 1 ]	C2
16	[ 2, 1 ]	C2
17	[ 2, 1 ]	C2
18	[ 2, 1 ]	C2
19	[ 4, 3 ]	D8
20	[ 2, 1 ]	C2
21	[ 8, 9 ]	C2 x D8
22	[ 2, 1 ]	C2
23	1
24	1
25	[ 2, 1 ]	C2
26	[ 2, 1 ]	C2
27	1
28	[ 24, 196 ]	((C2 x C2 x C2 x C2) : C3) : C2
29	[ 4, 3 ]	D8
30	[ 4, 3 ]	D8
31	[ 2, 1 ]	C2
32	[ 16, 9 ]	C2 x D8
33	[ 2, 1 ]	C2
34	[ 16, 9 ]	C2 x D8
35	[ 2, 1 ]	C2
36	[ 4, 2 ]	C2 x C2
37	[ 8, 41 ]	(((C2 x C2 x C2 x C2) : C2) : C3) : C2
38	192	((C2 x C2 x ((C2 x C2 x C2 x C2) : C2)) : C3) : C2
39	1
40	[ 4, 2 ]	C2 x C2
41	1
42	1
43	[ 2, 1 ]	C2
44	1
45	[ 2, 1 ]	C2
46	[ 2, 1 ]	C2
47	[ 12, 22 ]	C2 x S4
48	[ 12, 21 ]	C2 x S4
49	1
50	[ 2, 1 ]	C2
51	[ 3, 2 ]	S3
52	[ 4, 3 ]	D8
53	[ 4, 2 ]	C2 x C2
54	[ 2, 1 ]	C2
55	[ 12, 8 ]	S4
56	[ 16, 395 ]	(((C2 x C2 x C2 x C2) : C2) : C2) : C2
57	1
58	1
59	1
60	1
61	1
62	1
63	[ 2, 1 ]	C2
64	[ 2, 1 ]	C2
65	1
66	[ 2, 1 ]	C2
67	1
68	[ 4, 3 ]	D8
69	1
70	1
71	1
72	[ 2, 1 ]	C2
73	1
74	1
75	1
76	1
77	1
78	1
79	1
80	[ 6, 2 ]	S3
81	1
82	1
83	[ 2, 1 ]	C2
84	1
85	[ 2, 1 ]	C2
86	[ 2, 1 ]	C2
87	1
88	[ 2, 1 ]	C2
89	[ 4, 5 ]	S4
90	[ 12, 9 ]	S4
91	1
92	1
93	[ 2, 1 ]	C2
94	1
95	1
96	[ 2, 1 ]	C2
97	1
98	1
99	[ 4, 3 ]	D8
100	[ 2, 1 ]	C2
101	1
102	[ 2, 1 ]	C2
103	[ 2, 1 ]	C2
104	1
105	[ 4, 3 ]	D8
106	1
107	1
108	1
109	1
110	1
111	1
112	[ 2, 1 ]	C2
113	[ 2, 1 ]	C2
114	1
115	1
116	1
117	[ 2, 1 ]	C2
118	1
119	1
120	[ 2, 1 ]	C2
121	1
122	[ 2, 1 ]	C2
123	[ 2, 1 ]	C2
124	1
125	[ 2, 1 ]	C2
126	1
127	1
128	[ 4, 3 ]	D8
129	[ 2, 1 ]	C2
130	1
131	[ 2, 1 ]	C2
132	1
133	1
134	1
135	[ 2, 1 ]	C2
136	[ 2, 1 ]	C2
137	[ 2, 1 ]	C2
138	1
139	[ 4, 2 ]	C2 x C2
140	1
141	1
142	[ 6, 2 ]	S3
143	[ 2, 1 ]	C2
144	[ 2, 1 ]	C2
145	[ 2, 1 ]	C2
146	[ 12, 3 ]	D12
147	[ 4, 2 ]	C2 x C2
148	[ 2, 1 ]	C2
149	[ 2, 1 ]	C2
150	[ 2, 1 ]	C2
151	[ 4, 2 ]	C2 x C2
152	[ 8, 4 ]	D8
\end{verbatim}
\end{example}

\bigskip
\begin{example}[Proof of {Theorem \ref{thmain5}}: the computation of $\lambda_r=|{\rm WSEC}_{r,L}|=|Y\backslash {\rm Aut}(G)|$ as in Table $5$]\label{ex9.9}~\vspace*{-5mm}\\
\begin{verbatim}
gap> Read("BCAlgTori.gap");
gap> I4g:=List(I4,x->MatGroupZClass(x[1],x[2],x[3],x[4]));;
gap> for i in [1..7] do
> if i=7 then
> Y:=AutGLnZ(I4g[i]);
> else
> Y:=AutWSEC(I4g[i]);
> fi;
> AutG:=AutomorphismGroup(I4g[i]);
> d:=IdCoset(AutG,Y);
> if d=[1] then
> Print(i,"\t",1,"\n");
> else
> Print(i,"\t",d,"\t",StructureDescription(TransitiveGroup(d[1],d[2])),"\n");
> fi;
> od;
1	1
2	1
3	[ 2, 1 ]	C2
4	1
5	1
6	[ 2, 1 ]	C2
7	[ 2, 1 ]	C2
\end{verbatim}
\end{example}

\bigskip
\begin{example}[{Computations of Example \ref{exU2}}]\label{ex9.10}~\vspace*{-5mm}\\
\begin{verbatim}
gap> Read("BCAlgTori.gap");
gap> N3g:=List(N3,x->MatGroupZClass(x[1],x[2],x[3],x[4]));;
gap> U2WSEC:=ConjugacyClassesSubgroups1WSEC(N3g[2]);
rec( 
  ConjugacyClassesSubgroups1 := 
    [ Group([],[ [ 1, 0, 0 ], [ 0, 1, 0 ], [ 0, 0, 1 ] ])^G, 
      Group([ [ [ -1, -1, -1 ], [ 0, 0, 1 ], [ 0, 1, 0 ] ] ])^G, 
      Group([ [ [ -1, 0, 0 ], [ 0, -1, 0 ], [ 0, 0, -1 ] ] ])^G, 
      Group([ [ [ 0, -1, 0 ], [ -1, 0, 0 ], [ 1, 1, 1 ] ] ])^G, 
      Group([ [ [ 0, 0, -1 ], [ 1, 1, 1 ], [ -1, 0, 0 ] ] ])^G, 
      Group([ [ [ 0, 0, 1 ], [ -1, -1, -1 ], [ 1, 0, 0 ] ] ])^G, 
      Group([ [ [ 0, 1, 0 ], [ 1, 0, 0 ], [ -1, -1, -1 ] ] ])^G, 
      Group([ [ [ 1, 1, 1 ], [ 0, 0, -1 ], [ 0, -1, 0 ] ] ])^G, 
      Group([ [ [ -1, -1, -1 ], [ 0, 0, 1 ], [ 0, 1, 0 ] ], 
          [ [ -1, 0, 0 ], [ 0, -1, 0 ], [ 0, 0, -1 ] ] ])^G, 
      Group([ [ [ -1, -1, -1 ], [ 0, 0, 1 ], [ 0, 1, 0 ] ], 
          [ [ 0, -1, 0 ], [ -1, 0, 0 ], [ 1, 1, 1 ] ] ])^G, 
      Group([ [ [ -1, -1, -1 ], [ 0, 0, 1 ], [ 0, 1, 0 ] ], 
          [ [ 0, 0, 1 ], [ -1, -1, -1 ], [ 1, 0, 0 ] ] ])^G, 
      Group([ [ [ -1, 0, 0 ], [ 0, -1, 0 ], [ 0, 0, -1 ] ], 
          [ [ 0, -1, 0 ], [ -1, 0, 0 ], [ 1, 1, 1 ] ] ])^G, 
      Group([ [ [ -1, 0, 0 ], [ 0, -1, 0 ], [ 0, 0, -1 ] ], 
          [ [ 0, 0, -1 ], [ 1, 1, 1 ], [ -1, 0, 0 ] ] ])^G, 
      Group([ [ [ 0, -1, 0 ], [ -1, 0, 0 ], [ 1, 1, 1 ] ], 
          [ [ 0, 0, 1 ], [ -1, -1, -1 ], [ 1, 0, 0 ] ] ])^G, 
      Group([ [ [ 0, 0, -1 ], [ 1, 1, 1 ], [ -1, 0, 0 ] ], 
          [ [ 0, 1, 0 ], [ 1, 0, 0 ], [ -1, -1, -1 ] ] ])^G, 
      MatGroupZClass( 3, 3, 3, 3 )^G ], 
  WSEC := [ 0, 0, 0, 0, 0, 0, 0, 0, 0, 0, 1, 0, 0, 0, 0, 2 ] )
gap> GeneratorsOfGroup(N3g[2]);
[ [ [ -1, 0, 0 ], [ 0, -1, 0 ], [ 0, 0, -1 ] ], 
  [ [ 0, 0, 1 ], [ -1, -1, -1 ], [ 1, 0, 0 ] ], 
  [ [ 0, 1, 0 ], [ 1, 0, 0 ], [ -1, -1, -1 ] ] ]
gap> s1:=U2g[2];
[ [ 0, 0, 1 ], [ -1, -1, -1 ], [ 1, 0, 0 ] ]
gap> s2:=U2g[3];
[ [ 0, 1, 0 ], [ 1, 0, 0 ], [ -1, -1, -1 ] ]
gap> s3:=U2g[1];
[ [ -1, 0, 0 ], [ 0, -1, 0 ], [ 0, 0, -1 ] ]
gap> N3[2];
[ 3, 3, 3, 3 ]
gap> IsInvertibleF(N3g[2]);
false
gap> U2WSEC;
rec( 
  ConjugacyClassesSubgroups1 := 
    [ Group([],[ [ 1, 0, 0 ], [ 0, 1, 0 ], [ 0, 0, 1 ] ])^G, 
      Group([ [ [ -1, -1, -1 ], [ 0, 0, 1 ], [ 0, 1, 0 ] ] ])^G, 
      Group([ [ [ -1, 0, 0 ], [ 0, -1, 0 ], [ 0, 0, -1 ] ] ])^G, 
      Group([ [ [ 0, -1, 0 ], [ -1, 0, 0 ], [ 1, 1, 1 ] ] ])^G, 
      Group([ [ [ 0, 0, -1 ], [ 1, 1, 1 ], [ -1, 0, 0 ] ] ])^G, 
      Group([ [ [ 0, 0, 1 ], [ -1, -1, -1 ], [ 1, 0, 0 ] ] ])^G, 
      Group([ [ [ 0, 1, 0 ], [ 1, 0, 0 ], [ -1, -1, -1 ] ] ])^G, 
      Group([ [ [ 1, 1, 1 ], [ 0, 0, -1 ], [ 0, -1, 0 ] ] ])^G, 
      Group([ [ [ -1, -1, -1 ], [ 0, 0, 1 ], [ 0, 1, 0 ] ], 
          [ [ -1, 0, 0 ], [ 0, -1, 0 ], [ 0, 0, -1 ] ] ])^G, 
      Group([ [ [ -1, -1, -1 ], [ 0, 0, 1 ], [ 0, 1, 0 ] ], 
          [ [ 0, -1, 0 ], [ -1, 0, 0 ], [ 1, 1, 1 ] ] ])^G, 
      Group([ [ [ -1, -1, -1 ], [ 0, 0, 1 ], [ 0, 1, 0 ] ], 
          [ [ 0, 0, 1 ], [ -1, -1, -1 ], [ 1, 0, 0 ] ] ])^G, 
      Group([ [ [ -1, 0, 0 ], [ 0, -1, 0 ], [ 0, 0, -1 ] ], 
          [ [ 0, -1, 0 ], [ -1, 0, 0 ], [ 1, 1, 1 ] ] ])^G, 
      Group([ [ [ -1, 0, 0 ], [ 0, -1, 0 ], [ 0, 0, -1 ] ], 
          [ [ 0, 0, -1 ], [ 1, 1, 1 ], [ -1, 0, 0 ] ] ])^G, 
      Group([ [ [ 0, -1, 0 ], [ -1, 0, 0 ], [ 1, 1, 1 ] ], 
          [ [ 0, 0, 1 ], [ -1, -1, -1 ], [ 1, 0, 0 ] ] ])^G, 
      Group([ [ [ 0, 0, -1 ], [ 1, 1, 1 ], [ -1, 0, 0 ] ], 
          [ [ 0, 1, 0 ], [ 1, 0, 0 ], [ -1, -1, -1 ] ] ])^G, 
      MatGroupZClass( 3, 3, 3, 3 )^G ], 
  WSEC := [ 0, 0, 0, 0, 0, 0, 0, 0, 0, 0, 1, 0, 0, 0, 0, 2 ] )
gap> List(U2WSEC.ConjugacyClassesSubgroups1,x->IsInvertibleF(Representative(x)));
[ true, true, true, true, true, true, true, true, true, true, false, true, 
  true, true, true, false ]
gap> List(U2WSEC.ConjugacyClassesSubgroups1,x->[Order(Representative(x)),
> IsInvertibleF(Representative(x))]);
[ [ 1, true ], [ 2, true ], [ 2, true ], [ 2, true ], [ 2, true ], 
  [ 2, true ], [ 2, true ], [ 2, true ], [ 4, true ], [ 4, true ], 
  [ 4, false ], [ 4, true ], [ 4, true ], [ 4, true ], [ 4, true ], 
  [ 8, false ] ]
gap> List(U2WSEC.ConjugacyClassesSubgroups1,x->[Order(Representative(x)),
> CrystCatZClass(Representative(x)),IsInvertibleF(Representative(x))]);
[ [ 1, [ 3, 1, 1, 1 ], true ], [ 2, [ 3, 2, 1, 2 ], true ], 
  [ 2, [ 3, 1, 2, 1 ], true ], [ 2, [ 3, 2, 2, 2 ], true ], 
  [ 2, [ 3, 2, 2, 2 ], true ], [ 2, [ 3, 2, 1, 2 ], true ], 
  [ 2, [ 3, 2, 1, 2 ], true ], [ 2, [ 3, 2, 2, 2 ], true ], 
  [ 4, [ 3, 2, 3, 2 ], true ], [ 4, [ 3, 3, 2, 4 ], true ], 
  [ 4, [ 3, 3, 1, 3 ], false ], [ 4, [ 3, 2, 3, 2 ], true ], 
  [ 4, [ 3, 2, 3, 2 ], true ], [ 4, [ 3, 3, 2, 4 ], true ], 
  [ 4, [ 3, 3, 2, 4 ], true ], [ 8, [ 3, 3, 3, 3 ], false ] ]
gap> s1;
[ [ 0, 0, 1 ], [ -1, -1, -1 ], [ 1, 0, 0 ] ]
gap> s2;
[ [ 0, 1, 0 ], [ 1, 0, 0 ], [ -1, -1, -1 ] ]
gap> s3;
[ [ -1, 0, 0 ], [ 0, -1, 0 ], [ 0, 0, -1 ] ]
gap> U2e:=Elements(N3g[2]);
[ [ [ -1, -1, -1 ], [ 0, 0, 1 ], [ 0, 1, 0 ] ], 
  [ [ -1, 0, 0 ], [ 0, -1, 0 ], [ 0, 0, -1 ] ], 
  [ [ 0, -1, 0 ], [ -1, 0, 0 ], [ 1, 1, 1 ] ], 
  [ [ 0, 0, -1 ], [ 1, 1, 1 ], [ -1, 0, 0 ] ], 
  [ [ 0, 0, 1 ], [ -1, -1, -1 ], [ 1, 0, 0 ] ], 
  [ [ 0, 1, 0 ], [ 1, 0, 0 ], [ -1, -1, -1 ] ], 
  [ [ 1, 0, 0 ], [ 0, 1, 0 ], [ 0, 0, 1 ] ], 
  [ [ 1, 1, 1 ], [ 0, 0, -1 ], [ 0, -1, 0 ] ] ]
gap> U2e[1]=s1*s2;
true
gap> U2e[2]=s3;
true
gap> U2e[3]=s3*s2;
true
gap> U2e[4]=s3*s1;
true
gap> U2e[5]=s1;
true
gap> U2e[6]=s2;
true
gap> U2e[7]=s3^0;
true
gap> U2e[8]=s3*s1*s2;
true
gap> label:=["s1*s2","s3","s3*s2","s3*s1","s1","s2","1","s3*s1*s2"];
[ "s1*s2", "s3", "s3*s2", "s3*s1", "s1", "s2", "1", "s3*s1*s2" ]
gap> List(U2WSEC.ConjugacyClassesSubgroups1,x->[Order(Representative(x)),
> CrystCatZClass(Representative(x)),IsInvertibleF(Representative(x)),
> List(GeneratorsOfGroup(Representative(x)),y->label[Position(U2e,y)])]);
[ [ 1, [ 3, 1, 1, 1 ], true, [  ] ], [ 2, [ 3, 2, 1, 2 ], true, [ "s1*s2" ] ],
  [ 2, [ 3, 1, 2, 1 ], true, [ "s3" ] ], 
  [ 2, [ 3, 2, 2, 2 ], true, [ "s3*s2" ] ], 
  [ 2, [ 3, 2, 2, 2 ], true, [ "s3*s1" ] ], 
  [ 2, [ 3, 2, 1, 2 ], true, [ "s1" ] ], [ 2, [ 3, 2, 1, 2 ], true, [ "s2" ] ]
    , [ 2, [ 3, 2, 2, 2 ], true, [ "s3*s1*s2" ] ], 
  [ 4, [ 3, 2, 3, 2 ], true, [ "s1*s2", "s3" ] ], 
  [ 4, [ 3, 3, 2, 4 ], true, [ "s1*s2", "s3*s2" ] ], 
  [ 4, [ 3, 3, 1, 3 ], false, [ "s1*s2", "s1" ] ], 
  [ 4, [ 3, 2, 3, 2 ], true, [ "s3", "s3*s2" ] ], 
  [ 4, [ 3, 2, 3, 2 ], true, [ "s3", "s3*s1" ] ], 
  [ 4, [ 3, 3, 2, 4 ], true, [ "s3*s2", "s1" ] ], 
  [ 4, [ 3, 3, 2, 4 ], true, [ "s3*s1", "s2" ] ], 
  [ 8, [ 3, 3, 3, 3 ], false, [ "s3", "s1", "s2" ] ] ]
gap> List(U2e,Trace);
[ -1, -3, 1, 1, -1, -1, 3, 1 ]
gap> List(U2WSEC.ConjugacyClassesSubgroups1,x->[Order(Representative(x)),
> CrystCatZClass(Representative(x)),IsInvertibleF(Representative(x)),
> Collected(List(GeneratorsOfGroup(Representative(x)),Trace))]);
[ [ 1, [ 3, 1, 1, 1 ], true, [  ] ], 
  [ 2, [ 3, 2, 1, 2 ], true, [ [ -1, 1 ] ] ], 
  [ 2, [ 3, 1, 2, 1 ], true, [ [ -3, 1 ] ] ], 
  [ 2, [ 3, 2, 2, 2 ], true, [ [ 1, 1 ] ] ], 
  [ 2, [ 3, 2, 2, 2 ], true, [ [ 1, 1 ] ] ], 
  [ 2, [ 3, 2, 1, 2 ], true, [ [ -1, 1 ] ] ], 
  [ 2, [ 3, 2, 1, 2 ], true, [ [ -1, 1 ] ] ], 
  [ 2, [ 3, 2, 2, 2 ], true, [ [ 1, 1 ] ] ], 
  [ 4, [ 3, 2, 3, 2 ], true, [ [ -3, 1 ], [ -1, 1 ] ] ], 
  [ 4, [ 3, 3, 2, 4 ], true, [ [ -1, 1 ], [ 1, 1 ] ] ], 
  [ 4, [ 3, 3, 1, 3 ], false, [ [ -1, 2 ] ] ], 
  [ 4, [ 3, 2, 3, 2 ], true, [ [ -3, 1 ], [ 1, 1 ] ] ], 
  [ 4, [ 3, 2, 3, 2 ], true, [ [ -3, 1 ], [ 1, 1 ] ] ], 
  [ 4, [ 3, 3, 2, 4 ], true, [ [ -1, 1 ], [ 1, 1 ] ] ], 
  [ 4, [ 3, 3, 2, 4 ], true, [ [ -1, 1 ], [ 1, 1 ] ] ], 
  [ 8, [ 3, 3, 3, 3 ], false, [ [ -3, 1 ], [ -1, 2 ] ] ] ]
gap> List(U2WSEC.ConjugacyClassesSubgroups1,x->[Order(Representative(x)),
> CrystCatZClass(Representative(x)),IsInvertibleF(Representative(x)),
> Collected(List(Representative(x),Trace))]);
[ [ 1, [ 3, 1, 1, 1 ], true, [ [ 3, 1 ] ] ], 
  [ 2, [ 3, 2, 1, 2 ], true, [ [ -1, 1 ], [ 3, 1 ] ] ], 
  [ 2, [ 3, 1, 2, 1 ], true, [ [ -3, 1 ], [ 3, 1 ] ] ], 
  [ 2, [ 3, 2, 2, 2 ], true, [ [ 1, 1 ], [ 3, 1 ] ] ], 
  [ 2, [ 3, 2, 2, 2 ], true, [ [ 1, 1 ], [ 3, 1 ] ] ], 
  [ 2, [ 3, 2, 1, 2 ], true, [ [ -1, 1 ], [ 3, 1 ] ] ], 
  [ 2, [ 3, 2, 1, 2 ], true, [ [ -1, 1 ], [ 3, 1 ] ] ], 
  [ 2, [ 3, 2, 2, 2 ], true, [ [ 1, 1 ], [ 3, 1 ] ] ], 
  [ 4, [ 3, 2, 3, 2 ], true, [ [ -3, 1 ], [ -1, 1 ], [ 1, 1 ], [ 3, 1 ] ] ], 
  [ 4, [ 3, 3, 2, 4 ], true, [ [ -1, 1 ], [ 1, 2 ], [ 3, 1 ] ] ], 
  [ 4, [ 3, 3, 1, 3 ], false, [ [ -1, 3 ], [ 3, 1 ] ] ], 
  [ 4, [ 3, 2, 3, 2 ], true, [ [ -3, 1 ], [ -1, 1 ], [ 1, 1 ], [ 3, 1 ] ] ], 
  [ 4, [ 3, 2, 3, 2 ], true, [ [ -3, 1 ], [ -1, 1 ], [ 1, 1 ], [ 3, 1 ] ] ], 
  [ 4, [ 3, 3, 2, 4 ], true, [ [ -1, 1 ], [ 1, 2 ], [ 3, 1 ] ] ], 
  [ 4, [ 3, 3, 2, 4 ], true, [ [ -1, 1 ], [ 1, 2 ], [ 3, 1 ] ] ], 
  [ 8, [ 3, 3, 3, 3 ], false, [ [ -3, 1 ], [ -1, 3 ], [ 1, 3 ], [ 3, 1 ] ] ] ]
\end{verbatim}
\end{example}

\section{Proofs of ${\rm (2)}\Leftrightarrow {\rm (3)}\Leftrightarrow {\rm (4)}$ of {Theorem \ref{thmain2}} $(N_{3,i})$, {Theorem \ref{thmain4}} $(N_{4,i})$ and {Theorem \ref{thmain5}} $(I_{4,i})$}\label{S10}

We have ${\rm (2)}\Rightarrow {\rm (3)}\Rightarrow {\rm (4)}$. 
Hence we should show that ${\rm (4)}\Rightarrow {\rm (2)}$.\\

We made the following GAP \cite{GAP} algorithm 
in order to confirm that ${\rm (4)}\Rightarrow {\rm (2)}$.  
It is available as in \cite{BCAlgTori}.\\

\noindent
{\tt AutGWSECInvariantSmallDegreeTest(}$G${\tt )} 
returns the list $l=[l_1,\ldots,l_s]$ $(l_1\leq \cdots\leq l_s)$ 
of integers with the minimal $l_s,\ldots,l_1$ 
which satisfies $Z=Z^\prime$ 
where 
\begin{align*}
Z&=\{\sigma\in{\rm Aut}(G)\mid [M_H]^{fl}\sim [M_{H^\sigma}]^{fl}\ {\rm for}\ {\rm any}\ H\leq G\},\\
Z^\prime&=\{\sigma\in{\rm Aut}(G)\mid [M_H]^{fl}\sim [M_{H^\sigma}]^{fl}\ {\rm for}\ {\rm any}\ H\leq G\ {\rm with}\ [G:H]\in l\}
\end{align*}
for $G\leq \GL(n,\bZ)$ $(n=3,4)$.\\

By using the function 
{\tt AutGWSECInvariantSmallDegreeTest(}$G${\tt )}, we can verify  
${\rm (4)}\Rightarrow {\rm (2)}$
of {Theorem \ref{thmain2}}, 
{Theorem \ref{thmain4}} 
and {Theorem \ref{thmain5} 
(see Examples \ref{ex10.1}, \ref{ex10.2}, \ref{ex10.3} for GAP computations).\qed

\bigskip
\begin{example}[{Proof of ${\rm (4)}\Rightarrow {\rm (2)}$ of {Theorem \ref{thmain2}}}]\label{ex10.1}~\vspace*{-5mm}\\
\begin{verbatim}
gap> Read("FlabbyResolutionBC.gap");
gap> N3g:=List(N3,x->MatGroupZClass(x[1],x[2],x[3],x[4]));;
gap> for i in [1..15] do
> Print(i,"\t",Order(N3g[i]),"\t",AutGWSECInvariantSmallDegreeTest(N3g[i]),"\n");
> od;
1	4	[ 1 ]
2	8	[ 1, 2 ]
3	8	[ 1 ]
4	8	[ 1 ]
5	8	[ 1, 2 ]
6	8	[ 1, 2 ]
7	16	[ 1, 2 ]
8	12	[ 1 ]
9	24	[ 1 ]
10	24	[ 1 ]
11	24	[ 1 ]
12	24	[ 1 ]
13	24	[ 1 ]
14	48	[ 1 ]
15	48	[ 1, 2 ]
\end{verbatim}
\end{example}

\bigskip
\begin{example}[{Proof of ${\rm (4)}\Rightarrow {\rm (2)}$ of {Theorem \ref{thmain4}}}]\label{ex10.2}~\vspace*{-5mm}\\
\begin{verbatim}
gap> Read("FlabbyResolutionBC.gap");
gap> N4g:=List(N4,x->MatGroupZClass(x[1],x[2],x[3],x[4]));;
gap> for i in [1..152] do
> Print(i,"\t",Order(N4g[i]),"\t",AutGWSECInvariantSmallDegreeTest(N4g[i]),"\n");
> od;
1	4	[ 1 ]
2	8	[ 1 ]
3	8	[ 1, 2 ]
4	8	[ 1 ]
5	8	[ 1 ]
6	8	[ 1, 2 ]
7	8	[ 1 ]
8	8	[ 1, 2 ]
9	8	[ 1, 2 ]
10	16	[ 1 ]
11	16	[ 1, 2 ]
12	16	[ 1, 2 ]
13	8	[ 1 ]
14	8	[ 1 ]
15	8	[ 1, 2 ]
16	8	[ 1, 2 ]
17	8	[ 1, 2 ]
18	8	[ 1, 2 ]
19	16	[ 1, 2 ]
20	16	[ 1, 2 ]
21	16	[ 1, 4 ]
22	16	[ 1, 2 ]
23	8	[ 1 ]
24	8	[ 1 ]
25	8	[ 1, 2 ]
26	8	[ 1, 2 ]
27	16	[ 1 ]
28	16	[ 1, 2, 4 ]
29	16	[ 1, 2 ]
30	16	[ 1, 2 ]
31	16	[ 1, 2 ]
32	16	[ 1, 2, 4 ]
33	16	[ 1, 2 ]
34	16	[ 1, 2, 4 ]
35	16	[ 1, 2 ]
36	16	[ 1, 4 ]
37	32	[ 1, 2 ]
38	32	[ 1, 4 ]
39	8	[ 1 ]
40	16	[ 1, 2, 4 ]
41	16	[ 1 ]
42	16	[ 1 ]
43	16	[ 1, 2 ]
44	16	[ 1 ]
45	16	[ 1, 4 ]
46	16	[ 1, 2 ]
47	32	[ 1, 4 ]
48	32	[ 1, 4 ]
49	32	[ 1 ]
50	16	[ 1, 2 ]
51	16	[ 1, 2 ]
52	32	[ 1, 4 ]
53	32	[ 1, 4 ]
54	32	[ 1, 2 ]
55	32	[ 1, 2 ]
56	64	[ 1, 4 ]
57	9	[ 1 ]
58	18	[ 1 ]
59	18	[ 1 ]
60	18	[ 1 ]
61	18	[ 1 ]
62	18	[ 1 ]
63	36	[ 1, 2 ]
64	36	[ 1, 2 ]
65	36	[ 1 ]
66	36	[ 1, 2 ]
67	36	[ 1 ]
68	72	[ 1, 2 ]
69	24	[ 1 ]
70	24	[ 1 ]
71	24	[ 1 ]
72	48	[ 1, 2 ]
73	48	[ 1 ]
74	24	[ 1 ]
75	24	[ 1 ]
76	24	[ 1 ]
77	24	[ 1 ]
78	24	[ 1 ]
79	48	[ 1 ]
80	48	[ 1, 2 ]
81	48	[ 1 ]
82	48	[ 1 ]
83	48	[ 1, 2 ]
84	48	[ 1 ]
85	48	[ 1, 2 ]
86	48	[ 1, 6 ]
87	48	[ 1 ]
88	48	[ 1, 6 ]
89	96	[ 1, 2 ]
90	96	[ 1, 2 ]
91	18	[ 1 ]
92	18	[ 1 ]
93	36	[ 1, 2 ]
94	36	[ 1 ]
95	36	[ 1 ]
96	36	[ 1, 2 ]
97	36	[ 1 ]
98	36	[ 1 ]
99	72	[ 1, 2 ]
100	72	[ 1, 2 ]
101	72	[ 1 ]
102	72	[ 1, 2 ]
103	72	[ 1, 2 ]
104	72	[ 1 ]
105	144	[ 1, 2 ]
106	8	[ 1 ]
107	16	[ 1 ]
108	16	[ 1 ]
109	16	[ 1 ]
110	24	[ 1 ]
111	24	[ 1 ]
112	32	[ 1, 8 ]
113	32	[ 1, 4 ]
114	32	[ 1 ]
115	32	[ 1 ]
116	32	[ 1 ]
117	32	[ 1, 4 ]
118	48	[ 1 ]
119	48	[ 1 ]
120	64	[ 1, 4 ]
121	64	[ 1 ]
122	64	[ 1, 8 ]
123	64	[ 1, 8 ]
124	64	[ 1 ]
125	64	[ 1, 2 ]
126	96	[ 1 ]
127	96	[ 1 ]
128	128	[ 1, 8 ]
129	192	[ 1, 8 ]
130	192	[ 1 ]
131	192	[ 1, 8 ]
132	192	[ 1 ]
133	192	[ 1 ]
134	192	[ 1 ]
135	384	[ 1, 8 ]
136	384	[ 1, 8 ]
137	24	[ 1 ]
138	24	[ 1 ]
139	48	[ 1 ]
140	48	[ 1 ]
141	48	[ 1 ]
142	72	[ 1, 3 ]
143	96	[ 1, 12 ]
144	96	[ 1, 24 ]
145	96	[ 1 ]
146	144	[ 1, 3, 8 ]
147	192	[ 1, 12 ]
148	288	[ 1, 3 ]
149	576	[ 1, 24 ]
150	576	[ 1, 32 ]
151	576	[ 1, 3, 32 ]
152	1152	[ 1, 2, 32 ]
\end{verbatim}
\end{example}

\bigskip
\begin{example}[{Proof of ${\rm (4)}\Rightarrow {\rm (2)}$ of {Theorem \ref{thmain2}}}]\label{ex10.3}~\vspace*{-5mm}\\
\begin{verbatim}
gap> Read("FlabbyResolutionBC.gap");
gap> I4g:=List(I4,x->MatGroupZClass(x[1],x[2],x[3],x[4]));;
gap> for i in [1..7] do
> Print(i,"\t",Order(I4g[i]),"\t",AutGWSECInvariantSmallDegreeTest(I4g[i]),"\n");
> od;
1	20	[ 1 ]
2	20	[ 1 ]
3	40	[ 1, 2 ]
4	120	[ 1 ]
5	120	[ 1 ]
6	240	[ 1, 2 ]
7	24	[ 1 ]
\end{verbatim}
\end{example}

\section{Proof of Theorem \ref{thmain6} (Main theorem $6$): relationships among the $I_{4,i}$ cases}\label{S11}

Let $T_i$ $(1\leq i\leq 7)$ be an algebraic $k$-torus of dimension $4$ 
with the minimal splitting field $L_i$ and 
the character module $\widehat{T}_i=M_{G_i}$ which satisfies that 
$G_i\leq \GL(4,\bZ)$ is $\GL(4,\bZ)$-conjugate to $I_{4,i}$. 
Let $T_i^\sigma$ be the algebraic $k$-torus 
with the minimal splitting field $L_i$ and 
$\widehat{T}_i^\sigma=M_{G_i^\sigma}$ $(\sigma\in {\rm Aut}(G_i))$. 

We should show that 
$[M_{G_i}]^{fl}+[M_{G_i^\sigma}]^{fl}=0$ as $\widetilde{H}$-lattices 
for the cases 
(4) $G_6=I_{4,6}\simeq S_5\times C_2$ and 
(5) $G_7=I_{4,7}\simeq C_3\rtimes C_8$ 
because the cases 
(1) $G_1=I_{4,1}\simeq I_{4,2}\simeq F_{20}$, 
(2) $G_3=I_{4,3}\simeq F_{20}\times C_2$ 
and (3) $G_4=I_{4,4}\simeq I_{4,5}\simeq S_5$ 
are obtained by restricting the action of (4) $G_6=I_{4,6}$ 
to the subgroups $I_{4,1}, I_{4,2}, I_{4,3}, I_{4,4}, I_{4,5}\leq I_{4,6}$ 
(see Hoshi and Yamasaki \cite[Lemma 2.17]{HY17}) 
and ${\rm Aut}(G_6)/{\rm Inn}(G_6)\simeq C_2$ 
permutes $G_1=I_{4,1}\simeq F_{20}$ and $G_2=I_{4,2}\simeq F_{20}$
(resp. $G_4=I_{4,4}\simeq S_5$ and $G_5=I_{4,5}\simeq S_5$) 
(see Table $13$).\\

(4) Case $G=G_6=I_{4,6}\simeq S_5\times C_2$. 

By using the function 
{\tt FlabbyResolutionLowRank(}$G${\tt).actionF} (see Section \ref{S6}), 
we can obtain $F=[M_G]^{fl}$ with ${\rm rank}_\bZ F=16$ 
and $F^\prime=[M_{G^\sigma}]^{fl}$ for 
$1\neq\overline{\sigma}\in {\rm Aut}(G)/{\rm Inn}(G)$. 

By using the function 
{\tt PossibilityOfStablyPermutationM(}$F\oplus F^\prime${\tt )}, 
we get the list
\begin{align*}
&l=[0,0,0,0,0,0,0,0,0,0,0,0,0,0,0,0,0,0,0,0,1,0,0,0,0,0,0,0,0,0,0,-1,\\
&\qquad 0,0,-1,0,0,-1,0,-1,0,-1,0,0,0,1,1,1,1,0,0,-2,1,-1,-1,-1,2,1]
\end{align*} 
(see Hoshi and Yamasaki \cite[Section 5.5]{HY17}). 

Applying the function 
{\tt StablyPermutationMCheckP(}$F\oplus F^\prime${\tt ,Nlist(}$l${\tt ),Plist(}$l${\tt ))}, 
we may confirm the isomorphism 
\begin{align*}
&\bZ[G/H_{21}]\oplus\bZ[G/H_{46}]\oplus \bZ[G/H_{47}]\oplus 
\bZ[G/H_{48}]\oplus \bZ[G/H_{49}]\oplus \bZ[G/H_{53}]\oplus 
\bZ^{\oplus 2}\oplus F\oplus F^\prime\\
&\simeq 
 \bZ[G/H_{32}]\oplus \bZ[G/H_{35}]\oplus 
 \bZ[G/H_{38}]\oplus  \bZ[G/H_{40}]\oplus  \bZ[G/H_{42}]\oplus 
 \bZ[G/H_{52}]^{\oplus 2}\oplus 
 \bZ[G/H_{54}]\\
&\quad \oplus  \bZ[G/H_{55}]\oplus  \bZ[G/H_{56}]
\end{align*} 
as $G$-lattices 
with rank 
$40+12+10+10+10+4+2\times 1+16+16=24+20+20+20+20+2\times 5+2+2+2=120$ 
where 
$H_{21}\simeq C_6$, 
$H_{46}\simeq F_{20}$, 
$H_{47}\simeq A_4\times C_2$, 
$H_{48}\simeq S_4$, 
$H_{49}\simeq S_4$, 
$H_{53}\simeq A_5$, 
$H_{57}\simeq S_5\times C_2$, 
$H_{32}\simeq D_5$, 
$H_{35}\simeq A_4$, 
$H_{38}\simeq D_6$, 
$H_{40}\simeq D_6$, 
$H_{42}\simeq D_6$, 
$H_{52}\simeq S_4\times C_2$, 
$H_{54}\simeq S_5$, 
$H_{55}\simeq A_5\times C_2$, 
$H_{56}\simeq S_5$. 
This implies that $[F]+[F^\prime]=0$ (see Example \ref{ex11.1} for GAP computations).\\

(5) Case $G=G_7=I_{4,7}\simeq C_3\rtimes C_8$. 
By using the function 
{\tt FlabbyResolutionLowRank(}$G${\tt).actionF} (see Section \ref{S6}), 
we can obtain $F=[M_G]^{fl}$ with ${\rm rank}_\bZ F=20$ 
and $F^\prime=[M_{G^\sigma}]^{fl}$ for 
$1\neq\overline{\sigma}\in {\rm Aut}(G)/X\simeq C_2$ 
where $X={\rm Aut}_{\GL(4,\bZ)}(G)=\{\sigma\in{\rm Aut}(G)\mid G\ {\rm and}\ G^\sigma\ {\rm are}\ 
{\rm conjugate}\ {\rm in}\ \GL(4,\bZ)\}\simeq N_{\GL(4,\bZ)}(G)/Z_{\GL(4,\bZ)}(G)\simeq D_6$. 

By using the function 
{\tt PossibilityOfStablyPermutationM(}$F\oplus F^\prime${\tt )}, 
we get the list $l=[ 1, 1, 1, 0, -1, 0, 0, 0, -1 ]$ 
(see Hoshi and Yamasaki \cite[Section 5.5]{HY17}). 

Applying the function 
{\tt StablyPermutationMCheckP(}$F\oplus F^\prime${\tt ,Nlist(}$l${\tt ),Plist(}$l${\tt ))}, 
we may confirm the isomorphism 
\begin{align*}
\bZ[G]\oplus\bZ[G/H_2]\oplus\bZ[G/H_3]
\simeq 
\bZ[G/H_5]\oplus F\oplus F^\prime
\end{align*}
as $G$-lattices 
with rank $24+12+8=4+20+20=44$ 
where 
$H_2\simeq C_2$, 
$H_3\simeq C_3$, 
$H_5\simeq C_6$. 
Hence we conclude that $[F]+[F^\prime]=0$ 
(see Example \ref{ex11.2} for GAP computations).\qed

\bigskip
\begin{example}[{Proof of ${\rm (4)}$ of {Theorem \ref{thmain6}}}]\label{ex11.1}~\vspace*{-5mm}\\
\begin{verbatim}
gap> Read("BCAlgTori.gap");
gap> I4g:=List(I4,x->MatGroupZClass(x[1],x[2],x[3],x[4]));
[ MatGroupZClass( 4, 31, 1, 3 ), MatGroupZClass( 4, 31, 1, 4 ), 
  MatGroupZClass( 4, 31, 2, 2 ), MatGroupZClass( 4, 31, 4, 2 ), 
  MatGroupZClass( 4, 31, 5, 2 ), MatGroupZClass( 4, 31, 7, 2 ), 
  MatGroupZClass( 4, 33, 2, 1 ) ]
gap> G:=I4g[6]; # G=S5xC2
MatGroupZClass( 4, 31, 7, 2 )
gap> AutG:=AutomorphismGroup(G);
<group of size 240 with 4 generators>
gap> List(GeneratorsOfGroup(AutG),IsInnerAutomorphism); # AutG.4 is not Inn(G)
[ true, true, true, false ]
gap> F:=FlabbyResolutionLowRank(G).actionF; # F=[M_G]^fl
<matrix group with 3 generators>
gap> Rank(F.1); # F is of rank 16
16
gap> gf:=GroupHomomorphismByImages(G,F,GeneratorsOfGroup(G),GeneratorsOfGroup(F));;
gap> Gprime:=Group(List(GeneratorsOfGroup(G),x->Image(AutG.4,x))); # G'
<matrix group with 3 generators>
gap> Fprime:=Group(List(GeneratorsOfGroup(Gprime),x->Image(gf,x))); # F'
<matrix group with 3 generators>
gap> FFprime:=DirectSumMatrixGroup([F,Fprime]); # F+F'
<matrix group with 3 generators>
gap> IdSmallGroup(FFprime); 
[ 240, 189 ]
gap> IdSmallGroup(G); 
[ 240, 189 ]
gap> ll:=PossibilityOfStablyPermutationM(FFprime);;
gap> Length(ll);
18
gap> l:=ll[16];
[ 0, 0, 0, 0, 0, 0, 0, 0, 0, 0, 0, 0, 0, 0, 0, 0, 0, 0, 0, 0, 1, 0, 0, 0, 0, 
  0, 0, 0, 0, 0, 0, -1, 0, 0, -1, 0, 0, -1, 0, -1, 0, -1, 0, 0, 0, 1, 1, 1, 
  1, 0, 0, -2, 1, -1, -1, -1, 2, 1 ]
gap> Length(l);
58
gap> ss:=List(ConjugacyClassesSubgroups2(G),
> x->StructureDescription(Representative(x)));
[ "1", "C2", "C2", "C2", "C2", "C2", "C3", "C2 x C2", "C2 x C2", "C2 x C2", 
  "C2 x C2", "C2 x C2", "C2 x C2", "C4", "C4", "C2 x C2", "C5", "C6", "S3", 
  "S3", "C6", "S3", "C6", "S3", "C2 x C2 x C2", "D8", "D8", "C4 x C2", 
  "C2 x C2 x C2", "D8", "D8", "D10", "C10", "D10", "A4", "D12", "D12", "D12", 
  "C6 x C2", "D12", "D12", "D12", "C2 x D8", "C5 : C4", "D20", "C5 : C4", 
  "C2 x A4", "S4", "S4", "C2 x C2 x S3", "C2 x (C5 : C4)", "C2 x S4", "A5", 
  "S5", "C2 x A5", "S5", "C2 x S5" ]
gap> Length(ss);
57
gap> Nlist(l);
[ 0, 0, 0, 0, 0, 0, 0, 0, 0, 0, 0, 0, 0, 0, 0, 0, 0, 0, 0, 0, 0, 0, 0, 0, 0, 
  0, 0, 0, 0, 0, 0, 1, 0, 0, 1, 0, 0, 1, 0, 1, 0, 1, 0, 0, 0, 0, 0, 0, 0, 0, 
  0, 2, 0, 1, 1, 1, 0, 0 ]
gap> Plist(l);
[ 0, 0, 0, 0, 0, 0, 0, 0, 0, 0, 0, 0, 0, 0, 0, 0, 0, 0, 0, 0, 1, 0, 0, 0, 0, 
  0, 0, 0, 0, 0, 0, 0, 0, 0, 0, 0, 0, 0, 0, 0, 0, 0, 0, 0, 0, 1, 1, 1, 1, 0, 
  0, 0, 1, 0, 0, 0, 2, 1 ]
gap> Filtered([1..57],x->Nlist(l)[x]>0);
[ 32, 35, 38, 40, 42, 52, 54, 55, 56 ]
gap> ss{last};
[ "D10", "A4", "D12", "D12", "D12", "C2 x S4", "S5", "C2 x A5", "S5" ]
gap> Filtered([1..57],x->Plist(l)[x]>0);
[ 21, 46, 47, 48, 49, 53, 57 ]
gap> ss{last};                          
[ "C6", "C5 : C4", "C2 x A4", "S4", "S4", "A5", "C2 x S5" ]
gap> bp:=StablyPermutationMCheckP(FFprime,Nlist(l),Plist(l));;
gap> SearchPRowBlocks(bp);
rec( 
  bpBlocks := 
    [ [ 1, 2, 3, 4, 5, 6, 7, 8, 9, 10, 11, 12, 13, 14, 15, 16, 17, 18, 19, 
          20, 21, 22, 23, 24, 25, 26, 27, 28 ], 
      [ 29, 30, 31, 32, 33, 34, 35, 36, 37, 38, 39, 40, 41, 42, 43, 44, 45, 
          46, 47, 48, 49, 50, 51, 52, 53, 54, 55, 56, 57, 58 ], 
      [ 59, 60, 61, 62, 63, 64, 65, 66, 67, 68, 69, 70, 71, 72, 73, 74, 75, 
          76, 77, 78, 79, 80, 81, 82, 83, 84 ], 
      [ 85, 86, 87, 88, 89, 90, 91, 92, 93, 94, 95, 96, 97, 98, 99, 100, 
          101, 102, 103, 104, 105, 106, 107, 108, 109, 110 ], 
      [ 111, 112, 113, 114, 115, 116, 117, 118, 119, 120, 121, 122, 123, 
          124, 125, 126, 127, 128, 129, 130, 131, 132, 133, 134, 135, 136, 
          137, 138 ], 
      [ 139, 140, 141, 142, 143, 144, 145, 146, 147, 148, 149, 150, 151, 
          152, 153, 154 ], 
      [ 155, 156, 157, 158, 159, 160, 161, 162, 163, 164, 165, 166, 167, 
          168, 169, 170 ], 
      [ 171, 172, 173, 174, 175, 176, 177, 178, 179, 180, 181, 182 ], 
      [ 183, 184, 185, 186, 187, 188, 189, 190, 191, 192, 193, 194, 195, 196 ], 
      [ 197, 198, 199, 200, 201, 202, 203, 204, 205, 206, 207, 208 ] ],
  rowBlocks := 
    [ [ 1, 2, 3, 4, 5, 6, 7, 8, 9, 10, 11, 12, 13, 14, 15, 16, 17, 18, 19, 
          20, 21, 22, 23, 24 ], 
      [ 25, 26, 27, 28, 29, 30, 31, 32, 33, 34, 35, 36, 37, 38, 39, 40, 41, 
          42, 43, 44 ], 
      [ 45, 46, 47, 48, 49, 50, 51, 52, 53, 54, 55, 56, 57, 58, 59, 60, 61, 
          62, 63, 64 ], 
      [ 65, 66, 67, 68, 69, 70, 71, 72, 73, 74, 75, 76, 77, 78, 79, 80, 81, 
          82, 83, 84 ], 
      [ 85, 86, 87, 88, 89, 90, 91, 92, 93, 94, 95, 96, 97, 98, 99, 100, 
          101, 102, 103, 104 ], [ 105, 106, 107, 108, 109 ], 
      [ 110, 111, 112, 113, 114 ], [ 115, 116 ], [ 117, 118 ], [ 119, 120 ] ] )

# after some efforts we may get 
gap> n1:=
> [ 1, 1, 0, 1, 0, 1, 0, 1, 0, 0, 0, 1, 0, 0, 0, 0, 0, 1, 1, 1, 1, 1, 1, 1, 1, 
>   0, 1, 1, 1, 1, 1, 1, 0, 1, 0, 1, 0, 1, 1, 1, 0, 0, 0, 1, 1, 0, 0, 0, 1, 1, 
>   1, 0, 1, 0, 1, 1, 1, 0, 1, 0, 1, 0, 0, 0, 0, 1, 0, 1, 1, 1, 0, 1, 0, 1, 0, 
>   0, 1, 1, 1, 1, 1, 0, 0, 0, 1, 1, 1, 0, 1, 1, 1, 0, 0, 0, 1, 1, 0, 1, 1, 0, 
>   0, 1, 0, 0, 0, 0, 1, 1, 1, 0, 0, 0, 0, 0, 0, 1, 0, 0, 1, 0, 1, 1, 0, 1, 1, 
>   0, 1, 1, 1, 0, 1, 0, 0, 1, 0, 0, 0, 0, 0, 0, 0, 1, 0, 0, 1, 1, 1, 0, 1, 1, 
>   1, 0, 0, 1, 0, 0, 1, 1, 0, 1, 0, 1, 1, 1, 1, 0, 1, 0, 1, 0, 0, 1, -1, -1, 
>   0, 1, -1, 1, 1, 0, 0, 1, 0, 0, 1, 1, 0, 1, 1, -1, -1, 1, 1, -1, 1, 0, 3, 
>   1, 2, -2, 1, 0, 2, 0, -2, 0, 2, 1 ];;
gap> P:=n1*bp;;
gap> Size(P);
120
gap> Determinant(P);
1
gap> StablyPermutationMCheckMat(FFprime,Nlist(l),Plist(l),P);
true
\end{verbatim}
\end{example}

\bigskip
\begin{example}[{Proof of ${\rm (5)}$ of {Theorem \ref{thmain6}}}]\label{ex11.2}~\vspace*{-5mm}\\
\begin{verbatim}
gap> Read("BCAlgTori.gap");
gap> I4g:=List(I4,x->MatGroupZClass(x[1],x[2],x[3],x[4]));
[ MatGroupZClass( 4, 31, 1, 3 ), MatGroupZClass( 4, 31, 1, 4 ), 
  MatGroupZClass( 4, 31, 2, 2 ), MatGroupZClass( 4, 31, 4, 2 ), 
  MatGroupZClass( 4, 31, 5, 2 ), MatGroupZClass( 4, 31, 7, 2 ), 
  MatGroupZClass( 4, 33, 2, 1 ) ]
gap> G:=I4g[7]; # G=C3:C8
MatGroupZClass( 4, 33, 2, 1 )
gap> StructureDescription(G);
"C3 : C8"
gap> StructureDescription(Centre(G)); # Z(G)=C4
"C4"
gap> StructureDescription(G/Centre(G)); # Inn(G)=G/Z(G)=S3, |Inn(G)|=6
"S3"
gap> AutG:=AutomorphismGroup(G); 
<group of size 24 with 4 generators>
gap> StructureDescription(AutG); # Aut(G)=C2xC2xS3, |Aut(G)|=24
"C2 x C2 x S3"
gap> Order(AutGLnZ(G)); # |X|=|AutGLnZ(G)|=12 
12
gap> Order(Centre(AutGLnZ(G))); # X=AutGLnZ(G)=D6
2
gap> List(GeneratorsOfGroup(AutG),IsInnerAutomorphism); # AutG.1 is not in Inn(G)
[ false, false, true, true ]
gap> List(GeneratorsOfGroup(AutG),x->x in AutGLnZ(G)); # AutG.1 is not in X
[ false, true, true, true ]
gap> F:=FlabbyResolutionLowRank(G).actionF; # F=[M_G]^fl
<matrix group with 2 generators>
gap> Rank(F.1); # F is of rank 20
20
gap> gf:=GroupHomomorphismByImages(G,F,GeneratorsOfGroup(G),GeneratorsOfGroup(F));;
gap> Gprime:=Group(List(GeneratorsOfGroup(G),x->Image(AutG.1,x))); # G'
<matrix group with 2 generators>
gap> Fprime:=Group(List(GeneratorsOfGroup(Gprime),x->Image(gf,x))); # F'
<matrix group with 2 generators>
gap> FFprime:=DirectSumMatrixGroup([F,Fprime]); # F+F'
<matrix group with 2 generators>
gap> IdSmallGroup(FFprime);
[ 24, 1 ]
gap> IdSmallGroup(G);
[ 24, 1 ]
gap> PossibilityOfStablyPermutationM(FFprime);
[ [ 1, 1, 1, 0, -1, 0, 0, 0, -1 ] ]
gap> l:=last[1];
[ 1, 1, 1, 0, -1, 0, 0, 0, -1 ]
gap> Length(l);
9
gap> ss:=List(ConjugacyClassesSubgroups2(G),
> x->StructureDescription(Representative(x)));
[ "1", "C2", "C3", "C4", "C6", "C8", "C12", "C3 : C8" ]
gap> Length(ss);
8
gap> Nlist(l);
[ 0, 0, 0, 0, 1, 0, 0, 0, 1 ]
gap> Plist(l);
[ 1, 1, 1, 0, 0, 0, 0, 0, 0 ]
gap> Filtered([1..8],x->Nlist(l)[x]>0);
[ 5 ]
gap> ss{last};
[ "C6" ]
gap> Filtered([1..8],x->Plist(l)[x]>0);
[ 1, 2, 3 ]
gap> ss{last};                         
[ "1", "C2", "C3" ]
gap> bp:=StablyPermutationMCheckP(FFprime,Nlist(l),Plist(l));;
gap> SearchPRowBlocks(bp);
rec( 
  bpBlocks := [ [ 1, 2, 3, 4, 5, 6, 7, 8, 9, 10, 11, 12 ], 
      [ 13, 14, 15, 16, 17, 18, 19, 20, 21, 22, 23, 24, 25, 26, 27, 28, 29, 
          30, 31, 32, 33, 34, 35, 36, 37, 38, 39, 40, 41, 42, 43, 44, 45, 46, 
          47, 48, 49, 50, 51, 52 ], 
      [ 53, 54, 55, 56, 57, 58, 59, 60, 61, 62, 63, 64, 65, 66, 67, 68, 69, 
          70, 71, 72, 73, 74, 75, 76, 77, 78, 79, 80, 81, 82, 83, 84, 85, 86, 
          87, 88, 89, 90, 91, 92 ] ], 
  rowBlocks := 
    [ [ 1, 2, 3, 4 ], 
      [ 5, 6, 7, 8, 9, 10, 11, 12, 13, 14, 15, 16, 17, 18, 19, 20, 21, 22, 
          23, 24 ], 
      [ 25, 26, 27, 28, 29, 30, 31, 32, 33, 34, 35, 36, 37, 38, 39, 40, 41, 
          42, 43, 44 ] ] )

# after some efforts we may get 
gap> nn:=
> [ 1, -1, -1, 1, 1, 0, 0, 0, 0, -1, 0, 0, 1, 1, 1, 0, 1, 1, 1, 0, 0, 0, 1, 1, 
>   0, 1, 1, 1, 1, 0, 0, 1, 0, 0, 0, 0, 1, 1, 1, 0, 1, 1, 1, 1, 1, 0, 0, 0, 1, 
>   0, 1, 0, 0, 1, 0, 0, 1, 1, 0, 1, 0, 0, 0, 1, 0, 1, 1, 0, 0, 1, 0, 1, 1, 1, 
>   0, 0, 1, 0, 0, 0, 0, 1, 0, 0, 0, 1, 0, 1, 1, 0, 0, 1 ];;
gap> P:=nn*bp;;
gap> Size(P);
44
gap> Determinant(P);
-1
gap> StablyPermutationMCheckMat(FFprime,Nlist(l),Plist(l),P);
true
\end{verbatim}
\end{example}

\section{Proof of Theorem \ref{thmain7} (Main theorem $7$): higher dimensions}\label{S12}

%
By {Theorem \ref{thmain2}}, 
{Theorem \ref{thmain4}} and {Theorem \ref{thmain5}}, 
there exists 
an algebraic $k$-torus $T^{\prime\prime}$ of dimension $3$ or $4$ 
with $\widehat{T}^{\prime\prime}=M_{G^{\prime\prime}}$, 
$G^{\prime\prime}=N_{3,i}$ $(1\leq i\leq 15)$, 
$G^{\prime\prime}=N_{4,i}$ $(1\leq i\leq 152)$ or 
$G^{\prime\prime}=I_{4,i}$ $(1\leq i\leq 7)$ 
such that $[M_{G^{\prime\prime}}]^{fl}\sim [M_G]^{fl}$, 
i.e. there exists a subdirect product 
$\widetilde{H}\leq G^{\prime\prime}\times G$ 
of $G^{\prime\prime}$ and $G$ 
with surjections $\varphi_1: \widetilde{H} \rightarrow G^{\prime\prime}$ 
and $\varphi_2: \widetilde{H} \rightarrow G$ 
such that $[M_{G^{\prime\prime}}]^{fl}=[M_G]^{fl}$ 
as $\widetilde{H}$-lattices where $\widetilde{H}$ acts on 
$M_{G^{\prime\prime}}$ and $M_G$ through the surjections 
$\varphi_1$ and $\varphi_2$ respectively. 
 
We take a subdirect product 
$\widetilde{S}:=\{(h_1,h_2)\in\widetilde{H}_1\times \widetilde{H}_2\mid 
\varphi_2(h_1)=\varphi_2(h_2)\}\leq \widetilde{H}_1\times \widetilde{H}_2$ 
of $\widetilde{H}_1\simeq \widetilde{H}$ and 
$\widetilde{H}_2\simeq \widetilde{H}$ as $\widetilde{H}$-lattices 
with surjections 
$\theta_1: \widetilde{S}\rightarrow \widetilde{H}_1, (h_1,h_2)\mapsto h_1$ and 
$\theta_2: \widetilde{S}\rightarrow \widetilde{H}_2, (h_1,h_2)\mapsto h_2$. 
Then we have 
$[M_{\varphi_1\theta_1(\widetilde{S})}]^{fl}=[M_{\varphi_2\theta_1(\widetilde{S})}]^{fl}$ 
and 
$[M_{\varphi_1\theta_2(\widetilde{S})}]^{fl}=[M_{\varphi_2\theta_2(\widetilde{S})}]^{fl}$ as $\widetilde{S}$-lattices. 
On the other hand, it follows from the condition $\varphi_2\theta_1=\varphi_2\theta_2$ that 
$[M_{\varphi_2\theta_1(\widetilde{S})}]^{fl}=[M_{\varphi_2\theta_2(\widetilde{S})}]^{fl}$ as $\widetilde{S}$-lattices. 
Hence we obtain that 
$[M_{\varphi_1\theta_1(\widetilde{S})}]^{fl}=[M_{\varphi_1\theta_2(\widetilde{S})}]^{fl}$ as $\widetilde{S}$-lattices. 

We also consider a subdirect product 
$\widetilde{U}:=\{(\varphi_1\theta_1(x),\varphi_1\theta_2(x))\in G^{\prime\prime}\times G^{\prime\prime}\mid x\in \widetilde{S}\}
\leq 
\varphi_1\theta_1(\widetilde{S})\times \varphi_1\theta_2(\widetilde{S})$ 
of $\varphi_1\theta_1(\widetilde{S})=G^{\prime\prime}$ and 
$\varphi_1\theta_2(\widetilde{S})=G^{\prime\prime}$ 
with surjections 
$\mu_1:\widetilde{U}\rightarrow \varphi_1\theta_1(\widetilde{S})$ and 
$\mu_2:\widetilde{U}\rightarrow \varphi_1\theta_2(\widetilde{S})$. 
Then we have a surjection $\widetilde{S}\rightarrow \widetilde{U}$, 
$x\mapsto (\varphi_1\theta_1(x),\varphi_1\theta_2(x))$. 
Hence $[M_{\mu_1(\widetilde{U})}]^{fl}=[M_{\mu_2(\widetilde{U})}]^{fl}$ 
as $\widetilde{U}$-lattices. 
Then it follows from Theorem \ref{th8.1}, 
Theorem \ref{th8.2} and Theorem \ref{th8.3} that 
$\mu_1:\widetilde{U}\rightarrow \varphi_1\theta_1(\widetilde{S})\simeq G^{\prime\prime}$ and 
$\mu_2:\widetilde{U}\rightarrow \varphi_1\theta_2(\widetilde{S})\simeq G^{\prime\prime}$. 

Define $N_1=\varphi_1({\rm Ker}(\varphi_2))$. 
We will show that $N_1=1$. 
Suppose that $N_1\neq 1$ and take $n_1\in N_1$ with $n_1\neq 1$. 
Then $(1,1), (n_1,1)\in \widetilde{H}$, 
$((1,1),(n_1,1))\in \widetilde{S}$ and hence 
$(1,1), (1,n_1)\in \widetilde{U}$. 
This implies that $\mu_1$ is not injection and it yields a contradiction. 

We conclude that $N_1=1$. 
This implies ${\rm Ker}(\varphi_2)\subset {\rm Ker}(\varphi_1)$ 
and we have $L^{\prime\prime}\subset L$.\qed\\

As a consequence of Theorem \ref{thmain7} 
(Main theorem $7$ 
(with the aid of the proofs of Theorem \ref{th8.1}, Theorem \ref{th8.2} 
and Theorem \ref{th8.3})), 
we obtain the following: 

\begin{theorem}\label{th8.6}
Let $G^{\prime\prime}$ be a finite subgroup of $\GL(n,\bZ)$ $(n=3,4)$
which is conjugate to one of the groups 
{\rm (i)} $N_{3,i}$ $(1\leq i\leq 15)$; 
{\rm (ii)} $N_{4,i}$ $(1\leq i\leq 152)$; 
{\rm (iii)} $I_{4,i}$ $(1\leq i\leq 7)$ 
and $M_{G^{\prime\prime}}$ be the corresponding 
$G^{\prime\prime}$-lattice as in Definition \ref{d2.2}. 
If $[M_{G^{\prime\prime}}]^{fl}=[F]$ as $G^{\prime\prime}$-lattices, 
then $F$ is a faithful $G^{\prime\prime}$-lattice. 
\end{theorem}
\begin{proof}
Let $N^{\prime\prime}=\{\sigma^{\prime\prime}\in G^{\prime\prime}\mid 
f^{\sigma^{\prime\prime}}=f\ {\rm for\ any}\ f\in F\}
\lhd G^{\prime\prime}$ be the kernel 
of the action of $G^{\prime\prime}$ on $F$, i.e. 
$G^{\prime\prime}/N^{\prime\prime}$ acts on $F$ faithfully. 
We should show that $N^{\prime\prime}=1$. 

We take a coflabby resolution of $F$ 
as $(G^{\prime\prime}/N^{\prime\prime})$-lattice, 
i.e. an exact sequence of 
$(G^{\prime\prime}/N^{\prime\prime})$-lattices 
$0\to C\to P\to F\to 0$ 
with $C$ coflabby and $P$ permutation (see Definition \ref{defF}). 
Then there exists $G\leq {\rm GL}(m,\bZ)$ 
such that $M_{G}\simeq C$ as $(G^{\prime\prime}/N^{\prime\prime})$-lattices 
with surjections 
$\psi_1: G^{\prime\prime}\to G^{\prime\prime}/N^{\prime\prime}$ 
and
$\psi_2: G^{\prime\prime}/N^{\prime\prime}\to G$. 
Hence 
we have a 
surjection $\psi=\psi_2\circ\psi_1: G^{\prime\prime}\to G$, 
i.e. $G\simeq G^{\prime\prime}/{\rm Ker}(\psi)$ 
and $N^{\prime\prime}={\rm Ker}(\psi_1)\lhd {\rm Ker}(\psi)$. 
This implies that $[M_G]^{fl}=[M_{G^{\prime\prime}}]^{fl}=[F]$ 
as $G^{\prime\prime}$-lattices via surjection $\psi$. 
It follows from Theorem \ref{thmain7} that 
$L^{\prime\prime}\subset L$, 
i.e. there exists $N\lhd G$ such that $G/N\simeq G^{\prime\prime}$, 
and hence $N\simeq {\rm Ker}(\psi)=1$. 
We conclude that $N^{\prime\prime}=1$. 
\end{proof}

For GAP computations of Example \ref{ex1.42} for $G^{\prime\prime}=N_{5,324}\leq \GL(5,\bZ)$ with CARAT ID (5,100,12), see Example \ref{ex12.1} which needs {\tt MultInvField.gap} 
which is available as in \cite{MultInvField}.\\

\bigskip
\begin{example}[{Computationas of Example \ref{ex1.42}}]\label{ex12.1}~\vspace*{-5mm}\\
\begin{verbatim}
gap> Read("MultInvField.gap");
gap> Read("BCAlgTori.gap");
gap> Position(N5,[5,100,12]);
324
gap> N3g:=List(N3,x->MatGroupZClass(x[1],x[2],x[3],x[4]));;
gap> N5g:=List(N5,x->CaratMatGroupZClass(x[1],x[2],x[3]));;
gap> sd:=AllSubdirectProducts(N3g[1],N5g[324]);
[ <pc group with 5 generators>, <pc group with 4 generators>, 
  <pc group with 4 generators>, <pc group with 4 generators>, 
  <pc group with 3 generators> ]
gap> List(sd,Order);
[ 32, 16, 16, 16, 8 ]
gap> PossibilityOfStablyEquivalentFSubdirectProduct(sd[1]);
[  ]
gap> PossibilityOfStablyEquivalentFSubdirectProduct(sd[2]);
[  ]
gap> PossibilityOfStablyEquivalentFSubdirectProduct(sd[3]);
[  ]
gap> PossibilityOfStablyEquivalentFSubdirectProduct(sd[4]);
[  ]
gap> PossibilityOfStablyEquivalentFSubdirectProduct(sd[5]);
[ [ 0, 0, 0, 0, 0, 1, 0, 0, 1 ] ]
gap> l:=last[1]; 
# possibility for Z[H~/H6]+F=F' in the sense of the equation (8)
[ 0, 0, 0, 0, 0, 1, 0, 0, 1 ]
gap> bp:=StablyEquivalentFCheckPSubdirectProduct(sd[5],LHSlist(l),RHSlist(l));;
gap> SearchPRowBlocks(bp);
rec( bpBlocks := [ [ 1, 2, 3, 4, 5 ], [ 6, 7, 8, 9, 10, 11, 12, 13, 14, 15 ] ]
    , rowBlocks := [ [ 1, 2 ], [ 3, 4, 5, 6, 7 ] ] )
gap> r1:=SearchPFilterRowBlocks(bp,[1..5],[1,2],1);;
gap> Length(r1);
1
gap> r2:=SearchPFilterRowBlocks(bp,[6..15],[3..7],1);;
gap> Length(r2);
1
gap> P:=SearchPMergeRowBlock(r1,r2);
[ [ 1, 0, 0, 0, 0, 0, 0 ], 
  [ 1, 0, 0, 1, 1, -1, -1 ], 
  [ 1, 0, 0, 0, 0, -1, -1 ], 
  [ 0, 0, 0, 0, 0, 1, 0 ], 
  [ 0, 0, 0, 0, 1, 0, 0 ],
  [ 1, 0, -1, 1, 0, -1, -1 ], 
  [ 1, -1, -1, 1, 1, -1, -1 ] ]
gap> StablyEquivalentFCheckMatSubdirectProduct(sd[5],LHSlist(l),RHSlist(l),P);
true

gap> StructureDescription(N3g[1]);
"C2 x C2"
gap> GeneratorsOfGroup(N3g[1]);
[ [ [ 0, 0, 1 ], [ -1, -1, -1 ], [ 1, 0, 0 ] ], 
  [ [ 0, 1, 0 ], [ 1, 0, 0 ], [ -1, -1, -1 ] ] ]
gap> StructureDescription(N5g[324]);
"D8"
gap> GeneratorsOfGroup(N5g[324]);
[ [ [ 0, -1, -1, 0, 1 ], [ -1, 0, 0, 0, -1 ], [ 0, 0, 1, 0, 0 ], 
      [ 0, 0, 0, -1, 0 ], [ 0, 0, 1, 0, -1 ] ], 
  [ [ 1, 1, 1, 1, 0 ], [ 0, -1, 0, 0, 0 ], [ 0, 0, -1, 0, 0 ], 
      [ 0, 0, 0, -1, 0 ], [ 0, -1, -1, -1, 1 ] ] ]
gap> GeneratorsOfGroup(sd[5]);
[ f1, f2, f3 ]
gap> Image(Projection(sd[5],1),sd[5].1);
[ [ 0, 0, 1 ], [ -1, -1, -1 ], [ 1, 0, 0 ] ]
gap> Image(Projection(sd[5],2),sd[5].1);
[ [ 1, 1, 1, 1, 0 ], [ 0, -1, 0, 0, 0 ], [ 0, 0, -1, 0, 0 ], 
  [ 0, 0, 0, -1, 0 ], [ 0, -1, -1, -1, 1 ] ]
gap> Image(Projection(sd[5],1),sd[5].2);
[ [ 0, 1, 0 ], [ 1, 0, 0 ], [ -1, -1, -1 ] ]
gap> Image(Projection(sd[5],2),sd[5].2);
[ [ -1, 0, -1, 0, 0 ], [ 1, 0, 0, 0, 1 ], [ 0, 0, 1, 0, 0 ], 
  [ 0, 0, 0, -1, 0 ], [ 1, 1, 1, 0, 0 ] ]
gap> Image(Projection(sd[5],1),sd[5].3);
[ [ 0, 1, 0 ], [ 1, 0, 0 ], [ -1, -1, -1 ] ]
gap> Image(Projection(sd[5],2),sd[5].3);
[ [ 0, -1, -1, 0, 1 ], [ -1, 0, 0, 0, -1 ], [ 0, 0, 1, 0, 0 ], 
  [ 0, 0, 0, -1, 0 ], [ 0, 0, 1, 0, -1 ] ]
gap> s1:=N3g[1].1;
[ [ 0, 0, 1 ], [ -1, -1, -1 ], [ 1, 0, 0 ] ]
gap> s2:=N3g[1].2;
[ [ 0, 1, 0 ], [ 1, 0, 0 ], [ -1, -1, -1 ] ]
gap> Image(Projection(sd[5],1),sd[5].1)=s1;
true
gap> Image(Projection(sd[5],1),sd[5].2)=s2;
true
gap> Image(Projection(sd[5],1),sd[5].3)=s2;
true
gap> t1:=N5g[324].1;
[ [ 0, -1, -1, 0, 1 ], [ -1, 0, 0, 0, -1 ], [ 0, 0, 1, 0, 0 ], 
  [ 0, 0, 0, -1, 0 ], [ 0, 0, 1, 0, -1 ] ]
gap> t2:=N5g[324].2;
[ [ 1, 1, 1, 1, 0 ], [ 0, -1, 0, 0, 0 ], [ 0, 0, -1, 0, 0 ], 
  [ 0, 0, 0, -1, 0 ], [ 0, -1, -1, -1, 1 ] ]
gap> Centre(N5g[324]);
<matrix group with 1 generator>
gap> GeneratorsOfGroup(last);
[ [ [ 0, 1, 0, 0, -1 ], [ 0, -1, 0, 0, 0 ], [ 0, 0, 1, 0, 0 ], 
      [ 0, 0, 0, 1, 0 ], [ -1, -1, 0, 0, 0 ] ] ]
gap> t3:=last[1];
[ [ 0, 1, 0, 0, -1 ], [ 0, -1, 0, 0, 0 ], [ 0, 0, 1, 0, 0 ], 
  [ 0, 0, 0, 1, 0 ], [ -1, -1, 0, 0, 0 ] ]
gap> Image(Projection(sd[5],2),sd[5].1)=t2;
true
gap> Image(Projection(sd[5],2),sd[5].2)=t1*t3;
true
gap> Image(Projection(sd[5],2),sd[5].3)=t1;
true
gap> Order(t1*t2);
4
gap> t3=t1*t2*t1*t2;
true
gap> Image(Projection(sd[5],2),sd[5].2)=t2*t1*t2;
true
\end{verbatim}
\end{example}


\end{document}